\documentclass[11pt]{article}
\usepackage[utf8x]{inputenc}
\usepackage[margin=1.5in]{geometry}
\usepackage{microtype}
\usepackage[T1]{fontenc}
\usepackage{ae,aecompl}
\usepackage{times}
\usepackage{tikz-cd}
\usepackage{mathrsfs} 
\usepackage{harmony}
\usepackage{verbatim,amsmath,amsthm,amsfonts,amssymb,latexsym,graphicx,mathtools,extpfeil,color,mathabx}
\usepackage{tikz}
\usepackage{epstopdf,pinlabel}
\usepackage[all]{xy}
\usepackage{graphicx}
\usepackage{booktabs}
\usepackage{caption}
\usepackage{tabularx}
\usepackage{subcaption}
\usepackage{cancel}
\usepackage{slashed}
\DeclareMathAlphabet{\mathpzc}{OT1}{pzc}{m}{it}
\usepackage[colorlinks,pagebackref,hypertexnames=false]{hyperref} \usepackage[alphabetic,backrefs,msc-links]{amsrefs}

\usepackage{aliascnt}
\numberwithin{equation}{section}

\DeclareUnicodeCharacter{2212}{\ensuremath{-}}

\usepackage{accents}

\newcommand{\sfR}{{\sf R}}

\newcommand{\bC}{{\bf C}}

\newcommand{\bH}{{\bf H}}

\newcommand{\bN}{{\bf N}}

\newcommand{\bP}{{\bf P}}
\newcommand{\bQ}{{\bf Q}}
\newcommand{\bR}{{\bf R}}

\newcommand{\bZ}{{\bf Z}}

\newcommand{\cH}{\mathcal{H}}

\newcommand{\cO}{\mathcal{O}}

\newcommand{\cQ}{\mathcal{Q}}

\newcommand{\cS}{\mathcal{S}}

\newcommand{\sB}{\mathscr{B}}

\newcommand{\sD}{\mathscr{D}}

\newcommand{\sS}{\mathscr{S}}
\newcommand{\sT}{\mathscr{T}}

\newcommand{\fd}{{\mathfrak d}}

\newcommand{\fj}{{\mathfrak j}}

\newcommand{\fo}{{\mathfrak o}}
\newcommand{\fp}{{\mathfrak p}}

\newcommand{\fs}{{\mathfrak s}}

\newcommand{\fC}{{\mathfrak C}}

\newcommand{\fI}{{\mathfrak I}}

\newcommand{\fX}{{\mathfrak X}}

\newcommand{\Z}{\bZ}
\newcommand{\Q}{\bQ}
\newcommand{\R}{\bR}
\newcommand{\C}{\bC}

\DeclareMathOperator{\ind}{ind}

\renewcommand{\det}{\operatorname{det}}

\renewcommand{\epsilon}{\varepsilon}

\def\({\mathopen{}\left(}
\def\){\right)\mathclose{}}
\def\<{\mathopen{}\left<}
\def\>{\right>\mathclose{}}

\usepackage{multicol, color}

\definecolor{gold}{rgb}{0.85,.66,0}
\definecolor{cherry}{rgb}{0.9,.1,.2}
\definecolor{burgundy}{rgb}{0.8,.2,.2}
\definecolor{orangered}{rgb}{0.85,.3,0}
\definecolor{orange}{rgb}{0.85,.4,0}
\definecolor{olive}{rgb}{.45,.4,0}
\definecolor{lime}{rgb}{.6,.9,0}
\definecolor{green}{rgb}{.2,.7,0}
\definecolor{grey}{rgb}{.4,.4,.2}
\definecolor{brown}{rgb}{.4,.3,.1}

\def\makeautorefname#1#2{\AtBeginDocument{\expandafter\def\csname#1autorefname\endcsname{#2}}}

\newcommand{\mynewtheorem}[2]{
  \newaliascnt{#1}{equation}          
  \newtheorem{#1}[#1]{#2}
  \aliascntresetthe{#1}
  \makeautorefname{#1}{#2}
}
\mynewtheorem{theorem}{Theorem}
\mynewtheorem{prop}{Proposition}
\mynewtheorem{cor}{Corollary}
\mynewtheorem{construction}{Construction}
\mynewtheorem{lemma}{Lemma}
\mynewtheorem{conjecture}{Conjecture}

\numberwithin{substep}{step}
\makeautorefname{step}{Step}
\makeautorefname{substep}{Step}

\numberwithin{subcase}{case}
\makeautorefname{case}{Case}
\makeautorefname{subcase}{case}

\theoremstyle{remark}
\mynewtheorem{remark}{Remark}

\theoremstyle{definition}
\mynewtheorem{definition}{Definition}
\mynewtheorem{example}{Example}
\mynewtheorem{exercise}{Exercise}
\mynewtheorem{convention}{Convention}
\newtheorem*{convention*}{Convention}
\newtheorem*{conventions*}{Conventions}
\mynewtheorem{question}{Question}
        
\makeautorefname{chapter}{Chapter}
\makeautorefname{section}{Section}
\makeautorefname{subsection}{Section}
\makeautorefname{subsubsection}{Section}

\newcommand\hrC{{\widehat { C}}}
\newcommand\fhrC{{\widehat { \mathfrak C}}}

\newcommand\crC{{\widecheck { C}}}
\newcommand\fcrC{{\widecheck {\mathfrak C}}}

\newcommand\brC{{\overline {C}}}
\newcommand\fbrC{{\overline {\mathfrak C}}}
\newcommand\diamd{\blacklozenge}

\newcommand\opam{t}

\title{Unoriented skein exact triangles in\\ equivariant singular instanton Floer theory}
\author{Aliakbar Daemi\thanks{The work of AD was supported by NSF Grant DMS-1812033 and NSF FRG Grant DMS-1952762.} \hspace{1cm} Christopher Scaduto\thanks{The work of CS was supported by NSF Grant DMS-1952762.}}
\date{}

\newcommand{\Addresses}{{
  \bigskip
  \footnotesize
  Aliakbar Daemi, \textsc{Department of Mathematics, Washington University in St. Louis, One Brookings drive, Room 207A,
  St. Louis, MO 63130}\par\nopagebreak
  \textit{E-mail address}: \texttt{adaemi@wustl.edu}
  \vspace{.4cm}

Christopher Scaduto, \textsc{Department of Mathematics, University of Miami, 1365 Memorial Dr 515, Coral Gables, FL 33124}\par\nopagebreak
  \textit{E-mail address}: \texttt{cscaduto@miami.edu}
}}

\begin{document}
\maketitle

\begin{abstract}

Equivariant singular instanton Floer theory is a framework that associates to a knot in an integer homology $3$-sphere a suite of homological invariants that are derived from circle-equivariant Morse--Floer theory of a Chern--Simons functional for framed singular $SU(2)$-connections. These invariants generalize the instanton knot homology of Kronheimer and Mrowka. In the present work, these constructions are extended from knots to links with non-zero determinant, and several unoriented skein exact triangles are proved in this setting. As a particular case, a categorification of the behavior of the Murasugi signature for links under unoriented skein relations is established. In addition to the exact triangles, Fr\o yshov-type invariants for links are defined, and several computations using the exact triangles are carried out. The computations suggest a relationship between Heegaard Floer $L$-space knots and those knots whose instanton-theoretic categorification of the knot signature is supported in even gradings.\\

A main technical contribution of this work is the construction of maps for certain cobordisms between links on which obstructed reducible singular instantons are present. These constructions are inspired by recent work of the first author and Miller Eismeier in the setting of non-singular instanton theory for rational homology $3$-spheres. 

\end{abstract}

\newpage

\hypersetup{linkcolor=black}
\setcounter{tocdepth}{2}
\tableofcontents

\newpage


\section{Introduction}

In this work we construct several unoriented skein exact triangles in the setting of equivariant singular instanton Floer theory, as developed by the authors in \cite{DS1,DS2}, and which builds upon work of Kronheimer and Mrowka \cite{KM:YAFT,KM:unknot}. To motivate our results, we begin by explaining a special case which already exhibits several new features, and which categorifies the behavior of the knot signature under unoriented skein relations.

\subsection{Exact triangles for a categorification of the knot signature}

In recent decades, the study of knots and links has flourished from the introduction of various knot homology theories. Such theories typically come either in the form of a Floer homology, rooted in gauge theory or symplectic geometry, or from constructions in representation theory, which have a combinatorial flavor and are amenable to computations. For many knot homology theories $\mathcal{H}$, there is an associated long exact sequence
\begin{equation}\label{eq:formalexacttriangle}
	\begin{tikzcd}[]
\cdots   \arrow[r] & \mathcal{H}(L) \arrow[r] & \mathcal{H}(L') \arrow[r] & \mathcal{H}(L'') \arrow[r] & \mathcal{H}(L) \arrow[r] & \cdots
\end{tikzcd}
\end{equation}
whenever three links $L,L',L''$ in the $3$-sphere only differ in a small $3$-ball as depicted in Figure \ref{fig:skein}. We say $L,L',L''$ is an {\emph{unoriented skein triple}}, and $\mathcal{H}$ satisfies an {\emph{unoriented skein exact triangle}} when we have an exact sequence \eqref{eq:formalexacttriangle} for each such triple. Examples of $\mathcal{H}$ that satisfy an unoriented skein exact triangle include Khovanov homology \cite{khovanov}, Kronheimer and Mrowka's instanton homology \cite{KM:unknot}, and 3-manifold Floer homologies applied to double branched covers of links \cite{os:branched,KMOS,scadutothesis,daemi-pfh}.

Any unoriented skein triple $L,L',L''$ has a cyclic symmetry, in the sense that $L',L'',L$ and $L'',L,L'$ are also unoriented skein triples. This amounts to the fact that if one performs a certain rotation on the skein $3$-ball, then the pictures in Figure \ref{fig:skein} are cyclically permuted.

In recent work \cite{DS1,DS2}, the authors introduced a suite of homological invariants for knots $K$ derived from instanton Floer theory. One of these invariants,  the {\emph{irreducible instanton homology}} $I(K)$, is a $\Z/4$-graded abelian group, and is essentially a version of the orbifold instanton homology of Collin--Steer \cite{collin-steer}. It stands apart from all other known knot homology theories as having Euler characteristic one-half the knot signature:
\begin{equation}\label{eq:irredhomknoteuler}
	\chi\left( I(K) \right) = \frac{1}{2}\sigma(K).
\end{equation}
Under favorable circumstances, $I(K)$ is generated at the chain-level by the conjugacy classes of {\emph{irreducible}} $SU(2)$ representations of $\pi_1(S^3\setminus K)$ which are traceless around meridians; the differential counts singular instantons on $\R\times (S^3\setminus K)$. In general, this description is only an approximation, as perturbations may be required to achieve transversality.

\begin{figure}[t]
    \centering
    \includegraphics[scale=0.75]{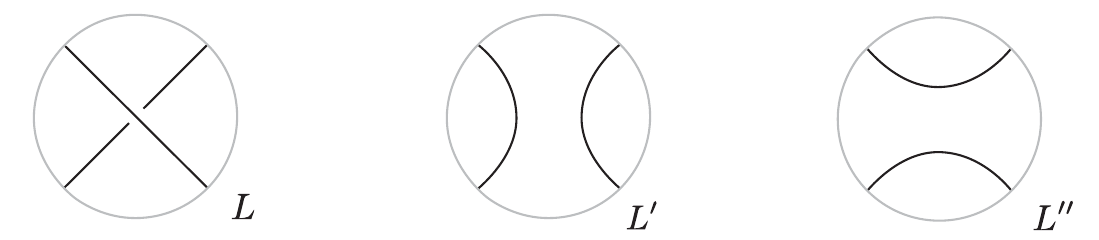}
    \caption{{\small{Local pictures for three links forming an unoriented skein triple.}}}
    \label{fig:skein}
\end{figure}

A natural question arises as to whether $I(K)$ satisfies an unoriented skein exact triangle. The first obstacle is that $I(K)$ has been defined in \cite{DS1} only for knots, but every unoriented skein triple has at least one link with multiple components. Fortunately, the construction of $I(K)$ generalizes in a straightforward way to links with {\emph{non-zero determinant}}, and this is the setting we develop. The authors expect that the story here adapts to arbitrary links, at the expense of some additional technicalities; see Subsection \ref{subsec:concludingremarks}.
 
Formula \eqref{eq:irredhomknoteuler} generalizes as follows. Recall that to a link $L$ with an orientation $o$, there is an associated signature $\sigma(L,o)\in\Z$. Let $-o$ denote the orientation obtained from $o$ by reversing the orientation of every component. Then $\fo=\{o,-o\}$ is called a {\emph{quasi-orientation}} of $L$. Write $\mathcal{Q}(L)$ for the set of quasi-orientations of $L$; note $|\mathcal{Q}(L)|=2^{|L|-1}$ where $|L|$ is the number of components. The signature $\sigma(L,o)$ only depends on $\fo$, and we often write $\sigma(L,\fo)$. The {\emph{Murasugi signature}} \cite{murasugi} of a link $L$ is the average of all $\sigma(L,\fo)$:
\[
	\xi(L) = \frac{1}{2^{|L|-1}} \sum_{\fo\in \cQ(L)} \sigma(L,\fo)	.
\]
For any link, $\xi(L)$ is an integer. Now, given a link $L$ with non-zero determinant, the irreducible instanton homology $I(L)$ is again a $\Z/4$-graded abelian group, and we will show that its Euler characteristic is a multiple of the Murasugi signature:
\begin{equation}\label{eq:irredhomlinkeuler}
	\chi\left( I(L) \right) = 2^{|L|-2} \cdot \xi(L).
\end{equation}

We now describe the behavior of the Murasugi signature for unoriented skein triples $(L,L',L'')$, implicit in \cite{murasugi}. Suppose that we can orient $L$ and $L'$ so that in the 3-ball we have Figure \ref{fig:orientskein}, and the components in $L$, $L'$ which do not intersect the $3$-ball are oriented the same way. Note that this can be done exactly when the difference between the number connected components of $L$ and $L'$ is one. Call the orientations $o$ and $o'$, respectively. Define
\[
  \epsilon(L,L') := \sigma(L,o) -  \sigma(L',o') \in \{ -1, 0 , 1 \}.
\]
This does not depend on the orientations $o$ and $o'$, as long as they are compatible in the way just described. If $L$ and $L'$ have non-zero determinants, then $\epsilon(L,L')=\pm 1$. By the symmetry of the skein triple, we may define $\epsilon$ for any two consecutive links for which the the difference between the number connected components is one. Note $\epsilon(L',L)=-\epsilon(L,L')$.

Now suppose $(L,L',L'')$ is an unoriented skein triple in which $L$ has one more component than $L'$ and $L''$. Then $\epsilon(L,L')$ and $\epsilon(L'',L)$ are defined, and we set
\begin{equation}\label{eq:deltadef}
	\delta = \frac{1}{2}(\epsilon(L,L') - \epsilon(L'',L))
\end{equation}
If the links have non-zero determinants, then $\delta\in \{-1,0,1\}$. In this case, the relation the Murasugi signatures satisfy, written in terms of Euler characteristics via \eqref{eq:irredhomlinkeuler}, is:
\begin{equation}\label{eq:eulercharskeinrel}
	\chi\left( I(L) \right) = \chi\left( I(L') \right)+ \chi\left( I(L'') \right) + \delta\cdot 2^{|L|-2}
\end{equation}
From this relation, another obstacle to an unoriented skein exact triangle for $I(L)$ appears. If such an exact triangle exists, and each of the maps in the triangle are homogeneously graded, there is an induced relation for Euler characteristics. However, the only version of relation \eqref{eq:eulercharskeinrel} that can come from an exact triangle is the one with $\delta=0$. In the cases $\delta\in \{1,-1\}$, some modifications must be made.

\begin{figure}[t]
    \centering
    \includegraphics[scale=0.75]{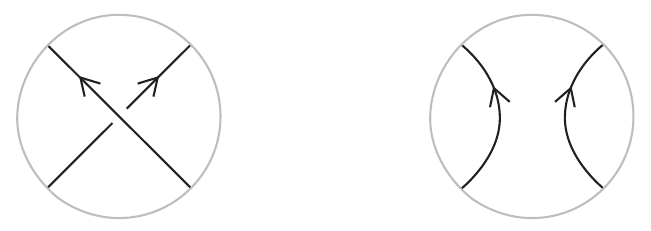}
    \caption{{\small{An oriented resolution from $L$ to $L'$.}}}
    \label{fig:orientskein}
\end{figure}

For a link $L$ with non-zero determinant, we shall define variations of $I(L)$, the ($\pm$) ``suspensions'', denoted $I_+(L)$ and $I_{-}(L)$. Each is related to $I(L)$ through an exact triangle:
            \begin{equation}
                \begin{tikzcd}[column sep=1.7ex, row sep=5ex, fill=none, /tikz/baseline=-10pt]
            & I_{+}(L)  \arrow[dr,"\text{odd deg}"]  & \\
             I(L)  \arrow[ur] & &  \Z^{2^{|L|-1}}  \arrow[ll,"(\delta_2)_\ast"] 
            \end{tikzcd}
\hspace{1.05cm}
                            \begin{tikzcd}[column sep=1.7ex, row sep=5ex, fill=none, /tikz/baseline=-10pt]
            & I_{ -}(L)  \arrow[dr]  & \\
             \Z^{2^{|L|-1}} \arrow[ur] & &  I(L) \arrow[ll,"(\delta_1)_\ast","\text{odd deg}"']
            \end{tikzcd} \label{eq:irrhomdeftriangles}
        \end{equation}
        The maps are of even degree except where indicated. The group $\Z^{2^{|L|-1}}$ is freely generated, in even gradings, by quasi-orientations; equivalently, its generators are the conjugacy classes of abelian traceless $SU(2)$ representations of $\pi_1(S^3\setminus L)$. The map $(\delta_1)_\ast$ is induced by counting singular instantons on $\R\times (S^3\setminus L)$ which are irreducible at $-\infty$ and reducible (abelian) at $+\infty$; the map $(\delta_2)_\ast$ is similar, with $-\infty$ and $+\infty$ swapped.
        
        Returning to \eqref{eq:eulercharskeinrel}, and noting that our assumptions imply $|L|=1+|L'|=1+|L''|$, we observe that an exact triangle {\emph{is possible}}, from the viewpoint of Euler characteristics, if $I(L')$ is replaced by $I_{+}(L')$ in the case $\delta=-1$, and if $I(L'')$ is replaced by $I_{-}(L'')$ in the case $\delta=1$. We show that such exact triangles indeed exist.
        
        \begin{theorem}\label{thm:irredexacttrianglesintro}
        	Let $L$, $L'$, $L''$ be an unoriented skein triple as in Figure \ref{fig:skein}. Suppose each link has non-zero determinant. Without loss of generality, assume $L$ has one more component than $L'$,  $L''$. There are three cases depending on the possible values of $\delta$. In each case we have an exact triangle as described in Figure \ref{fig:irredexacttrianglesintro}.
        \end{theorem}

    \begin{figure}[t]
  \centering
            \renewcommand{\arraystretch}{1.5}
            \begin{tabular}{ c | c | c }
        \textbf{Case I} & \textbf{Case II} & \textbf{Case III}\\ 
        \hline
        $\delta=0$ & $ \delta=-1$ & $\delta=+1$\\
        \hline
        $\epsilon(L,L') = + 1$ & $\epsilon(L,L') = -1$ & $\epsilon(L,L') = + 1$\\
        $\epsilon(L'',L) = + 1$ & $\epsilon(L'',L) = + 1$ & $\epsilon(L'',L) = -1$ \\
        \hline
        \begin{minipage}{4.2 cm}{
            \begin{equation*}
                \begin{tikzcd}[column sep=2ex, row sep=5ex, fill=none, /tikz/baseline=-10pt]
            &  I(L)  \arrow[dr]  & \\
             I(L'')    \arrow[ur] & &   I(L')  \arrow[ll] \\[-3mm]
            \end{tikzcd}
        \end{equation*}
        }\end{minipage} & 
        \begin{minipage}{4.5 cm}{
            \begin{equation*}
                \begin{tikzcd}[column sep=2ex, row sep=5ex, fill=none, /tikz/baseline=-10pt]
            &  I(L)  \arrow[dr]  & \\
           I(L'')    \arrow[ur] & & I_+(L')  \arrow[ll] \\[-3mm]
            \end{tikzcd}
        \end{equation*}
        }\end{minipage} & 
        \begin{minipage}{4 cm}{
            \begin{equation*}
                \begin{tikzcd}[column sep=2ex, row sep=5ex, fill=none, /tikz/baseline=-10pt]
            &I(L)  \arrow[dr]  & \\
            I_{-}(L'')   \arrow[ur] & & I(L')  \arrow[ll] \\[-3mm]
            \end{tikzcd}
        \end{equation*}
        }\end{minipage}  \\
        \hline 
    \end{tabular}
   \caption{{\small{The exact triangles of Theorem \ref{thm:irredexacttrianglesintro}.}}}
    \label{fig:irredexacttrianglesintro}
\end{figure}

The proof of Theorem \ref{thm:irredexacttrianglesintro}, as well as the more general results below, take as an initial template the proof of Kronheimer and Mrowka's exact triangle for $I^\natural(K)$ from \cite{KM:unknot}, which utilizes maps induced by metric families on cobordisms between the links. In contrast to the case of $I^\natural(K)$, however, we must deal with the presence of reducible connections. A singular connection in our setting is {\emph{reducible}} if it is compatible with a $U(1)$-reduction of the bundle, which behaves in a prescribed way along the singular locus.

A reducible singular instanton is {\emph{obstructed}} if it is not cut out transversely by the instanton equation. In our setup, reducible instantons with index less than or equal to $-3$ are necessarily obstructed, even after perturbations. Some of the cobordisms in the proof of the above exact triangles have obstructed reducibles. To deal with these we must use {\emph{obstructed gluing theory}} to construct the relevant chain-level maps.

The use of obstructed gluing theory to define maps for cobordisms with obstructed reducibles is a main technical contribution of the present work, and is developed in Section \ref{sec:obstructed}. A consequence is an extension of the functoriality of the equivariant singular instanton Floer theory package developed in \cite{DS1,DS2}, which constructed maps for cobordisms without any obstructed reducibles. The construction of cobordism maps in the presence of obstructed reducibles, as well as some of the algebra that appears below, is inspired by recent work of the first author and Miller Eismeier in the setting of non-singular instanton homology \cite{DME:QHS}. See Subsection \ref{subsec:concludingremarks} for further discussion.

\vspace{0.3cm}

\begin{remark}
	It is natural to expect that \eqref{eq:irredhomlinkeuler} is equal to an invariant defined by Benard and Conway \cite{benard-conway}, when the parameters $\alpha_i$ in that reference are set to $\pi/2$, corresponding to the traceless meridional holonomy condition. This is verified by \cite{liu-saveliev} for two-component links with non-zero linking number, which builds on work in \cite{benard-conway}. $\diamd$
\end{remark}

\subsection{Exact triangles in equivariant singular instanton theory}

We now turn to the more general setting, starting from the viewpoint established in \cite{DS1}. In this case we consider a {\emph{based}} knot $K$, that is, a knot with a chosen basepoint, which is often suppressed from notation. The main invariant constructed in \cite{DS1} is called an {\emph{$\cS$-complex}}, denoted $\widetilde C(K )$. As a $\Z/4$-graded group, there is a decomposition
\begin{equation}\label{eq:scomplexdecompintro}
	\widetilde C_\ast(K ) = C_\ast(K) \oplus C_{\ast-1}(K) \oplus \Z_{(0)}.
\end{equation}
The differential of $\widetilde C_\ast(K )$ anti-commutes with $\chi$, the map which sends $C_\ast(K)$ isomorphically to $C_{\ast-1}(K)$ and is otherwise zero. The $\cS$-complex $\widetilde C(K)$ is an $S^1$-equivariant Morse--Floer chain complex associated to the Chern--Simons functional for $SU(2)$-connections on $S^3\setminus K$ with limiting traceless meridional holonomy, and which have limiting holonomy exactly $\mathbf{i}\in SU(2)$ for meridians encircling the basepoint of $K$. The $\cS$-complex structure of $\widetilde C(K)$ is essentially that of a differential graded module over the algebra $H_\ast(S^1)\cong \Lambda^\ast(\chi)$, where the algebra structure is induced by the group structure of $S^1$. It is shown in \cite{DS1} that $\widetilde C(K)$ is an invariant of the based knot $K$ up to chain homotopy of $\cS$-complexes.

In \cite{DS1}, the authors showed that the total homology of $\widetilde C(K)$ is naturally isomorphic to Kronheimer and Mrowka's instanton knot homology $I^\natural(K)$ from \cite{KM:unknot}:
\begin{equation}\label{eq:introisotoinatural}
	H_\ast( \widetilde C(K) ) \cong I^\natural(K).
\end{equation}
Thus $\widetilde C(K)$ is an equivariant upgrade of $I^\natural(K)$. The invariant $I^\natural(K)$ has played an important role in low-dimensional topology in recent years. For example, Kronheimer and Mrowka used it to show that Khovanov homology detects the unknot \cite{KM:unknot}. 

The $\cS$-complex $\widetilde C(K)$, with varying additional layers of structure, has also been used to address problems in low-dimensional topology in ways beyond its predecessor $I^\natural(K)$. These include applications to the $4$-dimensional clasp number of knots \cite{DS2} and other aspects of knot concordance \cite{DISST:special-cycles}, existence results for non-abelian $SU(2)$-representations of fundamental groups of knot and surface complements \cite{DS2,imori,DISST:special-cycles}, as well as results on the homology cobordism of $3$-manifolds \cite{DISST:special-cycles}.

The construction of $\widetilde C(K)$ adapts in a straightforward manner to based links $L$ with non-zero determinant. It is simpler, in this generality, to work with $\Z/2$-graded $\cS$-complexes, although $\Z/4$-gradings can be defined (see Subsection \ref{subsec:gradings}). For such a link, the $\Z$ in \eqref{eq:scomplexdecompintro} is replaced by $\Z^{2^{|L|-1}}$. This summand is generated by the abelian (reducible) traceless flat $SU(2)$ connections on $S^3\setminus L$, which are in bijection with quasi-orientations of $L$.

At this point we mention that our notion of $\cS$-complex in the present work is more general than the one considered in \cite{DS1,DS2}, which required that the last summand in \eqref{eq:scomplexdecompintro}, which is more invariantly described as $\text{ker}(\chi)/\text{im}(\chi)$, is free of rank 1.

Everything discussed thus far is in fact valid for links with non-zero determinant in an arbitrary integer homology 3-sphere. We now state the main result of this work.\\

\vspace{0.3cm}

\begin{theorem}\label{thm:exacttriangles}
  Let $L$, $L'$, $L''$ be based links in an integer homology 3-sphere which form a skein triple as in Figure \ref{fig:skein}. Suppose each link has non-zero determinant. Write $\widetilde C$, $\widetilde C'$, $\widetilde C''$ for their respective singular instanton $\cS$-complexes. Without loss of generality, assume $L$ has one more component than $L'$,  $L''$. There are three cases depending on the possible values of $\delta$. In each case we have an exact triangle of $\cS$-complexes as described in Figure \ref{fig:exacttriangles}.
\end{theorem}

\begin{figure}[t]
    \centering
    \renewcommand{\arraystretch}{1.5}
    \begin{tabular}{ c | c | c }
        \textbf{Case I} & \textbf{Case II} & \textbf{Case III}\\ 
        \hline
                $\delta=0$ & $ \delta=-1$ & $\delta=+1$\\
        \hline
        $\epsilon(L,L') = + 1$ & $\epsilon(L,L') = -1$ & $\epsilon(L,L') = + 1$\\
        $\epsilon(L'',L) = + 1$ & $\epsilon(L'',L) = + 1$ & $\epsilon(L'',L) = -1$ \\
        \hline
        \begin{minipage}{4 cm}{
            \begin{equation*}
                \begin{tikzcd}[column sep=2ex, row sep=5ex, fill=none, /tikz/baseline=-10pt]
            & \widetilde C  \arrow[dr]  & \\
            \widetilde C''    \arrow[ur] & &  \widetilde  C'  \arrow[ll]
            \end{tikzcd}
        \end{equation*}
        }\end{minipage} & 
        \begin{minipage}{4 cm}{
            \begin{equation*}
                \begin{tikzcd}[column sep=2ex, row sep=5ex, fill=none, /tikz/baseline=-10pt]
            & \widetilde C  \arrow[dr]  & \\
            \widetilde C''    \arrow[ur] & & \Sigma \widetilde  C'  \arrow[ll]
            \end{tikzcd}
        \end{equation*}
        }\end{minipage} & 
        \begin{minipage}{4 cm}{
            \begin{equation*}
                \begin{tikzcd}[column sep=2ex, row sep=5ex, fill=none, /tikz/baseline=-10pt]
            & \widetilde C  \arrow[dr]  & \\
            \Sigma^{-1} \widetilde C''    \arrow[ur] & & \widetilde  C'  \arrow[ll]
            \end{tikzcd}
        \end{equation*}
        }\end{minipage}  \\
        & & \\
        \hline 
    \end{tabular}
    \caption{{\small{The exact triangles of Theorem \ref{thm:exacttriangles}.}}}
    \label{fig:exacttriangles}
\end{figure}

Two remarks are in order. First, an ``exact triangle of $\cS$-complexes'' is a variant of the notion of exact triangle in the setting of $A_\infty$-modules as described for example by Seidel \cite[Lemma 3.7]{seidel}; see Subsection \ref{subsec:scomplextriangles} for details. Any one $\cS$-complex in an exact triangle is a mapping cone, in the category of $\cS$-complexes, of the other two $\cS$-complexes.

Second, we explain the notation $\Sigma^{\pm 1} \widetilde C$. Given any $\cS$-complex $\widetilde C$, we can algebraically form the $(\pm)$ ``suspension'' $\Sigma^{\pm 1} \widetilde C$ which is a new $\cS$-complex. The homotopy-type of the suspension is simply the tensor product of $\cS$-complexes
\[
	\Sigma^{\pm 1} \widetilde C \simeq \widetilde C \otimes \widetilde \cO(\pm 1)
\]
where $\widetilde\cO(\pm 1)$ is a standard atomic $\cS$-complex. In fact, $\widetilde \cO(\pm 1)$ is isomorphic, up to a coefficient base change, to the $\cS$-complex of the trefoil $T_{\pm 2,3}$. In practice, we work with a smaller model for the suspended $\cS$-complex; see Subsection \ref{subsec:suspensiondefn}.

The algebraic formalism of suspension arises in the following way. When there is a surface cobordism of non-zero determinant links $S:L\to L'$ which has only unobstructed reducibles, we obtain a morphism $\widetilde C(L)\to \widetilde C(L')$, following \cite{DS1,DS2}. However, in the presence of obstructed reducibles, which occur for example on some cobordisms in the exact triangles, this is not the case. Suppose the obstructed reducibles on $S$ have index $-3$, which is the mildest deviation from the unobstructed case, and is sufficient for all considerations in the present work. Then there is associated to $S$ a morphism of $\cS$-complexes
\begin{equation}\label{eq:morphismsuspensionintro}
	\widetilde C(L) \longrightarrow \Sigma\widetilde C(L')
\end{equation}
By applying negative suspension, this data may alternatively be viewed as a morphism $\Sigma^{-1}\widetilde C(L) \to \widetilde C(L')$. Our use of obstructed gluing theory, mentioned earlier, is clarified as follows: the main technical contribution of the present work is the association of a morphism \eqref{eq:morphismsuspensionintro} to cobordisms with obstructed reducibles of index $-3$. This is taken up in Section \ref{sec:obstructed}. Several variations of this construction appear in the proof of Theorem \ref{thm:exacttriangles}.

Having parsed some of the essential aspects of Theorem \ref{thm:exacttriangles}, we now turn to its implications. First, Theorem \ref{thm:exacttriangles} implies Theorem \ref{thm:irredexacttrianglesintro}. Indeed, the irreducible instanton homology $I(L)$ can be extracted from the $\cS$-complex $\widetilde C(L)$; in the decomposition \eqref{eq:scomplexdecompintro} for a knot, $C(K)$ is the chain complex that computes $I(K)$. More invariantly, given any $\cS$-complex, the irreducible homology is defined to be $H_{\ast+1}(\text{im}\chi)$. Now the fact that Theorem \ref{thm:exacttriangles} implies Theorem \ref{thm:irredexacttrianglesintro} is just a matter of algebra; see Subsection \ref{subsec:scomplextriangles}. We note that $I_\pm(L)$ are by definition the irreducible homology groups of the suspensions $\Sigma^{\pm 1}\widetilde C(L)$.

Another invariant of an $\cS$-complex $\widetilde C$ is its homology $H_\ast(\widetilde C)$. As already mentioned in \eqref{eq:introisotoinatural}, it was shown in \cite{DS1} that the homology of $\widetilde C(K)$ is naturally isomorphic to Kronheimer and Mrowka's $I^\natural(K)$. The proof adapts to show that this is also true for based links with non-zero determinant. An important property of the suspension operation on $\cS$-complexes is that it leaves homology invariant: 
\[
	H_\ast( \Sigma^{\pm 1} \widetilde C)\cong H_\ast (\widetilde C).
\]
For this reason, the differences in the three exact triangles of Theorem \ref{thm:exacttriangles} disappear after taking homology, and we recover Kronheimer and Mrowka's exact triangle of \cite{KM:unknot}:
\begin{equation}\label{eq:inaturalexacttriangle}
	\begin{tikzcd}[]
\cdots  \arrow[r] &  I^\natural(L) \arrow[r] & I^\natural(L') \arrow[r] & I^\natural(L'') \arrow[r] & I^\natural(L)  \arrow[r] & \cdots
\end{tikzcd}
\end{equation}
As $I^\natural(L)$ is perhaps a more familiar invariant to some readers, we mention in passing that the $\cS$-complex structure of $\widetilde C(L)$ implies the existence of a spectral sequence
\begin{equation}
	E_2 = I_\ast(L)\oplus I_{\ast-1}(L)\oplus \Z^{2^{|L|-1}} \quad \rightrightarrows  \quad I^\natural(L) \label{eq:ssirrtonatintro}
\end{equation} 
for any link $L$ of non-zero determinant, which converges by the $E_4$-page.

Another set of invariants one can obtain from an $\cS$-complex is its {\emph{equivariant homology groups}}. This construction, as applied to the singular instanton $\cS$-complex of a knot, was studied in \cite{DS1}. The construction may be adapted to the $\cS$-complex of a non-zero determinant link $L$ (see Section \ref{subsec:equivhomgrps}), and we obtain a triad of $\Z[x]$-modules 
\[
	\widehat{I}(L), \qquad \widecheck{I}(L), \qquad \overline{I}(L).
\]
(Our notation is inspired by that of the formally similar equivariant homology groups in Seiberg--Witten monopole Floer homology \cite{km:monopole}.) An important property of suspensions of $\cS$-complexes is that the equivariant homology groups are invariant under suspensions (see Proposition \ref{prop:susequivar}). Thus some algebra combined with Theorem \ref{thm:exacttriangles} implies:

\begin{theorem}\label{thm:equivexacttriangle}
    Let $L$, $L'$, $L''$ be based links in an integer homology 3-sphere $Y$ which form a skein triple as in Figure \ref{fig:skein}, and suppose each has non-zero determinant. Then we have an exact triangle of equivariant singular instanton Floer groups as $\Z[x]$-modules:
\begin{equation}\label{eq:equivexacttriangle}
	\begin{tikzcd}[column sep=1ex, row sep=8ex, fill=none, /tikz/baseline=-10pt]
& \widehat{I}(Y,L) \arrow[dr] & \\
\widehat{I}(Y,L'') \arrow[ur] & & \widehat{I}(Y,L') \arrow[ll] 
\end{tikzcd}
\end{equation}
A similar statement holds for the equivariant homology groups $\widecheck{I}$ and $\overline{I}$, and the maps in \eqref{eq:equivexacttriangle} are compatible with the ones that relate $\widehat{I}$, $\widecheck{I}$, $\overline{I}$.
\end{theorem}

In \cite{DS1}, the $\cS$-complex for a based knot is defined using a local coefficient system $\Delta$, following a construction of Kronheimer and Mrowka \cite{KM:YAFT}. The result is an $\cS$-complex $\widetilde C(K;\Delta)$ over the coefficient ring $\sS=\Z[T^{\pm 1}]$. Roughly, the variable $T$ encodes information about holonomies of the knot in the longitudinal direction. This construction easily extends to the case of an $\cS$-complex for a non-zero determinant link.

\begin{theorem}\label{thm:localcoeffintro}
	There are exact triangles just as in Theorem \ref{thm:exacttriangles}, but where each $\cS$-complex has the local coefficient system $\Delta$ over the ring $\sS=\Z[T^{\pm 1}]$.
\end{theorem}

\noindent For a more general statement, see Theorem \ref{thm:localcoeffgen}.

The $\cS$-complex of a link can also be equipped with a {\emph{Chern--Simons filtration}}. In the terminology of \cite{DS1}, this gives rise to an $I$-graded $\cS$-complex, or more precisely an {\emph{enriched complex}}. This structure is used to provide several applications in \cite{DS2, DISST:special-cycles}. The exact triangles above are compatible with the Chern--Simons filtration in a reasonable way. This topic will be studied in future work.

\subsection{Exact triangles involving non-trivial bundles}

Another variation of Theorem \ref{thm:exacttriangles} is obtained by allowing different choices of bundles. Let $L$ be a link in an oriented, connected $3$-manifold $Y$, which is not necessarily an integer homology $3$-sphere. In every case discussed thus far, the bundle over $Y\setminus L$ has been the trivial $SU(2)$-bundle which extends (trivially) to all of $Y$. Following \cite{KM:unknot}, we represent a non-trivial $U(2)$-bundle over $Y\setminus L$ by an unoriented embedded 1-manifold $\omega\subset Y$ with $\partial \omega = \omega\cap L$. The Poincar\'{e} dual of $\omega$ is a representative for the second Stiefel--Whitney class of the bundle. More precisely, $\omega$ corresponds to an $SO(3)=PU(2)$-bundle; an orientation of $\omega$ determines a lift to $U(2)$.

We say that $(Y,L,\omega)$ is an {\emph{admissible link}} if there exists a closed oriented surface $\Sigma\subset Y$ transverse to $L$ and $\omega$ such that either $\Sigma$ is disjoint from $L$ and has odd intersection with $\omega$, or  $\Sigma$ has odd intersection with $L$. The admissibility condition guarantees that there are no {\emph{reducible}} critical points for the Chern--Simons functional. To an admissible link $(Y,L,\omega)$ with basepoint, we defined in \cite[Section 8.1]{DS1} a relatively $\Z/4$-graded $\cS$-complex
\[
    \widetilde C^{\omega}(Y,L)
\]
which is the mapping cone of the $v$-map acting on Kronheimer and Mrowka's instanton chain complex $C^\omega(Y,L)$, with homology $I^\omega(Y,L)$, constructed in \cite{KM:YAFT,KM:unknot}.

The following is a variation of our main theorem, Theorem \ref{thm:exacttriangles}, where exactly one of the links is equipped with a non-trivial bundle.

\begin{theorem}\label{thm:exacttriangles-withbundles}
  Let $L$, $L'$, $L''$ be based links in an integer homology 3-sphere $Y$ that form an unoriented skein triple. Assume $L$ has one more component than $L'$,  $L''$, and that $L',L''$ have ${\rm{det}}\neq 0$. Let $\omega\subset Y$ be an arc whose endpoints lie on the components of $L$ that intersect the skein 3-ball. Then $(Y,L, \omega)$ is admissible. There are two possible values of
  \begin{equation}\label{eq:epsilonavg}
  	\frac{1}{2}(\epsilon(L,L') +\epsilon(L'',L))\in \{0,1\}.
  \end{equation}
In each case we have an exact triangle of $\cS$-complexes as displayed in Figure \ref{fig:exacttriangles-withbundles}.
\end{theorem}

\begin{figure}[t]
    \centering
    \renewcommand{\arraystretch}{1.5}
    \begin{tabular}{ c | c }
        \textbf{Case A} & \textbf{Case B} \\
        \hline & \\[-0.6cm]
                $\frac{1}{2}(\epsilon(L,L') +\epsilon(L'',L))=0$ &   $\frac{1}{2}(\epsilon(L,L') +\epsilon(L'',L))=1$  \\[2mm]
        \hline & \\[-7.5mm]
        \begin{minipage}{6.5 cm}{
            \begin{equation*}
                \begin{tikzcd}[column sep=2ex, row sep=5ex, fill=none, /tikz/baseline=-10pt]
            &  \widetilde C^{\omega}(Y,L)  \arrow[dr]  & \\
           \widetilde C(Y,L'')    \arrow[ur] & & \widetilde  C(Y,L')   \arrow[ll] \\[-4.5mm]
            \end{tikzcd}
        \end{equation*}
        }\end{minipage} & 
        \begin{minipage}{6.5 cm}{
            \begin{equation*}
                \begin{tikzcd}[column sep=2ex, row sep=5ex, fill=none, /tikz/baseline=-10pt]
            &  \widetilde C^{\omega}(Y,L) \arrow[dr]  & \\
           	\Sigma \widetilde C(Y,L'')  \arrow[ur] & &  \widetilde  C(Y,L')  \arrow[ll] \\[-4.5mm]
            \end{tikzcd}
        \end{equation*}
        }\end{minipage} \\
        \hline 
    \end{tabular}
    \caption{{\small{The exact triangles of Theorem \ref{thm:exacttriangles-withbundles}.}}}
    \label{fig:exacttriangles-withbundles}
\end{figure}

The exact triangle of Case B may appear asymmetric regarding the placement of the suspension operation $\Sigma$. However, we may apply the inverse suspension $\Sigma^{-1}$ operation to the whole triangle, to obtain an exact triangle of the form
            \begin{equation*}
                \begin{tikzcd}[column sep=2ex, row sep=5ex, fill=none, /tikz/baseline=-10pt]
            &  \widetilde C^{\omega}(Y,L) \arrow[dr]  & \\
           	\widetilde C(Y,L'')  \arrow[ur] & &  \Sigma^{-1}\widetilde  C(Y,L')  \arrow[ll]
            \end{tikzcd}
        \end{equation*}
Indeed, $\Sigma^{-1}\Sigma \widetilde C(Y,L'')\simeq   \widetilde C(Y,L'')$, and $\Sigma\widetilde C^{\omega}(Y,L)=\widetilde C^{\omega}(Y,L)$. The latter follows from the definition of suspension and the fact that $\widetilde C^{\omega}(Y,L)$ has no reducibles. 

From Theorem \ref{thm:exacttriangles-withbundles} there follow two exact triangles for the associated irreducible homology groups, see Subsection \ref{subsec:euler}. Taking the Euler characteristics of these triangles reveals a narrative similar to that which led up to Theorem \ref{thm:irredexacttrianglesintro}.

For notational simplicity assume $Y$ is the 3-sphere. A key difference between the triangles in Theorem \ref{thm:exacttriangles} and those in Theorem \ref{thm:exacttriangles-withbundles} is that the bottom horizontal map has degree $1$ (mod $2$) in the former cases, whereas it has degree $0$ (mod $2$) in the latter cases. This leads to the relation, in the case of the triangles of Theorem \ref{thm:exacttriangles-withbundles}, given by
\begin{equation}\label{eq:eulercharskeinrel-withbundle}
	\chi\left( I^\omega(L) \right) = \pm \left( \chi\left( I(L') \right) - \chi\left( I(L'') \right)  - \widehat{\delta}\cdot 2^{|L|-2} \right)
\end{equation}

\noindent where $\widehat{\delta}$ is the expression \eqref{eq:epsilonavg}. Properties of the Murasugi signature dictate that the right side of \eqref{eq:eulercharskeinrel-withbundle} is a multiple of the linking number between the two components of $L$ that intersect the skein 3-ball. This is the content of the first part of the following result.

\vspace{0.2cm}

\begin{theorem}\label{thm:eulerchar-withbundles}
	Let $(Y,L,\omega)$ be an admissible link, where $Y$ is an integer homology 3-sphere. If $\omega$ is an arc connecting two distinct components $L_1,L_2\subset L$, then
	\begin{equation}
		|\chi\left( I^\omega(Y,L) \right)|= 2^{|L|-2}|\text{{\emph{lk}}}(L_1,L_2)| \label{eq:eulerchar-onearc}
	\end{equation}
	If the number of components $L_0\subset L$ with $\#( L_0\cap \partial \omega)\equiv 1 \pmod{2}$ is greater than $2$, then 
	\[
		\chi\left( I^\omega(Y,L) \right) = 0
	\]
\end{theorem}

\vspace{0.2cm}

Note Kronheimer and Mrowka's instanton homology $I^\omega(Y,L)$ of an admissible link is only a priori a relatively $\Z/4$-graded abelian group. We will show how a choice of quasi-orientation for $L$ naturally determines an absolute $\Z/2$-grading, and give a version of Theorem \ref{thm:eulerchar-withbundles} that involves signs; see Theorem \ref{thm:eulerchar-withbundles-signs}. We also describe a method for fixing an absolute $\Z/4$-grading, see Subsection \ref{subsec:absz4grading}.

Saveliev and Harper \cite{saveliev-harper} defined a Casson-Lin type invariant for two component links $L=L_1\cup L_2$ in the 3-sphere that is expected to agree with $\pm \chi(I^\omega(L))$. They computed that their invariant equals $\pm \text{lk}(L_1,L_2)$, which agrees with relation \eqref{eq:eulerchar-onearc}; see also \cite{boden-herald}.

The proof of the last statement in Theorem \ref{thm:eulerchar-withbundles} uses exact triangles which have non-trivial bundles on each link. Let $L$, $L'$, $L''$ be based links in an oriented $3$-manifold $Y$ that form an unoriented skein triple. Let $\omega$ be unoriented $1$-manifold in $Y$ away from the skein 3-ball and which makes each link admissible. Then there is an exact triangle of $\cS$-complexes:
              \begin{equation}\label{eq:kmequivtriangle}
                \begin{tikzcd}[column sep=2ex, row sep=5ex, fill=none, /tikz/baseline=-10pt]
            & \widetilde C^{\omega}(Y,L)  \arrow[dr]  & \\
            \widetilde C^{\omega}(Y,L'')    \arrow[ur] & &  \widetilde  C^{\omega}(Y,L')  \arrow[ll]
            \end{tikzcd}
        \end{equation}
        The assumptions in this situation guarantee the absence of any reducible connections. This result is essentially the exact triangle established by Kronheimer and Mrowka in \cite{KM:unknot}; the only additional layer of structure included in our formulation is that the exact triangle for $C^{\omega}(Y,L)$, $C^{\omega}(Y,L')$, $C^{\omega}(Y,L'')$ which follows from \cite{KM:unknot} is compatible with $v$-maps. A proof follows the lines of the proof of Theorem \ref{thm:exacttriangles-withbundles}, forgetting all contributions from reducibles. For this reason we do not discuss this case further.

\subsection{Computations and further invariants}

Exact triangles in any link homology theory are a useful computational tool. Here we survey some computations that follow from our results. For simplicity, we focus only on some immediate applications to irreducible instanton homology $I(L)$.

All of the link homology theories $\mathcal{H}$ mentioned at the start of the introduction gave the property that for a non-split alternating link, and more generally a quasi-alternating link as defined in \cite{os:branched}, the rank of $\mathcal{H}(L)$ is determined by the determinant of $L$. In all cases this can be proved at least in part by induction on the exact triangle \eqref{eq:formalexacttriangle}. As an example, Kronheimer and Mrowka's $I^\natural(L)$ is free abelian of rank $\det(L)$, as proved in \cite{KM:unknot}.

Although the irreducible instanton homology $I(L)$ does not satisfy such an unoriented skein exact triangle in general, one can still use Theorem \ref{thm:irredexacttrianglesintro} to show that for quasi-alternating links, its rank is determined by the determinant. 

\begin{theorem}\label{thm:quasialtintro}
	Let $L$ be a non-split alternating link, or more generally a quasi-alternating link. Then as (ungraded) abelian groups, we have
	\[
		I(L) \cong \Z^{\frac{1}{2}(\det(L)-2^{|L|-1})}
	\]
	As a consequence, the spectral sequence \eqref{eq:ssirrtonatintro} collapses at the $E_2$-page.
\end{theorem}

In a different direction, we consider links which have the smallest possible rank from the viewpoint of \eqref{eq:irredhomlinkeuler}: we say that a link $L$ with non-zero determinant is {\emph{$I$-basic}} if
\[
	\text{rank}_\Z I(L) = |2^{|L|-2} \cdot \xi(L)|
\]
Such links are rather special. For example, the only $I$-basic knots with 12 crossings or less are the torus knots and the pretzel knot $P(-2,3,7)$, and the only alternating knots that are $I$-basic are alternating torus knots. Here are some more examples: 

\begin{theorem}\label{thm:ibasiclistintro}
	The following knots are $I$-basic:
	\begin{itemize}
		\item[(i)] All torus knots;
		\item[(ii)] the pretzel knots $P(-2,3,n)$ where $n\in \Z_{>0}$ is odd;
		\item[(iii)] the twisted torus knots listed in Proposition \ref{prop:twistedtorusknots}.
	\end{itemize}
\end{theorem}

\noindent Items (ii) and (iii) of this result are proved inductively using Theorem \ref{thm:irredexacttrianglesintro}.

The reader familiar with $L$-space knots in Heegaard Floer theory \cite{os-lens} may notice that all of the properties listed above for $I$-basic knots also hold for $L$-space knots. An $L$-space knot is a knot in the $3$-sphere for which there is some Dehn surgery of $K$ such that the hat-flavor of Heegaard Floer homology is supported in even gradings. All of the $I$-basic knots listed in Theorem \ref{thm:ibasiclistintro} are $L$-space knots. We are led to the following:

\begin{question}\label{question:lspaceknots}
	Is every $I$-basic knot an $L$-space knot, and conversely? 
\end{question}

\noindent In a sequel article, the authors will explore this question in more detail.

Also constructed in the present work are numerical invariants for links derived from equivariant singular instanton Floer theory and modelled after Fr\o yshov's invariant for integer homology 3-spheres \cite{froyshov}. Such invariants for knots were studied in \cite{DS1,DS2}. Given an oriented link $L\subset S^3$ with a basepoint, let $n=|L|$ be the number of components. For an algebra $\mathscr{S}$ over the ring $\Z[T_1^{\pm 1},\ldots, T_n^{\pm 1}]$, we define a non-increasing function
\[
	d_L^\mathscr{S}:\Z \to \Z_{\geq 0}
\]
with the property that $d_L^\mathscr{S}(i)=2^{|L|-1}$ for $i\ll 0$ and $d_L^{\mathscr{S}}(i)=0$ for $i\gg 0$. If $L$ and $L'$ are oriented and based links, we may form their connected sum at the basepoints, and we have
\[
	d_L^\mathscr{S}(i)d_{L'}^\mathscr{S}(j) \leqslant d^{\mathscr{S}}_{L\# L'}(i+j)
\]
for all $i,j\in \Z$. If $mL$ is the mirror of $L$, with the same basepoint, then 
\[
	d_{mL}^\mathscr{S}(i) = 2^{|L|-1} - d_{L}^\mathscr{S}(-i).
\]
Perhaps most importantly, $d_L^\mathscr{S}$ is an invariant of $L$ under concordances which preserve the basepoints. See Subsection \ref{subsec:froyshovinvs} for more details. In the case of a knot $K$, we have
\[
	h_\mathscr{S}(K) = \max\left\{ i : \; d^\mathscr{S}_K(i)= 1 \right\}
\] 
where $h_\mathscr{S}(K)$ is the Fr\o yshov-type invariant studied in \cite{DS1,DS2}. In particular, by the results of \cite{DS2}, $d_K^{\mathscr{S}}$ recovers the knot signature, and it is natural to expect that $d^\mathscr{S}_L$ for general links is related to link signatures for certain choices of $\mathscr{S}$. In this work we only construct these invariants and exhibit a few basic properties, and leave a more extensive treatment for future work (although see Proposition \ref{prop:alternatingfroyshov}). Another possible future direction is the construction of concordance invariants for links which take into account the Chern--Simons filtration, following the case of knots as studied in \cite{DS1,DS2,DISST:special-cycles}.

\subsection{Concluding remarks}\label{subsec:concludingremarks}

The present work deals primarily with links that have non-zero determinant. The main reason for this restriction is that the reducible traceless flat $SU(2)$ connections of these link complements are {\it non-degenerate} critical points of the Chern--Simons functional, see Proposition \ref{prop:detnonzeronondeg}. For a link $L$ with zero determinant, such reducibles are still in correspondence with quasi-orientations, but they are no longer non-degenerate. We can always make a small perturbation of the Chern--Simons functional so that the reducible critical points are all non-degenerate (and are still in correspondence with the quasi-orientations). Following a construction similar to the one in Section \ref{sec:links}, we can then associate an $\cS$-complex to $L$ and any such choice of perturbation. However, the chain homotopy type of the $\cS$-complex is not expected to be an invariant of $L$ anymore.  

The standard approach in Floer theories to establish invariance with respect to the choice of perturbations and other auxiliary data uses {\it continuation maps}. These are instances of cobordism maps, where the underlying cobordism is a product, but the choice of perturbation is not translationally invariant. In the case at hand, this method breaks down to show invariance of the chain homotopy type of $\cS$-complex of $L$ with respect to perturbations because these cobordisms might have obstructed reducibles. However, one should be able to use our extension of functoriality in Section \ref{sec:obstructed} to show that the chain homotopy type of the $\cS$-complex of $L$ is a topological invariant of $L$ up to suspension of $\cS$-complexes. 

These ideas may be useful in other directions as well. For example, in \cite{imori}, singular instantons with other holonomy parameters along meridians of knots are used to define $\cS$-complexes for knots satisfying some non-degeneracy condition. We expect that a modification of the above strategy can be used to extend the construction of \cite{imori} to all holonomy parameters for any knot, and more generally for any link. In a separate context, the ideas here might also be used to provide a framework for Seiberg--Witten monopole Floer homology of rational homology $3$-spheres, alternative to \cite{km:monopole, froyshov-monopole}.

The methods here may also be useful for a version of singular instanton Floer theory for knots which is equivariant with respect to the Lie group $O(2)$. This is the symmetry group one obtains by incorporating the ``flip symmetry'' into the $S^1$-equivariant theory treated here. See \cite[Section 2(iv)]{km-embedded-i} and \cite[Section 2.3]{DS1}. The set of reducible critical points in this context may be identified with binary dihedral $SU(2)$ representations of the knot group. However, apart from the unique critical point which is reducible with respect to the $S^1$-action, the other reducibles in general require a choice of non-zero perturbation to achieve non-degeneracy. In this sense, the scenario is similar to the issue for links with determinant zero as discussed above.

The strategy to extend the definition of instanton $\cS$-complexes in the ways mentioned above is similar to the one used in \cite{DME:QHS}, where it is shown that a variation of $\cS$-complexes for rational homology 3-spheres defined using non-singular instantons in \cite{miller} is a topological invariant up to an appropriate notion of suspension. In fact, the definition of maps for obstructed cobordisms in Section \ref{sec:obstructed} is inspired by \cite{DME:QHS}. However, there are some essential differences. Unlike the $\cS$-complexes in our setup, the corresponding notion in \cite{miller,DME:QHS} is infinite-dimensional, which presents challenges in direct computations and other algebraic aspects of the story. Furthermore, our definition of suspension is more explicit and can be defined at the algebraic level for any $\cS$-complex. This is in contrast with suspension in \cite{DME:QHS}, where additional topological input is needed. 

The present work sets the stage for a development of surgery exact triangles in non-singular instanton theory, extending Floer's exact triangles for versions of irreducible instanton homology \cite{floer:inst2,bd-floer,scadutothesis} to the equivariant setting of \cite{miller,DME:QHS}. Similar to work here, cobordisms with obstructed reducible instantons are expected to appear, and the construction of their associated Floer-theoretic maps should play an important role.

Finally, as already explained above, the exact triangles established in this work may be viewed as equivariant upgrades of Kronheimer and Mrowka's unoriented skein exact triangle \eqref{eq:inaturalexacttriangle} for $I^\natural(L)$. A natural question that arises is whether this upgrade extends to Kronheimer and Mrowka's spectral sequence of \cite{KM:unknot}. In particular, one might ask if there is a relationship between the instanton $\cS$-complex of a link and Khovanov homology with one of its equivariant structures \cite{kh-frob}. Note that $\cS$-complexes for links with zero determinant necessarily enter in this story, as the chain complex for Khovanov homology involves the complete resolutions of a diagram, which are unlinks.\\

\vspace{0.25cm}

\noindent \textbf{Outline.} \; In Section \ref{sec:prelims}, the algebra of $\cS$-complexes is developed. Here we define and study suspensions of $\cS$-complexes. In Section \ref{sec:links}, the $\cS$-complex for a knot as defined in \cite{DS1} is extended to the case of links with non-zero determinant. We also explain in this section how the construction of cobordism maps with unobstructed reducibles, as developed in \cite{DS1,DS2}, carry over to this setting. In Section \ref{sec:obstructed}, we use obstructed gluing theory to define maps for cobordisms that have obstructed reducibles with index $-3$. In Section \ref{sec:proofs}, we prove our main result, Theorem \ref{thm:exacttriangles}. It is shown that Cases II and III of this theorem are equivalent, and the section is largely focused on proving Cases I and II. In Section \ref{sec:nontrivbundles} we prove Theorems \ref{thm:exacttriangles-withbundles} and \ref{thm:eulerchar-withbundles}. In Section \ref{sec:further}, Theorems \ref{thm:equivexacttriangle} and \ref{thm:localcoeffintro} are proved, and Fr\o yshov-type invariants for links are constructed. In Section \ref{sec:comps}, we consider computations of irreducible instanton homology $I(L)$, and prove Theorems \ref{thm:quasialtintro} and \ref{thm:ibasiclistintro}. \\

\vspace{0.25cm}

\noindent \textbf{Acknowledgments.} \; The authors thank Nikolai Saveliev for helpful comments.\\

\newpage

\newpage

\newpage


\section{Algebraic Preliminaries}\label{sec:prelims}

In this section we define the necessary algebra for the exact triangles in equivariant singular instanton Floer theory. In the first subsection we define $\cS$-complexes in a way which is slightly more general than in \cites{DS1,DS2}. We then introduce some basic operations of $\cS$-complexes, the notion of suspension, heights of morphisms, and discuss features particular to the cases of even and odd degree morphisms. In the last subsection we discuss exact triangles of $\cS$-complexes.

\subsection{\texorpdfstring{$\cS$}{S}-complexes}

The following is the algebraic object central to our framework.

\begin{definition}\label{defn:scomplex}
    An {\emph{$\cS$-complex}} over a commutative ring $R$ is a chain complex $(\widetilde C, \widetilde d )$ over $R$, finitely generated and free as an $R$-module, with a graded module decomposition
    \[
      \widetilde C = C \oplus C[-1] \oplus \sfR
    \]
    such that with respect to this decomposition the differential $\widetilde d$ takes the from
    \begin{equation}
        \widetilde d = \left[ \begin{array}{ccc} d & 0 & 0 \\ v & -d & \delta_2 \\ \delta_1 & 0 & r \end{array} \right].\label{eq:dtilde}
    \end{equation}
    In general, for a graded $R$-module $C$ we write $C[i]$ for the graded module obtained from $C$ by setting $C[i]_{j} = C_{i+j}$. $\diamd$ 
\end{definition}

In the above definition it is implicit that our complexes have $\Z$-gradings. In the context of equivariant singular instanton theory, we will consider $\Z/4$-graded $\cS$-complexes. For a given $\cS$-complex we often write $\widetilde C$ instead of $(\widetilde C, \widetilde d )$. Our convention is that the differential of a chain complex has degree $-1$.

A {\emph{morphism}} $\widetilde\lambda:\widetilde C \to \widetilde C'$ of $\cS$-complexes is a degree $0$ chain map of the form
\begin{equation}
  \widetilde\lambda =  \left[ \begin{array}{ccc} \lambda & 0 & 0 \\ \mu & \lambda & \Delta_2 \\ \Delta_1 & 0 & \rho \end{array} \right]\label{eq:lambdatilde}
\end{equation}
with respect to the decompositions $\widetilde C = C \oplus C[-1] \oplus \sfR$ and $\widetilde C' = C' \oplus C'[-1] \oplus \sfR'$. An $\cS$-chain homotopy  $\widetilde K: \widetilde C\to \widetilde C'$ is a chain homotopy between morphisms $\widetilde\lambda, \widetilde\lambda':\widetilde C \to \widetilde C'$, so that it satisfies $\widetilde d' \widetilde K + \widetilde K \widetilde d = \widetilde \lambda - \widetilde \lambda'$, and which takes the form 
\begin{equation}
    \widetilde K =  \left[ \begin{array}{ccc} K & 0 & 0 \\ L & -K & M_2 \\ M_1 & 0 & J \end{array} \right]\label{eq:Ktilde}
\end{equation}
with respect to the decompositions of $\widetilde C$ and $\widetilde C'$. A {\emph{chain homotopy equivalence}} of $\cS$-complexes $\widetilde C$ and $\widetilde C'$ is a pair of morphisms $\lambda:\widetilde C\to \widetilde C'$ and $\widetilde \lambda':\widetilde C'\to \widetilde C$ such that $\widetilde \lambda \widetilde \lambda'$ and $\widetilde \lambda'\widetilde\lambda$ are chain homotopic to the identity morphisms. We sometimes say {\emph{$\cS$-chain homotopy equivalence}} for emphasis; and we write $\widetilde C \simeq \widetilde C'$.

A {\emph{degree $k$ morphism}} $\widetilde \lambda:\widetilde C\to \widetilde C'$, where $k$ is some integer, is defined just as above except that the underlying chain map has degree $k$. Chain homotopies in this case are also defined similarly. Note that by our convention a degree $k$ morphism for $k\neq 0$ is not a morphism $\widetilde C\to \widetilde C'$ in the strict sense defined above. Indeed, in forming the category of $\cS$-complexes, one should use degree $0$ morphisms. Nonetheless, in the sequel, we often abuse terminology and use ``morphism'' to refer to a degree $k$ morphism for some integer $k$.

The definitions given above are more general than ones found in previous works \cites{DS1,DS2,DISST:special-cycles}. First, we allow the summand $\sfR\subset \widetilde C$ to be a free module of arbitrary rank, while in previous work we only considered the case in which $\sfR$ has rank 1 and is identified with the ring $R$. This generalization is necessary when defining the $\cS$-complex for a link, which will have $\sfR$ free of rank $2^{|L|-1}$ where $|L|$ is the number of components of $L$.

Second, in contrast to \cites{DS1,DS2,DISST:special-cycles}, we now have a potentially non-zero component $r:\mathsf{R}\to \mathsf{R}$ in the differential $\widetilde d$ of an $\cS$-complex. In the case that $\sf{R}$ is rank 1, $r=0$ for grading reasons, specializing to the case of the previous works. In fact most of our $\cS$-complexes in the sequel will have $r=0$. We allow the extra flexibility that $r\neq 0$ so that we may take mapping cones of $\cS$-complexes, for example.

Finally, we also diverge from \cites{DS1,DS2} by allowing the component $\rho:\sfR \to \sfR'$ of a morphism to be zero; and we allow the component $J:\sfR\to \sfR'$ of a chain homotopy $\widetilde K$ to be non-zero. This flexibility is required for the more general types of cobordism maps that we will encounter in the sequel.

\begin{definition}
    A morphism $\widetilde\lambda:\widetilde C\to \widetilde C'$ between $\cS$-complexes is {\emph{strong}} if it is degree $0$ and the component $\rho:(\sfR,r)\to (\sfR',r')$ is a quasi-isomorphism.  \; $\diamd$
\end{definition}

\noindent In the case that the summands $\sfR$ and $\sfR'$ are rank 1, a strong morphism recovers the notion of a morphism used in \cite{DS1}.

The structure of an $\cS$-complex $\widetilde C$ may be encoded as follows. Define
\[
  \chi =  \left[ \begin{array}{ccc} 0 & 0 & 0 \\ 1 & 0 & 0 \\ 0 & 0 & 0 \end{array} \right]
\]
with respect to the decomposition $\widetilde C = C \oplus C[-1] \oplus \sfR$. Thus $\chi$ maps $C$ to $C[-1]$ identically and is otherwise zero, and $\chi$ has degree 1. Now the condition that $\widetilde d$ takes the form \eqref{eq:dtilde} is equivalent to the condition that $\chi$ is an anti-chain map, i.e. $\widetilde d \chi + \chi \widetilde d = 0$. Similarly, given $\cS$-complexes $\widetilde C$ and $\widetilde C'$ with associated maps $\chi$ and $\chi'$, the conditions that $\widetilde \lambda$ and $\widetilde K$ take the forms \eqref{eq:lambdatilde} and \eqref{eq:Ktilde} are equivalent to $\chi' \widetilde \lambda =\widetilde \lambda \chi$ and $\chi' \widetilde K +\widetilde K \chi = 0$, respectively.

In this way, an $\cS$-complex is a differential graded (dg) module over the dg-algebra $R[\chi]/(\chi^2) = \Lambda_R(\chi)$ where $\chi$ has degree $1$, and where the differential on the algebra is zero. Conversely, suppose we have a finitely generated free dg-module $\widetilde C$ over $\Lambda_R(\chi)$ such that $\text{ker}(\chi)/\text{im}(\chi)$ is a free $R$-module. Then we can choose a complement $\sfR\subset\text{ker}(\chi)$ to $\text{im}(\chi)$, and defining $C$ by $C[-1] = \text{im}(\chi)$, we obtain a short exact sequence
\begin{equation*}
    0\longrightarrow{}C[-1] \oplus \sfR\xhookrightarrow{\hspace{.4cm}}\widetilde C \xrightarrow{\;\;\chi\;\;} C\longrightarrow{}0
\end{equation*}
A module splitting of this exact sequence expresses $\widetilde C$ as an $\cS$-complex. Thus every such dg module over $\Lambda_R(\chi)$ is equivalent to an $\cS$-complex. The typical $\cS$-complex that arises in the sequel has a natural decomposition as appears in Definition \ref{defn:scomplex}, and this explains our preference for using $\cS$-complexes over the language of dg-modules.

Identifying the dg-algebra $\Lambda_R(\chi)$ with $H_\ast(S^1;R)$, with its Pontryagin product, makes it clear that $\cS$-complexes are models for $S^1$-equivariant chain complexes.

Given an $\cS$-complex $\widetilde C = C \oplus C[-1] \oplus \sfR$ there are associated two complexes
\[
    (C,d) , \qquad (\sfR, r)  
\]
called the {\emph{irreducible}} and {\emph{{reducible}} complexes, respectively. The {\emph{irreducible}} homology of $\widetilde C$ refers to $H_\ast(C,d)$, and similarly $H_\ast(\sfR,r)$ is called the {\emph{reducible}} homology. As mentioned above, typically we will have $r=0$ in the reducible complex, and in this case the reducible homology can be identified with $\sfR$.

Let $(\widetilde C,\widetilde d)$ be an $\cS$-complex. For convenience, we list the relations that the components of the differential $\widetilde d$ satisfy which together are equivalent to the relation $\widetilde d^2=0$:
\begin{align}
	d^2 &= 0 \\
	\delta_1 d + r\delta_1 &=0 \label{eq:diff2}\\
	d \delta_2 - \delta_2 r & =0 \label{eq:diff3}\\
	d v-v d - \delta_2 \delta_1 & = 0 \label{eq:diff4}\\
    r^2 & = 0 \label{eq:diff5}
\end{align}
Given a morphism $\widetilde \lambda:(\widetilde C,\widetilde d)\to (\widetilde C',\widetilde d')$ the relation $\widetilde d'\widetilde \lambda = \widetilde \lambda \widetilde d$ is equivalent to:
\begin{align}
	d'\lambda - \lambda d &= 0 \\
	\Delta_1 d + \rho \delta_1 - \delta_1'\lambda - r'\Delta_1 &=0  \label{eq:lambda2}\\
	d'\Delta_2  - \delta'_2 \rho + \lambda \delta_2 + \Delta_2 r & =0  \label{eq:lambda3}\\
	\mu d + d'\mu + \lambda v - v'\lambda + \Delta_2\delta_1 - \delta_2'\Delta_1  & = 0 \label{eq:lambda4}\\
    \rho r - r' \rho &= 0 \label{eq:lambda5}
\end{align}
Finally, for a chain homotopy $\widetilde K$ the relation $\widetilde d' \widetilde K + \widetilde K \widetilde d = \widetilde \lambda - \widetilde \lambda'$ is equivalent to:
\begin{align}
	d'K + K d  - \lambda + \lambda' &= 0 \\
	\delta_1'K + r'M_1 + M_1d + J\delta_1 - \Delta_1 + \Delta_1' &=0\\
	-d'M_2 + \delta_2'J - K\delta_2 + M_2 r  - \Delta_2 + \Delta_2'& =0\\
	v'K -d'L + \delta_2'M_1 + Ld - Kv +M_2\delta_1 - \mu + \mu'  & = 0\\
    r'J + Jr - \rho + \rho' &= 0 \label{eq':K5}
\end{align}
Recall that in our applications we typically have $r=r'=0$, and so the relations \eqref{eq:diff5} and \eqref{eq:lambda5} are vacuous, while for $\widetilde K$ the last relation \eqref{eq':K5} is $\rho=\rho'$.

The following observation will be used in the proof of Theorem \ref{thm:exacttriangles}.

\begin{lemma}\label{lemma:morphismuseful}
	Let $\widetilde \lambda:\widetilde C\to \widetilde C'$ be a morphism of $\cS$-complexes with irreducible component $\lambda:C\to C'$ and the reducible component $\rho:\sfR\to \sfR'$, each chain homotopic to isomorphisms. Then $\widetilde \lambda$ is $\cS$-chain homotopic to an isomorphism of $\cS$-complexes.
\end{lemma}

\begin{proof}
	Let $\lambda$ be chain homotopic to $\lambda'$ via $K:C\to C'$, so that $d' K + K d - \lambda + \lambda' =0$; and let $\rho$ be chain homotopic to a $\rho'$ via $J:\sfR\to \sfR'$, so that $r' J + Jr - \rho + \rho'=0$. Define
\begin{equation*}
  \widetilde\lambda' =  \left[ \begin{array}{ccc} \lambda' & 0 & 0 \\ \mu + Kv - v' K & \lambda' & \Delta_2 + K \delta_2 -\delta_2' J \\ \Delta_1 - \delta_1' K -J\delta_1 & 0 & \rho' \end{array} \right] \qquad
 \widetilde K =  \left[ \begin{array}{ccc} K & 0 & 0 \\ 0 & -K & 0 \\ 0 & 0 & J \end{array} \right]
\end{equation*}
Then direct computation shows that $\widetilde \lambda$ is a morphism and is chain homotopic to $\widetilde \lambda'$ via $\widetilde K$. Now if $\lambda'$ and $\rho'$ are isomorphisms, then the map $\widetilde \lambda'$ is an isomorphism because with respect to ordering the decompositions as $C\oplus \sfR\oplus C[-1]$ and $C'\oplus \sfR'\oplus  C'[-1]$ it is upper triangular with diagonal entries $\lambda',\lambda',\rho'$. The inverse of $\widetilde \lambda'$ is also easily seen to be a morphism. Compare \cite[Lemma 6.29]{DS1}.
\end{proof}

\subsection{Basic operations}

We now discuss duals, tensor products, and mapping cones of $\cS$-complexes. The first two of these operations are adaptations of \cite[Section 4]{DS1} to our current setting.

Here and below, $\epsilon$ denotes the sign map which multiplies a homogeneously graded element of degree $i$ by $(-1)^i$.
For a graded $R$-module $V$ we define the dual $V^\dagger = \text{Hom}_R(V,R)$ with grading given as follows: if $f\in V^\dagger$ is homogeneous of the form $f:V_i\to R$ then the grading of $f$ is equal to $-i$. For any map $m:V\to W$ between graded vector spaces, we define $m^\dagger:W^\dagger\to V^\dagger$ as $m^\dagger(f) = -\varepsilon(f)f(m)$. 

Given an $\cS$-complex $(\widetilde C, \widetilde d, \chi)$, the dual $\cS$-complex $(\widetilde C^\dagger, \widetilde d^\dagger, \chi^\dagger)$ is defined using the conventions above. The dual $\cS$-complex has decomposition
\[
    \widetilde C^\dagger =   C^\dagger\oplus C^\dagger[-1] \oplus \sfR^\dagger
\]
and with respect to this decomposition the differential takes the form
\begin{equation}
    \widetilde d^\dagger = \left[ \begin{array}{ccc} d^\dagger & 0 & 0 \\ v^\dagger & -d^\dagger & \delta_1^\dagger \\ \delta_2^\dagger & 0 & r^\dagger \end{array} \right].\label{eq:dtildedual}
\end{equation}
Note that when taking duals, the roles of $\delta_1$ and $\delta_2$ are interchanged.

It is straightforward to define the tensor product of $\cS$-complexes $(\widetilde C, \widetilde d, \chi)$ and $(\widetilde C', \widetilde d', \chi')$ from the viewpoint of dg-modules: we simply take the tensor product chain complex and define the $\chi$-action to be $\chi\otimes 1 + \varepsilon \otimes \chi'$. Following the discussion in \cite[\S 4.5]{DS1} we can represent this tensor product as an $\cS$-complex $(\widetilde C^\otimes , \widetilde d^\otimes)$ in the following way:
\[
  \widetilde C^\otimes = C^\otimes \oplus C^\otimes [-1] \oplus \sfR^\otimes , \qquad \sfR^\otimes = \sfR\otimes \sfR'
\]
\[
    C^\otimes  = \left(  C\otimes C' \right) \oplus \left(  C\otimes C' \right)[-1]  \oplus \left(  C\otimes \sfR' \right) \oplus \left(  \sfR \otimes C' \right)
\]
The differential $\widetilde d^\otimes$ has the following components: 
\[
    d^\otimes =\left[
        \begin{array}{cccc}
            d\otimes 1 + \varepsilon \otimes d' &0&0&0\\
            -\varepsilon v \otimes 1 + \varepsilon \otimes v' & d\otimes 1 - \varepsilon \otimes d' & \varepsilon \otimes \delta_2' & -\epsilon\delta_2 \otimes 1\\
            \varepsilon \otimes \delta_1' &0&d\otimes 1 + \varepsilon \otimes r'& 0\\
            \delta_1\otimes 1& 0 & 0 & \epsilon\otimes d' + r\otimes 1\\
        \end{array}
        \right]    
\]
\begin{equation*}
	v^\otimes =\left[
					\begin{array}{cccc}
						v\otimes 1 &0& 0& \delta_2\otimes 1\\
						0 & v\otimes 1 & 0 &0\\
						0 & 0& v\otimes 1 & 0\\
						0 & \epsilon \delta_1\otimes 1& 0  & 1\otimes v'\\
					\end{array}
					\right] \label{eq:vtensor}
\end{equation*}
\[
    \delta_1^\otimes = [0 , 0 , \delta_1\otimes 1 , \varepsilon \otimes \delta_1'] \qquad \delta_2^\otimes = [0,0,\delta_2\otimes 1, 1 \otimes \delta_2']^\intercal  \qquad r^\otimes = r\otimes 1 + \varepsilon \otimes r'
\]
Note that this representation of the tensor product as an $\cS$-complex depends on the order of $\widetilde C$ and $\widetilde C'$, although changing the order gives an isomorphic $\cS$-complex.

Let $\lambda: C\to C'$ be a degree $k$ chain map of chain complexes. The mapping cone complex has underlying graded $R$-module $C[-1-k]\oplus C'$ and its differential is given by
\[
    \left[ \begin{array}{cc} -d & 0 \\ \lambda & d' \end{array} \right]  
\]
Let $\widetilde\lambda:\widetilde C\to \widetilde C'$ be a degree $k$ morphism of $\cS$-complexes. Define the mapping cone of $\widetilde \lambda$ to be the $\cS$-complex described by the ordinary mapping cone of $\widetilde \lambda$
\[
   {\text{Cone}}(\widetilde \lambda)=\widetilde C[-1-k]\oplus \widetilde C',\hspace{1.5cm}
   \widetilde{d}_{\text{Cone}(\widetilde \lambda)}=\left[ \begin{array}{cc} -\widetilde d & 0 \\ \widetilde\lambda & \widetilde d' \end{array} \right], 
\]
together with the dg module structure given by the map
\[
  {\chi}_{\text{Cone}(\widetilde \lambda)}=\left[ \begin{array}{cc} -\chi & 0 \\ 0 & \chi' \end{array} \right].
\]
This has the module splitting given by:
\[
    \text{Cone}(\widetilde \lambda) =  \left( C[-1-k]\oplus C' \right) \oplus \left( C[-2-k]\oplus C'[-1] \right) \oplus \left( \sfR[-1-k] \oplus \sfR' \right),
\]
Upon applying an automorphism to this decomposition, we obtain the formula
\begingroup
\renewcommand{\arraystretch}{1.25}
\[
    \widetilde{d}_{\text{Cone}(\widetilde \lambda)} = \left[\begin{array}{cc|cc|cc} 
        -d & 0 & 0 & 0 & 0 & 0 \\
        \lambda & d' & 0 & 0 & 0 & 0 \\
        \hline
        v & 0 & d & 0 & \delta_2 & 0\\
        \mu & v' & -\lambda & -d' & \Delta_2& \delta_2' \\
        \hline
        -\delta_1 & 0 & 0 & 0 & -r & 0\\
        \Delta_1 & \delta_1' & 0 & 0 & \rho & r'
     \end{array} \right].
\]
\endgroup
Clearly the irreducible and reducible complexes are just the mapping cones of $\lambda$ and $\rho$, respectively.  Note that a special case of the mapping cone is when $\widetilde \lambda = 0$, which yields the operation of direct sum of $\cS$-complexes.

\subsection{Suspension complexes}\label{subsec:suspensiondefn}

In this subsection we define an operation on $\cS$-complexes called {\emph{suspension}}. This construction plays a key role in the formulation of the unoriented skein exact triangles for equivariant singular instanton Floer theory. 

\begin{definition}\label{defn:suspension}
    Let $\widetilde C = C \oplus C[-1] \oplus \sfR$ be an $\cS$-complex with differential $\widetilde d$. We define the {\emph{suspension}} of $\widetilde C$, which is an $\cS$-complex, as follows:
    \begingroup
\renewcommand{\arraystretch}{1.25}
    \[
        \widetilde C_\Sigma = C_\Sigma \oplus C_\Sigma[-1] \oplus \sfR, \qquad C_\Sigma = C[-2]\oplus \sfR[-1],
    \]
    \[
        \widetilde d_\Sigma = \left[\begin{array}{cc|cc|c} 
            d & -\delta_2 & 0 & 0 & 0 \\
            0 & -r & 0 & 0 & 0 \\
            \hline
            v & 0 & -d & \delta_2 & v\delta_2 \\
            \delta_1 & 0 & 0 & r & \delta_1\delta_2 \\
            \hline
            0 & 1 & 0 & 0 & r
         \end{array} \right].
    \]
\endgroup
    We also write $\Sigma \widetilde C$ for the suspended $\cS$-complex $\widetilde C_\Sigma$. \; $\diamd$
\end{definition}

We have written the differential $\widetilde d_\Sigma$ in block form corresponding to the decomposition of $\widetilde C_\Sigma$ given. It is straightforward to verify that $\widetilde d_\Sigma^2 = 0$ follows from the relation $\widetilde d^2=0$.

The meaning of the suspension complex is clarified by the following result. To state it, let us define a certain standard $\cS$-complex as follows:
\[
    \widetilde \cO(1) = C \oplus C[-1] \oplus \sfR , \qquad C = R_{(1)}, \qquad \sfR = R_{(0)}
\]
The subscripts indicate gradings. The differential is given by $d=v=\delta_2=0$ and $\delta_1:R_{(1)}\to R_{(0)}$ is an isomorphism. Alternatively, $\cO(1)$ is the suspension of the trivial  $\cS$-complex that has $C=0$, $\sfR=R_{(0)}$ and trivial differential.

\begin{prop}\label{prop:sustensor}
    The suspension $\Sigma \widetilde C$ is $\cS$-chain homotopy equivalent to $\widetilde C\otimes \widetilde \cO(1)$.
\end{prop}

\begin{proof}
    Let $(C^\otimes,d^\otimes)$ be the tensor product $\cS$-complex description of $\widetilde C\otimes \widetilde \cO(1)$ as given in the previous subsection. Then $C^\otimes = C[-1]\oplus C[-2] \oplus C \oplus \sfR$ and $\sfR^\otimes = \sfR$ with
    \[
      d^\otimes = \left[         \begin{array}{cccc}
        d &0&0&0\\
        -\varepsilon v  & d  & 0 & -\epsilon\delta_2 \\
        \varepsilon  &0&d& 0\\
        \delta_1& 0 & 0 &  r\\
    \end{array} \right]  
    \qquad
    v^\otimes =\left[
					\begin{array}{cccc}
						v &0& 0& \delta_2\\
						0 & v & 0 &0\\
						0 & 0& v & 0\\
						0 & \epsilon \delta_1& 0  & 0\\
					\end{array}
					\right]
    \] 
    and $\delta_1^\otimes = [0,0,\delta_1,\epsilon]$, $\delta_2^\otimes = [0,0,\delta_2,0]^\intercal$, $r^\otimes = r$. Next, define a map $\widetilde\lambda:\widetilde C_\Sigma\to \widetilde C^\otimes$ by prescribing its components as follows:
    \[
      \lambda = \left[ \begin{array}{cc} 0 & 0\\ 1 & 0 \\ 0 & 0 \\ 0 & \epsilon \end{array} \right] \qquad \mu =  \left[ \begin{array}{cc} 0 & 0\\ 0 & 0 \\ 0 & 0 \\ 0 & r \end{array} \right]  \qquad 
      \Delta_2 = \left[ \begin{array}{c} \epsilon \delta_2  \\ \epsilon \delta_2r \\ 0 \\ 0   \end{array} \right] 
    \]
    and $\rho=1$, $\Delta_1=0$. Define a map $\widetilde\lambda':\widetilde C^\otimes \to \widetilde C_\Sigma$ by
    \[
        \lambda' = \left[ \begin{array}{cccc} d & 1  & v & 0 \\ 0 & 0  & \delta_1 & \epsilon \\\end{array} \right] 
         \qquad 
         \mu' =  \left[ \begin{array}{cccc} 0 & 0 & 0 & 0 \\ \delta_1 & 0  & 0 & r \\\end{array} \right] 
    \]
    and $\rho'=1$, $\Delta'_1=\Delta'_2=0$. A direct computation shows that $\widetilde\lambda$ and $\widetilde\lambda'$ are chain maps and thus morphisms of $\cS$-complexes. Further, $\widetilde \lambda'\widetilde\lambda=1$. Finally, defining $\widetilde K:\widetilde C^\otimes \to \widetilde C^\otimes$ as follows provides a chain homotopy from $\widetilde \lambda \widetilde \lambda'$ to the identity: the components are
    \[
      K =   \left[
        \begin{array}{cccc}
            0 &0& -\varepsilon & 0\\
            0 & 0 & -\varepsilon d &0\\
            0 & 0& 0 & 0\\
            0 & 0 & 0  & 0\\
        \end{array}
        \right] \qquad       L =   \left[
            \begin{array}{cccc}
                0 &0& 0 & 0\\
                0 & 0 & 0 &0\\
                0 & 0& 0 & 0\\
                0 & 0 & 0  & \varepsilon \\
            \end{array}
            \right]
            \qquad
            M_2 = \left[\begin{array}{c} 0 \\ -\varepsilon \delta_2 \\ 0 \\ 2r \end{array} \right]
    \]
    and $M_1=J=0$. This completes the proof.
\end{proof}

We can also suspend in the ``reverse direction''. One way to do this is to first take the dual $\cS$-complex, then take the suspended complex, and take the dual one more time. What emerges is the {\emph{negative suspension}} of $\widetilde C$, which is an $\cS$-complex given as follows:
    \begingroup
\renewcommand{\arraystretch}{1.25}
\[
    \widetilde C_{\Sigma^{-1}} = C_{\Sigma^{-1}} \oplus C_{\Sigma^{-1}}[-1] \oplus \sfR, \qquad C_{\Sigma^{-1}} = C[2]\oplus \sfR[2],
\]
\[
    \widetilde d_{\Sigma^{-1}} = \left[\begin{array}{cc|cc|c} 
        d & 0 & 0 & 0 & 0 \\
        \delta_1 & r & 0 & 0 & 0 \\
        \hline
        v & \delta_2 & -d & 0 & 0 \\
        0 & 0 & -\delta_1 & -r & 1 \\
        \hline
        \delta_1v &  \delta_1\delta_2 & 0 & 0 & r
     \end{array} \right].
\]
\endgroup
We also write $\Sigma^{-1} \widetilde C$ for the negatively suspended $\cS$-complex $\widetilde C_{\Sigma^{-1}}$. The dual $\cS$-complex of $\widetilde \cO(1)$, written $\widetilde \cO(-1)=\widetilde \cO(1)^\dagger$, may be described as follows: 
\[
    \widetilde \cO(-1) = C \oplus C[-1] \oplus \sfR , \qquad C = R_{(-2)}, \qquad \sfR = R_{(0)}
\]
with differential given by $d=v=\delta_1=0$ and $\delta_2:R_{(0)}\to R_{(-2)}$ is an isomorphism. The dual description of Proposition \ref{prop:sustensor} gives the following.

\begin{prop}
    $\Sigma^{-1} \widetilde C$ is $\cS$-chain homotopy equivalent to $\widetilde C \otimes \widetilde \cO(-1)$.
\end{prop}

To understand the relationship between $\Sigma \Sigma^{-1} \widetilde C$ and $\Sigma^{-1} \Sigma  \widetilde C$ we have the following. We write $\widetilde \cO(0)$ for the $\cS$-complex with $C=0$ and $\sfR=R_{(0)}$ and trivial differential.

\begin{lemma}
    $\widetilde \cO(1) \otimes \widetilde \cO(-1)$ is $\cS$-chain homotopy equivalent to $\widetilde \cO(0)$.
\end{lemma}

\begin{proof}
    For simplicity we suppress gradings. The $\cS$-complex $\widetilde C^\otimes = \widetilde \cO(1) \otimes \widetilde \cO(-1)$ has $C^\otimes = R^{\oplus 4}$ and $\sfR^\otimes = R$ with differential components
    \[
        d^\otimes  =   \left[
        \begin{array}{cccc}
            0 &0& 0 & 0\\
            0 & 0 & -1 &0\\
            0 & 0& 0 & 0\\
            1 & 0 & 0  & 0\\
        \end{array}\right] \qquad        
         v^\otimes  =   \left[
            \begin{array}{cccc}
                0 &0& 0 & 0\\
                0 & 0 & 0 &0\\
                0 & 0& 0 & 0\\
                0 & 1 & 0  & 0\\
            \end{array}\right]
    \]
    and $\delta_1^\otimes = [0, 0, 1,0]$, $\delta_2=[0,0,0,1]^\intercal$, $r^\otimes =0$. Define $\widetilde \lambda:\widetilde C^\otimes\to \widetilde \cO(0)$ by setting $\rho=1$, $\Delta_1=[0,1,0,0]$ and $\lambda=\mu=\Delta_2=0$. Define $\widetilde \lambda':\widetilde \cO(0)\to \widetilde C^\otimes$ by $\rho'=1$, $\Delta_2'=[1,0,0,0]^\intercal$ and $\lambda'=\mu'=\Delta_1'=0$. These are morphisms and $\widetilde \lambda \widetilde\lambda'=1$. Define the components of $\widetilde K:\widetilde C^\otimes \to  \widetilde C^\otimes$ by
    \[
        K  =   \left[
        \begin{array}{cccc}
            0 &0& 0 & 1\\
            0 & 0 & 1 &0\\
            0 & -1& 0 & 0\\
            0 & 0 & 0  & 0\\
        \end{array}\right] \qquad        
         L  =   \left[
            \begin{array}{cccc}
                0 &0& 1 & 0\\
                0 & 0 & 0 &0\\
                0 & 0& 0 & 0\\
                0 & 0 & 0  & 0\\
            \end{array}\right]
    \]
    and $M_1=M_2=J=0$. Then $\widetilde K$ provides a homotopy from the identity to $\widetilde\lambda'\widetilde \lambda$.
\end{proof}

\begin{cor}
    For all $i,j\in\Z$ we have $\Sigma^i\Sigma^j\widetilde C \simeq \Sigma^{i+j}\widetilde C$.
\end{cor}

Given $n\in \Z$, write $\widetilde \cO(n)$ for the tensor product $\widetilde \cO(1)^{\otimes n}$ for $n>0$, $\widetilde \cO(-1)^{\otimes -n}$ for $n<0$, and $\widetilde \cO(0)$ if $n=0$. From our discussion we have
\[
    \Sigma^n \widetilde C \simeq \widetilde C\otimes \widetilde\cO(n)
\]
The tensor product description of the suspension complex is more convenient for certain algebraic considerations. However, the definition of $\Sigma \widetilde C$ given is both smaller than the tensor product and is also what naturally appears in our proof of the exact triangles in equivariant singular instanton Floer homology. 

From the tensor product viewpoint it is clear how to define the suspension of a morphism. In terms of our suspension complexes, given a morphism $\widetilde \lambda:\widetilde C\to \widetilde C'$ of $\cS$-complexes the suspension $\Sigma\widetilde \lambda=\widetilde \lambda_\Sigma$, which is a morphism $\Sigma\widetilde C \to \Sigma\widetilde C'$, is given as follows:
    \begingroup
\renewcommand{\arraystretch}{1.25}
\[
    \widetilde \lambda_{\Sigma} = \left[\begin{array}{cc|cc|c} 
        \lambda & \Delta_2 & 0 & 0 & 0 \\
        0 & \rho & 0 & 0 & 0 \\
        \hline
        \mu & 0 & \lambda & \Delta_2 & \mu\delta_2 + v'\Delta_2 \\
        \Delta_1 & 0 & 0 & \rho & \Delta_1\delta_2 + \delta_1'\Delta_2\\
        \hline
        0 &  0 & 0 & 0 & \rho
     \end{array} \right].
\]
\endgroup
Given morphisms $\widetilde \lambda:\widetilde C\to \widetilde C'$ and $\widetilde \lambda':\widetilde C' \to \widetilde C''$ the suspension $\Sigma(\widetilde \lambda' \widetilde \lambda)$ is homotopic to $\Sigma\widetilde \lambda' \circ\Sigma\widetilde\lambda$ but in general they are not equal.

If $\widetilde K$ is an $\cS$-chain homotopy between morphisms $\widetilde \lambda$ and $\widetilde \lambda'$ then
    \begingroup
\renewcommand{\arraystretch}{1.25}
\[
    \widetilde K_{\Sigma} = \left[\begin{array}{cc|cc|c} 
        K & -M_2 & 0 & 0 & 0 \\
        0 & -J & 0 & 0 & 0 \\
        \hline
        L & 0 & -K & M_2 & v'M_2 + L \delta_2 \\
        M_1 & 0 & 0 & J & \delta_1'M_2 + M_1\delta_2 \\
        \hline
        0 &  0 & 0 & 0 & J
     \end{array} \right]
\]
\endgroup
provides an $\cS$-chain homotopy between $\Sigma \widetilde\lambda$ and $\Sigma \widetilde\lambda'$.

The tensor product description of suspension complexes also yields:

\vspace{2mm}

\begin{prop}\label{prop:susdual}
	The dual of the $\cS$-complex $\Sigma^n \widetilde C$ is isomorphic to $\Sigma^{-n}\widetilde C^\dagger$.
\end{prop}

\vspace{2mm}

\noindent Indeed, the dual of $\cO(1)$ is $\cO(-1)$, and duals respect tensor products.

Many properties of an $\cS$-complex are preserved under suspension. For example:
\vspace{2mm}

\begin{prop} \label{eq:tildehomologysuspension}
    $H_\ast( \Sigma \widetilde C) \cong H_\ast (\widetilde C)$.
\end{prop}

\vspace{2mm}

\noindent This follows from the tensor product description of $\Sigma\widetilde C$ and the K\"{u}nneth formula. Further, the Fr\o yshov invariants are related by $h(\Sigma^n \widetilde C) = h(\widetilde C) + n$, and suspension behaves nicely with respect to the equivariant homology theories; see Section \ref{sec:further} for more.

Finally, we observe that the Euler characteristic of the irreducible homology transforms in the following way under suspension; this follows directly from the definitions.

\vspace{2mm}

\begin{lemma}\label{lemma:irreuler}
    For $\widetilde C = C\oplus C[-1] \oplus \sfR$ we have $\chi(\Sigma^n C) = \chi(C) - n\cdot \chi(\sfR)$.
\end{lemma}

\subsection{Heights of morphisms}\label{subsec:poslevels}

Heights of morphisms were introduced in \cite[Section 4.2]{DS2} to handle certain types of cobordism maps arising in equivariant singular instanton Floer theory. We will see how this notion fits nicely with the concept of suspension.

For our purposes we carry these constructions out for the following class of $\cS$-complexes which have relatively simple reducible complexes.

\begin{definition}\label{defn:rperfect}
    An $\cS$-complex is {\emph{$r$-perfect}} if its reducible complex $\sfR$ is supported in even gradings. In particular, the reducible differential is zero: $r=0$.  \; $\diamd$
\end{definition}

\noindent Every $\cS$-complex in the context of equivariant singular instanton Floer theory is homotopic to an $r$-perfect one, and most are in fact $r$-perfect. Note that in addition to $r=0$, an $r$-perfect $\cS$-complex has $\delta_1\delta_2=0$, and more generally $\delta_1 v^i \delta_2 = 0$ for all $i$. 

    For a morphism $\widetilde \lambda:\widetilde C\to \widetilde C'$ of $r$-perfect $\cS$-complexes define
    \begin{equation}\label{level-n-rel}
        \tau_{0} := \rho, \qquad \tau_{i+1}:=\delta_{1}'v'^i\Delta_2+\Delta_1v^i\delta_2+\sum_{j=0}^{i-1}\delta_{1}'v'^j\mu v^{i-1-j}\delta_2 \quad (i\geqslant 0)
      \end{equation}

\noindent Note that if any of the maps $\tau_i:\sfR\to \sfR'$ are non-zero, then $\widetilde \lambda$ is necessarily of even degree. Throughout this subsection we will assume that the morphisms under consideration are between $r$-perfect $\cS$-complexes and have even degree. The case of odd morphisms requires some modifications, and is discussed in Subsection \ref{subsec:odddegmorphisms}.

\begin{definition}\label{defn:leveln}
      A {\emph{height $n$ morphism}} ($n\in \Z_{\geqslant 0}$) is an even degree morphism $\widetilde \lambda: \widetilde C\to \widetilde C'$ between $r$-perfect $\cS$-complexes satisfying $\tau_j=0$ for all $0\leqslant j<n$. The morphism is further called {\emph{strong height $n$}} if in addition $\tau_{n}:\sfR \to \sfR'$ is an isomorphism. $\diamd$
\end{definition}

In the case that $\sfR$ is free of rank 1, a morphism which is strong height $n$ agrees with what was called a morphism of height $n$ in \cite{DS2}, and a strong height $0$ morphism was a strong morphism. A height $n$ morphism is also a height $i$ morphism for all $0 \leqslant i\leqslant n$. If a morphism is said to have height $n$, then it is implicit that the morphism has even degree.

For the following result we must unravel the $n$-fold suspension $\Sigma^n \widetilde C = \widetilde C_{\Sigma^n}$ for positive $n$. We assume that $\widetilde C$ is an $r$-perfect $\cS$-complex. Then
\[
	C_{\Sigma^n} = C[-2n]\oplus \bigoplus_{i=0}^{n-1} \sfR[-2i-1]
\]
and $\sfR_{\Sigma^n} =\sfR$ with differential $\widetilde d_{\Sigma^n}$ given in components as follows:
\[
    d_{\Sigma^n} =   \left[
        \begin{array}{cccccc}
            d & -\delta_2 & -v\delta_2 & -v^2\delta_2 & \cdots & -v^{n-1}\delta_2\\
             & 0 & 0 & 0 &\cdots &0 \\
             & & 0 & 0 & & 0 \\
             & & &  \ddots &  & \vdots \\
            & & & & &  0 \\ 
             &  &  & & & 0\\
        \end{array}\right]    
\]
\[
  v_{\Sigma^n} =  \left[
    \begin{array}{cccccc}
        v &  &  &  & & \\
        \delta_1 & 0 &  & & &  \\
         & 1& 0 &   & &   \\
         & & 1 &  & &  \\
        & &  &  \ddots& 0 &  \\ 
         &  &  & &1  &0 \\
    \end{array}\right] \qquad   (\delta_2)_{\Sigma^n} =  \left[
        \begin{array}{c}
            v^n\delta_2 \\ 0 \\ 0 \\ \vdots \\ 0
        \end{array}\right]
\]
and $(\delta_1)_{\Sigma^n}=[0,0,\ldots,0,1]$, $r_{\Sigma^n}=0$. There is a natural morphism
\begin{equation}\label{eq:iotamap}
  \widetilde \iota_n : \widetilde  C \longrightarrow \Sigma^n\widetilde C  
\end{equation}
defined as follows: the $\lambda$-component is $[1,0,\ldots,0]^\intercal$, the $\Delta_2$-component is $[0,1,0,\ldots,0]^\intercal$, and all the other components are set to zero. It is straightforward to verify that $\widetilde \iota_n$ is a strong height $n$ morphism.

\begin{prop}\label{prop:levelsuspension}
    Suppose $\widetilde \lambda:\widetilde C\to \widetilde C'$ is a morphism of $r$-perfect $\cS$-complexes which is (strong) height $n\in \Z_{> 0}$. Then there exists a natural (strong) height $0$ morphism 
    \[
        \widetilde \lambda':\Sigma^n\widetilde C \to \widetilde C'
    \]
    such that it factors $\widetilde \lambda$ into the composition $\widetilde \lambda' \widetilde \iota_n = \widetilde \lambda$.
\end{prop}

\begin{proof}
    We may assume $n>0$. For each non-negative integer $i$ define $\mu_i= \sum_{j=0}^{i-1} v'^j \mu v^{i-j-1}$. Define the components of $\widetilde \lambda'$ as follows. Let $\lambda' = [\lambda'_0,\lambda'_1,\ldots,\lambda'_n]$ with $\lambda'_0=\lambda$ and
    \[
        \lambda_i' = \mu_{i-1}\delta_2 + (v')^{i-1}\Delta_2 \qquad (1\leqslant i \leqslant n)
    \]
    Also, $\mu'=[\mu,0,\ldots,0]$, $\Delta'_1=[\Delta_1,0,0,\ldots,0]$, $\rho'=\tau_n$ and $\Delta_2' = \mu_n\delta_2 + (v')^n\Delta_2$. 
    
    We now compute the relations that show $\widetilde \lambda'$ is a morphism. First, $d' \lambda' =  \lambda'  d_{\Sigma^n}$ is equivalent to $d'\lambda = \lambda d$, which is clear, along with the relations $R_i=0$ for $0\leqslant i \leqslant n-1$:
    \begin{align*}
        R_i := d'\mu_i\delta_2 + d'v'^i\Delta_2 + \lambda v^i\delta_2 = 0 
    \end{align*}
    Note $R_0$ is $d'\Delta_2+\lambda\delta_2 =0$, a relation for the morphism $\widetilde \lambda$. We claim the following:
    \begin{equation}\label{eq:rirelation}
        R_{i+1} = v' R_{i} +\delta_2'\tau_{i+1}
    \end{equation}
    for all non-negative integers $i$. To prove this we use $\mu_{i+1} = v'\mu_i + \mu v^i$ to write $R_{i+1}$ as:
    \begin{align*}
        &d'v'\mu_i\delta_2 + d'\mu v^i\delta_2 +  d'v'v'^{i}\Delta_2 + \lambda v v^{i}\delta_2
    \end{align*}
    Use relation \eqref{eq:diff4} on the two instances of $d'v'$ and \eqref{eq:lambda4} on $\lambda v$ to obtain:
    \begin{align*}
        &v'd'\mu_i\delta_2 + \delta_2'\delta_1'\mu_i\delta_2 + \cancel{d'\mu v^i\delta_2} +  v'd'v'^i\Delta_2 + \delta_2'\delta_1'v'^i\Delta_2    +  v'\lambda v^i\delta_2 \phantom{\sum_{j=0}^i }\\
        & \quad  + \delta_2'\Delta_1v^i\delta_2  - \mu d v^i \delta_2 - \cancel{d'\mu v^i\delta_2}  
    \end{align*}
    The terms with $v'$ in front collect together to form $v'R_i$ and those with $\delta_2'$ in front collect together to form $\delta_2'\tau_{i+1}$. Thus we are left with 
    \begin{align*}
        R_{i+1} & = v'R_i + \delta_2'\tau_{i+1}- \mu d v^i \delta_2 
    \end{align*}
    We have used $\delta_1 v^i \delta_2 = 0$ for all $i$, which follows from the assumption the the complexes are $r$-perfect. Now by \eqref{eq:diff3}, $\mu v^i d\delta_2=0$. Then successive applications of \eqref{eq:diff4} to $\mu v^i d\delta_2$ shows that $\mu d v^i \delta_2 $ vanishes, proving \eqref{eq:rirelation}. Now $R_i=0$ for $0\leqslant i\leqslant n-1$ holds by our assumption that $\widetilde\lambda$ has height $n$, relation \eqref{eq:rirelation}, and induction in $i$. Thus $d' \lambda' =  \lambda'  d_{\Sigma^n}$.

    Next, $\Delta'_1 d_{\Sigma^n} + \rho' (\delta_1)_{\Sigma^n} - \delta_1'\lambda' - r'\Delta'_1 =0$ is equivalent to the relations
    \begin{align*}
      &\delta_1'\lambda  = \Delta_1 d   \\[1mm]
      &\delta_1'\mu_{i-1} \delta_2 + \delta_1'(v')^{i-1}\Delta_2 = -\Delta_1v^{i-1}\delta_2 \quad (1\leqslant i\leqslant n-1)\\[1mm]
      &\delta_1'\mu_{n-1} \delta_2 + \delta_1'(v')^{n-1}\Delta_2 = -\Delta_1v^{n-1}\delta_2 + \tau_n
\end{align*}
These are immediate consequences of the assumption that $\widetilde\lambda$ is a morphism of height $n$. 

Next, $d' \Delta_2' - \delta_2'\rho' + \lambda'(\delta_2)_{\Sigma^n} + \Delta_2 r_{\Sigma^n} =0$ is equivalent to
\[
	R_n - \delta_2'\tau_n = 0,
\]
which holds by \eqref{eq:rirelation} and the established relation $R_{n-1}=0$. The relation
\[
	\mu'd_{\Sigma^n}+ d' \mu' + \lambda'v_{\Sigma^n} - v'\lambda' + \Delta_2'(\delta_1)_{\Sigma^n} - \delta_2'\Delta_1'=0
\]
is directly verified. Finally, $\rho'r_{\Sigma^n}=r'\rho'$ holds by $r_{\Sigma^n}=r'=0$. This verifies all relations that exhibit $\widetilde \lambda'$ as a morphism. As $\rho'=\tau_n$ clearly $\widetilde \lambda'$ is strong if $\widetilde \lambda$ is strong height $n$.

The verification $\widetilde \lambda' \widetilde \iota_n = \widetilde \lambda$ is a straightforward computation.
\end{proof}

The above construction suggests a characterization of height $n\in \Z_{\geq 0}$ morphisms in terms of the existence of factorizations $\widetilde C \to \Sigma^n \widetilde C \to \widetilde C'$. 

\begin{lemma}\label{lemma:leveladditivity}
    Let $\widetilde C$, $\widetilde C'$ and $\widetilde C''$ be $r$-perfect, and $\widetilde \lambda:\widetilde C\to\widetilde C'$ and $\widetilde \lambda':\widetilde C'\to \widetilde C''$ have (strong) heights $n$ and $m$, respectively. Then $\widetilde \lambda'\widetilde \lambda$ is a (strong) height $n+m$ morphism.
\end{lemma}

\begin{proof}

    We first establish some notation. Let $v_i=v^i$ for $i\in \Z_{\geqslant 1}$, $v_0=1$, and $v_i=0$ for $i$ negative. We use the summation convention on repeated indices. For $i\geq 0$, write
    \[  
        A_{i+1}=\delta_{1}'v'_i\Delta_2 \qquad B_{i+1} = \Delta_1v_i\delta_2 \qquad C_{i+1} = \delta_{1}'v'_j\mu v_{i-1-j}\delta_2
    \]
    and $A_0=\rho$, $B_0=C_0=0$, so that $\tau_{i} = A_i + B_i + C_i$ for $i\geqslant 0$. Also let each of these terms be zero for $i<0$. We write $A_i'$, $B_i'$, $C_i'$ for the similarly defined terms associated to $\widetilde \lambda'$. Denote by $\tau_{i+1}^\circ$ the expressions \eqref{level-n-rel} associated to $\widetilde \lambda'\widetilde \lambda$. We expand $\tau_{i+1}^\circ$ as follows:
    \begin{align*}\label{eq:taudoubleprime}
        \tau_{i+1}^\circ &=  \underbracket[.01cm]{\delta''_1v''_i\lambda'\Delta_2}_{\text{(i)}} + \underbracket[.01cm]{\delta''_1v''_i\Delta'_2\rho}_{\text{(ii)}} + \underbracket[.01cm]{\Delta_1'\lambda v_i\delta_2}_{\text{(iii)}}+ \underbracket[.01cm]{\rho'\Delta_2  v_i\delta_2}_{\text{(iv)}}\\[1mm] 
        & \qquad \quad + \underbracket[.01cm]{\delta''_1 v_j''\mu'\lambda v_{i-j-1} \delta_2}_{\text{(v)}} 
         + \underbracket[.01cm]{\delta_1'' v''_j \lambda'\mu v_{i-j-1}\delta_2}_{\text{(vi)}} + \underbracket[.01cm]{\delta_1''v_j'' \Delta_2'\Delta_1 v_{i-j-1}\delta_2}_{\text{(vii)}}
    \end{align*}
    We consider each term. First, successively applying \eqref{eq:lambda4} to $v''\lambda''$, term (i) becomes
    \begin{align*}
        \delta''_1v''_i\lambda'\Delta_2 &= \delta_1''\lambda'v'_i\Delta_2 + \delta_1'' v''_j \mu' d' v_{i-j-1} \Delta_2 + \cancel{\delta_1''v''_{j}d''\mu' v'_{i-j-1}\Delta_2}  \\[2mm]
        & \qquad  + \delta_1''v''_{j}\Delta_2'\delta_1' v'_{i-j-1}\Delta_2  - \cancel{\delta_1''v''_{j}\delta_2'' \Delta_1' v'_{i-j-1}\Delta_2} \nonumber
    \end{align*}
Apply relation \eqref{eq:diff4} successively to compute
    \[
        \delta_1'' v''_j \mu' d' v_{i-j-1} \Delta_2 = \delta_1'' v''_j\mu' v'_{i-j-1}d'\Delta_2 + \delta_1'' v''_k\mu' v'_{j-k-1}\delta'_2\delta'_1v'_{i-j-1}\Delta_2
    \]
    Then, using $d'\Delta_2 = -\lambda\delta_2+\delta_2'\rho$, we obtain the following expression for term (i):
    \begin{align*}
        \text{(i)} &= \delta_1''\lambda'v'_i\Delta_2 - \delta_1'' v''_j\mu' v'_{i-j-1} \lambda\delta_2 + C'_{j+1} A_{i-j}  + A'_{j+1}A_{i-j} - A'_0 A_{i+1} - A'_{i+1}A_0
    \end{align*}
    Similar computations for the other terms yields:
    \begin{align*}
        \text{(ii)} &= A'_{i+1}A_0\\[1mm]
        \text{(iii)} &= \Delta_1'v'_{i}\lambda\delta_2  - \Delta_1'd' v'_{j}\mu v_{i-j-1} \delta_2 + B'_{j+1} C_{i-j} + B'_{j+1} B_{i-j}\\[1mm]
        \text{(iv)} &= A'_0 B_{i+1}\\[1mm]
        \text{(v)} &=  \delta_1''v''_j\mu' v'_{i-j-1}\lambda\delta_2 -\delta_1''v''_{j}\mu' d' v'_k \mu v_{i-j-k-2} \delta_2  + C'_{j+1} C_{i-j} + C'_{j+1} B_{i-j}\\[1mm]
        \text{(vi)}  &= \Delta_1'd'v'_j \mu v_{i-j-1} \delta_2+\delta_1''v''_j \mu' d' v'_{k}\mu v_{i-j-k-2} \delta_2 + A'_{j+1} C_{i-j} \\[1mm]
        \text{(vii)} & = A'_{j+1} B_{i-j} - A'_0 B_{i+1}
\end{align*}
Note that using $d'\Delta_2 = -\lambda\delta_2+\delta_2'\rho$ and $\Delta'_1d' = \delta_1''\lambda'-\rho'\delta_1'$, along with \eqref{eq:diff4}, we obtain
\[
    \delta_1''\lambda'v'_i\Delta_2  +  \Delta_1'v'_{i}\lambda\delta_2 = B'_{j+1}A_{i-j} + A_0'A_{i+1}
\] 
Now summing the terms (i)--(vii) we obtain:
\[
    \tau_{i+1}^\circ = \sum_{j} \tau_{j+1}'\tau_{i-j}
\]
The statement follows easily from this relation.
\end{proof}

\begin{cor}
    Let $\widetilde \lambda:\widetilde C\to \widetilde C'$ be an even degree morphism of $r$-perfect $\cS$-complexes which factors as follows, where $\widetilde \lambda'$ is a (strong) morphism:
    \begin{equation*}
        \begin{tikzcd}[column sep=5ex, row sep=10ex, fill=none, /tikz/baseline=-10pt]
    \widetilde C \arrow[r, "\widetilde \iota_n"]  & \Sigma^n \widetilde C \arrow[r, "\widetilde \lambda'"]  & \widetilde C'
    \end{tikzcd}
    \end{equation*}
    Then $\lambda'$ is a (strong) height $n$ morphism.
\end{cor}

\begin{proof}
    The morphism $\widetilde \iota_n$ has height $n$ and $\widetilde \lambda'$ has height $0$, so by Lemma \ref{lemma:leveladditivity} the composition is a height $n+0=n$ morphism.
\end{proof}

\subsection{Negative height morphisms}\label{subsec:neglevels}
We next generalize the notion of non-negative height morphisms by allowing negative heights. As in the previous subsection, we restrict our attention here to the case of $r$-perfect $\cS$-complexes and even degree morphisms.

\begin{definition}\label{level--n-def}
	Suppose $\widetilde C$ and $\widetilde C'$ are $r$-perfect $\cS$-complexes. For an integer $n$, a height $n$ morphism $\widetilde \lambda:\widetilde C\to \widetilde C'$ is given by homomorphisms 
	\begin{equation}\label{level-mor-com}
		\lambda:C\to C', \hspace{1cm}\mu :C\to C', \hspace{1cm}\Delta_1:C\to \sfR',\hspace{1cm}\Delta_2:\sfR\to C',
	\end{equation}	
	as well as a sequence of homomorphisms
	\begin{equation}\label{tau-i}
		\hspace{2cm}\tau_{i}:\sfR\to \sfR'\hspace{1.5cm} i\in \Z.
	\end{equation}	
	The maps $\lambda,\mu,\Delta_1,\Delta_2,\tau_i$ have respective degrees $k,k-1,k,k-1,k-2i$ for some even integer $k$, called the {\emph{degree}} of $\widetilde \lambda$. We require that	$\tau_i=0$ for $i<n$, and for $i>0$ the relation  
	\begin{equation}\label{tau-i-relation}
	  \tau_{i}=\delta_{1}'v'^{i-1}\Delta_2+\Delta_1v^{i-1}\delta_2+\sum_{j=0}^{i-2}\delta_{1}'v'^j\mu v^{i-2-j}\delta_2 \quad
	\end{equation}
	is satisfied.  Furthermore, the following relations are satisfied:
	\begin{align}
	d'\lambda - \lambda d-\sum_{i=1}^{\infty}\sum_{j=0}^{i-1}v'^j\delta_2'\tau_{-i} \delta_1v^{i-1-j} &= 0 \label{eq:lambda1-1}\\
	- \delta_1'\lambda+\Delta_1 d +\sum_{i=0}^{\infty}\tau_{-i}\delta_1v^i  &=0  \label{eq:lambda2-1}\\
	\lambda \delta_2+d'\Delta_2  -\sum_{i=0}^{\infty}v'^i\delta_2'\tau_{-i}  & =0  \label{eq:lambda3-1}\\[2mm]
	\mu d + d'\mu + \lambda v - v'\lambda + \Delta_2\delta_1 - \delta_2'\Delta_1  & = 0 \label{eq:lambda4-1}
	\end{align}
	Note that our assumption on $\tau_i$ implies that all of the above identities involve finitely many terms. The morphism $\widetilde \lambda$ is a {\emph{strong height $n$}} morphism if $\tau_{n}$ is an isomorphism. $\diamd$
\end{definition}

A few remarks about this definition are in order. First note that the above definition agrees with Definition \ref{defn:leveln} in the case that $n$ is non-negative: in this case, any height $n$ morphism gives a morphism of $\cS$-complexes by forgetting the maps $\tau_i$ with $i\geq 1$. However, this does not apply to the case that $n$ is negative. (Although see Proposition \ref{decomp-neg-level}.) Any height $n$-morphism is also height $m$ for any $m\leq n$. On the other hand, it is a strong height $n$ morphism for at most one value of $n$. 

Non-negative height morphisms in the context of singular instanton Floer homology already appeared in \cite{DS2}. In the present work, we shall see in Section \ref{sec:obstructed} how one obtains height $(-1)$-morphisms from certain link cobordisms. For convenience, we write out the definition in this case: a height $(-1)$-morphism is determined by homomorphisms $\lambda$, $\mu$, $\Delta_1$ and $\Delta_2$ as in \eqref{level-mor-com} and $\tau_0, \, \tau_{-1}: \sfR\to \sfR'$ which satisfy the relations
\begin{align}
	d'\lambda - \lambda d-\delta_2'\tau_{-1} \delta_1&= 0 \label{eq:lambda1-1-lev-1}\\[1mm]
	- \delta_1'\lambda+\Delta_1 d +\tau_{-1}\delta_1v+\tau_{0}\delta_1  &=0  \label{eq:lambda2-1-lev-1}\\[1mm]
	\lambda \delta_2+d'\Delta_2  -v'\delta_2'\tau_{-1}-\delta_2'\tau_{0}  & =0  \label{eq:lambda3-1-lev-1}\\[1mm]
	\mu d + d'\mu + \lambda v - v'\lambda + \Delta_2\delta_1 - \delta_2'\Delta_1  & = 0 \label{eq:lambda4-1-lev-1}
\end{align}

\begin{example}
	For any non-negative integer $n$ and any $r$-perfect $\cS$-complex $\widetilde C$, we have a strong height $-n$ morphism $\widetilde \kappa_n: \Sigma^n \widetilde C\to \widetilde C$ where the $\lambda$-component is equal to $[1,0,\dots,0]$, the $\tau_{-n}$ component is the identity, 
	and the remaining components are zero. $\diamd$
\end{example}

The following describes how to compose two morphisms of arbitrary heights.

\begin{definition}\label{comp-neg-lev}
Suppose $\widetilde \lambda:\widetilde C\to \widetilde C'$ with components $(\lambda, \mu, \Delta_1,\Delta_2, \{\tau_{i}\})$ is a height $n$ morphism, and $\widetilde \lambda':\widetilde C'\to \widetilde C''$ with components $(\lambda', \mu', \Delta_1',\Delta_2', \{\tau_{i}'\})$ is a height $m$ morphism. Define $\widetilde \lambda'\widetilde \lambda:\widetilde C\to \widetilde C''$ with components $(\lambda^\circ, \mu^\circ, \Delta_1^\circ,\Delta_2^\circ, \{\tau_{i}^\circ\})$ as follows: 
\begin{align}
	\lambda^\circ:=&\lambda'\lambda+\sum_{i=0}^{\infty}\sum_{j=0}^iv''^j\delta_2''\tau_{-(i+1)}'\Delta_1v^{i-j}+\sum_{i=0}^{\infty}\sum_{j=0}^iv''^j\Delta_2'\tau_{-(i+1)}\delta_1v^{i-j}\nonumber\\[1mm]
	&+\sum_{i=0}^{\infty}\sum_{j=0}^i\sum_{k=0}^{i-j}v''^j\delta_2''\tau_{-(i+2)}'\delta_1'v'^k\mu v^{i-j-k}+\sum_{i=0}^{\infty}\sum_{j=0}^i\sum_{k=0}^{i-j}v''^j\mu'v'^k\delta_2'\tau_{-(i+2)}\delta_1 v^{i-j-k}\nonumber\\[2mm]
	\mu^\circ:=&\lambda'\mu+\mu'\lambda +\Delta_2'\Delta_1\nonumber\\[1mm]
	\Delta_1^\circ:=&\Delta_1'\lambda+\sum_{i=0}^\infty \tau_{-i}'\Delta_1v^i+\sum_{i=0}^{\infty}\sum_{j=0}^{i} \tau_{-(i+1)}'\delta_1'v'^j\mu v^{i-j} \nonumber\\[1mm]
	\Delta_2^\circ:=&\lambda'\Delta_2+\sum_{i=0}^\infty v''^i\Delta_2' \tau_{-i}+\sum_{i=0}^{\infty}\sum_{j=0}^{i} v''^j\mu'v'^{i-j}\delta_2'\tau_{-(i+1)} \nonumber\\[1mm]
	\tau_{i}^\circ:=&\sum_{k\in \Z}\tau_{i-k}'\tau_{k}\nonumber. \qquad \diamd
\end{align}
\end{definition}
The following lemma is a generalization of Lemma \ref{lemma:leveladditivity}, and we omit its proof.
\begin{lemma}\label{lemma:leveladditivity-general}
    Let $\widetilde \lambda:\widetilde C\to\widetilde C'$ and $\widetilde \lambda':\widetilde C'\to \widetilde C''$ be (strong) height $n$ and $m$ morphisms, respectively. Then $\widetilde \lambda'\widetilde \lambda$ is a (strong) height $n+m$ morphism.
\end{lemma}

There is also a variation of Proposition \ref{prop:levelsuspension} which applies to negative height morphisms.  
\begin{prop}\label{decomp-neg-level}
	Suppose $\widetilde \lambda:\widetilde C\to \widetilde C'$ is a morphism of (strong) height $-n$. Then there is a (strong) height $0$ morphism $\widetilde \lambda':\widetilde C\to \Sigma^n \widetilde C'$ such that $\widetilde \lambda$ is equal to the composition
	$\widetilde \kappa_n \widetilde \lambda'$.
\end{prop}

\begin{proof}
	Let the components of $\widetilde \lambda':\widetilde C\to \Sigma^n \widetilde C'$ be
	\[
	\lambda'=  \left[
        \begin{array}{c}
            \lambda \\[1mm] 
            \sum_{i=1}^n\tau_{-i}\delta_1v^{i-1} \\[1mm] 
            \sum_{i=2}^n\tau_{-i}\delta_1v^{i-2} \\[1mm]
            \sum_{i=3}^n\tau_{-i}\delta_1v^{i-3}\\[1mm]
             \vdots \\[1mm]
              \tau_{-n}\delta_1 
        \end{array}\right],\hspace{.2cm}
        	\mu'=  \left[
        \begin{array}{c}
            \mu \\ \Delta_1 \\0 \\0\\ \vdots \\ 0
        \end{array}\right],\hspace{.2cm}
        	\Delta_1'=0,\hspace{.2cm}
        	\Delta_2'=  \left[
        \begin{array}{c}
            \Delta_2 \\ \tau_0 \\\tau_{-1} \\\tau_{-2}\\ \vdots \\ \tau_{-(n-1)}
        \end{array}\right],\hspace{.2cm}
        \tau'=\tau_{-n}.
	\]
	It is straightforward to check that $\widetilde \lambda'$ has the required properties.
\end{proof}
In the special case that $\widetilde \lambda:\widetilde C\to \widetilde C'$ is a morphism of height $-1$, the morphism constructed in Proposition \ref{decomp-neg-level} has the following form: 
    \begingroup
\renewcommand{\arraystretch}{1.25}
\begin{equation}\label{eq:morphismfromheight-1}
   \left[ \begin{array}{c|c|c} \lambda & 0 & 0 \\\tau_{-1}\delta_1 & 0 & 0 \\\hline \mu & \lambda & \Delta_2 \\ \Delta_1 & \tau_{-1}\delta_1 & \tau_0\\\hline0&0&\tau_{-1} \end{array} \right].
\end{equation}
\endgroup

\begin{remark}
	Our terminology dictates that a ``height $n$ morphism'' $\widetilde\lambda:\widetilde C\to \widetilde C'$ in the case $n<0$ is in general {\emph{not}} a morphism of $\cS$-complexes, in any natural way. However, Proposition \ref{decomp-neg-level} associates to $\widetilde\lambda$ an honest morphism of $\cS$-complexes $\widetilde C\to \Sigma^n \widetilde C'$. $\diamd$
\end{remark}

\subsection{Odd degree morphisms}\label{subsec:odddegmorphisms}

In Subsections \ref{subsec:poslevels} and \ref{subsec:neglevels}, we restricted our attention to {\emph{even}} degree morphisms. Indeed, in the context of Definition \ref{defn:leveln}, if the morphism $\widetilde\lambda$ is instead {\emph{odd}} degree, then all the $\tau_i$ are necessarily zero, and thus the notion of ``height'' is superflous. 

Nonetheless, the algebra for height $n$ morphisms developed above can be adapted to the case in which odd degree morphisms are allowed. For example, Definition \ref{comp-neg-lev} gives a perfectly sensible composition rule for morphisms in the case when $\widetilde \lambda$ is height $n\in \Z$ (of even degree) and $\widetilde \lambda'$ is odd degree (with $\tau_i'=0$ for all $i$). However, in general this is not the correct rule for composing with odd degree morphisms for our purposes; studying the relations of morphisms that arise from instanton moduli spaces makes this clear. Indeed, in the odd degree case, a morphism $\widetilde \lambda:\widetilde C\to \widetilde C'$ between $r$-perfect $\cS$-complexes should come equipped with the additional data of a sequence of $R$-module homomorphisms
\begin{equation}\label{eq:numaps}
	\nu_i:\sfR\to \sfR' \hspace{1.5cm} i\in \Z_{\geq 0}
\end{equation}
where $\nu_i$ has degree $k-1-2i$, and $k$ is the degree of $\widetilde \lambda$. See Subsection \ref{subsubsec:odddegcobs} for a description of $\nu_i$ in the context of singular instanton Floer theory.

For the remainder of this subsection, assume that all $\cS$-complexes are $r$-perfect. Let $\widetilde \lambda:\widetilde C\to \widetilde C'$ be an even degree morphism of height $n$ and $\widetilde\lambda':\widetilde C'\to \widetilde C''$ an odd degree morphism. If $n\geq 0$, then $\widetilde\lambda$ is a morphism of $\cS$-complexes just as is $\widetilde \lambda'$, and composition is unambiguously defined as the $\cS$-complex morphism $\widetilde \lambda'\widetilde\lambda:\widetilde C\to \widetilde C''$. Note that in a more thorough treatment of the algebra at hand, one should also prescribe the homomorphisms \eqref{eq:numaps} for the composition. However, this is not needed in the sequel. A similar remark applies to the composition described in the following paragraph.

In the case that $n$ is negative, on the other hand, the maps $\nu'_i$ associated to $\widetilde\lambda'$ are expected to play a role in composition. We will describe the case $n=-1$, in which only
\[
	\nu':=\nu'_0:\sfR'\to \sfR''
\]
appears in the composition rule. This case, and minor variations thereof, is the only one that plays an important role in the sequel. Write $(\lambda, \mu, \Delta_1,\Delta_2, \tau_0,\tau_{-1})$ for the components of the height $-1$ morphism $\widetilde \lambda$. The composition $\widetilde \lambda'\widetilde \lambda:\widetilde C\to \widetilde C''$ is then defined to be the $\cS$-complex morphism defined by the following components:
\begin{align}
	\lambda^\circ &:= \lambda' \lambda + \Delta_2' \tau_{-1}\delta_1 \label{eq:compheight-1-odddeg-1} \\[1mm]
	\mu^\circ &:= \lambda'\mu + \mu' \lambda +\Delta_2' \Delta_1  \label{eq:compheight-1-odddeg-2}\\[1mm]
	\Delta_1^\circ & := \Delta_1' \lambda + \nu' \tau_{-1} \delta_1 \label{eq:compheight-1-odddeg-3} \\[1mm]
	\Delta_2^\circ &:= \lambda' \Delta_2 + \Delta_2' \tau_0 + \delta_2'' \nu' \tau_{-1} + v'' \Delta_2' \tau_{-1} + \mu'\delta_2'\tau_{-1} \label{eq:compheight-1-odddeg-4}
\end{align}

We may now describe a partial analogue of Proposition \ref{prop:levelsuspension} for odd degree morphisms. Let $\widetilde \lambda':\widetilde C'\to \widetilde C''$ be an odd degree morphism as above, with associated linear map $\nu'=\nu'_0:\sfR'\to \sfR''$. The claim is that there is a natural $\cS$-complex morphism
\[
	\widetilde \lambda'' : \Sigma \widetilde C' \to \widetilde C''
\]
satisfying $\widetilde \lambda' = \widetilde \lambda''\widetilde \iota_1$, where $\widetilde \iota_1$ is as in \eqref{eq:iotamap}. The formula for $\widetilde \lambda''$ is as follows:
    \begingroup
\renewcommand{\arraystretch}{1.25}
\begin{equation}\label{eq:lambdaprimesuspend}
\widetilde\lambda'' =  \left[ \begin{array}{cc|cc|c} \lambda'  & \Delta'_2 & 0 & 0 & 0  \\ \hline \mu' & 0 &  \lambda' & \Delta'_2 & \mu'\delta'_2+v''\Delta'_2 + \delta_2''\nu' \\ \hline \Delta'_1 & \nu' & 0 & 0 & 0 \end{array} \right]
\end{equation}
\endgroup
To see that this is the correct formula, one need only check that the composition of the $\cS$-complex morphism $\widetilde\lambda''$ with the morphism $\widetilde C\to \Sigma \widetilde C'$, constructed from an even degree height $-1$ morphism $\widetilde \lambda$ via Proposition \ref{decomp-neg-level} as in \eqref{eq:morphismfromheight-1}, agrees with the morphism obtained from composing $(\lambda, \mu, \Delta_1,\Delta_2,\tau_0,\tau_{-1})$ and $(\widetilde \lambda',\nu')$ using the rules \eqref{eq:compheight-1-odddeg-1}--\eqref{eq:compheight-1-odddeg-4}.

\subsection{Exact triangles}\label{subsec:scomplextriangles}

For our purposes it is convenient to define an exact triangle of $\cS$-complexes following the ``triangle detection lemma'', see \cite[Lemma 3.7]{seidel}, which includes a more general discussion, and also \cite[Lemma 4.2]{os:branched}, \cite[Lemma 7.1]{KM:unknot}. Note that our sign conventions are different than those of these references.

\begin{definition} \label{def:exacttriangle}
    An {\emph{exact triangle}} of $\cS$-complexes consists of the following data, for each $i\in \Z/3$: an $\cS$-complex $\widetilde C_i$; a morphism $\widetilde \lambda_i:(\widetilde C_i,\widetilde d_i,\chi_i)\to (\widetilde C_{i-1},\widetilde d_{i-1},\chi_{i-1})$ which for $i=0,2$ is a morphism of degree $0$ and for $i=1$ is a degree $-1$ morphism; a chain homotopy $\widetilde K_i:\widetilde C_i\to \widetilde C_{i-2}$ from $0$ to $\widetilde \lambda_{i-1}\widetilde \lambda_i$. In particular, for each $i\in \Z/3$ we have the relations
    \begin{align}
        \widetilde d_i^2 &= 0 \label{eq:exacttriangle1} \\[1mm]
        \widetilde d_{i-1} \widetilde \lambda_i  - \widetilde \lambda_i \widetilde d_i&=0 \label{eq:exacttriangle2} \\[1mm]
        \widetilde d_{i-2}\widetilde K_i + \widetilde K_{i} \widetilde d_i +  \widetilde \lambda_{i-1}\widetilde \lambda_i & =0 \label{eq:exacttriangle3}
    \end{align}
    We further require that for each $i\in \Z/3$ the morphism given by
    \begin{equation}
       \widetilde \lambda_{i-2} \widetilde K_i -   \widetilde K_{i-1} \widetilde \lambda_i : (\widetilde C_{i}, \widetilde d_i ,\chi_i) \longrightarrow (\widetilde C_{i} , -\widetilde d_i,-\chi_i) \label{eq:exacttriangle4} 
    \end{equation}
    is $\cS$-chain homotopy equivalent to an isomorphism. \; $\diamd$ 
\end{definition}

It is convenient to depict all of the data in an exact triangle as follows:
\begin{equation}\label{eq:exacttrianglegen}
	\begin{tikzcd}[column sep=5ex, row sep=10ex, fill=none, /tikz/baseline=-10pt]
& \widetilde C_2  \arrow[dr, "\widetilde \lambda_2"] \arrow[dl, bend right=60, "\widetilde K_2"'] & \\
\widetilde C_0  \arrow[rr, bend right=60, "\widetilde K_0"']  \arrow[ur, "\widetilde \lambda_0"] & & \widetilde C_1 \arrow[ll, "{[-1]}"', "\widetilde \lambda_1"] \arrow[ul, bend right=60, "\widetilde K_1"']
\end{tikzcd}
\end{equation}
The symbol $[-1]$ in the diagram reminds us that the degree of $\widetilde \lambda_1$ is $-1$, while $\widetilde \lambda_0$ and $\widetilde \lambda_2$ are degree $0$ morphisms. The last condition in the definition of an exact triangle asserts for each $i\in \Z/3$ the existence of a map $\widetilde N_i:\widetilde C_i\to \widetilde C_i$ commuting with the $\chi$-map and such that
\begin{equation}\label{eq:antischainhomotopy}
    \widetilde d_i \widetilde N_i - \widetilde N_i \widetilde d_i   + \widetilde \lambda_{i-2} \widetilde K_i -   \widetilde K_{i-1} \widetilde \lambda_i 
\end{equation}
is an isomorphism. We could have well included the data of the $\widetilde N_i$ in the definition of exact triangle, but for our purposes this makes no difference. 
The proof of the following is the same as for the case of the usual triangle detection lemma.

\begin{prop}
    Suppose we have an exact triangle as in \eqref{eq:exacttrianglegen}. Then we have $\cS$-chain homotopy equivalences given as follows:
    \[
    	\widetilde C_0 \simeq {\rm{Cone}}(\widetilde{\lambda}_2)[1], \qquad \widetilde C_1 \simeq {\rm{Cone}}(\widetilde{\lambda}_0), \qquad \widetilde C_2 \simeq {\rm{Cone}}(\widetilde{\lambda}_1).
    \]
\end{prop}

\noindent In this statement, the notation $\widetilde C[1]$ is short hand for the $\cS$-complex whose underlying graded $R$-module is given by $\widetilde C[1]$, and which has the same differential and $\chi$-map as $\widetilde C$.

Given an exact triangle of $\cS$-complexes as above we obtain an exact triangle, i.e. an associated long exact sequence, at the level of total homology: 
\begin{equation*}\label{eq:exacttrianglehom}
	\begin{tikzcd}[column sep=1ex, row sep=8ex, fill=none, /tikz/baseline=-10pt]
& H_\ast(\widetilde C_2 , \widetilde d_2) \arrow[dr, "(\widetilde\lambda_2)_\ast"] & \\
H_\ast(\widetilde C_0, \widetilde d_0) \arrow[ur, "(\widetilde\lambda_0)_\ast"] & & H_\ast(\widetilde C_1, \widetilde d_1) \arrow[ll, "{[-1]}"', "(\widetilde\lambda_1)_\ast"] 
\end{tikzcd}
\end{equation*}
We also obtain exact triangles for irreducible homology and reducible homology: 
\begin{equation*}\label{eq:exacttriangleirrred}
	\begin{tikzcd}[column sep=1ex, row sep=8ex, fill=none, /tikz/baseline=-10pt]
& H_\ast( C_2, d_2)  \arrow[dr, "(\lambda_2)_\ast"] & \\
H_\ast( C_0, d_0 ) \arrow[ur, "(\lambda_0)_\ast"] & & H_\ast(C_1, d_1) \arrow[ll, "{[-1]}"', "(\lambda_1)_\ast"] 
\end{tikzcd}\qquad 
	\begin{tikzcd}[column sep=1ex, row sep=8ex, fill=none, /tikz/baseline=-10pt]
& H_\ast(\sfR_2,r_2)  \arrow[dr, "(\rho_2)_\ast"] & \\
H_\ast(\sfR_0,r_0)  \arrow[ur, "(\rho_0)_\ast"] & & H_\ast(\sfR_1,r_1)  \arrow[ll, "{[-1]}"', "(\rho_1)_\ast"] 
\end{tikzcd}
\end{equation*}

\noindent In the sequel we deal mostly with $r$-perfect $\cS$-complexes, in which case the last exact triangle is simply an exact triangle relating the graded modules $\sfR_0,\sfR_1,\sfR_2$. 

\newpage


\newcommand{\qo}{{\fo}}
\newcommand{\orient}{{o}}
\newcommand{\sbd}{c}

\section{Invariants for links with non-zero determinant}\label{sec:links}

In this section we explain how the constructions in \cite{DS1} adapt to the case of a link in an integer homology 3-sphere with non-zero determinant. After discussing reducibles and gradings, we define the $\cS$-complex associated to a based link with non-zero determinant. We then discuss unobstructed cobordism maps in this setting. Along the way, notation and conventions for moduli spaces are introduced, mostly in alignment with \cites{DS1,DS2}. This section emphasizes the case of trivial bundles, which is most relevant to Theorem \ref{thm:exacttriangles}. For related material where the bundles are non-trivial, see Section \ref{sec:nontrivbundles}. In the following we assume familiarity with Sections 2 and 3 of \cite{DS1}.

\subsection{Reducibles and gradings}\label{subsec:gradings}

Let $(Y,L)$ be a pair consisting of an oriented integer homology 3-sphere and a link $L\subset Y$. For this pair we consider the setup of singular $SU(2)$ gauge theory in which the model connection near each component of $L$ is just as in the case for a knot, studied in \cite{DS1}: in particular, the holonomy of a singular connection around a sequence of shrinking meridians around any component of $L$ converges to a traceless element of $SU(2)$. For more details on singular connections and related technical aspects, see \cite{km-embedded-i, KM:YAFT, KM:unknot, DS1}.

We assume that the underlying bundle supporting the singular connections may be topologically extended to a trivial $SU(2)$ bundle over $Y$. The unperturbed critical set $\fC$ of the singular Chern--Simons functional is the set of flat singular connections modulo gauge, and is in correspondence, via holonomy, with the traceless character variety
\[
	\fX(Y,L) := \{\rho:\pi_1(Y\setminus L)\to SU(2): \text{tr} \rho(\mu_k)= 0 \}/SU(2)
\]
where $\mu_k$ ranges over meridians of the components $L_1,\ldots, L_n$ of $L$. Let $\text{Or}(Y,L)$ be the set of orientations of $L$. There is an involution on $\text{Or}(Y,L)$ which sends $\orient\in \text{Or}(Y,L)$ to $-\orient$, obtained from $\orient$ by reversing the orientation of each component. The classes in 
\[
     \mathcal{Q}=\mathcal{Q}(Y,L):=\text{Or}(Y,L)/\pm
\]
are called {\emph{quasi-orientations}}. A quasi-orientation $\qo=\{\orient,-\orient\}$ determines a reducible $[\rho_\qo]\in \fX(Y,L)$ as follows. Let $\mu_1,\ldots,\mu_k$ be meridians of $L$ oriented, via $\orient$, using the right-hand rule. Then $\rho_\qo$ factors through $H_1(Y\setminus L;\Z)=\bigoplus_{k=1}^n \Z\cdot \mu_k$ and is determined by
\[
	\rho_\qo(\mu_k) = i \in SU(2).
\]
Using $-\orient$ instead of $\orient$ gives a representation related to $\rho_\qo$ via conjugation by $j$ and thus $[\rho_\qo]$ only depends on $\qo$. The correspondence sending $\qo\in \mathcal{Q}(Y,L)$ to the class $[\rho_\qo]\in \fX(Y,L)$ is easily seen to induce a bijection from $\mathcal{Q}$ to the set of reducible elements of $\fX(Y,L)$. We write $\theta_\qo$ for the reducible singular flat connection class in $\fC$ corresponding to the quasi-orientation $\qo$, so that we have the following decomposition:
\[
	\fC=  \fC^\text{irr} \cup \{ \theta_\qo\}_{\qo\in \mathcal{Q}(Y,L)}
\]
Note that the reducibles $\theta_\qo$ have equivalent adjoint $SO(3)$ connections.

We recall that the determinant of a link is given by $\det(Y,L)=|\Delta_{Y,L}(-1)|$ where $\Delta_{Y,L}$ is the single-variable Alexander polynomial. The key observation for links with non-zero determinant, for our purposes, is the following.

\begin{prop}\label{prop:detnonzeronondeg}
    Suppose the determinant of $L\subset Y$ is non-zero. Then the set of reducibles $ \{ \theta_\qo\}_{\qo\in \mathcal{Q}(Y,L)}$ is non-degenerate and forms a discrete and isolated subset of $\fC$.
\end{prop}

\begin{proof}
    Let $\widetilde Y$ be the double cover of $Y$ branched over the link $L$. The adjoint $SO(3)$ connection of each reducible pulls back to a smooth trivial connection on $\widetilde Y$. The Zariski tangent space of a reducible in $\fC$ can then be identified with the subspace of $H^1(\widetilde Y;\R)\otimes \R^3$ which is invariant under the covering involution tensor $\text{diag}(1,-1,-1)$; see the discussion in \cite[Section 2.7]{DS1}. The hypothesis $\det(Y,L) \neq 0$ is equivalent to $H^1(\widetilde Y;\R)=0$, and so these tangent spaces vanish.
\end{proof}

Following \cite[Section 2.5]{DS1}, we may always choose a small perturbation $\pi$ of the Chern--Simons functional so that the reducibles remain isolated and non-degenerate critical points, and so that the perturbed critical set $\fC_\pi = \fC^\text{irr}_\pi \cup \{ \theta_\qo\}_{\qo\in \mathcal{Q}(Y,L)}$ is non-degenerate.

Choose a quasi-orientation $\qo \in \mathcal{Q}(Y,L)$ with the associated reducible $\theta_\qo$. We may define an absolute $\Z/4$-grading on $\alpha\in \fC_\pi$ denoted $\text{gr}[\qo]$ by the following formula:
\begin{equation}\label{eq:mod4grdef}
	\text{gr}[\qo](\alpha) :\equiv \text{gr}_z(\alpha,\theta_\qo)+ \dim\text{Stab}(\alpha)  \pmod{4}
\end{equation}
where $z$ is any homotopy class of paths from $\alpha$ to the distinguished reducible $\theta_\qo$. As in \cite{DS1}, $\text{gr}_z(\alpha,\theta_\qo)$ is the index of the associated linearized ASD operator with exponential decay at the ends. For instance, if $z$ is the constant homotopy class based at $\theta_\fo$, then we have $\text{gr}_z(\theta_\fo,\theta_\qo)=-1$, which implies that $\text{gr}[\qo](\theta_\qo)\equiv 0$ (mod 4) because $\text{Stab}(\theta_\fo)\cong S^1$. 

We may compare the gradings defined with respect to different reducibles in the following way. First note that additivity of the indices of the ASD operators implies that 
\[
  \text{gr}[\qo'](\alpha)\equiv\text{gr}[\qo](\alpha)-\text{gr}[\qo](\theta_{\qo'}) \pmod{4}.
\]
Applying Corollary \ref{index-cyl-red}, to be proved in Subsection \ref{sec:unobscobmaps}, for a path $z$ from $\theta_{\fo'}$ to $\theta_{\fo}$, we have
\[
  \text{gr}_z(\theta_{\fo'},\theta_\qo)\equiv\sigma(Y,L,\fo) - \sigma(Y,L,\fo')-1 \pmod{4}.
\]
By combining the above formulas we obtain
\begin{equation}\label{red-change-grading}
  \text{gr}[\qo](\alpha)-\text{gr}[\qo'](\alpha)\equiv \sigma(Y,L,\fo) - \sigma(Y,L,\fo') \pmod{4}.
\end{equation}
This relation can be rewritten in terms of $\lambda(Y,L,\fo)$, which is defined as
\[
  \lambda(Y,L,\fo) := -\sum_{1\leqslant i < j \leqslant n}\text{{lk}}_\fo(L_i,L_j)  
\]
for a link $L\subset Y$ in a homology 3-sphere with a quasi-orientation $\qo$. (Note that the linking number $\text{lk}_o(L_i,L_j)$ defined for a link $L_i\cup L_j$ oriented by $o$ depends only on the quasi-orientation $\fo=\{o,-o\}$, so we write $\text{lk}_\fo(L_i,L_j)$.) The signature of a link satisfies the following relation under a change of quasi-orientation:
\begin{equation}
	\sigma(Y,L,\fo) - \sigma(Y,L,\fo') =  \lambda(Y,L,\fo) - \lambda(Y,L,\fo') \label{eq:signaturechangeqo}
\end{equation}
This is due to Murasugi \cite{murasugi} in the case that $Y$ is the $3$-sphere, and generalizes to integer homology 3-spheres by the argument in \cite{kauffman-taylor}, for example. Note that the left hand side of \eqref{eq:signaturechangeqo} is an even integer. The identities in \eqref{red-change-grading} and \eqref{eq:signaturechangeqo} lead to the following.
\begin{prop} \label{prop:grdiff}
	For any $(Y,L)$ with non-zero determinant and quasi-orientations $\fo$ and $\fo'$, 
	the difference between the mod $4$ gradings $\text{{\emph{gr}}}[\fo]$ and $\text{{\emph{gr}}}[\fo']$ 
	is given by an overall shift: 
	\[
	  \text{{\emph{gr}}}[\qo] - \text{{\emph{gr}}}[\qo'] \equiv   \lambda(Y,L,\fo) - \lambda(Y,L,\fo') \pmod{4}.
	\]
	In particular, the mod $2$ value of $\text{{\emph{gr}}}[\qo]$ does not depend on $\fo$, and it induces 
	an absolute $\Z/2$-grading on $\fC_\pi$.
\end{prop}
    
Identity \eqref{eq:signaturechangeqo} implies that 
    $
        \sigma(Y,L) \in \Z/2,
    $
    given by the parity of $\sigma(Y,L,\fo)$ for any $\fo$, is a well-defined quantity independent of quasi-orientation. In fact, this quantity is given as follows. Let $|L|=n$ be the number of components of $L$. Define the {\emph{nullity}} of the link to be
    \[
    	\eta(Y,L) := b_1(\widetilde Y)
    \] 
    where $\widetilde Y$ is the double cover of $Y$ branched over $L$. 

    \begin{prop} $\sigma(Y,L) \equiv |L| + 1 + \eta(Y,L) \pmod{2}$.\label{prop:sigmod2}
    \end{prop}

    \begin{proof}
     For links in the 3-sphere, this is proved in \cite[Theorem 1]{shinohara} and \cite[Corollary 3.5]{kauffman-taylor}. The proof in the latter reference easily adapts to our current setting, replacing the $3$-sphere with an arbitrary integer homology $3$-sphere throughout.
    \end{proof}

     We remark that under the assumption that $\det(Y,L) \neq 0$, the nullity of the link is $0$, and the above formula simplifies to $\sigma(Y,L)\equiv |L|+1$ (mod $2$).

\subsection{The \texorpdfstring{$\cS$}{S}-complex of a link}\label{subsec:linkinv}

We now explain how the construction in \cite[Section 3]{DS1} of associating an $\cS$-complex to a knot using singular instanton Floer theory adapts in a straightforward manner to the more general setting of links with non-zero determinant. 

Let $(Y,L)$ be an integer homology 3-sphere with link, and assume $\det(Y,L)\neq 0$. Choose an orbifold Riemannian metric on $(Y,L)$ with cone angle $\pi$ along $L$, and choose a generic small perturbation $\pi$ of the Chern--Simons functional so that the critical set
\[
\fC_\pi =  \fC_\pi^\text{irr} \cup \{ \theta_\qo\}_{\qo\in \mathcal{Q}(Y,L)}
\]
is non-degenerate. For $\alpha,\beta\in \fC_\pi$ write $M(\alpha,\beta)$ for the moduli space of singular $SU(2)$ instantons on $\R\times (Y,L)$ with finite energy, as defined in \cite[Section 2]{DS1}. We always assume small generic perturbations are chosen such that the irreducible strata of these moduli spaces are smooth and unobstructed.

We briefly review some conventions from \cite{DS1}. Write $\breve{M}(\alpha,\beta)$ for the quotient of $M(\alpha,\beta)$ by the translation action, excluding the constant trajectories. For a homotopy class of paths $z$ in the configuration space from $\alpha$ to $\beta$ we let $M_z(\alpha,\beta)$ be the moduli space of instantons in the class $z$. Note $\text{gr}_z(\alpha,\beta) = \dim M_z(\alpha,\beta)$ when the moduli space $M_z(\alpha,\beta)$ consists of irreducibles and is non-empty. Write $M(\alpha,\beta)_d$ for the union of moduli spaces $M_z(\alpha,\beta)$ with dimension $d$, and $\breve M (\alpha,\beta)_{d-1}$ for its quotient.

Fix a commutative ring $R$. Recall that an $\cS$-complex $\widetilde C = C\oplus C[-1] \oplus \sfR$ over $R$ contains the data of two chain complexes, the {\emph{irreducible}} complex $(C,d)$ and the {\emph{reducible}} complex $(\sfR, r)$. We begin by defining these complexes for $(Y,L)$. We set
\begin{equation}\label{eq:irrcomplexdefn}
    C = C(Y,L) = \bigoplus_{\alpha \in \fC_\pi^\text{irr}} R\cdot \alpha  
\end{equation}
The differential $d$ is then defined by counting isolated instantons from $\alpha$ to $\beta$, i.e. $\langle d(\alpha),\beta \rangle = \# \breve{M}(\alpha,\beta)_0$. The proof that $d^2=0$ is the usual one, and we denote by 
\[
  I(Y,L) = H_\ast( C(Y,L), d )  
\]
the {\emph{irreducible}} instanton homology of $(Y,L)$. Next, we define the reducible complex 
\begin{equation}\label{eq:redcomplexdefn}
    \sfR = \sfR(Y,L) = \bigoplus_{\qo \in \cQ(Y,L)} R\cdot \theta_\qo  
\end{equation}
with zero differential $r=0$. Note that $\sfR$ is a free $R$-module of rank $2^{|L|-1}$.

From our discussion in the previous subsection, upon fixing a quasi-orientation $\qo$ of $(Y,L)$, each of the complexes defined thus far has a $\Z/4$-grading given by $\text{gr}[\qo]$. Furthermore, the induced $\Z/2$-grading is independent of $\qo$, and the gradings of all reducibles are even.

We now turn to define the total $\cS$-complex. To do this we require the choice of a basepoint $p\in L$. As an $R$-module, the $\cS$-complex is given by
\[
  \widetilde C = \widetilde C(Y,L,p) = C\oplus C[-1] \oplus \sfR 
\]
The differential $\widetilde d$ has non-zero components $d, \delta_1,\delta_2,v$, and $d$ has already been defined. For each $\qo\in \cQ(Y,L)$, we define $\delta_1:C\to \sfR$ and $\delta_2:\sfR\to C$ by 
\[
    \langle  \delta_1(\alpha) , \theta_\qo \rangle  = \# \breve{M}(\alpha,\theta_\qo)_0, \qquad \delta_2(\theta_\qo) = \sum \#\breve{M}(\theta_\qo,\alpha)_0
\]
where the sum is over $\alpha\in \fC_\pi^\text{irr}$. Then we have
\begin{equation}\label{eq:deltarels}
    \delta_1d = 0, \qquad d\delta_2 = 0.
\end{equation}
Write $\delta_1^\qo:C\to R$ for $\langle \delta_1 , \theta_\qo \rangle$ and $\delta_2^\qo:R\to C$ for $\delta_2(\theta_\qo)$. Then \eqref{eq:deltarels} expand to the set of relations $\delta_1^\qo d=0 $ and $d\delta_2^\qo  =0$ as $\qo$ ranges over $\cQ(Y,L)$. These are proved just as in the case for knots by counting the ends of the moduli spaces $\breve{M}(\alpha,\theta_\qo)_1$ and $\breve{M}(\theta_\qo,\alpha)_1$.

Finally, we define $v:C\to C$. This is the only data in the $\cS$-complex that depends on the basepoint $p\in L$. For each $\alpha,\beta\in \fC_\pi^\text{irr}$ there is a holonomy map $H_{\alpha\beta}:\breve{M}(\alpha,\beta)\to S^1$. Roughly, $H_{\alpha\beta}([A])$ is defined by first taking the adjoint $SO(3)$-connection of the instanton $[A]$, and then computing the holonomy along the line $\R\times \{ p \}$. As $p\in L$, the connection has a preferred reduction over this line and the holonomy lies in $S^1=SO(2)$. Then 
\[
    \langle v(\alpha) , \beta \rangle = \text{deg}\left( H_{\alpha\beta}: \breve{M}(\alpha,\beta)_1 \to S^1 \right).  
\]
In practice, the holonomy map $H_{\alpha\beta}$ must be modified to satisfy certain technical conditions so that its degree makes sense. For more details, see \cite[Section 3]{DS1}.

\begin{prop}
    $d v-v d -\delta_2\delta_1= 0$.
\end{prop}

\begin{proof}
The proof is essentially the same as \cite[Proposition 3.16]{DS1}. The idea is to count the ends of the moduli space $H_{\alpha\beta}^{-1}(g)\subset \breve{M}(\alpha,\beta)_2$ where $g$ is some generic element of $S^1\setminus \{1\}$. The terms $dv$ and $-vd$ come from factorizations of trajectories in 
\[
	\breve{M}(\alpha,\alpha')_1\times \breve{M}(\alpha',\beta)_0 \quad \text{ and } \quad \breve{M}(\alpha,\alpha')_0\times \breve{M}(\alpha',\beta)_1.
\]
The other possible ends come from breakings along the reducibles. Accounting for all of these contributes the term $-\delta_2\delta_1 =-\sum_{\qo\in\mathcal{Q}}\delta^\qo_2  \delta^\qo_1$. The details are explained in the proof of the above citation, where some technical modifications to this outline are also given.
\end{proof}

\begin{remark}\label{rmk:orientcyl}
    We use orientation conventions of moduli spaces which naturally extend the conventions from \cite[Section 2.9]{DS1}. See Subsection \ref{orientation} for details. $\diamd$ 
\end{remark}

\begin{prop}\label{prop:scomplexinvariance}
Let $(Y,L,p)$ be an integer homology 3-sphere with a link whose determinant is non-zero and a basepoint $p\in L$. Then the $\cS$-chain homotopy equivalence class of the $\Z/2$-graded $\cS$-complex $\widetilde C(Y,L,p)$ is an invariant of $(Y,L,p)$.
\end{prop}

Adapting the terminology of Definition \ref{defn:rperfect}, observe that $\widetilde C(Y,L,p)$ is an $r$-perfect $\Z/2$-graded $\cS$-complex, as the gradings of all reducibles are even. From our discussion in the previous subsection, we can upgrade $\widetilde C(Y,L,p)$ to a $\Z/4$-graded $\cS$-complex upon choosing a quasi-orientation for $L$. In particular, we always have a well-defined relative $\Z/4$-grading. The proof of Proposition \ref{prop:scomplexinvariance} is the same as \cite[Theorem 3.33]{DS1}, using the cobordism maps defined in the next subsection. We write 
\[
  \widetilde I(Y,L,p) = H_\ast ( \widetilde C(Y,L,p) , \widetilde d)  
\]
\[
   I(Y,L) = H_\ast ( C(Y,L) ,  d)  
\]
\[
  \sfR(Y,L) \cong R^{2^{|L|-1}}
\]
for the total homology, irreducible homology, and reducible homology, respectively. Note that the latter two of these do not depend on the basepoint, as indicated in the notation. We refer to $\widetilde I(Y,L,p)$ as the {\emph{framed}} instanton homology of the based link $(Y,L,p)$, and $I(Y,L)$ as the {\emph{irreducible instanton homology}} of the link.

There is another type of map that arises in the above setting:

\begin{definition}\label{defn:smap}
	For a based link $(Y,L)$ with $\cS$-complex $\widetilde C(Y,L,p)$ as above define
\begin{gather}
	s:\sfR(Y,L)\longrightarrow \sfR(Y,L), \label{eq:smapdef} \\[1mm]
	 \phantom{\;\;\; \diamd} \langle s(\theta_\fo),\theta_{\fo'}\rangle := \# \breve{M}(\theta_\fo , \theta_{\fo'})_0. \;\;\; \diamd \nonumber
\end{gather}
\end{definition}
\noindent The moduli spaces appearing in the above definition consist of irreducible instantons with reducible limits. Further, $\breve{M}(\theta_\fo , \theta_{\fo'})_0$ can only be non-empty when the two reducibles have a relative $\Z/4$-grading difference of $2\pmod{4}$, for index reasons; see Proposition \ref{prop:redind}. In particular, in the case that there is one reducible $\theta_\fo$, i.e. when $L$ is a knot, the only relevant moduli space is empty and the map $s$ is zero.

\vspace{.4cm}

\noindent \textbf{Example.} Let $H=L_1\cup L_2\subset S^3$ be the Hopf link. The traceless character variety $\fX(S^3,H)$ consists of two reducibles, corresponding to the two quasi-orientations $\fo_+$ and $\fo_-$ of the Hopf link, and no irreducibles. The quasi-orientation $\fo_{\pm}$ is chosen so that it makes $H$ into a link with $\pm$-crossings. Then $\widetilde C=\sfR=R\cdot \theta_{\fo_+}\oplus R\cdot \theta_{\fo_-}$. The differential $\widetilde d$ is zero. Fix a $\Z/4$-grading using either of the quasi-orientations. Note that $\theta_{\fo_+}$ and $\theta_{\fo_-}$ differ in $\Z/4$-grading by $2\cdot  \text{lk}(L_1,L_2) \equiv 2 $ (mod $4$). Thus
\[
  \widetilde I(S^3,H,p) = \sfR(S^3,H) \cong R_{(0)}\oplus R_{(2)}, \qquad  I(S^3,H) =   0.
\]
Here the point $p\in H$ can of course be any basepoint on $H$. The map $s$ is zero. However, if local coefficients over $\Z[T^{\pm 1}]$ are used, then $s(\theta_{\fo_-})=(T^2-T^{-2})\theta_{\fo_+}$, up to a unit of the ring, and $s(\theta_{\fo_-})=0$. See Proposition \ref{prop:smaphopf}. $\diamd$
\vspace{.4cm}

Recall that in the formation of the $\cS$-complex $\widetilde C(Y,L,p)$, a choice of metric and perturbation data is required. The map $s:\sfR\to \sfR$ also depends on these choices, in general. A different choice of such data leads to another $\cS$-complex $\widetilde C(Y,L,p)'$ which is $\cS$-chain homotopy equivalent to $\widetilde C(Y,L,p)$ via a morphism $\widetilde\lambda:\widetilde C(Y,L,p)\to \widetilde C(Y,L,p)'$, with an associated map $s':\sfR\to\sfR$. Then the maps $s$ and $s'$ differ as follows:
\begin{equation}\label{eq:smapdependency}
	s-s' = \delta'_1 \Delta_2  +  \Delta_1\delta_2
\end{equation}
where $\Delta_1:C(Y,L)\to \sfR$ and $\Delta_2:\sfR\to C(Y,L)'$ are components of $\widetilde \lambda$. This relation is a particular case of one for more general cobordisms, see \eqref{rel-tau-i+1}.

\begin{remark}
The importance of the maps \eqref{eq:smapdef}, from our viewpoint, is their appearance in the proofs of the exact triangles of Theorems \ref{thm:exacttriangles} and \ref{thm:exacttriangles-withbundles}. They first appear when constructing cobordism maps which have obstructed reducibles, in Section \ref{sec:obstructed}. $\diamd$
\end{remark}

\begin{remark}
An alternative to our Definition \ref{defn:scomplex} of $\cS$-complex might have included the data of a map $s:\sfR\to \sfR$, and the algebra of Section \ref{sec:prelims} could have been developed with this extra structure in mind.  Similar remarks hold for the $\nu_i:\sfR\to \sfR'$ of Subsection \ref{unob-cob-map}. $\diamd$
\end{remark}

In the above construction of $\widetilde C(Y,L,p)$, the singular bundle data for the pair $(Y,L)$ is trivial in the sense that the singular connections are defined on a fixed $SU(2)$ bundle over $Y\setminus L$ which extends to a trivial $SU(2)$ bundle over $Y$. As discussed in \cite[Section 8.1]{DS1}, following the work of Kronheimer and Mrowka \cites{KM:YAFT, KM:unknot}, the situation is even simpler for certain non-trivial singular bundle data.

An {\emph{admissible link}} $(Y,L,\omega)$ is the data of an oriented closed $3$-manifold $Y$, a link $L\subset Y$, and an unoriented embedded 1-manifold $\omega\subset Y$ with $\partial \omega = \omega\cap L$; this data is subject to the condition that there exists a closed oriented surface $\Sigma\subset Y$ transverse to $L$ and $\omega$ such that either $\Sigma$ is disjoint from $L$ and has odd intersection with $\omega$, or  $\Sigma$ has odd intersection with $L$. The admissibility condition guarantees that there are no reducible critical points for the Chern--Simons functional in this setting. To an admissible link $(Y,L,\omega)$ with basepoint $p\in L$ we defined in \cite[Section 8.1]{DS1} an $\cS$-complex
\begin{equation}\label{eq:mappingconeadmissible}
    \widetilde C^{\omega}(Y,L,p) = C \oplus C[-1] 
\end{equation}
with $C=C^{\omega}(Y,L)$. Its differential $d$ is defined similarly to the previous case, using irreducible critical points; the reducible complex is $\sfR=0$; and the only other component of $\widetilde d$ is given by the $v$-map defined using the basepoint $p\in L$. We write $\widetilde I^\omega(Y,L,p)$ and $I^\omega(Y,L,p)$ for the total and irreducible homologies of this $\cS$-complex. The latter is the instanton homology of admissible links defined in \cite{KM:unknot}. In Subsection \ref{subsec:absz2gr} we will define absolute $\Z/2$-gradings on these homology groups.

In the sequel we will often omit the basepoint $p\in L$ from notation, and write 
\[
    \widetilde C(Y,L)  
\]
for the $\cS$-complex $\widetilde C(Y,L,p)$ defined above; we think of $(Y,L)$ as a {\emph{based link}} with basepoint suppressed. A similar remark holds for the $\cS$-complex $\widetilde C^\omega(Y,L)$ associated to a based admissible link.

\subsection{Index computations}\label{sec:unobscobmaps}

We now turn to some index computations, much of which are straightforward generalizations from \cites{DS1,DS2}. These formulas are used to compute the degrees of the cobordism maps which are defined in the next subsection. 

Let $(Y,L)$ and $(Y',L')$ be oriented homology $3$-spheres with embedded links. Suppose $(W,S):(Y,L)\to (Y',L')$ is a cobordism of pairs consisting of a compact oriented 4-manifold $W$ with $\partial W = -Y\sqcup Y'$ and a surface $S\subset W$ with $\partial S = -L \sqcup L'$. With conventions as in \cites{DS1,DS2}, recall that the {\emph{energy}} of a singular connection $A$ is given by 
\[
  \kappa(A ) = \frac{1}{8\pi^2} \int_{W^+\setminus S^+} \text{Tr}(F_{\text{ad}(A)}\wedge F_{\text{ad}(A)} ).
\]
Here $W^+$ (resp. $S^+$) is obtained from $W$ (resp. $S$) bt attaching cylindrical ends to the boundary, and $\text{ad}(A)$ is the adjoint $SO(3)$ singular connection. We will always assume that the 4-manifold $W$ is connected. In general, by a singular connection $A$ ``on $(W,S)$'' is really meant a connection on $(W^+,S^+)$, after having attached cylindrical ends.

Following \cite[Section 4.2]{KM:unknot}, singular $SO(3)$ bundle data may be constructed from an embedded unoriented surface with boundary $\sbd\subset W$ which intersects $S$ in a disjoint union of circles; along the circles, $\sbd$ is normal to $S$. In the above citation, the interior of $c$ may also instersect $S$ transversely in some points, but we omit this possibility for simplicity. The singular bundle restricts to a bundle on $W\setminus S$ with second Stiefel--Whitney class $w_2$ given by the Poincar\'{e} dual of the relative homology class of $\sbd$. A choice of singular $U(2)$ bundle data lifting the singular $SO(3)$ bundle data amounts to orienting the surface $\sbd$. 

Henceforth $\sbd$ will denote some choice of surface representing singular $SO(3)$ bundle data. Most relevant to Theorem \ref{thm:exacttriangles} is the case in which $\sbd$ is empty, i.e. the trivial singular $SU(2)$ bundle data. In this case $\sbd$ is often omitted from notation. There will be occasions in which we consider other choices for $\sbd$, however (see Section \ref{sec:nontrivbundles}). In the sequel, ``singular bundle data'' will typically refer to a choice of surface $c$ as above.

We now compute the index of a singular connection with reducible limits on a cobordism of pairs $(W,S)$ with trivial singular bundle data. Note that in the statement to follow, there is no constraint on the determinants of the links (although $Y,Y'$ are still integer homology 3-spheres), and also the surface $S$ is not necessarily orientable. 

\begin{prop}\label{prop:redind}
    Let $A$ be a singular connection on $(W,S):(Y,L)\to (Y',L')$ and suppose that $A$ has reducible flat limits $\theta_\fo$ and $\theta_{\fo'}$ along $(Y,L)$ and $(Y',L')$, respectively. Then the index of the linearized ASD operator associated to $A$ is given by:
    \begin{align*}
      {\text{\emph{ind}}}(A) &=   8\kappa(A) - \frac{3}{2}\left( \chi(W) + \sigma(W) \right) + \chi(S) + \frac{1}{2} S\cdot S \\[1mm] & \qquad + \sigma(Y,L,\fo) - \sigma(Y',L',\fo')  - \eta(Y,L) - \eta(Y',L') - 1.
    \end{align*}
    Here, $S\cdot S$ is computed relative to Seifert framings induced at the ends by $\fo$ and $\fo'$.
\end{prop}

\noindent The index here requires some clarification. Let $\sD_A = -d_A^\ast + d_A^+$ be the ASD operator associated to $A$, where weighted Sobolev spaces are used on the (co)domain so as to force exponential decay along both ends, just as in \cite[\S 2.5]{DS1}. Then $\text{ind}(A)=\text{ind}(\sD_A)$.

\begin{proof}
The outline of the proof is similar to that of \cite[Lemma 2.26]{DS1}. First, consider the case of the usual ASD operator $d^\ast + d^+:\Omega^1\to \Omega^0\oplus \Omega^+$ on a smooth compact oriented 4-manifold $X$ with boundary $Y$. Attach a cylindrical end $\R_{\geq 0}\times Y$ to the boundary, and use weighted Sobolev spaces so that there is exponential decay along the cylindrical end. The index of this operator, written $\text{ind}(X)$, is given by
\[
	\text{ind}(X) = -\frac{1}{2}(\sigma(X) + \chi(X)) - \frac{1}{2}(b_1(Y) + b_0(Y))
\]
This follows from the discussion in \cite[\S 3.3.1]{donaldson-book}.

Now suppose $\Sigma\subset X$ is an orientable surface which is null-homologous, and $(X,\Sigma)$ has boundary $(Y,L)$. Here, unlike in the previous paragraph, we assume $Y$ is an integer homology 3-sphere. Write $\fo$ for the quasi-orientation of $L$ inherited as the boundary of $\Sigma$. We may also assume that $H_1(X;\Z/2)=0$, which allows us to form $\widetilde X$, be the double cover of $X$ branched over $\Sigma$. The covering involution $\tau:\widetilde X\to \widetilde X$ lifts to an involution of $\R^3\times \widetilde X$ defined as follows, for $(v_1,v_2,v_3)\in \R^3$ and $x\in \widetilde X$:
\[
	\widetilde \tau( (v_1,v_2,v_3), x) = ((v_1,-v_2,-v_3),\tau(x))
\] 
Let $\Theta$ be the trivial connection on $ \widetilde X \times \R^3$. Clearly $\text{ind}(\Theta)=3\,\text{ind}(\widetilde X)$. The quotient $X/\tau$ is an orbifold with singular locus $\Sigma$, and the quotient of $\Theta$ by $\widetilde \tau$ is the adjoint of an $SU(2)$ singular connection $\Theta_X$ on $(X,\Sigma)$ with limit $\theta_\fo$. As described in \cite[Lemma 2.26]{DS1}, 
\[
	\text{ind}(\Theta_X) = \dim\left( \text{ker}(\sD_\Theta)^{\widetilde \tau^\ast} \right)-\dim \left(\text{coker}(\sD_\Theta)^{\widetilde \tau^\ast} \right)
	\]
	where the invariant subspaces are taken with respect to $\widetilde \tau^\ast = \tau^\ast\otimes \text{diag}(1,-1,-1)$. Here $\tau^\ast$ is the induced map of $\tau$ on differential forms. We thus have
	\[
		\text{ind}(\Theta_X) = \text{ind}(\widetilde X)^\tau_+ + 2 \,\text{ind}(\widetilde X)^\tau_-
	\]
	 where $\text{ind}(\widetilde X)^\tau_+$ is the index of the ordinary ASD operator but restricted to the $\pm 1$-eigenspace of $\tau^\ast$. Note $\text{ind}(\widetilde X)^\tau_+=\text{ind}(X)$ and $\text{ind}(\widetilde X)^\tau_+ + \text{ind}(\widetilde X)^\tau_-=\text{ind}(\widetilde X)$. This gives
	 \begin{align}
	 \text{ind}(\Theta_X) &= 2 \,\text{ind}(\widetilde  X) - \text{ind}(X) \label{eq:indexplusminus} \\[1mm]
	 									&= -\frac{3}{2}(\sigma(X)  + \chi(\Sigma)) - \sigma(Y,L,\fo) + \chi(\Sigma)-\eta(Y,L) -\frac{1}{2} \nonumber
	 \end{align}
	 In this computation we have used the identities
	 \begin{align*}
	 		\sigma(\widetilde X) &= 2\sigma(X) + \sigma(Y,L,\fo),\\[1mm]
	 		\chi(\widetilde X)&=2\chi(X) - \chi(\Sigma).
	 \end{align*}
	 See \cite[Corollary 3.5]{kauffman-taylor}, for example. We have also used that $b_1(Y)=0$, and that $b_1$ of the branched cover of $(Y,L)$ is by definition the nullity $\eta(Y,L)$.

Applying a similar construction to the case where $(X',\Sigma')$ has boundary $(-Y',L')$, and noting the orientation reversal, we obtain the formula
\[
	\text{ind}(\Theta_{X'}) =  -\frac{3}{2}(\sigma(X')  + \chi(\Sigma')) + \sigma(Y',L',\fo') + \chi(\Sigma')-\eta(Y',L') -\frac{1}{2}
\] 
Write $\theta_{\fo'}$ for the corresponding reducible limit of $\Theta_{X'}$.

Next, let $(W,S):(Y,L)\to (Y',L')$ and the connection $A$ be as in the statement of the proposition. By index addivity, we have
\begin{align}
  \text{ind}(\overline{A}) & =	\text{ind}(\Theta_{X}) + h^0(\theta_\fo) + h^1(\theta_\fo) + \text{ind}(A) \label{eq:indexaddedred} \\[1mm]
  &\qquad \qquad  + h^0(\theta_{\fo'}) + h^1(\theta_{\fo'}) + \text{ind}(\Theta_{X'}) \nonumber
\end{align}
where $\overline{A}$ is the singular connection on $\overline{W}=X\cup W\cup W'$ obtained by gluing $\Theta_X,A,\Theta_{X'}$. See for example \cite[\S 3.4]{bd}. Here $h^i(\theta_{\fo})$ is the dimension of the orbifold cohomology $H^i(Y;\text{ad}\theta_{\fo})$. This vector space is identified with the $\widetilde \tau^\ast$-invariant part of $H^i(\widetilde Y;\underline{\R^3})$ where $\widetilde Y$ is the branched cover of $(Y,L)$. Thus, similar to \eqref{eq:indexplusminus}, we have
\[
	h^1(\theta_\fo ) = 2b_1(\widetilde Y)-b_1(Y) = 2\eta(Y,L)
\]
and $h^0(\theta_{\fo})=2b_0(\widetilde Y)-b_0(Y)=1$. This latter quantity is also simply the dimension of the stabilizer of $\theta_\fo$, which is isomorphic to $S^1$. Similar remarks hold for $h^i(\theta_{\fo'})$.

The index of $\overline{A}$, appearing on the left side of \ref{eq:indexaddedred}, is computed from the formula of Kronheimer and Mrowka \cite[Lemma 2.11]{KM:unknot} for closed pairs:
\begin{equation}\label{eq:indclosedcase}
	\text{ind}(\overline{A}) = 8\kappa(\overline{A}) - \frac{3}{2}\left( \chi(\overline{W}) + \sigma(\overline{W}) \right) + \chi(\overline{S}) + \frac{1}{2} \overline{S}\cdot \overline{S} 
\end{equation}
Here the surface $\overline{S}\subset \overline{W}$ is the gluing of $\Sigma,S,\Sigma'$.

Finally, the formula in the proposition for $\text{ind}(A)$ follows from \eqref{eq:indexaddedred} and the computations given above, making use of the additivity of the terms appearing in \eqref{eq:indclosedcase}. 
\end{proof}

We remark that term $S\cdot S$ is the normal Euler number of $S\subset W$ and is a well-defined integer even in the case that $S$ is non-orientable. To define it, let $S'\subset W$ be a push-off of $S$ which is transverse to $S$ and at the boundary component $\partial S'$ is the Seifert push-offs of the links $L\subset Y$ and $L'\subset Y'$ with respect to the quasi-orientations $\fo$ and $\fo'$, respectively. For $p\in S\cap S'$ any orientation of $T_p S$ determines one of $T_p S'$. Define a local intersection number $\pm 1$ by comparing the orientation of $T_p S\oplus T_p S'$ with that of $T_p W$. The sum of these over all points $p\in S\cap S'$ is then $S\cdot S'$.

\begin{lemma}\label{lemma:4kappamodz}
	Let $A$ be a singular connection on $(W,S,c)$, where $(W,S)$ is a closed pair, and the singular bundle data is represented by a surface $c$ whose boundary lies on $S$.  Suppose $H_1(W;\Z/2)=0$ and $[S]= 0\in H_2(W;\Z/2)$. Then
		\begin{equation}\label{eq:4kappamodz}
		 4 \kappa(A) \equiv -\frac{1}{4}S\cdot S - \frac{1}{2} (\partial c|_S)\cdot (\partial c|_S) \pmod{\Z}
	\end{equation}
	where the intersection number of $\partial c|_S$ is computed, mod $2$, on $S$. Furthermore, if $c$ is empty, then \eqref{eq:4kappamodz} holds mod $2\Z$.
\end{lemma}

\begin{proof}
The topological hypotheses imply that there is a unique double cover of $W$ branched over $S$, which we denote by $\pi:\widetilde W\to W$. Let $\widetilde S=\pi^{-1}(S)$ and $\widetilde c=\pi^{-1}(c)$. Note both are closed surfaces. The singular connection $A$ pulls back via $\pi$ to a singular connection $\widetilde A$ on $\widetilde W$. Furthermore, the adjoint of $\widetilde A$ is a smooth $SO(3)$ connection on a bundle $P$. As the singular bundle data of $A$ is represented by $c$, and the lift $\widetilde A$ has holonomy $-1$ around meridians of $\widetilde S$, we obtain that $w_2(P)$ is Poincar\'{e} dual to $[\widetilde c]+[\widetilde S]$. Thus we have
\[
	4\kappa(A) = 2\kappa(\widetilde A) = -\frac{1}{2}p_1(P) \equiv -\frac{1}{2} w_2(P)^2\equiv -\frac{1}{2} ([\widetilde c]+[\widetilde S])^2 \pmod{\Z}
\]
From the identity $\widetilde S\cdot \widetilde S = \frac{1}{2}S\cdot S$, it remains to prove $\widetilde c\cdot \widetilde c\equiv (\partial c|_S)\cdot (\partial c|_S) \pmod{2}$. Let $c'\subset W$ be a pushoff of $c$ which is transverse to $c$ on its interior and such that $\partial c'\subset S$ is transverse to $\partial c$. Then $\widetilde c'=\pi^{-1}(c')$ is a closed surface transverse to $\widetilde c$. The covering involution of $\widetilde W$ acts freely on points of $\widetilde c'\cap \widetilde c$ that do not lie on $S$, and on $S$ we have
\[
	\# (\widetilde c\cap \widetilde c' \cap S) \equiv \# \partial c\cap \partial c' = (\partial c|_S)\cdot (\partial c|_S) \pmod{2}
\]
Thus $\widetilde c\cdot \widetilde c\equiv (\partial c|_S)\cdot (\partial c|_S) \pmod{2}$, as claimed. 

As to the last statement, from the relation $p_1(P)\equiv w_2(P)^2\pmod{4}$, where the latter quantity is interpreted in the general case using the Prontryagin square, we can write
\begin{equation}
	4\kappa(A) =  -\frac{1}{2} ([\widetilde c]+[\widetilde S])^2 \pmod{2} \label{eq:4kappaupgrade}
\end{equation}
If $c$ is empty the right side is simply $-\frac{1}{4}S\cdot S \pmod{2}$.
\end{proof}

The topological hypotheses of Lemma \ref{lemma:4kappamodz} are an artifact of the proof method, which uses branched covers. Indeed, in \cite[\S 2.4]{KM:unknot}, Kronheimer and Mrowka express $\kappa(A)$ in terms of certain characteristic numbers associated to the singular bundle data. In particular, it follows immediately from their description that \eqref{eq:4kappamodz} is true without the topological assumptions on $(W,S)$ in all cases in which $c$ is empty. On the other hand, the assumptions in Lemma \ref{lemma:4kappamodz} are satisfied for all examples we have in mind, and allow us to avoid discussing further technicalities of singular bundle data.

With these comments in mind, for the remainder of this section we assume for simplicity that $H_1(W;\Z/2)=0$ and $[S]=0\in H_2(W;\Z/2)$.\\

\noindent \textbf{Example.} The following examples are taken from \cite[\S 2.7]{KM:unknot}. There are singular instantons $A_{\pm}^i$ on $(S^4,\mathbb{R}\mathbb{P}^2_\pm)$, where $i\in \{0,1\}$. Here $S=\mathbb{R}\mathbb{P}^2_\pm \subset S^4$ satisfies $S\cdot S=\pm 2$, and $A_\pm^i$ has trivial singular bundle data ($c$ empty) for $i=0$, and if $i=1$ then $\partial c$ is an essential loop on $S$. The invariants for these connections are given in Table \ref{table:rp2examples}. This table verifies the relation \eqref{eq:4kappamodz} for these examples. $\diamd$\\

\begin{figure}[t]
\centering{
\renewcommand{\arraystretch}{1.5}
\begin{tabular}{r | rrrr}
  & $A_+^0$ & $A_+^1$  & $A_-^0$ & $A_-^1$ \\
\hline
$\kappa$ & $\tfrac{1}{8}$ & $0$ & $\tfrac{3}{8}$ & $0$ \\
$S\cdot S$ & $2$  & $2$ & $-2$ & $-2$ \\
$(\partial c|_S)^2$ & $0$  & $1$ & $0$ & $1$ \\
\end{tabular}
}
\captionof{table}[]{}\label{table:rp2examples}
\end{figure}

From Lemma \ref{lemma:4kappamodz} and the proof of Proposition \ref{prop:redind} we obtain the following.

\begin{cor}\label{cor:indexred}
Let $A$ be a singular connection on $(W,S,c):(Y,L)\to (Y',L')$, where $c$ has empty intersection with $Y$ and $Y'$. Suppose that $A$ has reducible flat limits $\theta_\fo$ and $\theta_{\fo'}$ along $(Y,L)$ and $(Y',L')$, respectively. Then the index of $A$ is congruent modulo 2 to
    \begin{align*}
      & - \frac{3}{2}\left( \chi(W) + \sigma(W) \right) + \chi(S)  + \sigma(Y,L,\fo) - \sigma(Y',L',\fo')\\[1mm] & \qquad \qquad  - \eta(Y,L) - \eta(Y',L') - 1 + (\partial c|_S)\cdot (\partial c|_S)
    \end{align*}
    and in particular only depends on $(W,S)$ and $\partial c\subset S$. If $c$ is empty, then this relation for the index, with $(\partial c|_S)\cdot (\partial c|_S)=0$, holds modulo $4$.
\end{cor}

Applying this to a cylinder gives the following, used in the proof of Proposition \ref{prop:grdiff}.

\begin{cor}\label{index-cyl-red}
	Let $(Y,L)$ have non-zero determinant, and let $A$ be a singular connection on the cylinder 
	$(W,S)=I\times (Y,L)$ with empty $c$, connecting the reducibles $\theta_\fo$ and $\theta_{\fo'}$. Then $\ind(A)\equiv \sigma(Y,L,\fo) - \sigma(Y',L',\fo')-1 \pmod{4}$.
\end{cor}

Motivated by Corollary \eqref{cor:indexred}, we introduce the following terminology.

\begin{definition}\label{defn:cobordismparity}
Let $(W,S,c):(Y,L)\to (Y',L')$ be a cobordism with bundle data $c$ which has empty intersection with $Y$ and $Y'$. Suppose also that the two links in homology spheres have non-zero determinant. The {\emph{parity}} of $(W,S,c)$ is determined by the quantity: 
	 \begin{equation}\label{eq:cobparity}
         \frac{1}{2}\left( \chi(W) + \sigma(W) \right) + \chi(S) +|L| + |L'| + (\partial c|_S)\cdot (\partial c|_S)   \pmod{2},
    \end{equation}
    which is the mod $2$ index of a connection with reducible limits shifted by $1$. We refer to $(W,S,c)$ as {\emph{even}} or {\emph{odd}} accordingly. $\diamd$
\end{definition}

\begin{remark}\label{rmk:reducibleevendeg}
The index of any reducible singular connection on a cobordism $(W,S,c)$ has the parity of $1+b^1(W)+b^+(W)$; see the discussion in Subsection \ref{orientation}, for example. If $b^1(W)=b^+(W)=0$, then the index of any reducible is odd. In particular, following Definition \ref{defn:cobordismparity}, if $(W,S,c)$ satisfies $b^1(W)=b^+(W)=0$ and admits a reducible singular instanton for some metric, then it is necessarily an even cobordism. $\diamd$
\end{remark}

\subsection{Unobstructed cobordism maps}\label{unob-cob-map}

The main result of this subsection is that unobstructed cobordisms induce morphisms between instanton $\cS$-complexes of links with non-zero determinants. Just as in Section \ref{sec:prelims}, where the algebra of $\cS$-complex morphisms was developed, we divide the discussion into two cases, depending on the parity of the cobordism. 

\subsubsection{Even cobordisms}

We first treat the case of even cobordisms. By Remark \ref{rmk:reducibleevendeg}, assuming $b^1=b^+=0$, this is the case that is relevant when reducible singular instantons are present.

\begin{definition}\label{eq:levelncobordismmaps}
Let $(W,S):(Y,L)\to (Y',L')$ be a cobordism between links in homology 3-spheres with singular bundle data $\sbd$, where $c$ has empty intersection with $Y$ and $Y'$. Assume $(W,S,c)$ is even. We say $(W,S,\sbd)$ is {\emph{negative definite of height $i\in \Z$}} if 
\begin{equation}
	b^1(W)=b^+(W)=0 \label{eq:negdefcond}
\end{equation}
and the index of every reducible instanton on $(W,S,\sbd)$ is at least $2i-1$. Furthermore, if $i\geqslant 0$, we say the cobordism with bundle is {\emph{unobstructed}}. $\diamd$
\end{definition}

Fix an even cobordism with bundle $(W,S,c)$ satisfying \eqref{eq:negdefcond} as in Definition \ref{eq:levelncobordismmaps}. The assumption $b^+(W)=0$ guarantees that the condition for the existence of reducible singular instantons of a particular index is in fact metric-independent, as are the possible energies of such instantons. Call a reducible instanton $A$ on this cobordism {\emph{minimal}} if its energy $\kappa(A)$ is minimal among all reducible instantons. Note the minimum is taken over {\emph{all}} reducible instantons on $(W,S,\sbd)$, not fixing the limiting reducible connections at the ends. Define
\setlength{\jot}{2ex}
\begin{gather*}
    \eta_i(W,S,\sbd): \sfR(Y,L) \longrightarrow \sfR(Y',L')    \vspace{.4cm}\\
        \langle\eta_i(W,S,\sbd)(\theta_\fo) , \theta_{\fo'}\rangle  = \#\{\text{index } 2i-1 \text{ reducibles with limits } \theta_\fo, \theta_{\fo'}\}
\end{gather*}
\noindent where ``$\#$'' means a signed count determined by orienting the moduli spaces. The condition \eqref{eq:negdefcond} guarantees that there is always a finite set of reducible instantons of any given fixed energy. Suppose minimal reducibles on $(W,S)$ have index $2m-1$. Define
\begin{gather*}
    \eta= \eta(W,S,\sbd) :=\eta_{m}(W,S,\sbd) 
\end{gather*}
With this notation, $\langle\eta(W,S,\sbd)(\theta_\fo) , \theta_{\fo'}\rangle$ is the signed count of minimal reducibles with limiting critical points $\theta_\fo$ and $\theta_{\fo'}$.

\begin{remark}
    Following Remark \ref{rmk:orientcyl}, the conventions for orienting moduli spaces of instantons on cobordisms extend those of \cite[Section 2.9]{DS1}. See Subsection \ref{orientation}. $\diamd$
\end{remark}

\begin{definition}\label{eq:stronglevelncobordismmaps}
    An unobstructed cobordism $(W,S,\sbd)$ is {\emph{strong height $i\in \Z_{\geqslant 0}$}} if the minimal reducibles all have index $2i-1$ and $\eta(W,S,\sbd)$ is an isomorphism. $\diamd$
\end{definition}

Note that if an even cobordism $(W,S,\sbd)$ satisfies \eqref{eq:negdefcond} and it supports no reducible instantons, then it is negative definite of height $i$ for all $i\in \Z$ and in particular it is unobstructed. However, it is not strong for any height.

To define a morphism of $\cS$-complexes we also need to choose an embedded interval $\gamma\subset S\setminus c$ whose endpoints $p\in L$ and $p'\in L'$ give basepoints on the links. We write $(W,S,\gamma):(Y,L,p)\to (Y',L',p')$ for a cobordism with such data.

\begin{theorem}\label{thm:unobsmaps}
    Suppose $(W,S,\gamma):(Y,L,p)\to (Y',L',p')$ with singular bundle data $\sbd$ (where $c$ has empty intersection with $Y,Y'$) is unobstructed, and each link has $\det\neq 0$. Then there is a well-defined morphism of $\cS$-complexes
    \[
           \widetilde \lambda =  \widetilde\lambda_{(W,S,\sbd),\gamma}:\widetilde C(Y,L,p)\to \widetilde C(Y',L',p')
    \]
    If $(W,S,c)$ is negative definite of (strong) height $i$, then $\widetilde \lambda$ is a morphism of (strong) height $i$. In this case the expression $\tau_i$ in Definition \ref{defn:leveln} is equal to $\eta(W,S,c)$. 
\end{theorem}

As in our convention of dropping basepoints $p, p'$ from notation, we often write
\[
    \widetilde \lambda =  \widetilde\lambda_{(W,S)}:\widetilde C(Y,L)\to \widetilde C(Y',L')
\]
for the morphism given above, omitting both the path $\gamma$ and the bundle data $c$ from notation. The map also depends on metric and perturbation data, which is surpressed.

By Proposition \ref{prop:levelsuspension}, a negative definite height $i\geqslant 0$ cobordism induces a morphism 
\[
  \Sigma^i \widetilde C(Y,L) \longrightarrow \widetilde C(Y',L')  
\]
and if it is strong height $i$, then this morphism is strong and has $\rho=\eta(W,S,c)$.

The proof of Theorem \ref{thm:unobsmaps} follows along the lines of \cite[Section 3]{DS1} and \cite[Proposition 4.17]{DS2}, and there are no essential differences in the analysis. We proceed to sketch the construction, along the way introducing notation for moduli spaces on cobordisms. 

Fix $(W,S,\sbd)$ as above. A cylindrical end metric on $W^+$ with cone angle $\pi$ along $S^+$ is fixed, as is a connection $A_0$ on the determinant bundle data. An {\emph{instanton}} is a singular connection on $W^+\setminus S^+$ with bundle and conditions near $S$ determined by $\sbd$, such that $\det(A)=A_0$ and $F_{\text{ad}(A)}$ satisfies a perturbed anti-self-duality equation. We write 
\[
    M(W,S,\sbd;\alpha,\alpha')
\]
for the moduli space of finite energy singular instantons mod determinant 1 gauge on $(W,S,\sbd)$ with limiting critical points $\alpha\in \fC_\pi$ and $\alpha'\in \fC_{\pi'}$ on the ends. When the bundle data $\sbd$ is understood or trivial, we omit it from the notation. 

Consider the space of all singular connections for $(W,S,\sbd)$, modulo gauge, with fixed determinant and limits $\alpha$ and $\alpha'$, and write $z$ for a connected component. We refer to $z$ as a homotopy class of paths. The moduli space $M(W,S;\alpha,\alpha')$ is a disjoint union
\[
	M(W,S;\alpha,\alpha')=\bigcup_z M_z(W,S;\alpha,\alpha')
\]
where $z$ ranges over the homotopy classes of paths from $\alpha$ to $\alpha'$. 

The expected dimension of $M_z(W,S;\alpha,\alpha')$ is given by $\text{ind}(A)=\ind(z)\in \Z$, the index of the linearized ASD operator $\sD_A$ with exponential decay conditions. Here $A$ is any connection in the homotopy class $z$. We always assume perturbations have been chosen so that the irreducible strata of the moduli spaces are smooth manifolds of the expected dimension (or empty). Write $M(W,S;\alpha,\alpha')_d$ for the instantons $[A]$ with index $d$.

We now describe $\widetilde \lambda$ by giving its morphism components as in \eqref{eq:lambdatilde}. First,
\begin{equation}\label{eq:lambdafirstdefn}
  \langle\lambda(\alpha),\alpha'\rangle = \# M(W,S;\alpha,\alpha')_0  
\end{equation}
for $\alpha\in \fC_\pi^\text{irr}$ and $\alpha'\in \fC_{\pi'}^{\text{irr}}$. The maps $\Delta_1$ and $\Delta_2$ are defined similarly, but with reducible limits at one end and irreducible limits at the other:
\[
  \langle\Delta_1(\alpha),\theta_{\fo'}\rangle = \# M(W,S;\alpha,\theta_{\fo'})_0  \qquad   \langle\Delta_2(\theta_\fo),\alpha'\rangle = \# M(W,S;\theta_{\fo},\alpha')_0  
\]
Recall that $\gamma\subset S$ is an interval with endpoints $p\in L$ and $p'\in L'$. This determines an embedded real line $\gamma^+\subset S^+$. There is an associated (modified) holonomy map
\[
  H_{\alpha,\alpha'}^\gamma: M(W,S;\alpha,\alpha')_1\longrightarrow S^1
\]
which roughly computes the holonomy of the adjoint $SO(3)$ connection $\text{ad}(A)$ along $\gamma^+$ from $-\infty$ to $+\infty$ after fixing frames for the bundle at $\pm \infty$. As $\gamma^+$ is on the singular locus, the connection has holonomy lying in $S^1\cong SO(2)\subset SO(3)$. Then $\mu$ is defined by
\[
    \langle \mu ( \alpha), \alpha' \rangle  = \#(H^\gamma)^{-1}(-1)
\]
Finally, the component $\rho$ is defined as follows:
\[
    \rho = \begin{cases} \eta(W,S,\sbd) \quad  & \text{minimal reducibles have index } -1 \\ 0 & \text{ otherwise}
    \end{cases}
\]
With these definitions, the proof that $\widetilde\lambda$ is a morphism is essentially the same as in \cite[Section 3]{DS1}, and the statement about (strong) heights follows \cite[Proposition 4.17]{DS2}.

\begin{remark}\label{tau-i+1}
	Theorem \ref{thm:unobsmaps} can be upgraded as follows. If $(W,S,c)$ is negative definite of height $i\geq 0$, then for the associated height $i$ morphism
	$\widetilde \lambda$, we also have
	\begin{equation}\label{rel-tau-i+1}
	  \tau_{i+1}=\eta_{i+1}(W,S,c)-s'\eta_{i}(W,S,c)+\eta_{i}(W,S,c)s
	\end{equation}
	where $s$ and $s'$ are as in Definition \ref{defn:smap}. This relation can be proved in the same way as in \cite[Proposition 4.17]{DS2}. For instance, if $i=0$,
	one obtains this relation by inspecting the ends of $1$-dimensional moduli spaces $M^+(W,S;\theta_{\fo},\theta_{\fo'})_1$ for reducibles $\theta_{\fo}$ and $\theta_{\fo'}$. 
	It is reasonable to expect that analogues of \eqref{rel-tau-i+1} hold for $\tau_j$ with $j>i+1$. This would require analyzing the effect of bubbling at reducible 
	instantons in the relevant moduli spaces. $\diamd$
\end{remark}

\subsubsection{Odd cobordisms}\label{subsubsec:odddegcobs}

As above, let $(W,S,c):(Y,L)\to (Y',L')$ be a cobordism between non-zero determinant links in homology $3$-spheres with singular bundle data $c$ which has empty intersection with $Y$, $Y'$. Suppose further that $(W,S,c)$ is odd, and assume \eqref{eq:negdefcond}. By Remark \ref{rmk:reducibleevendeg}, this cobordism with bundle supports no reducible singular connections. Extending terminology from Definition \ref{eq:levelncobordismmaps}, we refer to $(W,S,c)$ as an {\emph{unobstructed}} cobordism. For such an odd unobstructed cobordism $(W,S,c)$, the construction of an $\cS$-complex morphism
\[
	\widetilde \lambda = \widetilde \lambda_{(W,S,c),\gamma}: \widetilde C(Y,L,p) \to \widetilde C(Y',L',p')
\]
is just as in Theorem \ref{thm:unobsmaps}, except that reducible instantons are absent. In particular, the component $\rho$ is defined to be zero. 

However, unlike the even case, there are potentially non-empty moduli spaces
\[
	M(W,S,c;\theta_{\fo},\theta_{\fo'})_{0}
\]
of zero dimension, with reducible limits, consisting of irreducible singular instantons. We are led to the following, a cobordism-analogue of the $s$-map from Definition \ref{defn:smap}.

\begin{definition}\label{defn:numap}
	For an odd cobordism $(W,S,c):(Y,L)\to (Y',L')$ as above, define
\begin{gather*}
	\nu=\nu(W,S,c):\sfR(Y,L)\longrightarrow \sfR(Y',L'), \\[1mm]
	 \phantom{\;\;\; \diamd}  \langle \nu(\theta_\fo),\theta_{\fo'}\rangle := \# M(W,S,c;\theta_{\fo},\theta_{\fo'})_{0}. \;\;\; \diamd
\end{gather*}
\end{definition}

\noindent The map $\nu$ is in fact the first in a family of maps
\begin{equation*}\label{eq:nu-i-maps-cob}
	\nu_i=\nu_i(W,S,c):\sfR(Y,L)\longrightarrow \sfR(Y',L')
\end{equation*}
where $i\in \Z_{\geq 0}$ and $\nu=\nu_0$. These are defined, roughly, by cutting down moduli spaces $M(W,S,c;\theta_{\fo},\theta_{\fo'})_{2i}$ by the $i^{\text{th}}$ power of the degree $2$ cohomology class in the configuration space determined by the Euler class of the reduced basepoint fibration, where the basepoint is somewhere on the interior of the singular locus $S\subset W$, away from $c$. We omit details, as the maps $\nu_i$ for $i>0$ are not used in the sequel.

The role of the $\nu_i$ maps in the algebra of $\cS$-complex morphisms has already been foretold in Subsection \ref{subsec:odddegmorphisms}. Similar to the map $s$, the map $\nu=\nu_0$ appears in the algebra of obstructed cobordism maps as developed in Section \ref{sec:obstructed}, as well as in several relations in the proofs of the exact triangles in Sections \ref{sec:proofs} and \ref{sec:nontrivbundles}.

\subsubsection{Degrees of cobordism maps}

We now compute the degrees of cobordism maps. Let $(W,S,c):(Y,L)\to (Y',L')$ be a cobordism between non-zero determinant links in homology $3$-spheres, where the singular bundle data does not intersect $Y$, $Y'$. (For the case in which $c$ intersects the ends, see Section \ref{sec:nontrivbundles}.) Assume that the cobordism $(W,S,c)$ is unobstructed, so that there is an induced $\cS$-complex morphism $\widetilde \lambda:\widetilde C(Y,L)\to \widetilde C(Y',L')$.

\begin{prop}\label{prop:degmorphism}
    The mod $2$ degree of $\widetilde \lambda$ agrees with the parity of $(W,S,c)$, given by \eqref{eq:cobparity}. Further, if absolute $\Z/4$-gradings for the $\cS$-complexes of $(Y,L)$ and $(Y',L')$ are fixed by quasi-orientations $\fo$ and $\fo'$ respectively, and $c=\emptyset$, then the $\Z/4$ degree of $\widetilde \lambda$ is given by
    \[
        - \frac{3}{2}\left( \chi(W) + \sigma(W) \right) + \chi(S)  + \sigma(Y,L,\fo) - \sigma(Y',L',\fo')   \pmod{4}
    \]
\end{prop}

\begin{proof}
    We prove that $\lambda:C\to C'$ has the claimed degree; the other components are similar. In this case we take $\alpha\in \fC_\pi^\text{irr}$ and $\alpha'\in \fC_{\pi'}^\text{irr}$. Let $[A]\in M(W,S,\sbd;\alpha,\alpha')_0$. Then the mod $4$ degree of $\lambda$ with respect to $\text{gr}[\fo]$ and $\text{gr}[\fo']$ is congruent to
    \[
      \text{gr}[\fo'](\alpha') - \text{gr}[\fo](\alpha) = \text{gr}_{z'}(\alpha',\theta_{\fo'}) -  \text{gr}_z(\alpha,\theta_\fo)
    \]
    where $z,z'$ are any homotopy classes of paths. By index additivity we have the relation
    \[
        \text{ind}(A) + \text{gr}_{z'}(\alpha',\theta_{\fo'}) =   \text{gr}_z(\alpha,\theta_\fo) + \dim \text{Stab}(\theta_\fo) +  \ind(A')
    \]
    where $A'$ is some singular connection on $(W,S,\sbd)$ with reducible limits $\theta_\fo$ and $\theta_{\fo'}$. Note $\text{ind}(A)=0$ and $\dim \text{Stab}(\theta_\fo)=1$. The result then follows from Corollary \ref{cor:indexred}, which computes $\text{ind}(A')$, and also Proposition \ref{prop:sigmod2} for the mod 2 reduction.
\end{proof}

\begin{figure}[t]
    \centering
     \centerline{ \includegraphics[scale=0.5]{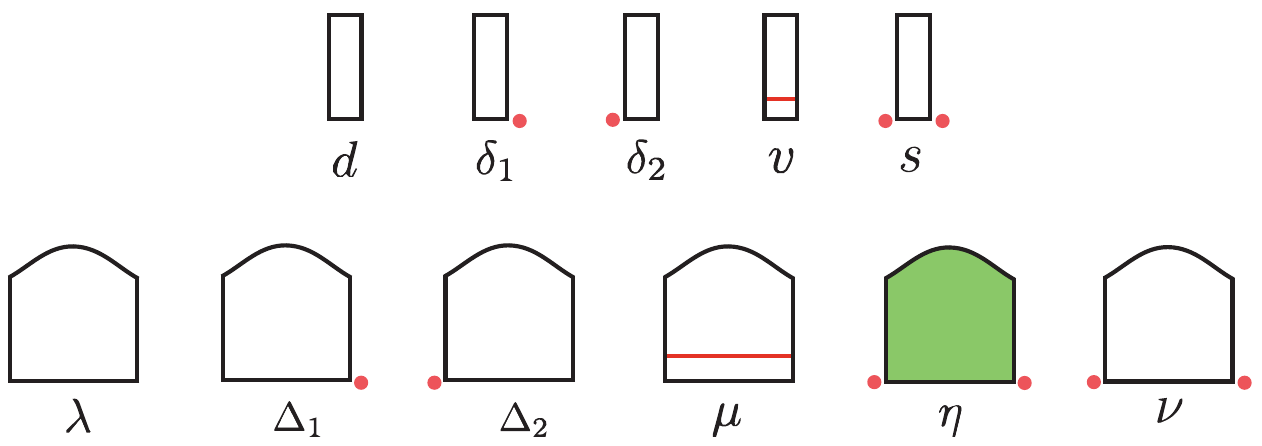}}
    \caption{{\small{Symbols for the maps defined for the cylinder (top) and a cobordism (bottom).}}}
    \label{fig:key1}
\end{figure}

We end this subsection by extending the pictorial calculus introduced in \cite{DS1}, see Figure \ref{fig:key1}. This will be useful for conceptualizing the relations in the proof of the exact triangles. The conventions are as follows. An undecorated cylinder or cobordism represents the map defined by counting isolated instantons (in the cylindrical case the trajectories are unparametrized). The presence of a horizontal line indicates that the moduli space is cut down by the holonomy along this line. The presence of a dot at an end indicates that the moduli spaces involve reducible limits at that end; in the absence of a dot the limits are irreducible critical points. Finally, if the cobordism is shaded (green), it indicates that the minimal (unobstructed) reducibles are counted.

With these conventions the terms in \eqref{eq:lambda4} for the cobordism map $\widetilde\lambda_{(W,S)}$ are depicted in Figure \ref{fig:pic-calc-ex}. The signs of the terms in the relation are not reflected in the picture.

\subsection{Cobordism maps in a cylinder}\label{subsec:cobcyl}

We now specialize the results of the previous subsection to the class of cobordisms most relevant for the proof of the exact triangle. In this case we take $W=I\times Y$ to be the product cobordism, and let $\sbd = \emptyset$, i.e. the trivial $SU(2)$ singular bundle data. Unless otherwise mentioned, $Y$ is a homology $3$-sphere and the links have non-zero determinant.

First suppose $S\subset I\times Y$ is orientable. Referring to \eqref{eq:cobparity}, we see that in this case our cobordism is {\emph{even}}. Indeed, for an orientable surface $S$, the Euler characteristic is congruent modulo $2$ to the number of boundary components. The group $H_1(W\setminus S;\Z)$ is isomorphic to a free abelian group on $|S|$ generators; each generator may be represented by a circle fiber in the boundary of a tubular neighborhood of $S$, one for each component. There are thus $2^{|S|-1}$ many homomorphisms $\pi_1(W\setminus S)\to SU(2)$ up to conjugacy that have abelian image and send each generator just described to a traceless element. These reducibles correspond to minimal reducible instantons; they are all flat and have zero energy.

Define a {\emph{quasi-orientation}} of a surface to be a pair of orientations which are related by reversing orientations on each component. Write $\cQ(W,S)$ for the set of quasi-orientations of $S\subset W$. Then similar to the 3-dimensional case, the minimal reducibles on $(W,S)$ are in bijection with $\cQ(W,S)$. There is a correspondence of sets
\begin{equation}\label{eq:qocorrespondence}
            \begin{tikzcd}[column sep=2ex, row sep=5ex, fill=none, /tikz/baseline=-10pt]
            &  \mathcal{Q}(W,S) \arrow[dl, "s"'] \arrow[dr, "t"] & \\
          \mathcal{Q}(Y,L) & & \mathcal{Q}(Y',L')
            \end{tikzcd}
        \end{equation} 
where $s$ and $t$ are defined by restriction. Suppose $\fo=s(\fs)$ and $\fo'=t(\fs)$ for some $\fs\in \cQ(W,S)$. We write $\Theta_\fs$ for the flat reducible instanton corresponding to the quasi-orientation $\fs$; it has limiting flat connections $\theta_{\fo}$ and $\theta_{\fo'}$.

\begin{figure}[t]
    \centering
     \centerline{ \includegraphics[scale=0.5]{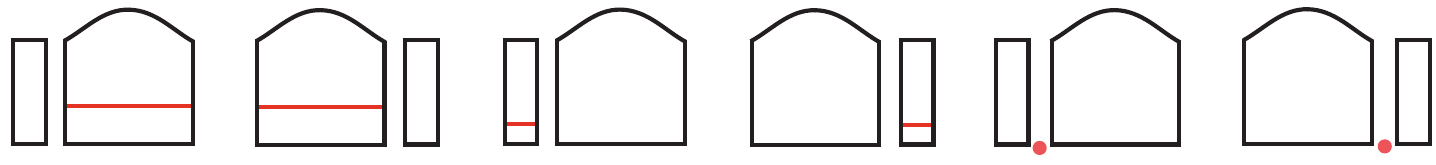}}
    \caption{{\small{Illustration of relation \eqref{eq:lambda4}.}}}
    \label{fig:pic-calc-ex}
\end{figure}

Corollary \ref{cor:indexred} implies that the index of a minimal reducible on this cobordism is 
\begin{equation}\label{eq:indminredcyl}
    \text{ind}(\Theta_\fs) = \chi(S) + \sigma(Y,L,\fo) - \sigma(Y',L',\fo') - 1
\end{equation}
where as above the quasi-orientation $\fs$ of $S$ restricts to $\fo$ and $\fo'$ at the ends. We note also that any reducibles that exist are in fact minimal. The cobordism is unobstructed if this quantity is at least $-1$ for all quasi-orientations of $S$. In the unobstructed case, because our cobordism is even, by Proposition \ref{prop:degmorphism} we obtain a morphism $\widetilde\lambda$ of even degree:
\[
    \text{deg}(\widetilde\lambda) \equiv 0 \pmod{2}
\]

Now consider the case in which $S\subset I\times Y$ is non-orientable. We again fix the trivial bundle data $\sbd$. Then there are no reducible instantons. Indeed, whenever the singular bundle data is trivial and the surface $S$ is non-orientable, all singular connections are irreducible even in any given neighborhood of a non-orientable component of $S$, as can be seen from the requirements in \cite[Section 2.2]{KM:unknot}. Thus we always obtain a morphism $\widetilde \lambda$ between the $\cS$-complexes and by Proposition \ref{prop:degmorphism} its mod $2$ degree is given by
\[
    \text{deg}(\widetilde\lambda) \equiv b_1(S) \pmod{2}
\]
which is the genus of $S$ modulo $2$, i.e. the number $k$ such that $S$ is homemorphic to connect-summing $k$ copies of $\bR\bP^2$ in some way to a collection of spheres with punctures. In particular, if $b_1(S)$ is odd, then the cobordism is odd.

\begin{figure}[t]
    \centering
    \includegraphics[scale=0.62]{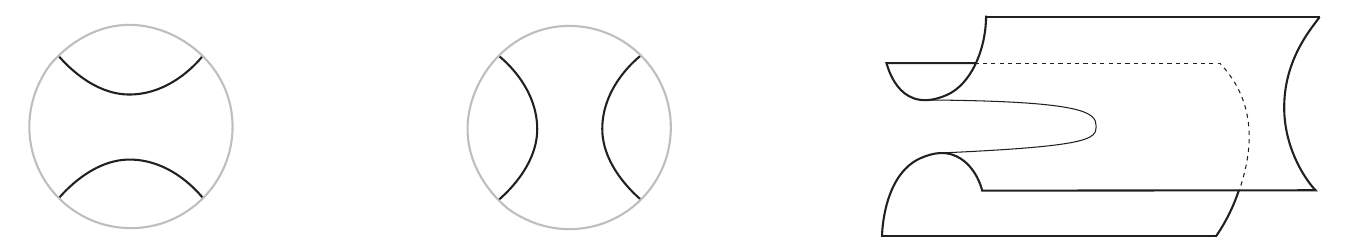}
    \caption{{\small{A saddle cobordism. The links $(Y,L)$ and $(Y,L')$ are the same except in a 3-ball neighborhood, where $(Y,L)$ is depicted on the left and $(Y,L')$ in the middle. The saddle cobordism is locally depicted on the right.}}}
    \label{fig:saddle}
\end{figure}

We further specialize to the case of a saddle cobordism. This is a surface cobordism in the cylinder $I\times Y$ where the links $(Y,L)$ And $(Y,L')$ differ only in a small 3-ball neighborhood $B\subset Y$ as depicted in Figure \ref{fig:saddle}, and $S:L\to L'$ is a product cobordism except in $I\times B$, where it is a standard saddle cobordism. Thus $S$ is obtained by attaching a single 1-handle to $I\times L$. In particular, $\chi(S)=-1$, and $|L|-|L'|=\pm 1$ if $S$ is orientable and $0$ otherwise. The cobordisms that induce maps in the exact triangles are saddle cobordisms.

In the non-orientable case, the cobordism is odd, and we have:

\begin{prop}\label{prop:saddlecobnonor}
    Let $(W,S):(Y,L)\to (Y,L')$ be a saddle cobordism with $S$ non-orientable, i.e. $|L|=|L'|$, and suppose the links have $\det\neq 0$. Then there is an induced morphism $\widetilde\lambda_{(W,S)}:\widetilde C(Y,L)\to \widetilde C(Y,L')$ of odd degree. 
\end{prop}

Now suppose $S$ is orientable. In this case the cobordism is even. Note also that the orientability of $S$ is equivalent to assuming that $L'$ is obtained from $L$ by an oriented resolution of some local crossing for some orientation of $L$. Then for any reducible $\Theta_\fs$ on $(W,S)$ corresponding to a quasi-orientation $\fs$ of the surface $S$ we have
\[
    \text{ind}(\Theta_\fs) =  \sigma(Y,L,\fo) - \sigma(Y',L',\fo') - 2 \; \; \in \; \;  \{-3, -2, -1 \}   
\]
where $\fo$ and $\fo'$ are the quasi-orientations induced by $\fs$. This follows from the well-known property that the difference of the signatures is either $0$ or $\pm 1$. If $L$ and $L'$ both have non-zero determinant, as we are assuming, then in fact $\text{ind}(\Theta_\fs)\in \{-3,-1\}$.

The constraint $\text{ind}(\Theta_\fs)\in \{-3,-2,-1\}$ can be seen directly from the deformation theory of the ASD complex, following \cite[Section 2.7]{DS1}. The adjoint of the reducible $\Theta_\fs$ pulls back to a trivial connection on $X$, the double cover of $W$ branched over $S$. With the assumption that the links have non-zero determinant, their double branched covers have $b_1=0$. The cohomology groups $H^\bullet_{\Theta_\fs}$ in the deformation complex of $\Theta_\fs$ are the $\tau^\ast\otimes \text{diag}(1,-1,-1)$ invariant subspaces of the groups $H^\bullet(X;\R)\otimes \R^3$ where $\bullet\in \{0,1,+\}$, the latter of which describe the cohomology groups of the deformation complex for the trivial $SO(3)$ connection on $X$. Here we have written $\tau$ for the covering involution on $X$. Write $h^\bullet$ for the dimension of $H^\bullet_{\Theta_\fs}$. Then 
\[
    \text{ind}(\Theta_\fs)=-h^0+h^1-h^+
\]
Clearly $h^0=1$, $h^1=0$. Furthermore, since $S$ is a 1-handle surface cobordism, $X$ is a 2-handle 4-dimensional cobordism and $H^2(X;\R)$ has dimension $1$; and because $W=I\times Y$, $\tau$ acts as $-1$ on $H^2(X;\R)$. Thus $h^+\in \{0,2\}$, and so $\text{ind}(\Theta_\fs)\in \{-3,-1\}$. Removing the assumption that determinants are non-zero gives the possibility $\text{ind}(\Theta_\fs)=-2$ using, for example, \cite[Proposition 3.15]{donaldson-book}.

We summarize the conclusions of our discussion as combined with Theorem \ref{thm:unobsmaps}.

\begin{prop}\label{prop:saddlecobor}
    Let $(W,S):(Y,L)\to (Y,L')$ be a saddle cobordism with $S$ orientable, i.e. $|L|-|L'|=\pm 1$, and suppose the links have $\det\neq 0$. Let $\fo,\fo'$ be any quasi-orientations of $L$, $L'$ that extend over $S$ to a quasi-orientation $\fs$. Then we have
    \[
      \epsilon = \epsilon(L,L') := \sigma(Y,L,\fo) - \sigma(Y,L',\fo') \in \{ -1, + 1 \}  
    \]
    The minimal reducibles on $(W,S)$ with trivial $SU(2)$ bundle data are all of index $\epsilon-2$ and are in natural bijection with the quasi-orientations of $S$. If $\epsilon=1$ then there is an induced strong (height $0$) morphism $\widetilde\lambda_{(W,S)}:\widetilde C(Y,L)\to \widetilde C(Y,L')$ of even degree.
\end{prop}

The case in which the difference of signatures $\epsilon = -1$, where every minimal reducible has index $\epsilon-2=-3$, is not an unobstructed cobordism, and its cobordism map has not been defined. This case is an example of an obstructed cobordism, as the analysis of the cobordism map will involve obstructed gluing theory, taken up in Section \ref{sec:obstructed}.

The discussion above relating the indices to the cohomology of double branched covers also gives the following useful observation.

\begin{prop}\label{prop:indbranchedcover}
	Let $(W,S):(Y,L)\to (Y,L')$ be a saddle cobordism, with $S$ orientable. Let $X$ be the double branched cover of $(W,S)$. Then $	\sigma(X) = -\epsilon(L,L')$.
\end{prop}

\subsection{Connected sum theorem and \texorpdfstring{$I^\natural(L)$}{I-natural}}

The connected sum theorem proved in \cite{DS1} carries over to the current setup. Given based links $(Y,L)$ and $(Y',L')$ with basepoints $p\in L$ and $p'\in L'$, we may form the connected sum $(Y\# Y', L\# L')$ along $p$ and $p'$; this is naturally a based link.

\begin{theorem}\label{thm:connectedsum}
    Let $(Y,L)$ and $(Y',L')$ be based links in homology 3-spheres with non-zero determinant. Then there is a chain homotopy equivalence of $\Z/2$-graded $\cS$-complexes:
    \[
        \widetilde C(Y\# Y',L\# L') \simeq \widetilde C(Y,L) \otimes  \widetilde C(Y',L')  
    \]
    The same result holds if either of the links is replaced by an admissible link. This equivalence is natural, up to homotopy, with respect to split cobordisms.
\end{theorem}

\noindent Furthermore, fix $\Z/4$-gradings by choosing quasi-orientations $\fo$ and $\fo'$ of $(Y,L)$ and $(Y',L')$, respectively. Then there is a naturally induced quasi-orientation on the connected sum: choose orientations representing $\fo$ and $\fo'$ that make the connected sum compatible with orientations. We then obtain a $\Z/4$-grading on the $\cS$-complex for $(Y\# Y', L\# L')$. In this case the homotopy equivalence above is then of $\Z/4$-graded $\cS$-complexes.

The proof of Theorem \ref{thm:connectedsum} is essentially the same as in \cite[Section 6]{DS1}. All of the cobordism maps are defined just as in the case for knots, except that there are more reducibles ocurring. However, the types of reducibles that occur are the same.

The connected sum theorem allows us to relate our construction to Kronheimer and Mrowka's $I^\natural(Y,L)$ defined in \cite{KM:unknot}. For any based link $(Y,L)$, this group is the homology of the chain complex $C^\natural(Y,L):=C^\omega(Y\# S^3, L\# H)$, where $H$ is a basepointed Hopf link and $\omega$ consists of a small arc joining the two Hopf link components.

\begin{theorem}\label{eq:inatural}
    Let $(Y,L)$ be a based link in an integer homology 3-sphere with non-zero determinant. Then there is a chain homotopy equivalence $\widetilde C(Y,L)\simeq C^\natural (Y,L)$, natural up to chain homotopy equivalence. Consequently, we have an isomorphism
    \[
       I^\natural(Y,L) \cong  \widetilde I(Y,L)
    \]
\end{theorem}

\noindent The proof is the same as that of \cite[Theorem 8.9]{DS1}. The discussion there for obtaining $I^\#(Y,L)$ from the $\cS$-complex also generalizes, and many of the other results discussed in \cite[Section 8]{DS1} adapt to this setting.

\newpage


\section{Maps from obstructed cobordisms}\label{sec:obstructed}

The goal of this section is to extend functoriality in equivariant singular instanton Floer theory to certain obstructed negative definite cobordisms. In Subsection \ref{subsec:compactified}, we recall some preliminaries, such as chain convergence and Floer compactifications of moduli spaces. In Subsection \ref{gluing-theory-obs}, we study the behavior of the compactified moduli spaces in neighborhoods of obstructed reducibles for negative definite cobordisms of height $-1$. In Subsection \ref{obs-cob-maps-level--1}, for any such cobordism $(Y,L)\to (Y',L')$, we construct a morphism of $\mathcal{S}$-complexes
\begin{equation}\label{eq:morphismsuspensionintro2}
	\widetilde C(Y,L) \longrightarrow \Sigma\widetilde C(Y',L'),
\end{equation}
where $\Sigma\widetilde C(Y',L')$ is the suspension of the $\mathcal{S}$-complex for $(Y',L')$. In the terminology of Subsection \ref{subsec:poslevels}, the map \eqref{eq:morphismsuspensionintro2} is a height $-1$ morphism $\widetilde C(Y,L) \to \widetilde C(Y',L')$. 

To construct \eqref{eq:morphismsuspensionintro2}, we adapt a construction of \cite{DME:QHS}, where non-singular instantons are used to define cobordism maps between rational homology spheres in the presence of obstructed reducibles. One difference is that the theory in \cite{DME:QHS} is equivariant with respect to $SO(3)$, while our setting of equivariant singular instanton homology is equivariant with respect to $S^1$. Another difference is that here we work with ``small models'' of equivariant complexes, which are of finite rank over the ground ring; in \cite{DME:QHS}, geometric chain complexes are used, which are of infinite rank. There is another key technical difference in the constructions explained in Remark \ref{difference-models}.

\subsection{Preliminaries}\label{subsec:compactified}

In preparation for the discussion of obstructed gluing theory, here we review some preliminaries, such as framed moduli spaces of singular instantons, Floer compactifications, and modified holonomy maps. For more background on singular instanton Floer theory, see \cite{KM:YAFT, DS1}. Some aspects below are straightforward adaptations from the non-singular $SU(2)$ instanton theory, for which we refer to \cite{donaldson-book, DME:QHS}.

For a cobordism $(W,S):(Y,L)\to (Y',L')$ with singular bundle data $c$, recall that
\[
M(W,S,c;\alpha,\alpha')_d
\]
denotes the associated moduli space of singular instantons of index $d$, with limiting critical points $\alpha$ and $\alpha'$ on $(Y,L)$ and $(Y',L')$, respectively. Each of $(Y,L)$ and $(Y',L')$ is a non-zero determinant link in an integer homology $3$-sphere. The moduli space depends on metric and perturbation data, which are surpressed. It is the quotient of a space of singular connections over $(W^+,S^+)$, the associated cylindrical end space, by a determinant-1 gauge transformation group (modulo $\pm 1$) whose elements converge at the ends to the stabilizers of $\alpha$ and $\alpha'$. With our topological assumptions, the possible stabilizers of critical points and of instantons are the trivial group (irreducibles) and a circle (reducibles).

A variation is to consider the moduli space framed at $-\infty$:
\begin{equation}\label{eq:framedmoduli}
	M(W,S,c;\widetilde{\alpha},\alpha')_{d+\dim \Gamma_\alpha}
\end{equation}
The subscript indicates the expected dimension of the irreducible stratum of this space. In constructing this moduli space, one mods out by the gauge transformations $g$ which limit at $-\infty$ to $\pm 1$ and at $+\infty$ to some gauge transformation in the stabilizer of $\alpha'$. There is an action of the (adjoint) stabilizer $\Gamma_\alpha$ on \eqref{eq:framedmoduli}, by extending $g\in\Gamma_\alpha$ to a gauge transformation over the cobordism. The quotient is $M(W,S,c;\alpha,\alpha')_d$. If $\alpha$ is irreducible, then $\Gamma_\alpha=1$, and these moduli spaces are the same. When $\alpha$ is reducible, then $\Gamma_\alpha\cong S^1$, and they are different, unless the moduli space consists itself entirely of reducibles.

Similarly, we have a moduli space framed at $+\infty$, denoted $M(W,S,c;\alpha,\widetilde{\alpha}')_{d+\dim \Gamma_{\alpha'}}$, which has a $\Gamma_{\alpha'}$-action. We also consider moduli framed at both $\pm \infty$:
\begin{equation}\label{eq:framedmoduli2}
	M(W,S,c;\widetilde{\alpha},\widetilde{\alpha}')_{d + \dim \Gamma_{\alpha} + \dim \Gamma_{\alpha'}}
\end{equation} 
and this carries a natural $\Gamma_\alpha \times \Gamma_{\alpha'}$-action. Similar notions are defined for the case of a cylinder $I\times (Y,L)$. The {\emph{index}} of a framed instanton refers to its index forgetting the framing data; in particular, instantons in \eqref{eq:framedmoduli2} have index $d$, as do those in the moduli space
\begin{equation}\label{eq:modulicylframed}
	\breve{M}(Y,L;\widetilde{\alpha}_1,\widetilde{\alpha}_2)_{d -1 + \dim \Gamma_{\alpha_1} + \dim \Gamma_{\alpha_2}}.
\end{equation}
We always assume that perturbation and metric data are chosen such that the critical sets $\mathfrak{C}_\pi(Y,L)$ and $\mathfrak{C}_{\pi'}(Y',L')$ are as in Subsection \ref{subsec:gradings} (non-degenerate, with reducibles in correspondence with quasi-orientations), and such that the irreducible strata of all (framed) moduli spaces are smooth.

For brevity we sometimes write $S$ in place of $(W,S,c)$, and $L$, $L'$ in place of $(Y,L)$, $(Y',L')$, respectively. A {\emph{broken}} framed singular instanton on $(W,S,c)$ is an element of
\begin{align}\label{eq:framedmoduli3}
	\breve{M}(L; \widetilde{\beta}_1,\widetilde{\beta}_2)\times_{\Gamma_{\beta_2}} \times \cdots \times_{\Gamma_{\beta_{k-1}}}\breve{M}(L; \widetilde{\beta}_{k-1},\widetilde{\beta}_{k})\times_{\Gamma_{\beta_{k}}} \times M(S;\widetilde{\beta}_k,\widetilde{\beta}_1') \nonumber \\
	\times_{\Gamma_{\beta_1'}} \breve{M}(L'; \widetilde{\beta}_{1}',\widetilde{\beta}_2')\times_{\Gamma_{\beta_2'}} \cdots \times_{\Gamma_{\beta_{l-1}'}} \breve{M}(L'; \widetilde{\beta}_{l-1}',\widetilde{\beta}_l') 
\end{align} 
Here $\Gamma_{\beta}$ acts diagonally on the two moduli spaces with $\beta$ as a limit, and \eqref{eq:framedmoduli3} is the product of the moduli spaces appearing by the product of the stabilizers $\Gamma_\beta$, as $\beta$ ranges over $\beta_i,\beta_j'$. For brevity we have omitted the dimensional subscripts from \eqref{eq:framedmoduli3}. An element in \eqref{eq:framedmoduli} is also a broken instanton. A tuple of connections $\mathbf{A}=(B_1,\ldots,B_{k-1},A,B_1',\ldots, B_{l-1}')$ representing a broken framed instanton in \eqref{eq:framedmoduli3} has index defined by 
\[
	\text{ind}(\mathbf{A}) = \sum_{i=1}^{k-1} \ind(B_i) + \dim \Gamma_{\beta_{i+1}} + \ind(A)  + \sum_{j=1}^{l-1} \ind(B'_j) + \dim \Gamma_{\beta'_{j}}.
\]
There is a notion of a sequence of elements in \eqref{eq:framedmoduli2} {\emph{chain converging}} to a broken framed instanton in \eqref{eq:framedmoduli3}, see \cite{donaldson-book,DME:QHS} for details. We define 
\begin{equation}\label{eq:framedmoduli4}
	M^+(W,S,c;\widetilde{\alpha},\widetilde{\alpha}')_{d + \dim \Gamma_{\alpha} + \dim \Gamma_{\alpha'}}
\end{equation} 
to be the moduli space of broken framed instantons of index $d$, with the topology of chain convergence. A fundamental result is that \eqref{eq:framedmoduli4} is {\emph{compact}} when $d\leq 3$; for $d>3$, non-compactness arises from bubbling of instantons along the singular locus. We may similarly define the broken instantons version of \eqref{eq:modulicylframed}, which is also compact for $d\leq 3$.

Given basepoints $p\in L$ and $p'\in L'$, and a path $\gamma$ in $W$ from $p$ to $p'$ whose image lies in the singular locus $S$, there is an associated holonomy map
\begin{equation}\label{eq:holmapintroduced}
	h^\gamma_{\alpha,\alpha'}: M(W,S,c;\widetilde{\alpha},\widetilde{\alpha}') \to S^1.
\end{equation}
This map is defined by choosing frames for the adjoint bundles at $(-\infty,p)$ and $(+\infty,p')$, and using parallel transport along $\gamma^+ \subset S^+$, the path obtained from $\gamma$ by attaching $\R_{\leq 0}\times \{p\}$ and $\R_{\geq 0}\times \{p'\}$. In the sequel we slightly abuse this description by referring to the holonomy as taken along $\gamma$. The holonomy takes values in $S^1$ because the instantons are reducible along the singular locus $S^+$. The map \eqref{eq:holmapintroduced} is equivariant with respect to the $\Gamma_{\alpha}\times \Gamma_{\alpha'}$-action in the sense that for $(g,g')\in \Gamma_{\alpha}\times \Gamma_{\alpha'}$ we have
\[
	h^\gamma_{\alpha,\alpha'}((g,g')A) = g' h^\gamma_{\alpha\alpha'}(A) g^{-1}. 
\]
We also have a holonomy map in the case of a cylinder, $h_{\alpha_1,\alpha_2}: \breve{M}(Y,L;\widetilde{\alpha}_1,\widetilde{\alpha}_2) \to S^1$, constructed using parallel transport along $\R\times \{p\}$. The map
\begin{equation*}\label{eq:holmaponbroken}
		h_{{\beta}_1,{\beta}_2}\times \cdots \times h_{{\beta}_{k-1},{\beta}_k}\times   h^\gamma_{{\beta}_k,{\beta}'_1} \times h_{{\beta}_1',{\beta}'_2}\times \cdots \times h_{{\beta}'_{l-1},{\beta}'_l} 
\end{equation*}
descends to the space of broken instantons \eqref{eq:framedmoduli3}, by equivariance. This defines a continuous extension of the holonomy map \eqref{eq:holmapintroduced} to the moduli space of broken instantons.

In practice, we must modify the holonomy map \eqref{eq:holmapintroduced} so that it behaves nicely with respect to the strata of the moduli space. We write
\[
	H^\gamma=H^\gamma_{\alpha,\alpha'}: M(W,S,c;\widetilde{\alpha},\widetilde{\alpha}') \to S^1.
\]
for such a {\emph{modified holonomy map}}, and similarly $H_{\alpha_1,\alpha_2}$ in the case of a cylinder. Such modified holonomy maps were already used in Section \ref{sec:links}. For the details of the construction, see \cite[\S 3.3.2, Appendix]{DS1}. The modified holonomy maps are chosen such that they send $0$-dimensional irreducible strata to $1\in S^1$, and are transverse on all strata to $-1\in S^1$.

Finally, we remark that a neighborhood of an unobstructed broken framed singular instanton in a compactified moduli space such as \eqref{eq:framedmoduli4} only has, a priori, the structure of a topological manifold. Ideally, there would be the structure of a smooth manifold with corners. This is a technical issue and is related to the smoothness of gluing maps, and was dealt with in \cite{DS1}, for example; see Remark 3.17 of that reference. To circumvent this problem, in order to carry out standard differential topological constructions, we work in the framework of {\emph{stratified-smooth spaces}}, which are certain stratified spaces whose strata are smooth manifolds. See \cite[Appendix A]{DME:QHS} for details.

\subsection{Obstructed reducible solutions} \label{gluing-theory-obs}

Let $(W,S):(Y,L)\to (Y',L')$ be a cobordism between links with non-zero determinants in integer homology spheres, with singular bundle data $c$. In this subsection and the following one, we assume that $(W,S,c)$ is negative definite of height $-1$. In particular, $b^1(W)=b^+(W)=0$, and the index of every reducible instanton on $(W,S,c)$ is at least $-3$. We may also assume that $(W,S,c)$ has at least one reducible instanton of index $-3$, for otherwise this cobordism is unobstructed. In particular, in the terminology of Definition \ref{defn:cobordismparity}, $(W,S,c)$ is an {\emph{even}} cobordism, and following Definition \ref{eq:stronglevelncobordismmaps} it is {\emph{strong}} height $-1$. 

 In the subsection to follow, we define a height $-1$ morphism $\widetilde C(Y,L)\to \widetilde C(Y',L')$ associated to $(W,S,c)$. In this subsection, we introduce the relevant results from obstructed gluing theory, adapted from \cite{DME:QHS}, that will be used towards this goal.

Recall that reducible critical points $\theta_\mathfrak{o}$ for $(Y,L)$ are labelled by quasi-orientations $\mathfrak{o}$. In this section, to simplify notation, we write $\theta=\theta_\mathfrak{o}$ for any of these reducible critical points, and similarly $\theta'$ for a reducible on $(Y',L')$. We also frequently write $S$ in place of $(W,S,c)$, and $L,L'$ in place of $(Y,L)$, $(Y',L')$, respectively. We also assume that basepoints $p\in L$ and $p'\in L'$ are fixed, as well as a path $\gamma$ in $S$ from $p$ to $p'$.

After possibly perturbing the ASD equation, we may assume that all irreducible elements in $M(S;\alpha,\alpha')_d$ for $d\leq 3$ are unobstructed. We cannot achieve the same regularity at reducibles of index $-3$. Consequently, the moduli space $M^+(S;\alpha,\alpha')_d$ does not necessarily have the structure of a stratified-smooth space, even after a generic perturbation. For instance, let $\Theta$ be a reducible over $(W,S,c)$ with index $-3$, asymptotic to $\theta$ and $\theta'$, and $[B]\in \breve M(L;\alpha,\theta)_1$. Then there might exist a sequence of elements in $M^+(S;\alpha,\theta')_0$ convergent to the broken solution $[B,\Theta]$. Thus $M^+(S;\alpha,\theta')_0$ is not necessarily a finite discrete set of points. In particular, the map $\Delta_1:C(Y,L)\to \sfR'$ cannot be defined in the usual way without first modifying the space $M^+(S;\alpha,\theta')_0$.

We define three types of obstructed broken singular instantons over $(W,S,c)$. Let $\Theta$ be a reducible over $(W,S,c)$ with index $-3$, with reducible limits $\theta$ and $\theta'$. A {\emph{type I}} obstructed solution is an element of the subspace of broken instantons
\begin{equation}\label{type-I-obs}
  \breve M^+(L;\alpha,\theta)_{d+1}\times\{\Theta\}
\end{equation}
in $M^+(S;\alpha,\theta')_d$. {\emph{Type II}} obstructed solutions form the subspace of broken instantons 
\begin{equation}\label{type-II-obs}
	\{\Theta\}\times \breve M^+(L';\theta',\alpha')_{d+1}
\end{equation}
inside $M^+(S;\theta,\alpha')_d$. {\emph{Type III}} solutions form the subspace of $M^+(S;\alpha,\alpha')_d$ given as
\begin{equation}\label{type-III-base}
  \bigsqcup_{0\leq i\leq d-1} \breve M^+(L;\alpha,\widetilde \theta)_{i+1} \times_{S^1} \overline{\underline \Theta} \times_{S^1} \breve M^+(L';\widetilde \theta',\alpha')_{d-i},
\end{equation}
where for each $i$, the corresponding subspace is empty unless $i\equiv \deg(\alpha)-1$ mod $4$. The moduli space $M(S;\widetilde \theta,\widetilde \theta')_{-1}$ admits an $S^1\times S^1$ action, and $ \overline{\underline \Theta} $ is the orbit of this action on the element $\Theta$. Since $\Theta$ is reducible, this orbit gives a copy of $S^1$. We have a holonomy map $M(S;\widetilde \theta,\widetilde \theta')_{-1}\to S^1$ along the path $\gamma$. In particular, there is a distinguished element of $ \overline{\underline \Theta} $ with holonomy $1$. This induces an identification of \eqref{type-III-base} with 
\begin{equation}\label{type-III-base-2}
  \bigsqcup_{0\leq i\leq d-1} \breve M^+(L;\alpha,\widetilde \theta)_{i+1} \times_{S^1} \breve M^+(L';\widetilde \theta',\alpha')_{d-i},
\end{equation}
Similar to \eqref{type-III-base-2}, \eqref{type-I-obs} and \eqref{type-II-obs} can be respectively written as 
\[
  \breve M^+(L;\alpha,\widetilde \theta)_{d+1}\times_{S^1} \underline \Theta, \hspace{1cm} \overline \Theta \times_{S^1} \breve M^+(L';\widetilde \theta',\alpha')_{d+2},
\]
where $\underline \Theta\subset M(S;\widetilde \theta,\theta')_{-2}$ and $\overline \Theta\subset M(S; \theta,\widetilde \theta')_{-2}$ are the orbits given by $\Theta$.

A description of the behavior of the moduli spaces $M^+(S;\alpha,\alpha')_d$ in a neighborhood of obstructed broken solutions is provided by obstructed gluing theory. Before stating these gluing theory results, we need to define the relevant {\it obstruction bundles}. Consider $\breve M(L;\alpha,\widetilde \theta)_{d+1}$ as a principal $S^1$-bundle over $\breve M(L;\alpha,\theta)_{d}$, and let
\[
  \mathcal H_{\alpha,\theta}:=\breve M(L;\alpha,\widetilde \theta)_{d+1}\times_{S^1} \C
\] 
This smooth line bundle extends over the stratified-smooth space $\breve M^+(L;\alpha,\theta_{\fo})_{d}$ as a stratified complex line bundle, which will also be denoted $\mathcal H_{\alpha,\theta}$. One should regard $\mathcal H_{\alpha,\theta}$ as the space of the cokernels of the ASD operator for the obstructed solutions of type I. To be more precise, the cokernel of the ASD operator for $\Theta$, which has complex dimension $1$, gives rise to $S^1$-bundles $\mathcal H^+(\underline \Theta)$ over $\underline \Theta$ and $\mathcal H^+(\overline \Theta)$ over $\overline \Theta$. Then we have an identification
 \[
	\mathcal H_{\alpha,\theta}=\breve M^+(L;\alpha,\widetilde \theta)_{d+1}\times_{S^1}\mathcal H^+(\underline \Theta).
\]
We may similarly define a complex line bundle $\mathcal H_{\theta,\alpha}$ over $\breve M^+(L;\theta,\alpha)_{d}$, which can be identified with $\mathcal H^+(\overline \Theta) \times_{S^1} \breve M^+(L';\widetilde \theta',\alpha')_{d+2}$.

Finally, to define the obstruction bundle over the space of type III obstructed solutions, we note that the cokernel of $\Theta$ induces a complex line bundle $\mathcal H^+(\overline{\underline \Theta})$ over $\overline{\underline \Theta}$ together with a lift of the $S^1\times S^1$ action on the base. This induces a line bundle on \eqref{type-III-base}, which we denote by $\mathcal H_{\alpha,\alpha'}$. There is a natural projection map 
\[
	\breve M^+(L;\alpha,\widetilde \theta)_{i+1} \times_{S^1} \breve M^+(L';\widetilde \theta',\alpha')_{d-i} \to \breve M^+(L;\alpha,\widetilde \theta)_{i}
\] 
and $\mathcal H_{\alpha,\alpha'}$ is the pullback of $\mathcal H_{\alpha,\theta}$ with respect to this projection map. Similarly, $\mathcal H_{\alpha,\alpha'}$ is the pullback of $\mathcal H_{\theta',\alpha'}$ with respect to the projection map to $\breve M^+(L';\widetilde \theta',\alpha')_{d-i-1}$. In particular, any section of $\mathcal H_{\alpha,\theta}$ or $\mathcal H_{\theta',\alpha'}$ determines a section of $\mathcal H_{\alpha,\alpha'}$.

In what follows, $\overline \R_{+}$ denotes the union of the set of positive real numbers $\R_+$ and $\{\infty\}$, topologized in the obvious way. In particular, $\overline \R_{+}$ is a smooth stratified space with the two strata $\{\infty\}$ and $\R_{+}$. The space $\overline \R_{+}$ is used to parametrize the length of the {\it neck}, where $\infty\in \overline \R_{+}$ is used when the length of the neck is infinity, the location of broken elements of a moduli space of ASD connections. 

The following proposition gives a description for a neighborhood of type I obstructed elements of $M^+(S;\alpha,\theta')_d$. For the proof of this proposition (and Proposition \ref{type-3-obs-gluing} below), we refer the reader to \cite[Subsection 5.1]{DME:QHS} and the references therein.  
\begin{prop}\label{type-1-obs-gluing}
	Fix an integer $-1 \leq d\leq 2$, and $\alpha\in \fC_{\pi}(Y,L)$. Then there is a continuous section $\Psi_{\alpha,\Theta}$ of $\mathcal H_{\alpha,\theta}\times \overline \R_{+}$ over 
	$\breve M^+(L;\alpha,\theta)_{d+1}\times \overline \R_{+}$ satisfying:
	\begin{itemize}
		\item[(i)] $\Psi_{\alpha,\Theta}$ vanishes on $\breve M^+(L;\alpha,\theta)_{d+1}\times \{\infty\}$, and $\Psi_{\alpha,\Theta}$ is transverse to the zero section over any stratum of  $\breve M^+(L;\alpha,\theta)_{d+1}\times \R_{+}$.
		\item[(ii)] There is a map $\Phi_{\alpha,\Theta}:\Psi_{\alpha,\Theta}^{-1}(0)\to M^+(S;\alpha,\theta')_d$ which is a homeomorphism onto an open subset of $M^+(S;\alpha,\theta')_d$, 
		its restriction to each stratum is a diffeomorphism, and it maps $\breve M^+(L;\alpha,\theta)_{d+1}\times \{\infty\}$ to the space of type I obstructed solutions by the identity map. 
		Furthermore, denoting by $\pi$ the projection map 
		\[\breve M^+(L;\widetilde \alpha,\widetilde \theta)_{d+1+\dim\Gamma_\alpha}\times \overline \R_{+}\to\breve M^+(L;\alpha,\theta)_{d+1}\times \overline \R_{+},\] there is a $\Gamma_\alpha\times S^1$-equivariant map 
		\[\widetilde \Phi_{\alpha,\Theta}:\pi^{-1}(\Psi_{\alpha,\Theta}^{-1}(0))\to M^+(S;\widetilde \alpha,\widetilde \theta')_{d+1+\dim\Gamma_\alpha}\] such that $\Phi_{\alpha,\Theta}$ is induced by $\widetilde \Phi_{\alpha,\Theta}$.
		The map $\widetilde \Phi_{\alpha,\Theta}$ sends the point $(x,\infty)$ in the space $\breve M^+(L;\widetilde \alpha,\widetilde \theta)_{d+1+\dim\Gamma_\alpha}\times \{\infty\}$ to 
		$(x,\Theta)\in M^+(S;\widetilde \alpha,\widetilde \theta')_{d+1+\dim\Gamma_\alpha}$. 
		\item[(iii)] $\Psi_{\alpha,\Theta}$ and $\Phi_{\alpha,\Theta}$ are compatible with the stratification of $\breve M^+(L;\alpha,\theta)_{d+1}$ in the following sense: for any 
		$[x,y]\in \breve M^+(L;\alpha,\widetilde \beta)_{i}\times_{\Gamma_\beta} \breve M^+(L;\widetilde \beta,\widetilde \theta)_{d-i+\dim\Gamma_{\beta}}$, 
		we have 
		\[\Psi_{\alpha,\Theta}([x,y],t)=\Psi_{\beta,\Theta}([y],t),\]
		and if in addition we have $\Psi_{\beta,\Theta}([y],t)=0$, then 
		\[\widetilde \Phi_{\alpha,\Theta}([x,y],t)=[x,\widetilde \Phi_{\beta,\Theta}(y,t)].\]
	\end{itemize}
\end{prop}

Roughly speaking, Proposition \ref{type-1-obs-gluing} states that we can glue an element $x$ of the moduli space $\breve M^+(L;\alpha,\theta)_{d+1}$ to the reducible $\Theta$ with a prescribed length of the neck $t\in \R_{+}$ and obtain an element of $M^+(S;\alpha,\theta')_d$ if and only if an obstruction given by the obstruction section $\Psi_{\alpha,\Theta}$ vanishes, and the space of such solutions of $M^+(S;\alpha,\theta')_d$ determines a neighborhood of type I obstructed solutions.

There is a similar proposition for type II obstructed solutions: for any $\alpha'\in \fC_{\pi'}(Y',L')$ and $-1 \leq d\leq 2$, there is a continuous section $\Psi_{\Theta,\alpha'}$ of the bundle
\[
	\mathcal H_{\theta',\alpha'}\times \overline \R_{+}\to \breve M^+(L';\theta',\alpha')_{d+1}\times \overline \R_{+}
\]
and a homeomorphism $\Phi_{\Theta,\alpha'}:\Psi_{\Theta,\alpha'}^{-1}(0)\to M^+(S;\theta,\alpha')_d$, and the analogues of (i)--(iii) are satisfied. The next proposition describes the situation for type III solutions.

\begin{prop}\label{type-3-obs-gluing}
	Suppose $\Theta$ is as in Proposition \ref{type-1-obs-gluing}, $\alpha\in \fC_{\pi}(Y,L)$, $\alpha'\in \fC_{\pi'}(Y',L')$, $1\leq d\leq 3$ and $0\leq i\leq d-1$ are integers such that $i\equiv \deg(\alpha)-1$ mod $4$.
	There is a continuous section $\Psi_{\alpha,\Theta,\alpha'}$ of $\mathcal H_{\alpha,\alpha'}\times \overline \R_{+}\times \overline \R_{+}$ over 
	\begin{equation}\label{domian-C-a-T-a}
	  \breve M^+(L;\alpha,\widetilde \theta)_{i+1}\times_{S^1} \breve M^+(L';\widetilde \theta',\alpha')_{d-i}
	  \times \overline \R_{+}\times \overline \R_{+}
	\end{equation}
	such that the following conditions are satisfied.
	\begin{itemize}
		\item[(i)] For any $[x,y]\in \breve M^+(L;\alpha,\widetilde \theta)_{i+1} \times_{S^1} \breve M^+(L';\widetilde \theta',\alpha')_{d-i}$, we have
		\[
		  \Psi_{\alpha,\Theta,\alpha'}([x,y],t,\infty)=\Psi_{\alpha,\Theta}([x],t),\hspace{0.6cm}\Psi_{\alpha,\Theta,\alpha'}([x,y],\infty,t)=\Psi_{\Theta,\alpha'}([y],t).
		\]
		In particular, $\Psi_{\alpha,\Theta,\alpha'}([x,y],\infty,\infty)=0$.
		\item[(ii)] $\Psi_{\alpha,\Theta,\alpha'}$ is transverse to the zero section over any stratum of  
		\[\breve M^+(L;\alpha,\widetilde \theta)_{i+1}\times_{S^1} \overline{\underline \Theta} \times_{S^1} \breve M^+(L';\widetilde \theta',\alpha')_{d-i}
		\times \R_{+}\times \R_{+}.\]
		\item[(iii)]There is a map $\Phi_{\alpha,\Theta,\alpha'}:\Psi_{\alpha,\Theta,\alpha'}^{-1}(0)\to M^+(S;\alpha,\alpha')_d$ which is a homeomorphism into an open subspace of $M^+(S;\alpha,\alpha')_d$, its restriction to each stratum of $\Psi_{\alpha,\Theta,\alpha'}^{-1}(0)$ 
		is a diffeomorphism, and for any $[x,y]\in \breve M^+(L;\alpha,\widetilde \theta)_{i+1} \times_{S^1} \breve M^+(L';\widetilde \theta',\alpha')_{d-i}$, 
		\[
		  \Phi_{\alpha,\Theta,\alpha'}([x,y],t,\infty)=[\widetilde \Phi_{\alpha,\Theta}(x,t),y],\hspace{.6cm}\Phi_{\alpha,\Theta,\alpha'}([x,y],\infty,t)=[x,\widetilde \Phi_{\Theta,\alpha'}(y,t)].
		\]
		Further, there is a $\Gamma_{\alpha}\times \Gamma_{\alpha'}$-equivariant lift $\widetilde  \Phi_{\alpha,\Theta,\alpha'}$ of $ \Phi_{\alpha,\Theta,\alpha'}$ as in Proposition \ref{type-3-obs-gluing}.
	\end{itemize} 
	\begin{itemize}
		\item[(iv)] For $[x,y,z]\in  \breve M^+(L;\alpha,\widetilde \beta)_{j+\dim(\Gamma_{\beta})} \times_{\Gamma_{\beta}} 
		\breve M^+(L;\widetilde \beta,\widetilde \theta)_{i-j} \times_{S^1} \breve M^+(L';\widetilde \theta',\alpha')_{d-i}$, 
		\[\Psi_{\alpha,\Theta,\alpha'}([x,y,z],t,t')=\Psi_{\beta,\Theta,\alpha'}([y,z],t,t'),\]
		and if in addition we have $\Psi_{\beta,\Theta,\alpha'}([y,z],t,t')=0$, then
		\[
		  \Phi_{\alpha,\Theta,\alpha'}([x,y,z],t,t')=[x,\Phi_{\beta,\Theta,\alpha'}([y,z],t,t')]
		\]
		Similar claims hold for the restriction of $\Psi_{\alpha,\Theta,\alpha'}$ to subspaces of instantons broken along some critical point $\beta'$ for $L'$.	
	\end{itemize}
\end{prop}

We now describe some constructions using the data from these propositions that will be used in the next subsection. For any reducible $\Theta$ of index $-3$, fix a positive real number $T_\Theta$. Write $\psi_{\alpha,\Theta}$ and $\psi_{\Theta,\alpha'}$ for the restrictions of the sections $\Psi_{\alpha,\Theta}$ and $\Psi_{\Theta,\alpha'}$ to $\breve M^+(L;\alpha,\theta)_{d+1}\times \{T_\Theta\}$ and $\{T_\Theta\}\times \breve M^+(L';\theta',\alpha')_{d+1}$. For generic $T_\Theta$, $\psi_{\alpha,\Theta}$ and $\psi_{\Theta,\alpha'}$ are transverse to the zero section. Assume $T_\Theta$ has this property for all $\alpha$ and $\alpha'$. Define
\begin{equation}\label{S-theta}
  S_\Theta:=\{(t,t')\in \overline \R_{+}\times \overline \R_{+} \mid \frac{1}{t}+\frac{1}{t'}=\frac{1}{T_\Theta}\},
\end{equation}
\begin{equation}\label{U-theta}
  U_\Theta:=\{(t,t')\in \overline \R_{+}\times \overline \R_{+} \mid \frac{1}{t}+\frac{1}{t'}\leq \frac{1}{T_\Theta}\}.
\end{equation}
Denote by $\psi_{\alpha,\Theta,\alpha'}$ the restriction of $\Psi_{\alpha,\Theta,\alpha'}$ to the subspace of \eqref{domian-C-a-T-a} where the last two coordinates belong to $S_\Theta$. We may again assume that $\psi_{\alpha,\Theta,\alpha'}$ is transverse to the zero section. Note that $(T_\Theta,\infty)$, $(\infty,T_\Theta) \in S_\Theta$ and as a consequence of Proposition \ref{type-3-obs-gluing}, 
\[
  \psi_{\alpha,\Theta,\alpha'}([x,y],T_\Theta,\infty)=\psi_{\alpha,\Theta}([x],T_\Theta),\hspace{0.4cm}\psi_{\alpha,\Theta,\alpha'}([x,y],\infty,T_\Theta)=\psi_{\Theta,\alpha'}([y],T_\Theta),
\]
for any $[x,y]\in \breve M^+(L;\alpha,\widetilde \theta)_{i+1} \times_{S^1} \breve M^+(L';\widetilde \theta',\alpha')_{d-i}$. We need $T_\Theta$ to be large enough and we will make this property more explicit in the next subsection.

The line bundles $\mathcal H_{\alpha,\theta}$, $\mathcal H_{\theta',\alpha'}$ can be trivialized when $\alpha$ and $\alpha'$ are irreducible, in the following way. Since the modified holonomy map $H_{\alpha, \theta}$ is $S^1$-equivariant, the assignment
\[
  [A]\in \breve M^+(L;\alpha,\widetilde \theta)_{d+1} \mapsto H_{\alpha, \theta}([A])\in S^1\subset \C
\]
determines a nowhere vanishing global section of $\mathcal H_{\alpha,\theta}$. Define $\psi_{\alpha,\Theta}'$ to be this section of $\mathcal H_{\alpha,\theta}$. Similarly, a nowhere vanishing section of $\mathcal H_{\theta',\alpha'}$ is given by
\[
  [A]\in \breve M^+(L';\widetilde \theta',\alpha')_{d+1} \mapsto H_{\theta',\alpha'}([A])^{-1}\in S^1\subset \C,
\]
which we denote by $\psi_{\Theta,\alpha'}'$.

To define the cobordism map associated to $(W,S,c)$, we will need homotopies from $\psi_{\alpha,\Theta}$ to $\psi_{\alpha,\Theta}'$ and from $\psi_{\Theta,\alpha'}$ to $\psi_{\Theta,\alpha'}'$. Suppose $\Xi_{\alpha,\Theta}$ is a section of $\mathcal H_{\alpha,\theta}\times [0,1]$ over $\breve M^+(L;\alpha,\theta)_{d}\times[0,1]$ such that its restriction to each stratum is transverse to the zero section. We require that the restriction of $\Xi_{\alpha,\Theta}$ to $\breve M^+(L;\alpha,\theta)_{d}\times\{0\}$ and $\breve M^+(L;\alpha,\theta)_{d}\times\{1\}$ are respectively $\psi_{\alpha,\Theta}$ and $\psi_{\alpha,\Theta}'$. The properties of the modified holonomy maps imply that the restriction of $\psi_{\alpha,\Theta}'$ to the subspace  $\breve M^+(L;\alpha,\beta)_{i} \times \breve M^+(L;\beta,\theta)_{d-i-1}$ of $\breve M^+(L;\alpha,\theta)_{d}$ is
\begin{equation}\label{co-dim-1-res}
  \psi_{\alpha,\Theta}'(x,y)=H_{\alpha,\beta}(x)\cdot \psi_{\beta,\Theta}'(y).
\end{equation}
For $(x,y,t) \in \breve M^+(L;\alpha,\beta)_{i} \times \breve M^+(L;\beta,\theta)_{d-i-1}\times [0,1]$, consider the section
\begin{equation} \label{co-dim-1-res-2}
  \left( (1-t) + t H_{\alpha\beta}(x)\right)\Xi_{\beta,\Theta}(y,t).
\end{equation}
When $d=i+1$, our assumptions on $H_{\alpha,\beta}$ and $\Xi_{\beta,\Theta}$ imply that \eqref{co-dim-1-res-2} is transverse to the zero section, and in this case we require that $\Xi_{\alpha,\Theta}$ agrees with \eqref{co-dim-1-res-2} over $\breve M^+(L;\alpha,\beta)_{i} \times \breve M^+(L;\beta,\theta)_{d-i-1}\times [0,1]$. Proposition \ref{type-1-obs-gluing} implies that our requirements on $ \Xi_{\alpha,\Theta}$ are consistent. Similarly, we pick a section $\Xi_{\Theta,\alpha'}$ of $\mathcal H_{\theta',\alpha'}\times [0,1]$ over $\breve M^+(L';\theta',\alpha')_{d}\times [0,1]$ such that they have the analogues of the above properties. 

We need a similar construction starting with the section $\psi_{\alpha,\Theta,\alpha'}$. We fix a section $\Xi_{\alpha,\Theta,\alpha'}$ of $\mathcal H_{\alpha,\alpha'}\times S_\Theta\times [0,1]$ which is transverse to the zero section on each stratum and has the following properties for any $[x,y]\in \breve M^+(L;\alpha,\widetilde \theta)_{i+1} \times_{S^1} \breve M^+(L';\widetilde \theta',\alpha')_{d-i}$. First, 
\[
	\Xi_{\alpha,\Theta,\alpha'}([x,y],(t,t'),0)=\psi_{\alpha,\Theta,\alpha'}([x,y],(t,t'))
\]
for any $(t,t')\in S_\Theta$.  Recall that $\mathcal H_{\alpha,\alpha'}$ over $ \breve M^+(L;\alpha,\widetilde \theta)_{i+1} \times_{S^1} \breve M^+(L';\widetilde \theta',\alpha')_{d-i}$ is isomorphic to the pullbacks of $\mathcal H_{\alpha,\theta}$ over $ \breve M^+(L;\alpha,\theta)_{i} $ and $\mathcal H_{\theta',\alpha'}$ over $ \breve M^+(L';\theta',\alpha')_{d-i-1}$. If $\alpha$ (resp. $\alpha'$) is irreducible, we require that $\Xi_{\alpha,\Theta,\alpha'}([x,y],(T_\Theta,\infty),t)=\Xi_{\alpha,\Theta}([x],t)$ (resp. $\Xi_{\alpha,\Theta,\alpha'}([x,y],(\infty,T_\Theta),t)=\Xi_{\Theta,\alpha'}([y],t)$).

We can also be more specific about the restriction of $\Xi_{\alpha,\Theta,\alpha'}$ to 
\[
	\breve M^+(L;\alpha,\widetilde \theta)_{i+1} \times_{S^1} \breve M^+(L';\widetilde \theta',\alpha')_{d-i}\times S_\Theta\times \{1\},
\]
which we denote by $\psi_{\alpha,\Theta,\alpha'}'$. Let $S_-$ and $S_+$ be copies of $S^1$, equipped with $S^1$ actions
\begin{equation}\label{S-m-p}
  (z,x)\in S^1\times S_\pm \mapsto z^{\pm 1}\cdot x \in S_\pm.
\end{equation}
Then we have the following isomorphisms of $S^1$-spaces:
\begin{align}\label{S1-space-id-1}
  &(\pi,H_{\alpha, \theta}):\breve M^+(L;\alpha,\widetilde \theta)_{i+1}\to \breve M^+(L;\alpha,\theta)_{i}\times S_-,
\\
  &(\pi',H_{ \theta',\alpha'}):\breve M^+(L';\widetilde \theta',\alpha')_{d-i}\to \breve M^+(L';\theta',\alpha')_{d-i}\times S_+, \label{S1-space-id-2}
\end{align}
where $\pi$ and $\pi'$ are projection maps. The $S^1$-equivariant maps 
\[
 \widetilde \sigma_\pm: (x_-,x_+)\in S_-\times S_+ \mapsto (x_{\pm})^{\pm 1} \in S^1\subset \C 
\]
determine sections $\sigma_\pm$ of the complex line bundle $\mathcal H$ over $S_-\times_{S^1}S_+$ which is the quotient of $ S_-\times S_+\times \C$ by the diagonal $S^1$ action.
Fix a section $u$ of $\mathcal H\times [-1,1]$ over $S_-\times_{S^1}S_+\times [-1,1]$ which agrees with $\sigma_\pm$ over $S_-\times_{S^1}S_+\times \{\pm 1\}$ and has a unique transversal intersection with the zero section at $([1,1],0)\in S_-\times_{S^1}S_+\times [-1,1]$.

We may use \eqref{S1-space-id-1} and \eqref{S1-space-id-2} to obtain an identification 
\[
  \breve M^+(L;\alpha,\widetilde \theta)_{i+1} \times_{S^1} \breve M^+(L';\widetilde \theta',\alpha')_{d-i} \cong  \breve M^+(L;\alpha,\theta)_{i}\times \breve M^+(L'; \theta',\alpha')_{d-i}\times S_- \times_{S^1}S_+
\]
and with respect to this identification, $\mathcal H_{\alpha,\alpha'}$ is identified with the pullback of the bundle $\mathcal H$. After further identifying $S_\Theta$ with a closed interval, say $[-1,1]$, $\psi_{\alpha,\Theta,\alpha'}'$ can be regraded as a section of the pullback of $\mathcal H$ to the space 
\begin{equation}\label{domain-c-a-T-a-p}
  \breve M^+(L;\alpha,\theta)_{i}\times \breve M^+(L'; \theta',\alpha')_{d-i}\times S_- \times_{S^1}S_+ \times [-1,1].
\end{equation}
Our previous requirements on $\Xi_{\alpha,\Theta,\alpha'}$ imply that over the subspace 
\[
\breve M^+(L;\alpha,\theta)_{i}\times \breve M^+(L';\theta',\alpha')_{d-i}\times S_- \times_{S^1}S_+ \times \{\pm 1\},
\]
$\psi_{\alpha,\Theta,\alpha'}'$ is given by $\sigma_{\pm}$. We further demand that the section $\psi_{\alpha,\Theta,\alpha'}'$ over \eqref{domain-c-a-T-a-p} equals the section $u$ as described in the previous paragraph.

\vspace{.2cm}

\begin{remark}\label{difference-models}
	There is a difference between how we construct the sections $\Xi_{\alpha,\Theta}$ here and the corresponding ones in \cite{DME:QHS}. 
	For instance, the section $\Xi_{\alpha,\Theta}$ here interpolates 
	between the section $\psi_{\alpha,\Theta}$ given by obstructed gluing theory and $\psi_{\alpha,\Theta}'$, which depends only on the modified 
	holonomy map associated to the incoming end of the cobordism. On the other hand, the analogous sections in \cite{DME:QHS} interpolate between obstructed gluing 
	theory sections and some fixed section associated to the incoming end. In particular, there is more flexibility in the choice of this section, as it is not necessarily given by (modified) holonomy maps. 
	
	In fact, the setup of \cite{DME:QHS} does not allow using modified holonomy maps 
	in the same way as in the present work to get the counterpart of $\psi_{\alpha,\Theta}'$, because of the term $H_{\alpha,\beta}(x)$ in \eqref{co-dim-1-res}. 
	In \cite{DME:QHS}, the counterpart of $\psi_{\alpha,\Theta}'$ satisfies the analogue of \eqref{co-dim-1-res}, where  $H_{\alpha,\beta}(x)$  is replaced with $1$. 
	The necessity to use modified holonomy maps to obtain $\psi_{\alpha,\Theta}'$ and its variations in the present work arises from the point that we use smaller 
	(in particular, finite rank) models to define our $\cS$-complexes. In the next subsection, we use these sections to modify our moduli spaces, and the 
	more specific form of $\psi_{\alpha,\Theta}'$ and its variations allows us to have better control on the boundary of our modified moduli spaces.
 \end{remark}

\subsection{Height $-1$ cobordism maps}\label{obs-cob-maps-level--1}

We are ready to define the cobordism map associated to $(W,S,c):(Y,L)\to (Y',L')$, a given cobordism which is negative definite of height $-1$. In particular, we define maps $\lambda$, $\mu$, $\Delta_1$, $\Delta_2$, $\tau_0$ and $\tau_{-1}$ satisfying the following:
\begin{align}
	d'\lambda - \lambda d-\delta_2'\tau_{-1} \delta_1&= 0 \label{lambda-bdry-relation}\\
	- \delta_1'\lambda+\Delta_1 d +\tau_{-1}\delta_1v+\tau_{0}\delta_1  &=0  \label{Delta-1-bdry-relation}\\
	\lambda \delta_2+d'\Delta_2  -v'\delta_2'\tau_{-1}-\delta_2'\tau_{0}  & =0  \label{Delta-2-bdry-relation}\\
	\mu d + d'\mu + \lambda v - v'\lambda + \Delta_2\delta_1 - \delta_2'\Delta_1  & = 0 \label{mu-bdry-relation}
\end{align}

We start by defining the map $\lambda:C(Y,L)\to C(Y',L')$. As in the unobstructed case, this map is defined using the moduli spaces $M^+(W,S,c;\alpha,\alpha')_0$ with $\alpha\in \fC_\pi^\text{irr}$ and $\alpha'\in \fC_{\pi'}^{\text{irr}}$. In general, the moduli space $M^+(W,S,c;\alpha,\alpha')_d$ contains obstructed solutions, which are necessarily of type III, only if $d\geq 1$. In particular, $M(W,S,c;\alpha,\alpha')_0$ is a compact smooth $0$-dimensional manifold. After orienting this moduli space (see Subsection \ref{orientation} below), define the homomorphism $\lambda:C(Y,L)\to C(Y',L')$ by
\[
  \langle\lambda(\alpha),\alpha'\rangle = \# M(W,S,c;\alpha,\alpha')_0 . 
\]
The following proposition asserts that $\lambda$ is not quite a chain map. The failure of $\lambda$ being a chain map is measured by the operators $\delta_1$ and $\delta_2'$. The map $\eta_{-1}(W,S,c):\sfR\to \sfR'$ in this proposition is given by a signed count of reducibles of index $-3$ over $(W,S,c)$, where the signs are specified in Subsection \ref{orientation}.

\begin{prop}\label{bdry-rel-lambda}
	The homomorphism $\lambda:C(Y,L)\to C(Y',L')$ satisfies \eqref{lambda-bdry-relation}, where 
	\[
		\tau_{-1}:=\eta_{-1}(W,S,c).
	\]
\end{prop}

The proof of this proposition uses the moduli spaces $M^+(W,S,c;\alpha,\alpha')_1$. Since these spaces might contain obstructed solutions of type III, they need to be modified in a neighborhood of any such obstructed solution to obtain a stratified-smooth space.  For $d\leq 3$, we define $N^+(W,S,c;\alpha,\alpha')_d$ by first removing 
\begin{equation}\label{first-part-N-complement}
	\bigsqcup_{\Theta,i} \Phi_{\alpha,\Theta,\alpha'}\(\Psi_{\alpha,\Theta,\alpha'}^{-1}(0)\cap(\breve M^+(L;\alpha,\widetilde \theta)_{i+1} \times_{S^1} \breve M^+(L';\widetilde \theta',\alpha')_{d-i}\times U_\Theta)\)
\end{equation}
from $M^+(S;\alpha,\alpha')_d$ and then taking the disjoint union with
\begin{equation}\label{second-part-N}
	\bigsqcup_{\Theta} (\Xi_{\alpha,\Theta,\alpha'})^{-1}(0)\cap \(\breve M^+(L;\alpha,\widetilde \theta)_{i+1} \times_{S^1} \breve M^+(L';\widetilde \theta',\alpha')_{d-i}\times S_\Theta\times [0,1]\),
\end{equation}
where $\Theta$ is a reducible over $(W,S,c)$ with index $-3$ and limits $\theta,\theta'$. Analogous to other moduli spaces associated to the cobordism $(W,S,c)$, we write $N^+(S;\alpha,\alpha')_d$ for this moduli space when it does not cause any confusion. The topological space $N^+(S;\alpha,\alpha')_d$ is a compact $d$-dimensional stratified-smooth space because a neighborhood of the obstructed solutions are removed from $M^+(S;\alpha,\alpha')_d$ and the sections $\psi_{\alpha,\Theta,\alpha'}$ and $\Xi_{\alpha,\Theta,\alpha'}$ are transverse to the zero section. Our convention in orienting this moduli space is discussed in Subsection \ref{orientation}. The orientation on the complement of \eqref{second-part-N} in $N^+(S;\alpha,\alpha')_d$ is inherited from $M^+(S;\alpha,\alpha')_d$. This space and \eqref{second-part-N} have common boundary given by the zeroes of $\psi_{\alpha,\Theta,\alpha'}$, and the orientation on \eqref{second-part-N} is chosen that the induced boundary orientations on this boundary component disagrees with each other.

\begin{proof}
	The boundary strata of $N^+(S;\alpha,\alpha')_1$ can be described as follows. The boundary components from the complement of \eqref{first-part-N-complement} 
	in $M^+(S;\alpha,\alpha')_1$ are given by
	\[
	  \breve M^+(L;\alpha,\beta)_{0}\times M^+(S;\beta,\alpha')_0,\hspace{1cm}M^+(S;\alpha,\beta')_0\times \breve M^+(L';\beta',\alpha')_{0},
	\]
	and $\Phi_{\alpha,\Theta,\alpha'}(\psi_{\alpha,\Theta,\alpha'}^{-1}(0))$ where $\beta\in \fC^{\rm irr}_\pi$, $\beta'\in \fC^{\rm irr}_{\pi'}$ and $\Theta:\theta\to \theta'$ is a reducible over $(W,S,c)$ with index $-3$.
	In particular, a signed count of the boundary components of this part of $N^+(W,S,c;\alpha,\alpha')_1$ implies that
	\begin{equation}\label{cut-moduli-bdry}
	  \langle d'\lambda( \alpha),\alpha'\rangle -\langle \lambda d(\alpha),\alpha'\rangle+ \sum_{\Theta}\# \(\psi_{\alpha,\Theta,\alpha'}^{-1}(0)\).
	\end{equation}
	
	Now, we analyze the boundary points of the second part of $N^+(S;\alpha,\alpha')_1$ given in \eqref{second-part-N}. In this case, $i=0$. Our transversality assumption on $\Xi_{\alpha,\Theta}$ over $\breve M^+(L;\alpha,\theta)_{0} \times [0,1]$ 
	and $\Xi_{\Theta,\alpha'}$ over $\breve M^+(L';\theta',\alpha')_{0}\times [0,1]$ and the relation between these sections and $\Xi_{\alpha,\Theta,\alpha'}$
	imply that the boundary of \eqref{second-part-N} is equal to
	\begin{equation}\label{bdry-part-two}
	  \bigsqcup_{\Theta} \((\psi_{\alpha,\Theta,\alpha'}')^{-1}(0)\sqcup -\psi_{\alpha,\Theta,\alpha'}^{-1}(0)\),
	\end{equation}
	where $\psi_{\alpha,\Theta,\alpha'}'$ is the restriction of $\Xi_{\alpha,\Theta,\alpha'}$ to $\breve M^+(L;\alpha,\widetilde \theta)_{1} \times_{S^1} \breve M^+(L';\widetilde \theta',\alpha')_{1}\times S_\Theta\times \{1\}$.
	Any connected component of $\breve M^+(L;\alpha,\widetilde \theta)_{1}$ (resp. $ \breve M^+(L';\widetilde \theta',\alpha')_{1}$) can be identified with the space $S_-$ (resp. $S_+$) in \eqref{S-m-p}. By our assumptions on $\psi_{\alpha,\Theta,\alpha'}'$, it has exactly one zero over $S_-\times_{S^1}S_+ \times S_\Theta\times \{1\}$, where the sign is determined by the orientations of $S_-$ and $S_+$. From this we see that the signed count of $(\psi_{\alpha,\Theta,\alpha'}')^{-1}(0)$ equals 
	\[
	  -\#\(\breve M^+(L;\alpha,\theta)_{0}\)\cdot \#\(\breve M^+(L';\theta',\alpha')_{0}\)
	\]
	Combining \eqref{cut-moduli-bdry}, \eqref{bdry-part-two} and the characterization of the signed count of $(\psi_{\alpha,\Theta,\alpha'}')^{-1}(0)$ establishes the claim. 	
\end{proof}

The  definition of the operator $\Delta_1$ associated to $(W,S,c)$ uses a modified version of $M^+(S;\alpha,\theta')_0$. The moduli spaces $M^+(W,S,c;\alpha,\theta')_0$ contain obstructed solutions of type I, parametrized by the spaces $\breve M^+(L;\alpha,\theta)_{1}\times\{\Theta\}$ where $\Theta$ is an index $-3$ reducible from some reducible flat connection $\theta$ to the fixed reducible $\theta'$ on $(Y',L')$. Define $N^+(W,S,c;\alpha,\theta')_0$ (usually abbreviated as $N^+(S;\alpha,\theta')_0$) to be the disjoint union of 
\begin{equation}\label{Delta-1-1-def-mod-space}
  M^+(S;\alpha,\theta')_0\setminus \bigsqcup_{\Theta}\Phi_{\alpha,\Theta}\(\Psi_{\alpha,\Theta}^{-1}(0)\cap (\breve M^+(L;\alpha,\theta)_{1}\times (T_\Theta,\infty])\)
\end{equation}
\begin{equation}\label{Delta-1-2-def-mod-space}
 \text{and} \qquad  \bigsqcup_\Theta (\Xi_{\alpha,\Theta})^{-1}(0)\cap \(\breve M^+(L;\alpha,\theta)_{1}\times [0,1]\),
\end{equation}
where $\Theta$ runs over all index $-3$ reducibles which are asymptotic to $\theta'$ on the outgoing end. The space $N^+(S;\alpha,\theta')_0$ is a compact oriented $0$-dimensional manifold. We use this space to define $\Delta_1:C(Y,L)\to \sfR'$ as follows:
\begin{equation}\label{Delta-1-def}
  \langle\Delta_1(\alpha),\theta'\rangle = \# N^+(	S;\alpha,\theta')_0 .   
\end{equation}
Alternatively, we may replace $N^+(S;\alpha,\theta')_0$ in \eqref{Delta-1-def} with the space in \eqref{Delta-1-1-def-mod-space}, denoted by $M_c^+(S;\alpha,\theta')_0$, to define a variation $\Delta_1^1:C(Y,L)\to \sfR'$ of $\Delta_1$. The space in \eqref{Delta-1-2-def-mod-space} gives us another operator $\Delta_1^2:C(Y,L)\to \sfR'$. In particular, we have $\Delta_1=\Delta_1^1+\Delta_1^2$.

\begin{prop}\label{Delta-1}
	The homomorphism $\Delta_1:C(Y,L)\to \sfR'$ satisfies \eqref{Delta-1-bdry-relation}, where
	\begin{align}
	  \tau_{0}&:=\eta_0(W,S,c)-s'\eta_{-1}(W,S,c)+\eta_{-1}(W,S,c)s, \label{tau-obs}\\[1mm]
	  \tau_{-1}&:=\eta_{-1}(W,S,c). \label{tau-obs2}
	\end{align}
	Recall that $s:\sfR\to \sfR$ and $s':\sfR'\to \sfR'$ are introduced in Definition \ref{defn:smap}.
\end{prop}
\begin{proof}
	We wish to form a modification $N^+(S;\alpha,\theta')_1$ of $1$-dimensional moduli spaces $M^+(S;\alpha,\theta')_1$ and derive 
	\eqref{Delta-1-bdry-relation} from the vanishing of the signed count of its boundary points.
	In addition to type I obstructed solutions, the 
	moduli space $M^+(S;\alpha,\theta')_d$ for $d\geq 1$ has type III obstructed solutions of the form
	\begin{equation}\label{obs-red-new}
	  \breve M^+(L;\alpha,\widetilde \theta)_{i+1} \times_{S^1} \overline{\underline \Theta} \times_{S^1} \breve M^+(L';\widetilde \rho',\theta')_{d-i}
	\end{equation}
	for any reducible $\Theta$ of index $-3$ over $(W,S,c)$, with limits $\theta$ and $\rho'$, and $0\leq i\leq d-1$. In the same way as before, we can identify this space with 
	$\breve M^+(L;\alpha,\widetilde \theta)_{i+1} \times_{S^1} \breve M^+(L';\widetilde \rho',\theta')_{d-i}$.
	Define $N^+(S;\alpha,\theta')_d$ as the disjoint union of 
	\begin{align}
	  M^+(S;\alpha,\theta')_d\setminus \bigsqcup_{\Theta}\Phi_{\alpha,\Theta}\(\Psi_{\alpha,\Theta}^{-1}(0)\cap (\breve M^+(L;\alpha,\theta)_{d+1}\times (T_\Theta,\infty])\) \label{first-part-1-Delta-1} \\
	  \setminus \bigsqcup_{\Theta,i}\Phi_{\alpha,\Theta,\theta'}\(\Psi_{\alpha,\Theta,\theta'}^{-1}(0)\cap(\breve M^+(L;\alpha,\widetilde \theta)_{i+1} \times_{S^1} \breve M^+(L';\widetilde \rho',\theta')_{d-i}\times U_\Theta)\),\label{first-part-2-Delta-1}
	\end{align}
	\begin{equation}\label{second-part-N-Delta-1}
	\bigsqcup_\Theta (\Xi_{\alpha,\Theta})^{-1}(0)\cap \(\breve M^+(L;\alpha,\theta)_{d+1}\times [0,1]\),
	\end{equation}
	\begin{equation}\label{third-part-N-Delta-1}
	  \bigsqcup_{\Theta,i} (\Xi_{\alpha,\Theta,\theta'})^{-1}(0)\cap\(\breve M^+(L;\alpha,\widetilde \theta)_{i+1} \times_{S^1} \breve M^+(L';\widetilde \rho',\theta')_{d-i}\times S_\Theta\times [0,1]\).
	\end{equation}
	In \eqref{first-part-1-Delta-1} and \eqref{second-part-N-Delta-1}, $\Theta$ is a reducible of index $-3$ from $\theta$ to $\theta'$, and in \eqref{first-part-2-Delta-1} and \eqref{third-part-N-Delta-1}, 
	$\Theta$ is an index $-3$ reducible from $\theta$ to $\rho'$ and $0\leq i\leq d-1$. 
	Propositions \ref{type-1-obs-gluing} and \ref{type-3-obs-gluing} and the transversality properties of $\Xi_{\alpha,\Theta}$ and $\Xi_{\alpha,\Theta,\alpha'}$ imply that $N^+(S;\alpha,\theta')_d$
	is a $d$-dimensional stratified-smooth space. As it is explained in Subsection \ref{orientation}, this moduli space is oriented and a choice of an orientation for this 
	moduli space is specified there.
	
	We analyze the boundary components of the different parts of $N^+(S;\alpha,\theta')_1$. The boundary components of the subspace in \eqref{first-part-1-Delta-1}--\eqref{first-part-2-Delta-1} are given by
	\[
	  (\bigsqcup_\beta \breve M(L;\alpha,\beta)_{0} \times M_c^+(S;\beta,\theta')_0)\sqcup (\bigsqcup_{\beta'} M^+(S;\alpha,\beta')_0\times \breve M(L';\beta',\theta')_{0} ),
	\]
	\begin{equation}\label{bdry-com-N-Delta-1-2}
	  (\bigsqcup_{\Theta} \breve M(L';\alpha,\theta)_{0}  \times \{\Theta\}) \cup  (\bigsqcup_{\Theta}\Phi_{\alpha,\Theta}(\psi_{\alpha,\Theta}^{-1}(0))) \cup  (\bigsqcup_{\Theta}\Phi_{\alpha,\Theta,\theta'}(\psi_{\alpha,\Theta,\theta'}^{-1}(0)))
	\end{equation}
	where the first disjoint union in \eqref{bdry-com-N-Delta-1-2} runs over reducibles $\Theta$ from $\theta$ to $\theta'$ of index $-1$, 
	in the second disjoint union $\Theta$ is similar but has index $-3$, and in the third, 
	$\Theta$ from $\theta$ to $\rho'$ has index $-3$. The sum of the boundary components of this part of $N^+(S;\alpha,\theta')_1$ gives
	\begin{equation}\label{identity-1-Delta-1}
	  \langle- \delta_1'\lambda(\alpha)+\Delta_1^1 d(\alpha)+\eta_0\delta_1(\alpha),\theta'\rangle+\sum_{\Theta}\#\psi_{\alpha,\Theta}^{-1}(0)+\sum_{\Theta}\#\psi_{\alpha,\Theta,\theta'}^{-1}(0)=0
	\end{equation}
	with the first sum runs over all reducibles $\Theta$ from $\theta$ to $\theta'$ of index $-3$ and the second sum runs over all reducibles $\Theta$ from $\theta$ to $\rho'$ of index $-3$.
	
	The boundary points of \eqref{second-part-N-Delta-1} are given by the zeros of $\Xi_{\alpha,\Theta}$ on the subspaces 
	\begin{gather}\label{second-part-N-Delta-1-bdry-1}
		\breve M^+(L;\alpha,\theta)_{2}\times\{0\},\hspace{1cm}\breve M^+(L;\alpha,\theta)_{2}\times\{1\},\\
\label{second-part-N-Delta-1-bdry-2}
		\breve M^+(L;\alpha,\beta)_{i}\times \breve M^+(L;\beta,\theta)_{1-i}\times [0,1],\hspace{0.3cm}\forall \beta\in \fC^{\rm irr}_{\pi},\, i\in \{0, 1\}\\
\label{second-part-N-Delta-1-bdry-3}
		\text{ and } \qquad \breve M^+(L;\alpha,\widetilde \rho)_{1}\times_{S^1} \breve M^+(L;\widetilde \rho,\theta)_{1}\times [0,1]
	\end{gather}
	of $\breve M^+(L;\alpha,\theta)_{2}\times[0,1]$. The restriction of $\Xi_{\alpha,\Theta}$ to the first space in \eqref{second-part-N-Delta-1-bdry-1} equals $\psi_{\alpha,\Theta}$, and its restriction to the second subspace of \eqref{second-part-N-Delta-1-bdry-1}, 
	$\psi_{\alpha,\Theta}'$, does not have any zero. Our assumption on $\Xi_{\alpha,\Theta}$ implies that its restriction to \eqref{second-part-N-Delta-1-bdry-2} for $i=0$ is equal to $\Xi_{\beta,\Theta}$. In particular, the signed count of the 
	zeros of $\Xi_{\alpha,\Theta}$ over this subspace for all $\Theta$ from $\theta$ to $\theta'$ with index $-3$, $\beta\in \fC^{\text{irr}}_{\pi}$ and $i=0$, gives the term $\langle \Delta_{1}^2d(\alpha),\theta'\rangle$. Continuing with \eqref{second-part-N-Delta-1-bdry-2}, for $i=1$, recall from \eqref{co-dim-1-res-2} that the restriction of $\Xi_{\alpha,\Theta}$ is given by the section
	\begin{align}\label{fL-alt}
	 (x,y,t)\in \breve M^+(L;\alpha,\beta)_{1}\times &\breve M^+(L;\beta,\theta)_{0}\times [0,1]  \nonumber\\
	  \mapsto &((1-t)+t H_{\alpha,\beta}(x))\cdot \Xi_{\beta,\Theta}(y,t) 
	\end{align}	
	of $\mathcal H_{\alpha,\theta}$ over the boundary strata of $\breve M^+(L;\alpha,\beta)_{1}\times \breve M^+(L;\beta,\theta)_{0}\times [0,1]$. Moreover, \eqref{fL-alt} has vanishing set given by
	\[
	  \{(x,y,\frac{1}{2})\mid x\in \breve M^+(L;\alpha,\beta)_{1},\, y\in \breve M^+(L;\beta,\theta)_{0},\, H_{\alpha,\beta}(x)=-1\}.
	\]
	Thus the signed count of the zeros of $\Xi_{\alpha,\Theta}$ over \eqref{second-part-N-Delta-1-bdry-2} for all $\beta\in \fC^{\rm irr}_{\pi}$ and $i=1$ is equal to $\langle \delta_1 v (\alpha),\theta\rangle$.	
	The same argument as in the proof of Proposition \ref{lambda-bdry-relation} implies that the signed count of the boundary components of the zeros of $\Xi_{\alpha,\Theta}$ over \eqref{second-part-N-Delta-1-bdry-3} 
	with the induced boundary orientation is equal to the negative of the product of 
	$\#\breve M^+(L;\alpha,\rho)_{0}$ and $\#\breve M^+(L;\rho,\theta)_{0}$. In summary, summing over all boundary components of \eqref{second-part-N-Delta-1} for all choices of $\Theta$ gives the relation
	\begin{equation}\label{identity-2-Delta-1}
	  \langle \Delta_1^2 d(\alpha)+\eta_{-1}\delta_1v(\alpha)+\eta_{-1}s \delta_1(\alpha),\theta'\rangle-\sum_{\Theta}\#\psi_{\alpha,\Theta}^{-1}(0)=0.
	\end{equation}
	
	An argument as in the proof of Proposition \ref{lambda-bdry-relation} shows that the signed count of the zeros of $\Xi_{\alpha,\Theta,\alpha'}$ over any connected component of 
	\[
	  \breve M^+(L;\alpha,\widetilde \theta)_{1} \times_{S^1} \breve M^+(L';\widetilde \rho',\theta')_{1}\times S_\Theta\times \{1\}
	\]
	is $-1$ and the restriction of  $\Xi_{\alpha,\Theta,\alpha'}$ to
	\[
	  \breve M^+(L;\alpha,\widetilde \theta)_{1} \times_{S^1} \breve M^+(L';\widetilde \rho',\theta')_{1}\times S_\Theta\times \{0\}
	\]
	equals $\psi_{\alpha,\Theta,\alpha'}$, and $\Xi_{\alpha,\Theta,\alpha'}$ does not have any zero over 
	\[
	  \breve M^+(L;\alpha,\widetilde \theta)_{1} \times_{S^1} \breve M^+(L';\widetilde \rho',\theta')_{1}\times \{(T_\Theta,\infty),(\infty,T_\Theta)\}\times [0,1].
	\]	
	Therefore, the sum over the boundary components of \eqref{third-part-N-Delta-1} gives 
	\begin{equation}\label{identity-3-Delta-1}
		-\langle s'\eta_{-1} \delta_1(\alpha),\theta'\rangle -\sum_{\Theta}\#\psi_{\alpha,\Theta,\theta'}^{-1}(0)=0.
	\end{equation}
	Combining \eqref{identity-1-Delta-1}, \eqref{identity-2-Delta-1} and \eqref{identity-3-Delta-1} proves the claim.
\end{proof}

\begin{remark}
	The compatibility of $\Phi_{\alpha,\Theta}$, $\Psi_{\alpha,\Theta}$ from Proposition \ref{type-1-obs-gluing} and 
	$\Phi_{\alpha,\Theta,\alpha'}$, $\Psi_{\alpha,\Theta,\alpha'}$ 
	from Proposition \ref{type-3-obs-gluing} can be used to define the framed variation $N^+(S;\alpha,\widetilde \theta')_{d+1}$ of $N^+(S;\alpha, \theta')_{d}$. In particular, there is an $S^1$ action on $N^+(S;\alpha,\widetilde \theta')_{d+1}$ such that the quotient is 
	$N^+(S;\alpha, \theta')_{d}$. Furthermore, we can define analogues of modified holonomy maps as $S^1$-equivariant maps 
	$N^+(S;\alpha,\widetilde \theta')_{d+1} \to S^1$ as it is explained below. 
\end{remark}

As in the proof of Proposition \ref{Delta-1}, we can modify the moduli spaces $M^+(S;\theta,\alpha')_{d}$ to $N^+(S;\theta,\alpha')_{d}$ by first removing a neighborhood of obstructed solutions of type II and III in these moduli spaces and then using the sections $\Xi_{\alpha,\Theta}$ and $\Xi_{\alpha,\Theta,\alpha'}$ to add a new part to the moduli space which allows us to have a better control on the boundary of the modified space. The moduli spaces $N^+(S;\theta,\alpha')_{0}$ can be used to define the homomorphism $\Delta_2:\sfR\to C(Y',L')$, and the following proposition can be verified by considering the boundary components of $N^+(S;\theta,\alpha')_1$.  Since this proposition can be proved using essentially the same argument as in Proposition \ref{Delta-1}, we omit the details.

\begin{prop}\label{Delta-2}
	The homomorphism $\Delta_2:\sfR\to C(Y',L')$ satisfies relation \eqref{Delta-2-bdry-relation}, where $\tau_{-1}:\sfR\to \sfR'$ and $\tau_{0}:\sfR\to \sfR'$ are the same as in Proposition \ref{Delta-1}.
\end{prop}

As the last ingredient of the height $-1$ morphism associated to $(W,S,c)$, we define a homomorphism $\mu:C(Y,L)\to C(Y',L')$ using the moduli spaces $N^+(S;\alpha,\alpha')_d$ cut down by the (modified) holonomy of connections along a given path $\gamma$ on $S$ connecting the basepoints of  $L$ and $L'$. First, we define maps 
\begin{equation}\label{eq:holmapobs}
  \bH_{\alpha,\alpha'}^\gamma: N^+(S;\widetilde \alpha,\widetilde \alpha')_d\to S^1
\end{equation}
for any $d\leq 2$, which are the counterparts of the $S^1$-valued maps $H^\gamma_{\alpha,\alpha'}$, defined on $M^+(S;\alpha,\alpha')_d$. We assume that for any such $d$, $\alpha$ and $\alpha'$, the modified holonomy map
\[
  H^\gamma_{\alpha,\alpha'}:M^+(S;\widetilde \alpha,\widetilde \alpha')_d \to S^1
\]
is transverse to $-1\in S^1$ on any stratum of $M^+(S;\alpha,\alpha')_d$. For any reducible $\Theta$ of index $-3$ and any generic value of $T_\Theta$, the restriction of $H^\gamma_{\alpha,\alpha'}$ to $\widetilde \Phi_{\alpha,\Theta,\alpha'}(\psi_{\alpha,\Theta,\alpha'}^{-1}(0))$ is also traverse to $-1$ on each stratum. Similarly, if $\alpha$ (resp. $\alpha'$) is reducible, the restriction of $H^\gamma_{\alpha,\theta'}$ (resp. $H^\gamma_{\theta,\alpha'}$) to $\widetilde \Phi_{\alpha,\Theta}(\psi_{\alpha,\Theta}^{-1}(0))$   (resp. $\widetilde \Phi_{\Theta,\alpha'}(\psi_{\Theta,\alpha'}^{-1}(0))$) is transverse. We pick $T_\Theta$ such that all of these transversality conditions are satisfied. We set $\bH_{\alpha,\alpha'}^\gamma=H^\gamma_{\alpha,\alpha'}$ on the part of $N^+(S;\alpha,\alpha')_d$ which is a subset of $M^+(S;\alpha,\alpha')_d$.

Next, we need to extend $\bH_{\alpha,\alpha'}^\gamma$ to the subspaces of the form: 
\begin{equation}\label{domain-2-H-a-a-p}
   (\Xi_{\alpha,\Theta,\alpha'})^{-1}(0)\cap\(\breve M^+(L;\widetilde \alpha,\widetilde \theta)_{i+1} \times_{S^1} \breve M^+(L';\widetilde \theta',\widetilde \alpha')_{d-i}\times S_\Theta\times [0,1]\);
\end{equation}
to the subspaces of the following form for reducible $\alpha'$:
\begin{equation}\label{domain-2-H-a-a-p-2}
	(\Xi_{\alpha,\Theta})^{-1}(0)\cap \(\breve M^+(L;\widetilde \alpha, \widetilde \theta)_{d+1}\times [0,1]\);
\end{equation}
and to the subspaces of the following form for reducible $\alpha$:
\begin{equation}\label{domain-2-H-a-a-p-3}
	(\Xi_{\Theta,\alpha'})^{-1}(0)\cap \(\breve M^+(L';\widetilde\theta',\widetilde\alpha')_{d+1}\times [0,1]\).
\end{equation}
In fact, we claim that there are such maps that are transverse to $-1$ on all strata and they have the following properties on the boundary components of \eqref{domain-2-H-a-a-p}, \eqref{domain-2-H-a-a-p-2} and \eqref{domain-2-H-a-a-p-3}:
\begin{itemize}
	\item[(i)] If $([x,y],(t,t'),0)$ is an element of \eqref{domain-2-H-a-a-p}, then: 
		\[\bH^\gamma_{\alpha,\alpha'}([x,y],(t,t'),0)=H^\gamma_{\alpha,\alpha'}\circ\widetilde \Phi_{\alpha,\Theta,\alpha'}([x,y],(t,t')).\]
		Similarly, for $(x,0)$ in \eqref{domain-2-H-a-a-p-2} and $(y,0)$ in \eqref{domain-2-H-a-a-p-3}, we have
		\[\bH_{\alpha,\theta'}^\gamma(x,0)=H^\gamma_{\alpha,\theta'}\circ\widetilde \Phi_{\alpha,\Theta}(x,T_\Theta),\hspace{1cm}
		\bH_{\theta,\alpha'}^\gamma(y,0)=H^\gamma_{\theta,\alpha'}\circ\widetilde \Phi_{\Theta,\alpha'}(y,T_\Theta).\]
	\item[(ii)] If $([x,y],(t,t'),1)$ is an element of \eqref{domain-2-H-a-a-p}, then: 
		\[\bH^\gamma_{\alpha,\alpha'}([x,y],(t,t'),1)=H_{\alpha,\theta}(x)H_{\theta',\alpha'}(y).\]	
	\item[(iii)] If $([x,y],(T_\Theta,\infty),t)$ is an element of \eqref{domain-2-H-a-a-p}, then: 
		\[\bH^\gamma_{\alpha,\alpha'}([x,y],(T_\Theta,\infty),t)=\bH^\gamma_{\alpha,\theta'}(x,t)H_{\theta',\alpha'}(y).\]	
	\item[(iv)] If $([x,y],(\infty,T_\Theta),t)$ is an element of \eqref{domain-2-H-a-a-p}, then: 
		\[\bH^\gamma_{\alpha,\alpha'}([x,y],(\infty,T_\Theta),t)=H_{\alpha,\theta}(x)\bH^\gamma_{\theta,\alpha'}(y,t).\]	
	\item[(v)] These maps are multiplicative on broken moduli spaces. For instance, if 
	\[[x,y,z]\in  \breve M^+(L;\widetilde \alpha,\widetilde \beta)_{j+\dim(\Gamma_{\beta})} \times_{\Gamma_{\beta}} 
	\breve M^+(L;\widetilde \beta,\widetilde \theta)_{i-j} \times_{S^1} \breve M^+(L';\widetilde \theta',\widetilde \alpha')_{d-i},\]
	such that $([x,y,z],(t,t'),s)$ is an element of \eqref{domain-2-H-a-a-p}, then: 
	\[
	  \bH^\gamma_{\alpha,\alpha'}([x,y,z],(t,t'),s)=H_{\alpha,\theta}(x)\bH^\gamma_{\beta,\alpha'}([y,z],(t,t'),s).
	\]
\end{itemize}

To verify the existence of the maps $\bH^\gamma_{\alpha,\alpha'}$, $\bH^\gamma_{\alpha, \theta'}$ and $\bH^\gamma_{ \theta,\alpha'}$, note that the given properties specify the maps on the boundary and we need to check there is no obstruction to extend these maps to the interior of the domain of these maps. To see this, note that if we set $\bH^\gamma_{\alpha,\alpha'}([x,y],(t,t'),s)=H^\gamma_{\alpha,\theta}(x)H_{\theta',\alpha'}(y)$, $\bH_{\alpha,\theta'}^\gamma(x,s)=H_{\alpha,\theta}(x)$ and $\bH_{\theta,\alpha'}^\gamma(y,s)=H_{\theta',\alpha'}(y)$, conditions (ii)-(v) are satisfied. Even though (i) is not necessarily satisfied, we obtain the required identities in (i) in the limit as $T_\theta\to \infty$. Thus there is no obstruction. 

These maps can be perturbed, if necessary, to be transverse to $-1$, while still satisfying (i)--(v). To check that this is possible, we first show that $\bH^\gamma_{\alpha,\alpha'}$ does not take the value $-1$ on the subspace of its domain given by the zeros of $\Xi_{\alpha,\Theta,\alpha'}$ on 
\begin{equation}\label{final-bdry-str}
	\breve M^+(L;\alpha,\widetilde \theta)_{i+1} \times_{S^1} \breve M^+(L';\widetilde \theta',\alpha')_{d-i}\times S_\Theta\times \{1\}.
\end{equation}
After identifying $S_\Theta\times \{1\}$ with $[-1,1]$, our earlier discussion identifies \eqref{final-bdry-str} with 
\[
  \breve M^+(L;\alpha, \theta)_{i} \times \breve M^+(L';\theta',\alpha')_{1-i}\times S_-\times_{S^1} S_+\times [-1,1],
\]
and $\Xi_{\alpha,\Theta,\alpha'}$ on this space is given by the pullback of the section $u$ over the space $S_-\times_{S^1}S_+\times [-1,1]$. In particular, the zeros of $\Xi_{\alpha,\Theta,\alpha'}$ on \eqref{final-bdry-str} are given by
\begin{equation}\label{cut-down-mod-ob-sec}
	\breve M^+(L;\alpha, \theta)_{i} \times \breve M^+(L';\theta',\alpha')_{1-i}\times \{[1,1]\}\times \{0\}.
\end{equation}	
Property (ii) implies that the restriction of $\bH^\gamma_{\alpha,\alpha'}$ to \eqref{final-bdry-str} is the pullback of the map $S_-\times_{S^1} S_+\to S^1$ given by $[x,x']\to xx'$. In particular, all elements in \eqref{cut-down-mod-ob-sec} are mapped to $1$ via $\bH^\gamma_{\alpha,\alpha'}$. This shows that $(\bH^\gamma_{\alpha,\alpha'})^{-1}(-1)$ has an empty intersection with \eqref{final-bdry-str}. Given this, it is easy to see now that a small perturbation may be chosen such that the maps $\bH^\gamma_{\alpha,\alpha'}$, $\bH^\gamma_{\alpha, \theta'}$ and $\bH^\gamma_{\theta,\alpha'}$ satisfy conditions (i)--(v) and are transverse to $-1$.

The properties of $\bH^\gamma_{\alpha,\alpha'}: N^+(S;\alpha,\alpha')_d\to S^1$ imply that for $d\leq 2$, the space 
\[
  (\bH^\gamma_{\alpha,\alpha'})^{-1}(-1)\cap N^+(S;\alpha,\alpha')_d 
\]
is a compact oriented stratified-smooth space of dimension $d-1$. In particular, in the case that $d=1$, we obtain an oriented compact $0$-dimensional manifold. We define the homomorphism $\mu:C(Y,L)\to C(Y',L')$ by
\begin{equation}\label{mu-def}
  \langle\mu(\alpha),\alpha'\rangle = \# \left((\bH^\gamma_{\alpha,\alpha'})^{-1}(-1)\cap N^+(S;\alpha,\alpha')_1\) .   
\end{equation}
\begin{prop}\label{prop:mu-obs-rel}
	The homomorphism $\mu:C(Y,L)\to C(Y',L')$ satisfies \eqref{mu-bdry-relation}.
\end{prop}
\begin{proof}
	To prove the relation we analyze the boundary components of the $1$-dimensional stratified-smooth space $(\bH^\gamma_{\alpha,\alpha'})^{-1}(-1)\cap N^+(S;\alpha,\alpha')_2$. There are six different possibilities.
	The first type is given by the intersection of $(\bH^\gamma_{\alpha,\alpha'})^{-1}(-1)$ and broken solutions of the form
	\[
	  \breve M^+(L;\alpha,\beta)_{1}\times M^+(S;\beta,\alpha')_0,\hspace{1cm}M^+(S;\alpha,\beta')_0\times \breve M^+(L';\beta',\alpha')_{1}.
	\]
	For elements $(x,y)$ and $(x',y')$ in these two types of spaces, we have
	\[
	  \bH^\gamma_{\alpha,\alpha'}(x,y)=H_{\alpha,\beta}(x),\hspace{1cm}\bH^\gamma_{\alpha,\alpha'}(x',y')=H_{\beta',\alpha'}(y').
	\]
	In particular, the signed count of these boundary points with the induced orientation gives
	\begin{equation} \label{mu-bdry-tye-1}
		 \langle\lambda v(\alpha)-v'\lambda(\alpha),\alpha'\rangle.	
	\end{equation}
	The solutions of the following form give the second type of codimension one stratum:
	\[
	  \breve M^+(L;\alpha,\beta)_{0}\times N^+(S;\beta,\alpha')_1,\hspace{1cm}N^+(S;\alpha,\beta')_1\times \breve M^+(L';\beta',\alpha')_{0}.
	\]
	For elements $(x,y)$ and $(x',y')$ belonging to these two spaces, we have
	\[
	  \bH^\gamma_{\alpha,\alpha'}(x,y)=\bH^\gamma_{\beta,\alpha'}(y),\hspace{1cm}\bH^\gamma_{\alpha,\alpha'}(x',y')=\bH^\gamma_{\alpha,\beta'}(x').
	\]
	Thus, these points have the following contribution in the signed count of boundary points:
	\begin{equation} \label{mu-bdry-tye-2}
		 \langle\mu d(\alpha)+d'\mu (\alpha),\alpha'\rangle.	
	\end{equation}

	The third type of boundary stratum for the space $N^+(S;\alpha,\alpha')_2$ consists of broken solutions of the following form:		
	\[
	  \breve M^+(L;\alpha,\widetilde \theta)_{1} \times_{S^1} N^+(S;\widetilde \theta,\alpha')_1,\hspace{.75cm} N^+(S;\alpha,\widetilde \theta')_1\times_{S^1}\breve M^+(L';\widetilde \theta',\alpha')_{1}.
	\]
	Each connected component of any of these spaces is identified with a copy of $S^1$ using the map $\bH^\gamma_{\alpha,\alpha'}$. In particular, the intersection of
	 $(\bH^\gamma_{\alpha,\alpha'})^{-1}(-1)$ with this space is given by
	\[
	  \breve M^+(L;\alpha, \theta)_{0} \times N^+(S;\widetilde \theta,\alpha')_0,\hspace{.75cm} N^+(S;\alpha,\theta')_0\times\breve M^+(L';\widetilde \theta',\alpha')_{0}.
	\]	
	An analysis of the induced orientation of these boundary points shows that the signed count of these boundary points is given by
	\begin{equation} \label{mu-bdry-tye-3}
		 \langle\Delta_2\delta_1(\alpha)- \delta_2'\Delta_1 (\alpha),\alpha'\rangle.	
	\end{equation}	
	
	Two other types of the boundary strata of $N^+(S;\alpha,\alpha')_2$ are given by $\Phi_{\alpha,\Theta,\alpha'}(\psi_{\alpha,\Theta,\alpha'}^{-1}(0))$ 
	in the complement of \eqref{first-part-N-complement} in $M^+(S;\alpha,\alpha')_2$ and the subspace $\psi_{\alpha,\Theta,\alpha'}^{-1}(0)$ of 
	\eqref{second-part-N} for any reducible $\Theta$ of index $-3$. The restrictions of $\bH^\gamma_{\alpha,\alpha'}$ to these spaces are given by $H_{\alpha,\beta}^\gamma$ and $H_{\alpha,\beta}^\gamma\circ \Phi_{\alpha,\Theta,\alpha'}$. In particular, $\Phi_{\alpha,\Theta,\alpha'}$ gives an orientation-reversing identification of the intersection of $(\bH^\gamma_{\alpha,\alpha'})^{-1}(-1)$ with these boundary strata of $(\bH^\gamma_{\alpha,\alpha'})^{-1}(-1)\cap N^+(S;\alpha,\alpha')_2$. Thus, their signed counts cancel each other.
	Finally, $N^+(S;\alpha,\alpha')_2$ has the boundary strata
	\[
	  (\psi_{\alpha,\Theta,\alpha'}')^{-1}(0)\cap\(\breve M^+(L;\alpha,\widetilde \theta)_{i+1} \times_{S^1} \breve M^+(L';\widetilde \theta',\alpha')_{2-i}\times S_\Theta\)
	\]
	for a reducible $\Theta$ of index $-3$ and $i\in \{0,1\}$. However, the discussion preceding this proposition shows that $\bH^\gamma_{\alpha,\alpha'}$ always takes the value $1$ on this space, 
	and hence this type of boundary component does not contribute to the count of the boundary points of 
	$(\bH^\gamma_{\alpha,\alpha'})^{-1}(-1)\cap N^+(S;\alpha,\alpha')_2$. Now the claim follows from the vanishing of the sum of the terms in \eqref{mu-bdry-tye-1},
	\eqref{mu-bdry-tye-2} and \eqref{mu-bdry-tye-3}.
\end{proof}

\begin{figure}[t]
    \centering
    \centerline{\includegraphics[scale=0.37]{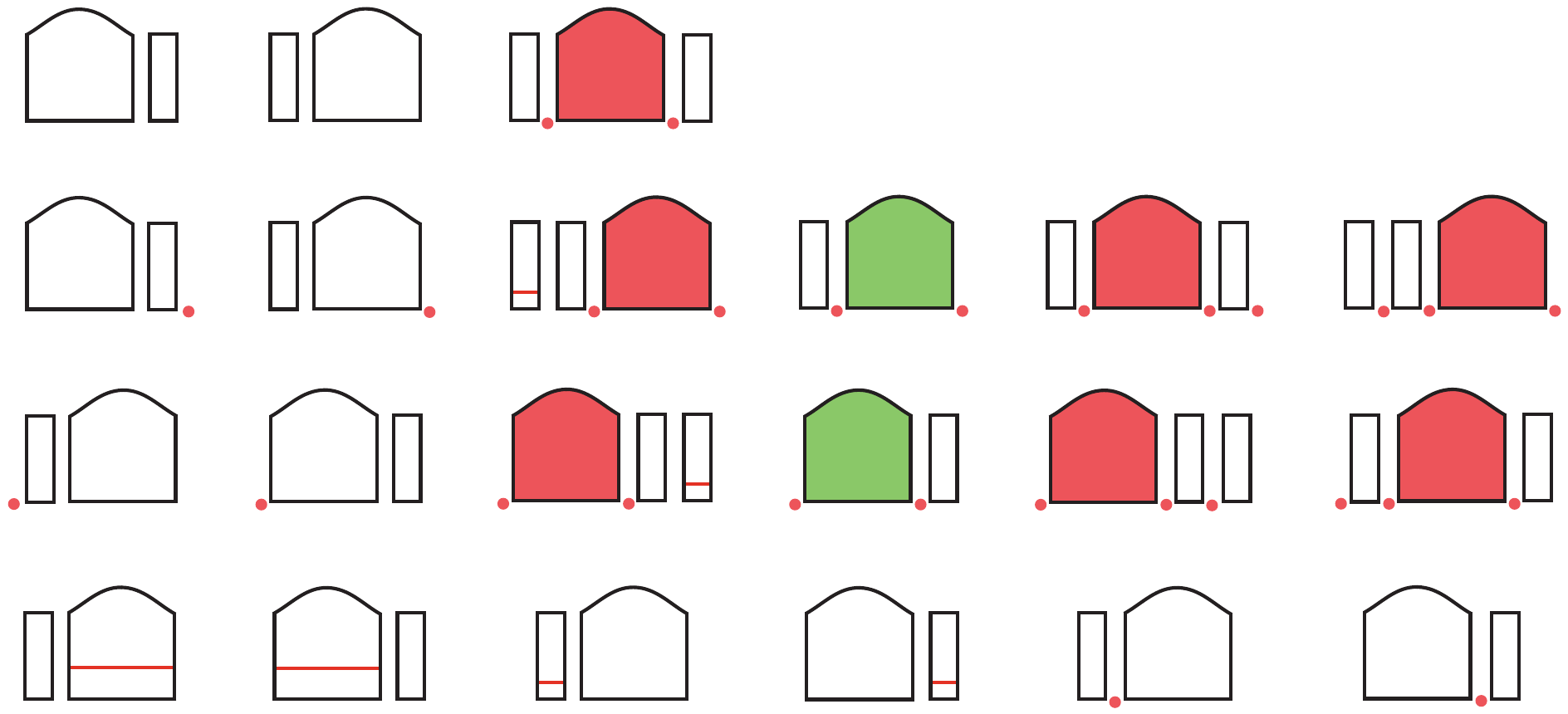}}
    \caption{{\small{Relations \eqref{lambda-bdry-relation}--\eqref{mu-bdry-relation}.}}}
    \label{fig:height-1rels}
\end{figure}

Using the pictorial calculus introduced in Subsection \ref{unob-cob-map}, we depict in Figure \ref{fig:height-1rels} the four relations proved in this section that exhibit $\widetilde\lambda$ as a height $-1$ morphism. Note that by \eqref{tau-obs}, terms involving $\tau_{0}$ expand into three types of terms. In these illustrations, red is used for $\eta_{-1}$, which counts obstructed reducibles of index $-3$, and green is for $\eta_0$, which counts unobstructed reducibles of index $-1$.

We conclude this subsection by comparing the construction above with an alternative approach. The height $-1$ morphism associated to a cobordism $(W,S,c):(Y,L)\to (Y',L')$ as above induces, by Proposition \ref{decomp-neg-level}, a morphism of $\cS$-complexes
\begin{equation}\label{eq:height-1morphismdiscussion}
	\widetilde \lambda: \widetilde C(Y,L)\to \Sigma \widetilde C(Y',L').
\end{equation}
Recall also from Proposition \ref{prop:sustensor} that $\Sigma \widetilde C(Y',L')$ is homotopy equivalent to the tensor product $\widetilde C(Y',L')\otimes \widetilde \cO(1)$ where $\widetilde \cO(1)$ is a certain atomic $\cS$-complex. If one works over an appropriate local coefficient system, then $\widetilde \cO(1)$ is isomorphic to an $\cS$-complex for the right-handed trefoil $T_{2,3}$, see \cite[Section 9.2.3]{DS1}. In this context, an analogue of Theorem \ref{thm:connectedsum} then implies that the suspension $\Sigma \widetilde C(Y',L')$ is homotopy equivalent to an $\cS$-complex for the connected sum of the based link $(Y',L')$ with $(S^3,T_{2,3})$. 

It is thus reasonable to expect that in this local coefficient setting the morphism \eqref{eq:height-1morphismdiscussion} is essentially one which is obtained from $(W,S,c)$ by splicing on a standard cobordism from the unknot to $T_{2,3}$. The effect of such a topological operation should convert the obstructed reducibles into unobstructed ones of index $-1$. This provides an alternative approach to constructing the morphism \eqref{eq:height-1morphismdiscussion} that avoids obstructed gluing theory. It is reminiscent  of Fr\o yshov's work in monopole Floer homology \cite{froyshov-monopole}, where punctures with certain metric and perturbation data are introduced on cobordisms to avoid obstructed gluing theory. 

This trick involving the trefoil was in fact previously utilized in \cite{DS2} and \cite{DISST:special-cycles}. There are some disadvantages to this approach, however. First, the construction only works in the setting of local coefficients, after having inverted certain elements of the ground ring, and in particular does not apply to the case of $\cS$-complexes with ordinary coefficient rings such as $\Z$, $\Q$ or $\Z/2$. Second, in applying this construction in the context of Chern--Simons filtrations, it appears to produce less than optimal results. 

These are not issues for the construction of the morphism \eqref{eq:height-1morphismdiscussion} using obstructed gluing theory given in this section. In addition, our approach is more direct, and, as we will see in Section \ref{sec:proofs}, fits in naturally with the demands of unoriented skein triangles, in which obstructed reducibles appear in several ways. Finally, our construction here may be employed to handle a variety of related problems in Floer theory, ones already mentioned in Subsection \ref{subsec:concludingremarks}, to which the other method does not seem to apply.

\subsection{Homotopies}\label{subsec:homotopies}

In the previous subsection, we associated a height $-1$ morphism of $\cS$-complexes to a cobordism $(W,S,c):(Y,L)\to (Y',L')$ which is negative definite of height $-1$, after fixing some auxiliary choices (including the Riemannian metric and perturbations of the ASD equation). The goal of this subsection is to show that this morphism is independent of these choices up to chain homotopy. To be more specific, let $\mathfrak a_+$ and $\mathfrak a_-$ be two choices of auxiliary data (see below) with associated height $-1$ morphism components given by
\begin{align}
  \lambda^\pm,\, \mu^\pm:C(Y,L)\to C(Y',L'),  \label{maps-aux-choices} \\
   \Delta^\pm_1:C(Y,L) \to  \sfR', \label{maps-aux-choices2} \\
   \Delta^\pm_2: \sfR \to C(Y',L'). \label{maps-aux-choices3}
\end{align}
By choosing a 1-parameter family of auxiliary data that interpolates between the choices $\mathfrak{a}_+$ and $\mathfrak{a}_-$, we shall define homomorphisms 
\[
  K, L:C(Y,L)\to C(Y',L'),\hspace{1cm} M_1:C(Y,L) \to  \sfR',\hspace{1cm}M_2: \sfR \to C(Y',L'),
\]
which are related to the maps in \eqref{maps-aux-choices}--\eqref{maps-aux-choices3} via the following homotopy relations:
\begin{align}
	Kd+d'K-\lambda^++\lambda^-&= 0 \label{htpy-lambda-bdry-relation}\\
	\delta_1'K  + M_1d - \Delta_1^+ + \Delta_1^- &=0  \label{htpy-Delta-1-bdry-relation}\\
	-d'M_2 - K\delta_2    - \Delta_2^+ + \Delta_2^-& =0  \label{htpy-Delta-2-bdry-relation}\\
	Ld-d'L-Kv+v'K+M_2\delta_1+\delta_2'M_1-\mu^++\mu^-  & = 0 \label{htpy-mu-bdry-relation}
\end{align}

Let us explain in detail what the auxiliary data $\mathfrak{a}_+$ contains. First, we choose (i) a cylindrical end metric $g_+$ on the space obtained from $(W,S)$ by attaching cylindrical ends (with cone angle $\pi$ along the surface), (ii) compactly supported perturbation data for the singular ASD equation on the cobordism and (iii) a way in which to modify the holonomy maps (see the Appendix of \cite{DS1}). The choices (i) and (iii) must be compatible at the cylindrical ends with fixed auxiliary data for $(Y,L)$ and $(Y',L')$ which determine the $\cS$-complexes $\widetilde C(Y,L)$ and $\widetilde C(Y',L')$. Regarding (ii), note that the perturbation data for the ends $(Y,L)$ and $(Y',L')$ already induce a perturbation of the ASD equation and the compactly supported perturbation on $(W,S)$ is used to further perturb the ASD equation.

In addition to items (i)--(iii), for each obstructed reducible $\Theta$ of index $-3$ on the cobordism $(W,S,c)$ (with respect to the choices (i) and (ii)), the data $\mathfrak{a}_+$ contains the choices that are used in the proofs of Propositions \ref{type-1-obs-gluing} (and its type II analogue) and \ref{type-3-obs-gluing} which determine the sections $\Psi_{\alpha,\Theta},\Psi_{\Theta,\alpha'},\Psi_{\alpha,\Theta,\alpha'}$ and their equivariant lifts, as well as the maps $\Phi_{\alpha,\Theta},\Phi_{\Theta,\alpha'},\Phi_{\alpha,\Theta,\alpha'}$. Furthermore, for each such obstructed reducible there is the choice of a positive constant $T_{\Theta}\in \R$, as well as the sections $\Xi_{\alpha,\Theta}$, $\Xi_{\Theta,\alpha'}$, $\Xi_{\alpha,\Theta,\alpha'}$. The symbol $\mathfrak{a}_+$ represents all of these choices. Similar remarks hold for the auxiliary data $\mathfrak{a}_-$.

To construct a 1-parameter family $\{\mathfrak{a}_\opam\}_{\opam\in [-1,1]}$ where $\mathfrak{a}_{\pm 1}=\mathfrak{a}_\pm$, we first choose (i') a family of metrics $G=\{g_\opam\}_{\opam\in [-1,1]}$ where $g_{\pm 1}=g_{\pm}$ and (ii') a 1-parameter family of perturbations interpolating between the choices at $\pm 1$. Using these choices we form
\begin{equation}\label{fam-mod-space-hom}
	M(S ;\alpha,\alpha')^G_d= \bigcup_{\opam\in [-1,1]} M(S;\alpha,\alpha')_{d-1}^{g_\opam}.
\end{equation}
for any $\alpha$, $\alpha'$; we will only need to consider the cases $d\leq 2$. Each moduli space appearing on the right in \eqref{fam-mod-space-hom} is defined using the choices of (i') and (ii') specialized to the parameter $\opam\in [-1,1]$. By convention, the metric $g_\opam$ is included in the notation of each such moduli space, while the choice of perturbation surpressed. 

The family moduli space \eqref{fam-mod-space-hom} is a subspace of $\sB(W ,S,c ;\alpha,\alpha')\times [-1,1]$ and is equipped with the subspace topology. We need a compactification of \eqref{fam-mod-space-hom} given as
\begin{equation}\label{fam-mod-space-hom2}
	M^+(S;\alpha,\alpha')^G_d= \bigcup_{\opam\in [-1,1]} M^+(S ;\alpha,\alpha')_{d-1}^{g_\opam}
\end{equation}
where broken trajectories are included. We require these family moduli spaces to be as regular as possible. More specifically, we make choices (i') and (ii') such that:
\begin{itemize}
	\item  Any irreducible element of $M^+(S;\alpha,\alpha')^G_d$ is cut down transversally.
	\item Up to the action of the gauge group, there is a unique reducible $\Theta_\opam$ with respect to $\mathfrak{a}_\opam$ for any isomorphism class of $U(1)$-reduction of the singular bundle data $(W,S,c)$. The reducible $\Theta_\opam$ is regular if $\ind(\Theta_\opam)\geq -1$. If $\ind(\Theta_\opam)=-3$, then $H^1(\Theta_\opam)$ and $H^+(\Theta_\opam)$ respectively have dimensions $0$ and $2$. 
\end{itemize}
Here $H^1(\Theta_\opam)$ and $H^+(\Theta_\opam)$ refer to the cohomology groups in the associated deformation complex of the linearized ASD operator. That the first item above can be arranged is standard. For the second item, one may for example adapt the arguments of \cite[\S 7.3]{CDX}.

Continuing our description of $\{\mathfrak{a}_\opam\}_{\opam\in [-1,1]}$, we next choose (iii') a $1$-parameter family of modified holonomy map data; for this, we again refer to the Appendix of \cite{DS1}. Finally, we must choose $1$-parameter data related to the obstructed gluing data, which we now explain. We fix our attention on a 1-parameter family of reducibles $\{\Theta_\opam\}_{\opam\in [-1,1]}$ of index $-3$, and consider the setting of Proposition \ref{type-1-obs-gluing}. Recalling that the proof of this proposition is a straightforward adaptation of \cite[Proposition 5.5]{DME:QHS}, the essential choice in the data $\mathfrak{a}_{\pm}$ which determines the section $\Psi_{\alpha,\Theta_{\pm}}$ where $\Theta_\pm = \Theta_{\pm 1}$ is that of a linear map
\begin{equation}\label{lin-map-choice-obs-gluing}
	\sigma_{\pm} = \sigma_{\Theta_\pm}:H^+(\Theta_\pm) \to \Omega^{+,g_\pm}(W^+,S^+;\text{ad}P)
\end{equation}
where the codomain is to be interpreted as bundle-valued $g_{\pm}$-self-dual two forms on the orbifold $(W,S)$ with respect to the metric in $\mathfrak{a}_\pm$, and where $\text{ad}P$ is the adjoint bundle data of the singular bundle data on $(W^+,S^+)$, the pair obtained from $(W,S)$ by attaching cylindrical ends. The requirements on \eqref{lin-map-choice-obs-gluing} are that its image consists of elements with compact support, and that the linear map
\begin{equation}\label{lin-map-choice-obs-gluing-requirement}
	d_{\Theta_\pm}^++\sigma_{\Theta_\pm}:\Omega^1(W^+,S^+;\text{ad}P)\oplus H^+(\Theta_\pm) \to \Omega^{+, g_{\pm}}(W^+,S^+;\text{ad}P),
\end{equation}
with the appropriate Sobolev completions (see \cite[\S 2]{DS1}), is a surjection. Here $d_{\Theta_\pm}^+$ is the linearized (perturbed) ASD operator associated to $\Theta_{\pm}$. The section $\Psi_{\alpha,\Theta_{\pm}}$ of Proposition \ref{type-1-obs-gluing} is then constructed in a natural way from the choice of the linear map $\sigma_\pm$; see \cite[Proposition 5.5]{DME:QHS} for more explanation.

Write $\Omega_\opam$ for the subspace of $\Omega^{+,g_\opam}(W^+,S^+;\text{ad}P)$ which consists of compactly supported sections, so that $\sigma_{{\pm}}$ appearing in \eqref{lin-map-choice-obs-gluing} are maps $H^+(\Theta_{\pm})\to \Omega_{\pm 1}$. The condition regarding \eqref{lin-map-choice-obs-gluing-requirement} is then that the following composition
\begin{equation}\label{lin-map-choice-obs-gluing-requirement-2}
	\pi_{\pm 1}\circ \sigma_{{\pm}}:H^+(\Theta_\pm)\to H^+(\Theta_\pm)
\end{equation}
is an isomorphism, where here $\pi_\opam:\Omega_\opam\to H^+(\Theta_\opam)$ is the projection to the cokernel of $d_{\Theta_\opam}^+$. We may furthermore require that \eqref{lin-map-choice-obs-gluing-requirement-2} is the identity on $H^+(\Theta_\opam)$.

With these conventions set, it is straightforward to see that for fixed choices $\sigma_{{\pm}}$ as above, we can find a smooth family of linear maps
\begin{equation}\label{lin-map-choice-obs-gluing-requirement-3}
	\sigma_{\opam}:H^+(\Theta_\opam)\to \Omega_\opam, \qquad \opam\in [-1,1]
\end{equation}
such that $\pi_{\opam}\circ \sigma_{\opam}=\text{id}$ and also $\sigma_{\pm 1}=\sigma_{\pm}$. Indeed, for a fixed parameter $\opam$, the space of such choices described is connected; the family $\sigma_\opam$ may thus be viewed as a section of a fiber bundle over $[-1,1]$ with connected fiber.

 The family of linear maps \eqref{lin-map-choice-obs-gluing-requirement-3} induces, in the context of Proposition \ref{type-1-obs-gluing}, a family of sections, which we may view as a section $\Psi_{\alpha,\Theta^G}$ of the bundle
\[
	\mathcal H_{\alpha,\theta}\times \overline \R_{+}\times G \to \breve M^+(L;\alpha,\theta)_{d}\times \overline \R_{+}\times G.
\]
Here we identify the metric family $G=\{g_\opam\}_{\opam\in [-1,1]}$ with $[-1,1]$ and also denote by $\Theta^G= \{\Theta_\opam\}_{\opam\in[-1,1]}$ the 1-parameter family of obstructed reducibles under consideration. Here we restrict our attention to the cases $0\leq d\leq 3$.

The items of Proposition \ref{type-1-obs-gluing} all have analogues for the section $\Psi_{\alpha,\Theta^G}$. For example, $\Psi_{\alpha,\Theta^G}$ vanishes on $\breve M^+(L;\alpha,\theta)_{d}\times \{\infty\}\times G$ and is transverse to the zero section over any stratum of  $\breve M^+(L;\alpha,\theta)_{d}\times \R_{+} \times G$. Furthermore, there is a map 
\[
	\Phi_{\alpha,\Theta^G}:\Psi_{\alpha,\Theta^G}^{-1}(0)\to M^+(S;\alpha,\theta')^G_d
\]
which is a homeomorphism onto an open subset of $M^+(S;\alpha,\theta')^G_d$. Similarly, we can obtain sections $\Psi_{\Theta^G,\alpha'},\Psi_{\alpha,\Theta^G,\alpha'}$ and associated maps $\Phi_{\Theta^G,\alpha'},\Phi_{\alpha,\Theta^G,\alpha'}$.

Next, choose a smooth function
\[
	T:[-1,1]\to \R_+
\]
such that $T(\pm 1)$ is equal to the constant $T_{\Theta_{\pm}}$ in the auxiliary data $\mathfrak{a}_\pm$. We may choose $T$ generically, just as for the cases of $T_{\Theta_{\pm}}$, so that the constructions below have the requisite transversality. Having made these choices, it is straightforward to construct sections $\Xi_{\alpha,\Theta^G}$, $\Xi_{\Theta^G,\alpha'}$, $\Xi_{\alpha,\Theta^G,\alpha'}$ which extend the ones already defined in Section \ref{gluing-theory-obs} with respect to the auxiliary choices $\mathfrak{a}_\pm$. For example, $\Xi_{\alpha,\Theta^G}$ is a section of the bundle $\mathcal H_{\alpha,\theta}\times [0,1] \times G$ which is transverse to the zero section on each stratum and such that it restricts on $\mathcal H_{\alpha,\theta}\times [0,1] \times \{g_\opam\}$ to the type of section $\Xi_{\alpha,\Theta^{g_\opam}}$ described in Section \ref{gluing-theory-obs} with respect to the auxiliary data $\mathfrak{a}_\opam$. This completes our description of the 1-parameter auxiliary data $\{\mathfrak{a}_\opam\}_{\opam\in[-1,1]}$.

We now describe the maps \eqref{maps-aux-choices}--\eqref{maps-aux-choices3}. Define $K:C(Y,L)\to C(Y',L')$ by
\begin{equation}\label{K-homotopy}
  \langle K(\alpha),\alpha'\rangle = \# M(S;\alpha,\alpha')^G_0 . 
\end{equation}
Like $\lambda$, the map $K$ does not involve any of the obstructed gluing data. To obtain relation \eqref{htpy-lambda-bdry-relation}, one inspects the boundaries of moduli spaces $M^+(S;\alpha,\alpha')^G_1$. There are boundary components $M(S;\alpha,\alpha')_0^{g_\pm}$ corresponding to the boundary of $G$, and which contribute the two terms $\lambda^{\pm}$. The other boundary components are of the form
\[
	M(S;\alpha,\beta')^G_0 \times \breve{M}(L';\beta',\alpha')_0, \qquad \breve{M}(L;\alpha,\beta)_0 \times M(S;\beta,\alpha')^G_0
\]
and these contribute the terms $d' K$ and $K d$ respectively.

Next, to define $M_1$, similar to the case of $\Delta_1$, we introduce the space $N^+(S;\alpha,\theta')^G_0$ defined as the disjoint union of the spaces
\begin{equation*} 
  M^+(S;\alpha,\theta')^G_0\setminus \bigsqcup_{\Theta^G}\Phi_{\alpha,\Theta^G}\(\Psi_{\alpha,\Theta^G}^{-1}(0)\cap (\breve M^+(L;\alpha,\theta)_{0}\times R_G )\)
\end{equation*}
\begin{equation*}
 \text{and} \qquad  \bigsqcup_{\Theta^G} (\Xi_{\alpha,\Theta^G})^{-1}(0)\cap \(\breve M^+(L;\alpha,\theta)_{0}\times [0,1]\times G \),
\end{equation*}
where $R_G = \{ (s,g_\opam)\in \R\times G \mid s> T(\opam)\}$. Here the disjoint unions range over topological types of reductions of the singular bundle data of index $-3$; by our assumptions, for each such reduction there is a 1-parameter family $\Theta^G$ of index $-3$ reducibles. Then set
\begin{equation*} 
  \langle M_1(\alpha),\theta'\rangle = \# N^+(	S;\alpha,\theta')^G_0 .   
\end{equation*}
To obtain relation \eqref{htpy-Delta-1-bdry-relation}, consider the boundaries of moduli spaces $N^+(	S;\alpha,\theta')^G_1$ which are defined analogously to \eqref{first-part-1-Delta-1}--\eqref{third-part-N-Delta-1} using the 1-parameter auxiliary data $\{\mathfrak{a}_\opam\}_{\opam\in[-1,1]}$. There are boundary components $N^+(	S;\alpha,\theta')^{g_\pm}_0$ associated to the boundary of $G$ and these contribute the terms $\Delta_1^\pm$. The other boundary components are of the form
\[
	M(S;\alpha,\beta')^G_0 \times \breve{M}(L';\beta',\theta')_0, \qquad \breve{M}(L;\alpha,\beta)_0 \times N^+(S;\beta,\theta')^G_0
\]
and these contribute the terms $\delta_1' K$ and $M_1 d$ respectively. (There are also some boundary components of $N^+(	S;\alpha,\theta')^G_1$ that obviously cancel.) The case of the map $M_2$ and the relation \eqref{htpy-Delta-2-bdry-relation} is completely analogous.

Finally, to define $L$, one constructs holonomy maps
\begin{equation} \label{eq:holmaphomotopy}
	\bH^{\gamma,G}_{\alpha,\alpha'}:N^+(S;\alpha,\alpha')^G_d\to S^1
\end{equation}
for the path $\gamma$ along $S$ just as in the definition of the maps $\mu^\pm$, but performing the construction for the 1-parameter family $G$, or more precisely, for the data $\{\mathfrak{a}_{\opam}\}_{\opam\in [-1,1]}$. Here $N^+(S;\alpha,\alpha')^G_d$ is defined as in \eqref{first-part-N-complement}--\eqref{second-part-N} but for the 1-parameter auxiliary data. We assume the restriction of \eqref{eq:holmaphomotopy} to the space $N^+(S;\alpha,\alpha')^{g_\pm}_{d-1}$ agrees with the data used to define the maps $\mu^\pm$. Assuming transversality of \eqref{eq:holmaphomotopy} at $-1\in S^1$, we define $L$ by
 \begin{equation*} 
  \langle L(\alpha),\alpha'\rangle = \# \left((\bH^{\gamma,G}_{\alpha,\alpha'})^{-1}(-1)\cap N^+(S;\alpha,\alpha')^G_1\) .   
\end{equation*}
The verification of relation \eqref{htpy-mu-bdry-relation} is similar to previous arguments.

\subsection{Compositions}\label{subsec:functoriality}

We now consider the composition of a height $-1$ morphism from an obstructed cobordism, as constructed in Subsection \ref{obs-cob-maps-level--1}, with a morphism from an unobstructed cobordism, as defined earlier in Subsection \ref{unob-cob-map}. We divide the discussion into two parts, depending on whether the unobstructed cobordism is even or odd.

\subsubsection{The even case}

Suppose $(W,S):(Y,L)\to (Y',L')$ and $(W',S'):(Y',L')\to (Y'',L'')$ are cobordisms of pairs between non-zero determinant links in homology $3$-spheres and $c$, $c'$ give the data of singular bundles on $(W,S)$, $(W',S')$. Let the composition of $(W,S)$ and $(W',S')$ be denoted by $(W^\circ,S^\circ):(Y,L)\to (Y'',L'')$, with singular bundle data $c^\circ=c\sqcup c'$. We also assume that $(W,S,c)$, $(W',S',c')$ are negative definite of respective heights $n$, $n'$ where $n,n'\geq -1$. Then the composite $(W^\circ,S^\circ,c^\circ)$ is negative definite of height $n+n'\geq -1$. We have morphisms of respective heights $n$, $n'$, $n+n'$:
\begin{align*}
  \widetilde \lambda_{(W,S,c)}&:\widetilde C(Y,L) \to \widetilde C(Y',L'), \\[1mm]
  \widetilde \lambda_{(W',S',c')}&:\widetilde C(Y',L') \to \widetilde C(Y'',L''), \\[1mm]
  \widetilde \lambda_{(W^\circ,S^\circ,c^\circ)}&:\widetilde C(Y,L) \to \widetilde C(Y'',L'').
\end{align*}
It is reasonable to expect that $\widetilde \lambda_{(W^\circ,S^\circ,c^\circ)}$ is chain homotopy equivalent to the composition of $\widetilde \lambda_{(W,S,c)}$ and $\widetilde \lambda_{(W',S',c')}$ in the sense of Definition \ref{comp-neg-lev}.

In the case that $n$ and $n'$ are both non-negative, and in particular $(W,S,c)$ and $(W',S',c')$ are both unobstructed, similar arguments as in \cite{DS1} can be used to see that the map associated to $(W^\circ,S^\circ,c^\circ)$ of height $n+n'$ is chain homotopy equivalent to the composition of $\widetilde \lambda_{(W,S,c)}$ and $\widetilde \lambda_{(W',S',c')}$ as morphisms of height $n+n'$. 

The situation in the case that $n=-1$ or $n'=-1$ is more complicated because of obstructed gluing theory. The goal of this subsection is to establish this instance of functoriality for $n=-1$ and $n'=0$, stated as Theorem \ref{func-obs}. The case $n=0$ and $n'=-1$ can be treated in essentially the same way. In the statement of the following theorem, we abbreviate $\widetilde \lambda_{(W,S,c)}$, $\widetilde \lambda_{(W',S',c')}$ and $\widetilde \lambda_{(W^\circ,S^\circ,c^\circ)}$ to $\widetilde \lambda$, $\widetilde \lambda'$ and $\widetilde \lambda^\circ$. We follow a similar convention to denote the components of these morphisms. In the proof of the exact triangles in the next section, several adaptations of the following argument will be used.

\begin{theorem}\label{func-obs}
	Let $(W,S,c):(Y,L)\to (Y',L')$ and $(W',S',c'):(Y',L')\to (Y'',L'')$ be negative definite cobordisms of heights $-1$ and $0$, respectively.
	After fixing auxiliary choices for these cobordisms, consider the associated height $-1$ and $0$ morphisms 
	\[
	  \widetilde \lambda:\widetilde C(Y,L) \to \widetilde C(Y',L'), 
	  \hspace{1cm}\widetilde \lambda':\widetilde C(Y',L') \to \widetilde C(Y'',L''),
	\]
	Then the height $-1$ morphism $\widetilde \lambda^\circ$, associated to the composite of 
	$(W,S,c)$ and $(W',S',c')$, is $\cS$-chain homotopy equivalent to 
	$\widetilde \lambda'\circ \widetilde \lambda$. 
	To be more specific, there are homomorphisms 
	\begin{equation}\label{chain-htpy}
	  K,\, L:C(Y,L) \to C(Y'',L''), \hspace{0.3cm} M_1:C(Y,L) \to \sfR'', \hspace{0.3cm} M_2:\sfR \to C(Y'',L''),
	\end{equation}
	such that the components of $\widetilde \lambda$, $\widetilde \lambda'$ and 
	$\widetilde \lambda^\circ$ satisfy
	\begin{align}
		&\lambda^\circ+Kd+d''K =\,\lambda'\lambda+\Delta_2'\tau_{-1}\delta_1
		\label{lambda-circ-relation}\\
		&\Delta_1^\circ+\delta_1''K+M_1d =\,\Delta_1'\lambda+ \tau_{0}'\Delta_1
		\label{Delta-1-circ-relation}\\
		&\Delta_2^\circ-d''M_2-K\delta_2 =\,\lambda'\Delta_2+\Delta_2' \tau_{0}+v''\Delta_2' \tau_{-1}+\mu'\delta_2' \tau_{-1}
		 \label{Delta-2-circ-relation}\\
		&\mu^\circ+Ld-d''L-Kv+v''K+M_2\delta_1+\delta_2''M_1=\,\lambda'\mu+\mu'\lambda +\Delta_2'\Delta_1 
		\label{mu-circ-relation}\\
		&\tau_{i}^\circ=\,\sum_{k\in \Z}\tau_{i-k}'\tau_{k}\label{tau-circ-relation}
	\end{align}
	In particular, we have the following:
	\begin{align}\label{red-rel}
	 \tau_0^\circ &=\tau_0'\tau_0+(\delta_1''\Delta_2'+\Delta_1'\delta_2')\tau_{-1}\\
	  \tau_{-1}^\circ &=\tau_{0}'\tau_{-1}\label{red-rel-minus}
	\end{align}
\end{theorem}

This theorem is proved by following a strategy similar to the previous subsections. Fix auxiliary data to form the $\cS$-complexes $\widetilde C(Y,L)$, $\widetilde C(Y',L')$ and $\widetilde C(Y'',L'')$. Let $\mathfrak a$, $\mathfrak a'$ be auxiliary choices used to form the morphisms $\widetilde \lambda$, $\widetilde \lambda'$ which are compatible with the auxiliary data of the $\cS$-complexes of $L$, $L'$ and $L''$. Similar to the previous subsection, we use a 1-parameter family of auxiliary data $\{\mathfrak{a}_\opam\}_{\opam\in [1,\infty)}$ to define the homotopy required for Theorem \ref{func-obs}. This 1-parameter family is induced by $\mathfrak a$ and $\mathfrak a'$ in the following way. 

First, we use the orbifold Riemannian metrics on $W$ and $W'$ to define an orbifold metric $g_t$ on $W^\circ$ for each $t\in [1,\infty)$. The metric $g_t$ has a neck isometric to $[-t,t]\times Y'$ with the product metric in the middle of the composed cobordism, where the orbifold metric on $Y'$ is the same as used in the definition of $\widetilde C(Y',L')$. On the complement of the neck region, $g_t$ is independent of $t$ and is determined by the metrics of $W$ and $W'$. We extend this family to $t=\infty$ by letting $g_\infty$ be the {\it broken} metric on $W^\circ$ along $Y'$ which is equal to the chosen metrics on $W$ and $W'$. Write $G$ for this 1-parameter family of metrics.

Similarly, for any $t$, we have a perturbation of the ASD equation on $W^\circ$, where the ASD equation is defined using the metric $g_t$. This perturbation over the cylindrical ends and the subspace of the neck region given by $[-t+1,t-1]\times Y'$ is equal to the downward gradient flow of the perturbations for the Chern--Simons functional of $(Y,L)$, $(Y',L')$ and $(Y'',L'')$. On the complement of the ends and the neck region, the perturbation is given by the perturbation of the ASD equation associated to $\mathfrak a$ and $\mathfrak a'$. As usual, perturbation data is typically omitted from notation.

To define the homomorphisms in \eqref{chain-htpy} and prove the desired relations in Theorem \ref{func-obs}, we use the instanton moduli spaces over $W^\circ$ defined with respect to $G$ and with dimension $d\leq 2$. For any pair of critical points $\alpha$, $\alpha''$ on $(Y,L)$ and $(Y'',L'')$, define
\begin{equation}\label{fam-mod-space}
	M(S^\circ  ;\alpha,\alpha'')^G_d= \bigcup_{t\in [1,\infty)} M(S^\circ  ;\alpha,\alpha'')_{d-1}^{g_t}.
\end{equation}
This family moduli space is a subspace of $\sB(W^\circ ,S^\circ ,c^\circ ;\alpha,\alpha'')\times [1,\infty)$ and is equipped with the subspace topology. We need a partial compactification of \eqref{fam-mod-space} given as
\begin{equation}\label{fam-mod-space-comp}
	M^+(S^\circ ;\alpha,\alpha'')^G_d= \bigcup_{T\in [1,\infty]} M^+(S^\circ ;\alpha,\alpha'')_{d-1}^{g_t},
\end{equation}
where the moduli space $M^+(S^\circ ;\alpha,\alpha'')_{d-1}^{g_\infty}$ is the union of all {\it broken} solutions on $(W^\circ ,S^\circ ,c^\circ )$ with respect to the broken metric $g_\infty$. Thus this space includes
\begin{equation}\label{fib-prod-comp}
  \bigcup_{\alpha', i}M(S;\alpha,\widetilde \alpha')_{i+\dim(\Gamma_{\alpha'})}\times_{\Gamma_{\alpha'}}M(S';\widetilde \alpha',\alpha'')_{d-i-1},
\end{equation}
and lower strata where a sequence of broken solutions on the cylinders associated to $L$, $L'$ and $L''$ are inserted in the above fiber product. 
Similar to previous cases, we can arrange the data of $G$ so that the above moduli spaces are as regular as possible in the following sense:
\begin{itemize}
	\item  Any irreducible element of $M(S^\circ;\alpha,\alpha'')^G_d$ is cut down transversally.
	\item Up to the action of the gauge group, there is a unique reducible $\Theta^\circ_\opam$ with respect to $\mathfrak{a}_\opam$ for any isomorphism class of $U(1)$-reduction of the singular bundle data $(W^\circ,S^\circ,c^\circ)$. The reducible $\Theta^\circ_\opam$ is regular if 
	$\ind(\Theta^\circ_\opam)\geq -1$. If $\ind(\Theta_\opam)=-3$, then $H^1(\Theta_\opam)$ and $H^+(\Theta_\opam)$ respectively have dimensions $0$ and $2$. 
\end{itemize}

Any isomorphism class of $U(1)$-reduction of the singular bundle data $(W^\circ,S^\circ,c^\circ)$ determines, and is determined by, isomorphism classes of $U(1)$-reductions of the singular bundle data of $(W,S,c)$ and $(W',S',c')$ such that the induced reductions over $(Y',L')$ agree with each other. In particular, for any reducible $\Theta$ over $(W,S,c)$ with limits $\theta$, $\theta'$ and reducible $\Theta'$ over $(W',S',c')$ with limits $\theta'$, $\theta''$, and any given $t\in [1,\infty)$, there is a unique corresponding reducible $\Theta^\circ_\opam$ over $(W^\circ,S^\circ,c^\circ)$. Furthermore, 
\[
	\ind(\Theta^\circ_\opam)=\ind(\Theta)+\ind(\Theta')+1.
\]
Our assumptions on $(W,S,c)$ and $(W',S',c')$ imply that $\ind(\Theta_\opam)=-3$ if and only if $\ind(\Theta)=-3$ and $\ind(\Theta')=-1$.

Let $\Theta$ be a reducible instanton of index $-3$ over $(W,S,c)$ with limits $\theta$ and $\theta'$ on the ends. Then any element of the subspace of the moduli space $M^+(S^\circ;\alpha,\alpha'')^G_d$ given as
\begin{equation}\label{type-1-3}
  \breve M^+(L;\alpha,\widetilde \theta)_i \times_{S^1} \overline{\underline \Theta} \times_{S^1} M^+(S';\widetilde \theta',\alpha'')_{d-i+1},
\end{equation}
where $1\leq i\leq d+1$, is an obstructed solution.  Here $i=d+1$ implies that $\alpha''$ is a reducible flat connection $\theta''$ and any element of $M^+(S';\widetilde \theta',\theta'')_{0}$ is given by a reducible instanton of index $-1$. Similarly, any element of the space 
\begin{equation}\label{type-2}
  \overline{\Theta} \times_{S^1} M^+(S';\widetilde \theta',\alpha'')_{d+2}
\end{equation}
determines an obstructed solution of $M^+(S^\circ;\theta,\alpha'')^G_d$.

We may use obstructed gluing theory analogous to the previous subsections to describe neighborhoods of the obstructed loci of \eqref{type-1-3} and \eqref{type-2}. We use the same notation as in Propositions \ref{type-1-obs-gluing} and \ref{type-3-obs-gluing} to denote the obstruction sections and the homeomorphisms into the zero sets of the obstruction sections. For instance, there is a section $\Psi_{\Theta,\alpha}$ over the product of \eqref{type-2} and $\overline \R_{+}$, which is transverse to the zero section over the subspace given by the product of \eqref{type-2} and $\R_{+}$, and its zero set is homeomorphic to a neighborhood of the obstructed solutions \eqref{type-2} in $M^+(S^\circ;\alpha,\alpha'')^G_d$ via a homeomorphism $\Phi_{\Theta,\alpha}$. More precisely, the interval $\overline \R_{+}$ is the embedding of $\overline \R_{+}$ as a (closed) neighborhood of $\infty$ in the parameter space of $G$. In particular, replacing the parameter space $[1,\infty]$ of $G$ with $[t_0,\infty]$ for large enough $t_0$, we can assume that $\Psi_{\Theta,\alpha}$ is defined over the product of \eqref{type-2} and $G$.

An obstructed reducible instanton associated to $g_t\in G$ is obtained by gluing an obstructed reducible of index $-3$ over $(W,S,c)$ and an index $-1$ reducible over $(W',S',c')$. We may form broken instantons for the perturbed ASD equation associated to $g_t\in G$, which contain such an obstructed reducible as the component over $(W^\circ,S^\circ,c^\circ)$ in the same way as in \eqref{type-I-obs}, \eqref{type-II-obs} and \eqref{type-III-base}. Then obstructed elements of $M^+(S^\circ;\alpha,\alpha'')^G_d$ are given by such broken connections or the elements in \eqref{type-1-3}, \eqref{type-2}.

By choosing $t_0$ large enough, we may assume that neighborhoods of any such elements are given by the neighborhoods of the subspaces in \eqref{type-1-3} and \eqref{type-2}, which are discussed in the previous paragraph. 
By picking $t_0$ generically, we may also assume that any irreducible element of $M(S^\circ;\alpha,\alpha'')^{g_{t_0}}_d$ is cut down transversally for any $\alpha$ and $\alpha''$. In the following, we assume that $\widetilde \lambda^\circ$ is defined using the metric $g_{t_0}$ and the corresponding perturbation of the ASD equation. The rest of the auxiliary data required to form $\widetilde \lambda^\circ$ are chosen concurrently with the auxiliary data required to define the homomorphisms in \eqref{chain-htpy}. Note that using the result of the previous subsection, we are free to pick any auxiliary data for $(W^\circ,S^\circ,c^\circ)$.

Now we are ready to define the homomorphisms in \eqref{chain-htpy}. First, we define the map $K:C(Y,L) \to C(Y'',L'')$ similar to \eqref{K-homotopy} as
\[
  \langle K(\alpha),\alpha''\rangle = \# M(S^\circ;\alpha,\alpha'')^G_0. 
\]
As $\alpha$ and $\alpha''$ are irreducible, we see that $M(S^\circ;\alpha,\alpha'')^G_0$ is compact using a straightforward index argument. Relation \eqref{lambda-circ-relation} follows from inspecting the ends of the compactified moduli space $M^+(S^\circ;\alpha,\alpha'')^G_1$. This moduli space has obstructed solutions of the type 
\begin{equation}\label{obs-type1-family}
  \breve M(L;\alpha,\widetilde \theta)_{1} \times_{S^1} \overline{\underline \Theta} \times_{S^1} M(S';\widetilde \theta',\alpha'')_{1}
\end{equation}
where $\Theta$ is a reducible of index $-3$ over $(W,S,c)$, asymptotic to reducibles $\theta$ and $\theta'$.
Following the same argument as in Proposition \ref{bdry-rel-lambda}, we can analyze the local behavior of the moduli space around such obstructed solutions and obtain \eqref{lambda-circ-relation}. In particular, the obstructed solutions in \eqref{obs-type1-family} give rise to the term $\Delta_2'\tau_{-1}\delta_1$, which is the counterpart of the term $\delta_2'\tau_{-1} \delta_1$ in \eqref{lambda-bdry-relation}. The term $\lambda'\lambda$ in \eqref{lambda-circ-relation} corresponds to the ends of the moduli space $M(S^\circ;\alpha,\alpha'')_{1}$ modeled on the broken solutions $M(S;\alpha,\alpha')_{0}\times M(S';\alpha',\alpha'')_{0}$. The remaining terms in \eqref{lambda-circ-relation} are obtained in the same way as in the proof of \eqref{htpy-lambda-bdry-relation}.

The moduli spaces $M^+(S^\circ;\alpha,\theta'')^G_d$, with $d\in\{0,1\}$, can be used to define the homomorphism $M_1:C(Y,L) \to \sfR''$ and obtain relation \eqref{Delta-1-circ-relation}. The moduli space $M^+(S^\circ;\alpha,\theta'')^G_d$ contains obstructed solutions of the form
\begin{equation}\label{obs-sol-M1}
  \breve M^+(L;\alpha,\widetilde \theta)_i \times_{S^1} \overline{\underline \Theta} \times_{S^1} M^+(S';\widetilde \theta',\theta'')_{d-i+1},
\end{equation}  
where $1\leq i\leq d+1$, and $\Theta$ is a reducible instanton of index $-3$ over $(W,S,c)$. Furthermore, these elements of $M^+(S^\circ;\alpha,\theta'')^G_d$ give all obstructed solutions with $t=\infty$. Since $(W,S,c)$ is an even cobordism, $M^+(S';\widetilde \theta',\theta'')_{d-i+1}$ is non-empty only if $d-i+1$ is an even integer. Thus the only possible value for $i$ is $d+1$, and in this case the \eqref{obs-sol-M1} has the form 
\begin{equation}\label{obs-t=infty}
  \bigsqcup_{(\Theta,\Theta')} \breve M^+(L;\alpha, \theta)_{d},
\end{equation}
where $\Theta$ is a reducible instanton of index $-3$ over $(W,S,c)$ asymptotic to $\theta$ and $\theta'$, and $\Theta'$ is a reducible instanton of index $-1$ over $(W',S',c')$ asymptotic to $\theta'$ and $\theta''$. As in Subsection \ref{obs-cob-maps-level--1} and for $d=\{0,1\}$, we define the moduli space $N^+(S^\circ;\alpha,\theta'')^G_d$ as the disjoint union of
\begin{equation*} 
  M^+(S^\circ;\alpha,\theta'')^G_d\setminus \bigsqcup_{(\Theta,\Theta')}\Phi_{\alpha,\Theta,\theta''}\(\Psi_{\alpha,\Theta,\theta''}^{-1}(0)\cap (\breve M^+(L;\alpha,\theta)_{d}\times G \times (T_{\Theta^\circ},\infty] )\)
\end{equation*}
\begin{equation*}
 \text{and}  \qquad  \bigsqcup_{(\Theta,\Theta')} (\Xi_{\alpha,\Theta,\theta''})^{-1}(0)\cap \([0,1]\times  \breve M^+(L;\alpha,\theta)_{d}\times G \),
\end{equation*}
where the possible range for $(\Theta,\Theta')$ is as in \eqref{obs-t=infty}. The constant $T_{\Theta^\circ}$ and the section $\Xi_{\alpha,\Theta^\circ}$ are defined in the same way as before. As the moduli spaces of the form $M^+(S';\widetilde \theta',\theta'')_{0}$ for reducibles $\theta'$ and $\theta''$ are empty, the definition of  $N^+(S;\alpha,\theta')^G_d$ does not involve terms of the form \eqref{first-part-2-Delta-1} or \eqref{third-part-N-Delta-1}. Now we set
\begin{equation*} 
  \langle M_1(\alpha),\theta'\rangle = \# N^+(S^\circ;\alpha,\theta'')^G_0.
\end{equation*}
A similar argument as in the proof of relation \eqref{htpy-Delta-1-bdry-relation} using the moduli spaces $N^+(S^\circ;\alpha,\theta'')^G_1$ gives \eqref{Delta-1-circ-relation}. In particular, the terms on the right hand side of \eqref{Delta-1-circ-relation} are given by the following boundary components of $N^+(S^\circ;\alpha,\theta'')^G_1$:
\[
  M^+(S;\alpha,\alpha')_0\times M^+(S';\alpha',\theta'')_0, \qquad \bigsqcup_{\Theta'} N(S;\alpha,\theta')_0\times \{\Theta'\},
\]
where the disjoint union ranges over all reducibles of index $-1$ over $(W',S',c')$ which are asymptotic to $\theta'$ and $\theta''$.

To define $M_2$ and verify \eqref{Delta-2-circ-relation}, use the moduli spaces $M^+(S^\circ;\theta,\alpha'')^G_d$ with $d\in\{0,1\}$ after some modifications. In particular, define $N^+(S^\circ;\theta,\alpha'')^G_0$ as the disjoint union of
\begin{equation*} 
  M^+(S^\circ;\theta,\alpha'')^G_0\setminus \bigsqcup_{\Theta}\Phi_{\Theta,\alpha''}\(\Psi_{\Theta,\alpha''}^{-1}(0)\cap (M^+(S'; \theta',\alpha'')_{1}\times G )\)
\end{equation*}
\begin{equation*}
\text{and} \qquad  \bigsqcup_{\Theta} (\Xi_{\Theta,\alpha''})^{-1}(0)\cap \([0,1]\times  M^+(S'; \theta',\alpha'')_{1}\),
\end{equation*}
where the disjoint unions range over all obstructed reducibles of index $-3$ over $(W,S,c)$, and the section $\Xi_{\Theta,\alpha''}$ is defined in the same way as before. The definition of $N^+(S^\circ;\theta,\alpha'')^G_1$ is slightly different from the definition of $N^+(S^\circ;\theta,\alpha'')^G_0$ because $M^+(S^\circ;\theta,\alpha'')^G_1$ contains obstructed solutions of the following form:
\[
  \breve M^+(L;\theta,\widetilde \rho)_1 \times_{S^1} \overline{\underline \Theta} \times_{S^1} M^+(S';\widetilde \theta',\alpha'')_{1},
\]
where $\rho$ is a reducible flat connection associated to $L$ and $\Theta$ is a reducible instanton of index $-3$ over $(W,S,c)$.  Appearance of such obstructed solutions is similar to the appearance of \eqref{obs-red-new} in the proof of Proposition \ref{Delta-1}. In particular, we define $N^+(S^\circ;\theta,\alpha'')^G_1$ similar to \eqref{first-part-1-Delta-1}--\eqref{third-part-N-Delta-1}, and in the same way that \eqref{Delta-1-bdry-relation} is verified in Proposition \ref{Delta-1}, we obtain relation \eqref{Delta-2-circ-relation} using the moduli space $N^+(S^\circ;\theta,\alpha'')^G_1$. In particular, the terms $\lambda'\Delta_2$, $\Delta_2' \tau_{0}$, $(v''\Delta_2'+\mu'\delta_2')\tau_{-1}$ are respectively counterparts of $\Delta_1 d$, $\tau_{0}\delta_1$, $\tau_{-1}\delta_1v$.

The moduli space $M^+(S^\circ;\alpha,\alpha'')^G_d$ contains obstructed solutions of the form \eqref{type-1-3}, and for $d\in\{1, 2\}$, we define the moduli space $N^+(S^\circ;\alpha,\alpha'')^G_d$ in the same way as in Subsection \ref{obs-cob-maps-level--1} by removing the counterpart of \eqref{first-part-N-complement} and including the counterpart of \eqref{second-part-N}. Recall that $U_\Theta$ and $S_\Theta$ in \eqref{first-part-N-complement} and \eqref{second-part-N} are defined in \eqref{U-theta} and \eqref{S-theta}, and are subspaces of $\overline \R_{+}\times \overline \R_{+}$. In the present setup $U_\Theta$ and $S_\Theta$ are replaced by $(T_\theta,\infty] \times G$ and $\{T_\theta\}\times G$ for a large generic constant $T_\Theta$. Next, we define the map 
\[
  \bH^{\gamma^\circ,G}_{\alpha,\alpha''}:N^+(S^\circ;\alpha,\alpha'')^G_d\to S^1
\]
similar to \eqref{eq:holmapobs} and \eqref{eq:holmaphomotopy}. Here $\gamma^\circ$ is the path in $S^\circ$ connecting the basepoints on the two links obtained by composing the paths $\gamma\subset S$ and $\gamma'\subset S'$. Define $L$ in \eqref{chain-htpy} as
\[
  \langle L(\alpha),\alpha''\rangle = \# \left((\bH^{\gamma^\circ,G}_{\alpha,\alpha''})^{-1}(-1)\cap N^+(S^\circ;\alpha,\alpha'')^G_1\).  
\] 
Repeating a similar argument as in the proof of Proposition \ref{prop:mu-obs-rel} using the moduli space $(\bH^{\gamma^\circ,G}_{\alpha,\alpha''})^{-1}(-1)\cap N^+(S^\circ;\alpha,\alpha'')^G_2$ gives \eqref{mu-circ-relation}. In particular, the terms $\lambda'\mu$, $\mu'\lambda$ and $\Delta_2'\Delta_1$ respectively appear from the boundary components of $N^+(S^\circ;\alpha,\alpha'')^G_2$ given by
\[
  N^+(S;\alpha,\alpha')_1\times M^+(S';\alpha',\alpha'')_0,\hspace{1cm} M^+(S;\alpha,\alpha')_0\times M^+(S';\alpha',\alpha'')_1
\]
\[
  \text{and} \qquad N^+(S;\alpha,\widetilde \theta')_1\times_{S^1} M^+(S';\widetilde \theta',\alpha'')_1.
\]

To obtain relation \eqref{tau-circ-relation}, first note that the counts or reducibles of index $-1$ and $-3$ on $S^\circ$ are related to the corresponding counts on $S$ and $S'$ by the following relations:
\begin{equation}\label{eta-rels}
  \eta_{-1}(S^\circ)=\eta_{0}(S') \eta_{-1}(S),\hspace{1cm}\eta_{0}(S^\circ)=\eta_{0}(S') \eta_{0}(S)+\eta_{1}(S')\eta_{-1}(S)
\end{equation}
The first relation together with the definition of $\tau_{-1}$ for a height $-1$ cobordism in \eqref{tau-obs2} gives relation \eqref{red-rel-minus}. Relation \eqref{red-rel} follows from combining \eqref{eta-rels}, \eqref{tau-obs}, \eqref{tau-obs2} and \eqref{rel-tau-i+1} applied to the cobordism $S'$.

\begin{figure}[t]
    \centering
    \centerline{\includegraphics[scale=0.23]{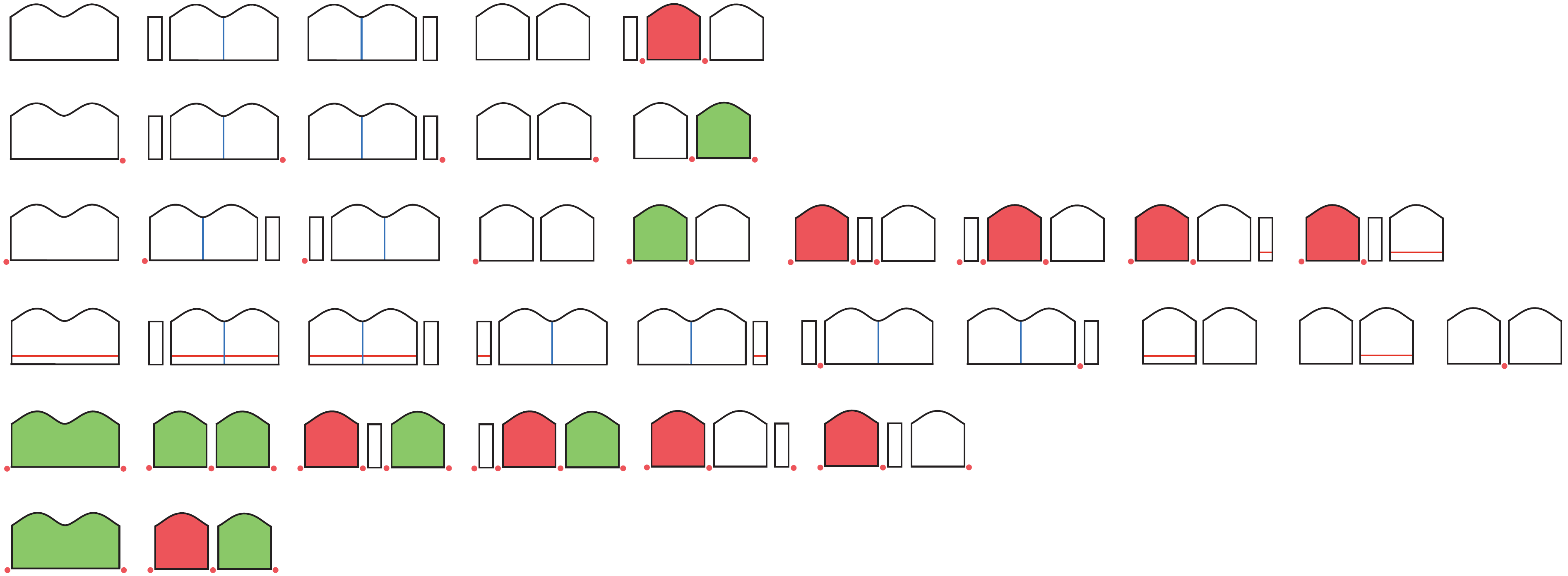}}
    \caption{{\small{Relations \eqref{lambda-circ-relation}--\eqref{mu-circ-relation} and \eqref{red-rel}--\eqref{red-rel-minus}, row by row.}}}
    \label{fig:height-1compeven}
\end{figure}

We depict the relations of Theorem \ref{func-obs} in Figure \ref{fig:height-1compeven}. Pictures that involve a vertical blue line represent maps that are defined using the 1-parameter metric family, i.e. $K$, $M_1$, $M_2$, $L$. The first column, in order, depicts the maps defined on the composition cobordism, i.e. the maps $\lambda^\circ$, $\Delta_1^\circ$, $\Delta_2^\circ$, $\mu^\circ$ and $\tau^\circ$. As before, green (resp. red) is used for unobstructed (resp. obstructed) reducible instantons. Similar illustrations will be used in Section \ref{sec:proofs}.

\begin{remark}
	In the present work, we have constructed maps associated to cobordisms which have the mildest type of obstructed reducibles, of index $-3$. On the other hand, the tools developed here may be used in an approach to define cobordism maps in much greater generality, as we now explain.
	
	 Consider any cobordism $(W,S):(Y,L)\to (Y',L')$ between non-zero determinant links in homology $3$-spheres, with $b^1(W)=b^+(W)=0$. Suppose $(W,S)$ can be decomposed as
\begin{equation}\label{eq:strategyforcompositioningeneral}
	(W,S)= (W_n,S_n)\circ (W_{n-1},S_{n-1})\circ \cdots \circ (W_1,S_1)
\end{equation}
where each $(W_i,S_i)$ is between non-zero determinant links in homology $3$-spheres, and its reducible instantons for any metric have index at least $-3$. In particular, by the constructions of this section, there is an associated height $-1$ morphism for each $(W_i,S_i)$. Then one may obtain a morphism $\widetilde C(Y,L) \to \widetilde C(Y',L')$ associated to $(W,S)$, which in general can be of any negative height, by algebraically composing the height $-1$ morphisms associated to the $(W_i,S_i)$. That this algebraic composition is indeed a plausible definition for the morphism of $(W,S)$ is suggested by Theorem \ref{func-obs}. 

A major step in carrying this strategy out is to show that if two different decompositions as in \eqref{eq:strategyforcompositioningeneral} are chosen for $(W,S)$, then the resulting morphisms constructed are chain homotopic in the appropriate sense. This step should involve studying more general family moduli spaces that interpolate between the choices.

Alternatively, in \eqref{eq:strategyforcompositioningeneral}, one may take all $(W_i,S_i)$ but one to be cylinders, and on each cylinder introduce a perturbation such that the only reducible instantons present have index $-3$. Such perturbations are necessarily non-trivial at the reducible critical points. This approach, which is an adaptation of the ideas in \cite{DME:QHS}, has the advantage that such a decomposition clearly exists, and the independence of the construction may be handled just as the above citation.  $\diamd$
\end{remark}

\subsubsection{The odd case}

Above, we considered the composition of two negative definite cobordisms, one of height $-1$ and the other of some non-negative height. In particular, both cobordisms involved were {\emph{even}}. The following variation of Theorem \ref{func-obs} considers the case in which the unobstructed cobordism is {\emph{odd}}. In this version, the $\nu$-maps of Definition \ref{defn:numap} appear.

\begin{theorem}\label{func-obs-odd}
	Suppose $(W,S,c)$ is a negative definite cobordism of height $-1$, and $(W',S',c')$ is an unobstructed odd cobordism. 
	After fixing auxiliary choices for these cobordisms, let $\widetilde \lambda$ be the height $-1$ morphism associated to $(W,S,c)$
	and $\widetilde \lambda'$ be the odd degree morphism associated to $(W',S',c')$.
	Then the odd degree morphism $\widetilde \lambda^\circ$, associated to the composite of 
	$(W,S,c)$ and $(W',S',c')$, is $\cS$-chain homotopy equivalent to 
	the composition $\widetilde \lambda'\widetilde \lambda$ as described in Subsection \ref{subsec:odddegmorphisms}. 
	More precisely, there are maps
	\begin{equation}\label{chain-htpy}
	  K,\, L:C(Y,L) \to C(Y'',L''), \hspace{0.3cm} M_1:C(Y,L) \to \sfR'', \hspace{0.3cm} M_2:\sfR \to C(Y'',L''),
	\end{equation}
	such that the components of $\widetilde \lambda$, $\widetilde \lambda'$ and 
	$\widetilde \lambda^\circ$ satisfy
	\begin{align}
		\lambda^\circ+Kd+d''K =&\,\lambda'\lambda+\Delta_2'\tau_{-1}\delta_1
		\label{lambda-circ-relation-odddeg}\\
		\Delta_1^\circ+\delta_1''K+M_1d =&\,\Delta_1'\lambda+ \nu' \tau_{-1} \delta_1
		\label{Delta-1-circ-relation-odddeg}\\
		\Delta_2^\circ-d''M_2-K\delta_2 =&\,\lambda'\Delta_2+\Delta_2' \tau_{0}+\delta_2'' \nu' \tau_{-1}+v''\Delta_2' \tau_{-1}+\mu'\delta_2' \tau_{-1}
		 \label{Delta-2-circ-relation-odddeg}\\
		\mu^\circ+Ld-d''L-Kv+v''K+&M_2\delta_1+\delta_2''M_1=\,\lambda'\mu+\mu'\lambda +\Delta_2'\Delta_1 
		\label{mu-circ-relation-odddeg}
	\end{align}
\end{theorem}

\begin{figure}[t]
    \centering
    \centerline{\includegraphics[scale=0.23]{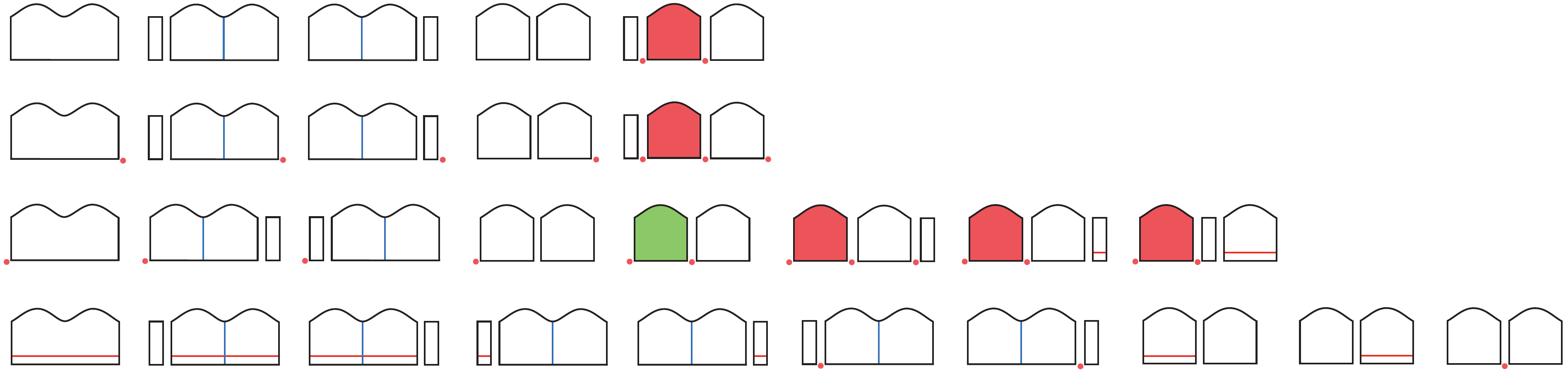}}
    \caption{{\small{Relations \eqref{lambda-circ-relation-odddeg}--\eqref{mu-circ-relation-odddeg}.}}}
    \label{fig:height-1compodd}
\end{figure}

Similar to the previous case, we depict the relations of this result in Figure \ref{fig:height-1compodd}.
This theorem is proved using a similar argument as in the proof of Theorem \ref{func-obs}. Therefore we omit most of the details and only focus on the aspects of the proof of Theorem \ref{func-obs-odd} that differ from those of Theorem \ref{func-obs}. 

We again form a family of metrics $G$ and perturbations of the ASD equation for $(W^\circ,S^\circ,c^\circ)$ parametrized by $[t_0,\infty]$. The moduli space $M^+(S^\circ;\alpha,\alpha'')^G_d$ has obstructed solutions for the parameter $t=\infty$ of the forms given in \eqref{type-1-3} and \eqref{type-2}, and we can use obstructed gluing theory to describe neighborhoods of these solutions. One difference in the present case is that there are no obstructed solutions when $t$ is finite. 

We define the homomorphism $K$ in the same way as before, and the same argument shows \eqref{lambda-circ-relation}. On the other hand, the definition of $M_1$ needs to be slightly modified. For $d\in\{0, 1\}$, the moduli space $M^+(S^\circ;\alpha,\theta'')^G_d$ contains obstructed solutions only of the form
\begin{equation}\label{obs-sol-M1-odd}
  \breve M^+(L;\alpha,\widetilde \theta)_i \times_{S^1} \overline{\underline \Theta} \times_{S^1} M^+(S';\widetilde \theta',\theta'')_{d-i+1},
\end{equation}  
where $1\leq i\leq d+1$, and $\Theta$ is a reducible instanton of index $-3$ over $(W,S,c)$. Our assumption on the parity of $(W',S',c')$ implies that $d-i+1$ is an odd integer. In particular, $M^+(S^\circ;\alpha,\theta'')^G_0$ does not have any obstructed solution, and we define $M_1$ using these moduli spaces without any modifications. To check the relation satisfied by $M_1$, we observe that the moduli space $M^+(S^\circ;\alpha,\theta'')^G_1$ has the following obstructed solutions:
\begin{equation}\label{obs-sol-Delta-1-new}
  \breve M(L;\alpha,\widetilde \theta)_{1} \times_{S^1} \overline{\underline \Theta} \times_{S^1} M(S';\widetilde \theta',\theta'')_{1}.
\end{equation}
Using this moduli space, we obtain relation \eqref{Delta-1-circ-relation-odddeg}, where the term $\nu\tau_{-1}\delta_1$ can be obtained in the same way that the term $\delta_2'\tau_{-1}\delta_1$ appears in relation \eqref{lambda-bdry-relation} for height $-1$ cobordism maps by analyzing neighborhoods of the obstructed solutions in \eqref{obs-sol-Delta-1-new}.

The definition of the homomorphism $M_2$ and the proof of relation \eqref{Delta-2-circ-relation-odddeg} are essentially the same as in the previous case. The new feature is that the moduli spaces $M^+(S^\circ;\theta'',\alpha'')^G_1$ contain obstructed solutions of the form
\[
  \overline{\Theta} \times_{S^1} M(S';\widetilde \theta',\widetilde \theta'')_{2}\times_{S^1} \breve M(L';\widetilde \theta'',\alpha'')_{1},
\]
which give rise to the term $\delta_2'' \nu' \tau_{-1}$ in \eqref{Delta-2-circ-relation-odddeg} in the same way that $\eta_{-1}s \delta_1$ shows up in \eqref{identity-2-Delta-1} as a part of the proof of Proposition \ref{Delta-1}. Finally, we define $L$ in the same way as in the previous cases and obtain the same relation as in \eqref{mu-circ-relation}.

\subsection{Orientations}\label{orientation}

The conventions we use for orienting moduli spaces build upon those of \cite{DS1}. The general scheme follows \cite[\S 3.6]{KM:unknot}; see also the discussion of orientations in the setting of non-singular instanton Floer homology as given in \cite[\S 5.4]{donaldson-book}.

We review and extend the set up from \cite[\S 2.9]{DS1}. Let $(W,S):(Y,L)\to (Y',L')$ be a cobordism of pairs, with bundle data described by an unoriented surface $c\subset W$ with $c \cap \partial W = \omega \sqcup \omega'$. For now we do not assume $Y,Y'$ are integer homology $3$-spheres, nor that the links have $\det \neq 0$. Fix irreducible critical points $\alpha\in \fC_\pi^\omega$, $\alpha' \in \fC_{\pi'}^{\omega'}$. Let $z$ be a homotopy class of singular connections mod gauge on $(W,S,c)$ relative to $\alpha,\alpha'$. Write $\mathscr{B}_z(W,S,c;\alpha,\alpha')$ for the configuration space of singular connections in the homotopy class $z$ (where we remind the reader that cylindrical ends are attached). The index bundle $l_z(W,S,c;\alpha,\alpha')$ of the family of ASD operators $\mathscr{D}_A$ (which are Fredholm with respect to standard Sobolev completions) is a trivializable real line bundle over $\mathscr{B}_z(W,S,c;\alpha,\alpha')$. Denote the two-element set of orientations of $l_z(W,S,c;\alpha,\alpha')$ by
\[
	\Lambda_z[W,S,c;\alpha,\alpha'].
\]
If $\alpha=\theta$ is a reducible, then we consider two index bundles $l_z(W,S,c;\theta_{\pm},\alpha')$, where the ASD operators $\mathscr{D}_A$ are defined on Sobolev spaces whose elements belong to $e^{\varepsilon |t|} L^2$ over $\R_{\leq 0} \times (Y,L)$ for $\pm \varepsilon > 0$ and $|\varepsilon|$ sufficiently small. We write $\Lambda_z[W,S,c;\theta_{\pm},\alpha']$ for the set of orientations. Similar remarks for when $\alpha'$ is reducible. In any of these cases, for different choices of homotopy classes $z_1$ and $z_2$, there are natural identifications
\[
	\Lambda_{z_1}[W,S,c;\alpha,\alpha'] \cong \Lambda_{z_2}[W,S,c;\alpha,\alpha'],
\]
and thus we may drop these decorations from our notation.

Suppose we have another cobordism $(W',S'):(Y',L')\to (Y'',L'')$ with bundle data $c'\subset W'$, where $ c' \cap \partial W'= \omega' \sqcup \omega''$. Write $W'\circ W$ for the composite cobordism, and so forth. Then there is a composition map, which is an isomorphism:
\[
  \Phi:\Lambda[W,S,c;\alpha,\alpha']\otimes_{\Z/2\Z}\Lambda[W',S',c';\alpha',\alpha'']\to \Lambda[W'\circ W,S'\circ S,c'\circ c;\alpha,\alpha''].
\]
Here we view each set of orientations as a $\Z/2$-torsor. If $\alpha'$ is a reducible $\theta'$, then it should be decorated by ``$-$'' in one of the orientation sets in the domain, and by ``$+$'' in the other. The maps $\Phi$ are associative with respect to triple composites.

Consider $ (Y,L)$ with bundle data $\omega$ as above, and fix some choice of {\emph{basepoint}} $\theta$, which is any class of singular connections mod gauge on $(Y,L,\omega)$. Define
\[
	\Lambda[\alpha ] := \Lambda[I\times Y, I\times L; \alpha_+, \theta_-].
\]
If $\alpha$ (resp. $\theta$) is irreducible, ignore the subscript ``$+$'' (resp. ``$-$''). If $\omega$ is empty and $Y$ is an integer homology $3$-sphere with the link $L$ having non-zero determinant, we may choose some quasi-orientation and take $\theta$ to be the associated reducible. 

The irreducible complex for $(Y,L)$ introduced in \eqref{eq:irrcomplexdefn} is more precisely defined as
\begin{equation}\label{eq:irrcomplexprecise}
	C(Y,L) = \bigoplus_{\alpha\in \fC^{\text{irr}}_\pi}R  \Lambda[\alpha]
\end{equation}
and similarly for the admissible case. Here $R\Lambda[\alpha]$ is the infinite cyclic module of $R$ with generators the elements of $\Lambda[\alpha]$. To identify \eqref{eq:irrcomplexprecise} and \eqref{eq:irrcomplexdefn}, one chooses an element in each $\Lambda[\alpha]$. Similarly, the more precise version of \eqref{eq:redcomplexdefn} is
\[
	\sfR(Y,L) = \bigoplus_{\fo\in \mathcal{Q}(Y,L)} R \Lambda[\theta_{\fo}].
\]

Next, consider a cobordism $(W,S,c):(Y,L,\omega)\to (Y',L',\omega')$ as before. Borrowing terminology from \cite{KM:unknot}, an {\emph{I-orientation}} is an element 
\[
	 o_{\theta,\theta'}\in \Lambda[W,S,c;\theta_+,\theta_-']
\]
where $\theta$ and $\theta'$ are the respective basepoint critical points for $(Y,L,\omega)$ and $(Y',L',\omega')$. Given such an I-orientation, and elements $o_\alpha\in \Lambda[\alpha]$ and $o_{\alpha'}\in \Lambda[\alpha']$ where $\alpha\in \fC_{\pi}^\omega$ and $\alpha'\in \fC_{\pi'}^{\omega'}$, we obtain $o_{\alpha,\alpha'}\in \Lambda[W,S,c;\alpha_+,\alpha_-']$ by the rule
\begin{equation}\label{eq:cobordismorientationrule}
	\Phi( o_\alpha \otimes o_{\theta,\theta'}) = \Phi(o_{\alpha,\alpha'} \otimes o_{\alpha'}).
\end{equation}
An orientation of the irreducible stratum of $M_z(W,S,c;\alpha,\alpha')$ is induced by an element of $\Lambda[W,S,c;\alpha_-,\alpha'_-]$. In particular, if $\alpha$ is irreducible, then $o_{\alpha,\alpha'}$ induces an orientation of this moduli space. When $\alpha$ is reducible (and $\omega=\emptyset$), we proceed as follows. Use that $\Lambda[I\times Y,I\times L;\alpha_-,\alpha_-]$ has a distinguished element $o_{\alpha}^{-}$ induced by an orientation of a base-pointed component of $L$ and the canonical homology orientation (see below). Then  
\begin{equation}\label{eq:orientingredsit}
	\Phi(o_{\alpha}^{-}\otimes o_{\alpha,\alpha'} ) 
\end{equation}
gives an element in $\Lambda[W,S,c;\alpha_-,\alpha_-']$, which orients the moduli space. 

Suppose $A_0$ is a reducible singular connection on $(W,S,c)$ with limits in the basepoints $\theta$ and $\theta'$. The adjoint of $A_0$ is compatible with a bundle splitting $\underline{\R}\oplus K_0^{\pm 1}$ where $K_0$ is a complex line bundle. One of $K_0^{\pm 1}$ is distinguished over the other, say $K_0$, by choosing a $U(2)$-lift of the bundle data, which amounts to orienting $c$; and choosing an ordering of the $U(2)$-reduction at point $p\in S\setminus \partial c$, which amounts to orienting the component of $S\setminus \partial c$ containing $p$. Given such choices, there is a natural identification of real vector spaces
\begin{equation}
	l_z(W,S;\theta_-,\theta'_-)|_{[A_0]} \cong \wedge^\text{top} (H^1(W) \oplus H^+(W)^\ast\oplus H^0(W)^\ast).\label{eq:reddetlinemin}
\end{equation}
Indeed, the operator $\mathscr{D}_{A_0}$ splits into a sum according to the decomposition $\underline{\R}\oplus K_0$. The index bundle of the operator coupled to $K_0$ is naturally oriented by its complex structure, and has trivial real determinant. The index bundle of the operator associated to $\underline{\R}$ remains, and appears on the right side of \eqref{eq:reddetlinemin}. Similarly, we have an identification
\begin{equation}\label{ori-bdle}
	l_z(W,S;\theta_+,\theta'_-)|_{[A_0]} \cong \wedge^\text{top} (H^1(W) \oplus H^+(W)^\ast).
\end{equation}
In this way, a reducible singular connection, an orientation of $c$ and of a component of $S\setminus \partial c$, and a choice of a {\emph{homology orientation}}, i.e. an orientation of
\[
	H^1(W)\oplus H^+(W),
\]
	determine an I-orientation $(W,S,c)$.

\begin{remark}\label{rmk:orientation-reversal}
If $K_0$ is replaced by $K_0^{-1}$ above, then the identifications in \eqref{eq:reddetlinemin} and \eqref{ori-bdle} change sign respectively according to the parities of $d+d'-1$ and $d+d'$, where\footnote{This corrects some sign typos in \cite[2.9]{DS1}, where $d,d'$ are defined slightly differently. These typos that appear in \cite{DS1} are isolated and do not affect anything else in that work.}
\[
	d=b^1(W)-b^+(W), \hspace{1cm} d'=\frac{\ind(\sD_{A_0})-d+1}{2}. \;\;\diamd
\]
\end{remark}

Consider $(Y,L,\omega)$, and critical points $\alpha,\beta\in \fC_\pi^\omega$. Assume a basepoint connection $\theta$ is chosen, as described above. We also choose a basepoint $p\in L\setminus  \partial \omega$, just as in the formation of its $\cS$-complex, and orient $\omega$ and the component of $L\setminus  \partial \omega$ containing $p$. Given $o_\alpha\in \Lambda[\alpha]$ and $o_{\beta}\in \Lambda[\beta]$ we obtain an element $o_{\alpha,\beta}\in \Lambda[I\times Y,I\times L,I\times \omega;\alpha_+,\beta_-]$ by the rule
\[
	\Phi(o_{\alpha,\beta}\otimes o_{\beta} ) = o_\alpha.
\]
If $\theta$ is reducible, this convention agrees with the above general rules applied to the cylinder, with I-orientation induced by the canonical homology orientation, and orientation of $[0,1]\times \omega$ and the component of $[0,1]\times (L\setminus \partial \omega)$ containing the basepoint, as in the previous paragraph. The element $o_{\alpha,\beta}$ is used to orient the irreducible stratum of $M_z(\alpha,\beta)$. 

Let $\tau_s$ be the diffeomorphism of $\R\times Y$ given by $\tau_s(t,y)=(t-s,y)$. An identification
\begin{equation}\label{eq:cylmodulitranslationorient}
	 \R\times \breve{M}_z(\alpha,\beta) = M_z(\alpha,\beta)
\end{equation}
is given by $(s,[A])\mapsto [\tau_s^\ast A_0]$, where $A_0$ is gauge equivalent to $A$ and 
\[
	\int_{\R\times Y} t |F_{A_0}|^2=0.
\]
The identification \eqref{eq:cylmodulitranslationorient}, with ordering as written, induces an orientation of $\breve{M}_z(\alpha,\beta)$ from one of $M_z(\alpha,\beta)$ and the usual orientation of $\R$. 

These conventions determine sign conventions for the maps of type $d,\delta_1,s,\lambda,\Delta_1,\Delta_2$ described in Subsection \ref{subsec:linkinv}. For $\delta_2$, we use the opposite of what would follow from above. Specifically, if $\theta_{\fo}$ is the basepoint, given $o_\alpha\in \Lambda[\alpha]$ and $o_{\theta_{\fo'}}\in \Lambda[\theta_{\fo'}]$ we demand
\[
	\langle \delta_2(o_{\alpha}) , o_{\theta_{\fo'}} \rangle = -\# \breve{M}(\alpha,\theta_{\fo'})_0
\]
where the moduli space appearing is oriented as in the previous paragraph, and ``$\#$'' is the signed count of this $0$-manifold with respect to this orientation. To define maps such as $v,\mu$ we must cut down by modified holonomy maps. Relevant to $v$ are the maps 
\[
	H_{\alpha,\beta}:\breve{M}_z(\alpha,\beta)\longrightarrow S^1,
\]
and we orient the preimage of a regular value using the normal-fibers-first convention, where $\breve{M}_z(\alpha,\beta)$ has already been oriented using the above conventions. To define $\mu$ we use the same convention for modified holonomy maps on moduli $M_z(W,S,c;\alpha,\alpha')$.

For the maps $\eta_i = \eta_i(W,S,\sbd)$, proceed as follows. Assume $b^1(W)=b^+(W)=0$. Fix basepoint classes $\theta$ and $\theta'$ for the ends, which may as well be reducibles. Assume that there is a component of $S\setminus \partial c$ that intersects the basepoints at the ends, and choose an orientation of this component, as well as an orientation of $c$. These choices also determine the choices of orientations for $\omega,\omega'$ and the components of $L\setminus \partial \omega$, $L'\setminus \partial \omega'$ containing the basepoints. Finally, fix a homology orientation of $W$ and an I-orientation of $(W,S,c)$. 

Now consider a reducible instanton $\Theta$ on $(W,S,c)$ of index $2i-1$ with limiting connections $\theta_\fo$ and $\theta_{\fo'}$ at the ends; these may be different from the basepoint connections $\theta$ and $\theta'$. The data of $\Theta$ together with the orientation choices involving $S$ and $c$ in the previous paragraph, with the homology orientation of $W$, determine an element of
\begin{equation}\label{eq:detlineindexreds}
	  \Lambda[W,S,c;(\theta_\fo)_-,(\theta_{\fo'})_-].
\end{equation}
There is another element of \eqref{eq:detlineindexreds} which is obtained by using the fixed I-orientation of $(W,S,c)$. Then the contribution of $\Theta$ to $\langle \eta_i(\theta_{\fo}) ,\theta_{\fo'}\rangle$ is defined to be $+1$ if the two orientations obtained in \eqref{eq:detlineindexreds} agree, and $-1$ otherwise. 

A moduli space $M(W,S,c;\alpha,\alpha')^G$ parametrized by a metric family $G$ is oriented, via the normal-fibers-first convention, using the projection $M(W,S,c;\alpha,\alpha')^G\to G$, a choice of orientation for $G$, and the orientations of the fibers $M(W,S,c;\alpha,\alpha')^{g}$ for $g\in G$ as determined above. When further cutting down by a modified holonomy map $H_{\alpha,\alpha'}$, first orient $M(W,S,c;\alpha,\alpha')^G$ as just described, and then orient the primage of a regular value of $H_{\alpha,\alpha'}:M(W,S,c;\alpha,\alpha')^G\to S^1$ using the normal-fibers-first convention.

Finally, we remark on the orientation conventions relevant to the obstructed gluing theory used in this section. First, a framed moduli space such as 
\begin{equation}\label{eq:framedmoduliinorientationsec}
	M(W,S,c;\widetilde{\alpha},\widetilde{\alpha}')
\end{equation}
is a principal $\Gamma_\alpha\times \Gamma_{\alpha'}$-bundle over the irreducible stratum of $M(W,S,c; \alpha, \alpha')$. We use the normal-fibers-first convention to orient the irreducible stratum of \eqref{eq:framedmoduliinorientationsec}, given our orientation of the base and the product orientation of $\Gamma_\alpha\times \Gamma_{\alpha'}$. Maps such as $\Phi_{\alpha,\Theta}$ and $\widetilde \Phi_{\alpha,\Theta}$ in Proposition \ref{type-1-obs-gluing} are declared orientation-preserving, and this induces orientations on the domains of such maps given our conventions. The complex line bundles that appear, as well as the zero sets of sections such as $\Psi_{\alpha,\Theta}$, are oriented in the natural way, given an orientation of the base space.

\newpage


\section{Proof of the exact triangles}\label{sec:proofs}

In this section we prove Theorem \ref{thm:exacttriangles}, the three exact triangles that involve trivial singular bundle data. In the first subsection, we discuss topological aspects of the exact triangles and along the way introduce the basic constructions; this is a review of the situation in \cite{KM:unknot} with a number of extensions (and with different notation). In this preliminary subsection we also explain how Cases I, II, III of Theorem \ref{thm:exacttriangles} are indeed all possible cases, and also that Cases II and III are logically equivalent. In the two subsections that follow, we prove Case I and then Case II, respectively.

\subsection{Setup} \label{subsec:triangleprelim}

As most cobordisms take place in the cylinder $I\times Y$, we sometimes omit it from notation. We assume $L,L',L''$ are links that form a skein triangle as in Figure \ref{fig:skein}, each having non-zero determinant. Fix basepoints on the links that are located away from the local skein 3-ball, so that we may view the basepoints on $L,L',L''$ as the same point $p$. 

Throughout this section, we always use the trivial $SU(2)$ singular bundle data. A metric $g_L$ defined on $Y$ with cone angle $\pi$ along $L$ is fixed, as is small generic perturbation data, and similarly for the other links $L',L''$. Then the $\cS$-complexes for the links are defined:
\begin{equation}\label{eq:scomplexesprooftriangles}
    \widetilde C = \widetilde C(Y,L), \qquad \widetilde C' = \widetilde C(Y,L'), \qquad \widetilde C'' = \widetilde C(Y,L'').
\end{equation}
Write $S:L\to L'$, $S':L'\to L''$ and $S'':L''\to L$ for the saddle cobordisms which are products outside of the local skein region. These are depicted schematically as follows:
\begin{center}
    \includegraphics[scale=0.65]{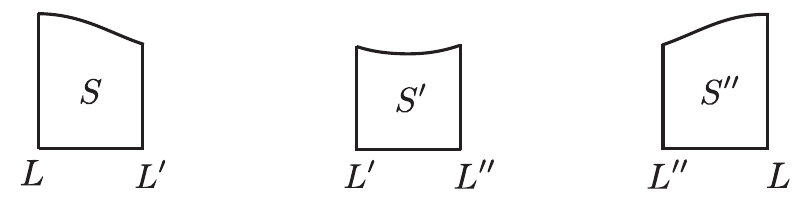}
\end{center}
The height of the  vertical line representing $L$ reminds us that it is the link with more components than $L',L''$. Under this assumption, $S$ and $S''$ are orientable, and $S'$ is non-orientable. A cylindrical end metric $g_S$ on $\R\times Y$ with cone angle $\pi$ along $S^+$ is fixed, where $S^+$ is the surface $S$ with cylindrical ends attached. This metric agrees, at the ends, with $dt^2 + g_L$ and $dt^2+g_{L'}$. Similarly for the other cobordisms. In addition to the $\cS$-complexes \eqref{eq:scomplexesprooftriangles}, the $s$-maps for the links, as given in Definition \ref{defn:smap}, will also play a role in the construction of the exact triangles below.

Recall that $\sfR=\sfR(Y,L)$ is the free abelian group with basis the set of quasi-orientations $\cQ(Y,L)$. Call an automorphism of $\sfR$ a {\emph{sign-change}} if it is a diagonal matrix with entries in $\{1,-1\}$ with respect to this natural basis. Similar remarks apply to $\sfR'$ and $\sfR''$.

 The simple topology of the saddle cobordisms gives the following: each quasi-orientation $\fo$ of $L$ extends over the cobordisms $S$ and $S'$ to unique quasi-orientations $\fo'$ on $L'$ and $\fo''$ on $L''$ respectively; and $\fo$ is uniquely determined by $\fo'$ and $\fo''$. Thus we may write
\begin{equation*}
    \cQ(Y,L) = \cQ(Y,L')\sqcup \cQ(Y,L'')
\end{equation*}
and from this we obtain the following natural identification:
\begin{equation}\label{eq:reddirectsum}
    \sfR = \sfR' \oplus \sfR''
\end{equation}
As discussed in Subsection \ref{subsec:cobcyl}, the orientable cobordisms $S$ and $S''$ have minimal reducibles that are flat and are labelled by their quasi-orientations. Consider the maps
\[
  \eta=\eta(I\times Y, S) : \sfR\to \sfR'  \qquad \eta''=\eta(I\times Y, S''):\sfR''\to \sfR
\]
which are signed counts of minimal reducibles, as in Subsection \ref{unob-cob-map}. Up to sign-changes, $\eta$ and $\eta''$ are the natural projection and inclusion with respect to \eqref{eq:reddirectsum}, respectively. For the cobordism $S'$ we have no reducibles and $\eta'=0$. Recall from Subsection \ref{subsec:cobcyl} that the orientable cobordisms $S$ and $S''$ are {\emph{even}}, while the non-orientable cobordism $S'$ is {\emph{odd}}. 

The signed counts defining $\eta,\eta''$ follow the orientation conventions described in Subsection \ref{orientation}, after fixing arbitrary I-orientations for each of the three cobordisms $S,S',S''$. 

Recall that when defining general cobordism maps we also require a path $\gamma$ that joins the basepoints of the links. In all cases we take this path to be $\gamma = I\times \{p\} \subset I\times Y$, which always lies on the cobordism surface.

For the cobordism $S$ we define the maps $\lambda$, $\mu$, $\Delta_1$, $\Delta_2$ in the usual way as described in Subsection \ref{sec:unobscobmaps}. These definitions, which we display below, will be used in the proof of Case I; for Case II, these maps will be modified in a way similar to what was done in Section \ref{sec:obstructed}, as will be explained in Subsection \ref{subsec:caseii}. For irreducible critical points $\alpha\in \fC_\pi$ and $\alpha'\in \fC_{\pi'}$ and reducible critical points $\theta_{\fo}$ and $\theta_{\fo'}$ where $\fo\in \cQ(Y,L)$ and $\fo'\in \cQ(Y,L')$ we set
\begin{align*}
    \lambda:C\to C' && \langle \lambda(\alpha), \alpha' \rangle & = \# M(I\times Y, S;\alpha,\alpha')^{g_S}_0 & \lambda &= \vcenter{\hbox{\includegraphics[scale=0.35]{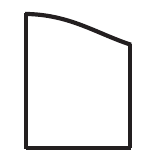} }}\\
    \Delta_1:C\to \sfR' && \langle \Delta_1(\alpha), \theta_{\fo'} \rangle &= \# M(I\times Y, S;\alpha,\theta_{\fo'})^{g_S}_0 & \Delta_1& = \vcenter{\hbox{\includegraphics[scale=0.35]{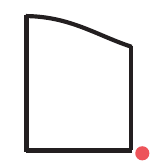} }}\\
    \Delta_2:\sfR\to C' && \langle  \Delta_2(\theta_{\fo}), \alpha' \rangle &= \# M(I\times Y, S;\theta_{\fo},\alpha')^{g_S}_0 &\Delta_2 &= \vcenter{\hbox{\includegraphics[scale=0.35]{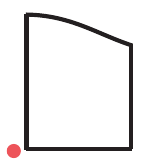} }}\\  
    \mu:C\to C' &&\langle \mu(\alpha), \alpha' \rangle &= \# M_{\gamma}(I\times Y, S;\alpha,\alpha')^{g_S}_0 &\mu &= \vcenter{\hbox{\includegraphics[scale=0.35]{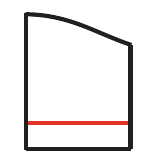} }}
\end{align*}
Here and below, ``$\#$'' refers to a signed-count using the orientation conventions given in Subsection \ref{orientation}. We have indicated the metric used in defining each moduli space as a superscript, as in Section \ref{sec:obstructed}; the choice of perturbations and modified holonomy maps (used for $\mu$) are surpressed from notation. Following notation from \cite[Section 6]{DS1}, we write 
\[
    M_{\gamma}(I\times Y, S;\alpha,\alpha')^{g_S}_d = \left\{ [A]\in M(I\times Y,S;\alpha,\alpha')^{g_S}_{d+1}: H^{\gamma}(A)= -1 \right\}
\]
Recall that $H^\gamma$ is the modified holonomy map to $S^1$ which, roughly, evaluates the holonomy of the adjoint connection along the line $\gamma^+  = \R\times \{p\} \subset \R\times Y$. We define the maps $\lambda',\Delta_1',\Delta_2',\mu'$ for $S'$ and $\lambda'',\Delta_1'',\Delta_2'',\mu''$ for $S''$ in the same way. The relations that these maps satisfy depend on the indices of the minimal reducibles, and these indices will depend on which of Case I--III we are in. 

Now consider the composite cobordisms $T=S'\circ S:L\to L''$, $T'=S''\circ S':L'\to L$ and $T''=S\circ S'':L''\to L$. As shown in \cite[Lemma 7.2]{KM:unknot} we have
\begin{equation}
  (I\times Y, T ) =  (I\times Y\setminus B^4, \overline S''\setminus B^2 ) \cup_{(S^3,U_1)} (B^4, F_0) \label{eq:doublecomp}
\end{equation}
where $F_0$ is $\bR\bP^2$ minus a disk, and $\overline S$ is the reverse cobordism of $S$. Further,
\[
	F_0 \cdot F_0 = + 2
\]
In the above, $(S^3,U_1)$ is a cross-section for a connected sum region; $U_1$ is an unknot. We form a 1-dimensional family of metrics, denoted $G_T$, which is parametrized by $t\in [-\infty,\infty]$. At $t= -\infty$, the metric $g_t\in G_T$ is the metric broken along $(Y,L')$ which is simply $g_S\sqcup g_{S'}$ on the disjoint union $(\R\times Y,S^+)\sqcup (\R\times Y,S'^+)$. At $t=+\infty$, the metric is broken along $(S^3,U_1)$. Write $g_{\overline S''\setminus B^2}\sqcup g_{F_0}$ for this latter broken metric, where $g_{\overline S''\setminus B^2}$ and $g_{F_0}$ are cylindrical end metrics on the respective pairs. The metrics $g_t$ for $t\in \R$ are smooth, and interpolate between the above two broken metrics. Similar remarks, with the obvious notation, hold for the composite cobordisms $T'$ and $T''$. 

We next define maps $K, M_1, M_2, L$ associated to the composite cobordism $T$. (There should be no confusion regarding the map $L$ and the link $L$.) Write $\alpha$ and $\alpha''$ (resp. $\theta_\fo$ and $\theta_{\fo''}$) for irreducible (resp. reducible) critical points for $L$ and $L''$, respectively. Define
\begin{align*}
    K:C\to C'' && \langle K(\alpha), \alpha'' \rangle & = \# M(I\times Y, T;\alpha,\alpha'')^{G_T}_0 & K &= \vcenter{\hbox{\includegraphics[scale=0.35]{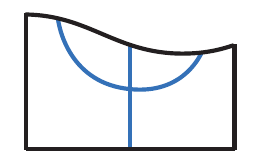} }}\\
    M_1:C\to \sfR'' && \langle M_1(\alpha), \theta_{\fo''} \rangle &= \# M(I\times Y, T;\alpha,\theta_{\fo''})^{G_T}_0 & M_1& = \vcenter{\hbox{\includegraphics[scale=0.35]{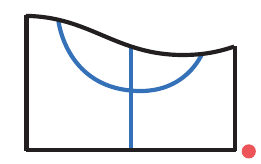} }}\\
    M_2:\sfR\to C'' && \langle  M_2(\theta_{\fo}), \alpha'' \rangle &= \# M(I\times Y, T;\theta_{\fo},\alpha'')^{G_T}_0 &M_2 &= \vcenter{\hbox{\includegraphics[scale=0.35]{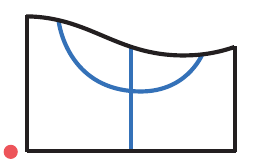} }}\\  
    L:C\to C'' &&\langle L(\alpha), \theta_{\fo''} \rangle &= \# M_{\gamma}(I\times Y, T;\alpha,\alpha'')^{G_T}_0 & L &= \vcenter{\hbox{\includegraphics[scale=0.35]{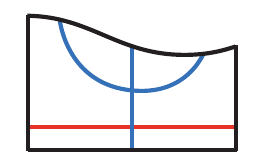} }}
\end{align*}

\noindent Here, as in Section \ref{sec:obstructed}, the superscripts of $G_T$ in the notation indicate that the moduli spaces consist of pairs $([A],g_t)$ where $[A]$ is a singular instanton with respect to $g_t$. There is also chosen a corresponding family of perturbation data. As a general remark that also applies to later constructions in this section, this and other choices analogous to the auxiliary data that appeared in Subsection \ref{subsec:homotopies} are omitted from notation. The subscripts continue to denote the expected dimension of the moduli spaces. For example,
\begin{align*}
    M(I\times Y, S;\alpha,\alpha')^{G_T}_d &= \left\{([A],g_t): \, g_t\in G_T, \, [A]\in M(I\times Y, T;\alpha,\alpha'')^{g_t}_{d-1}  \right\} \\[1mm]
    M_\gamma(I\times Y, S;\alpha,\alpha')^{G_T}_d &= \left\{([A],g_t)\in M(I\times Y, S;\alpha,\alpha')^{g_T}_{d-2}: \; H^\gamma(A)=-1 \right\} 
\end{align*}
The maps $K',M_1',M_2',L'$ and $K'',M_1'',M_2'',L''$ associated to the respective composite cobordisms $T'$ and $T''$ are defined completely analogously. Further commentary on these definitions is given below. See Lemma \ref{lemma:welldefined}, for example.

Next we consider the triple composite $V = S''\circ S' \circ S: L \to L$. Following \cite[Lemma 7.4]{KM:unknot}, there is a distinguished hypersurface pair $(S^3,U_2)\subset (I\times Y, V)$ such that:
\begin{equation}\label{eq:twocompunlinkdecomp}
    (I\times Y, V)  = (W, F_1 )\cup_{(S^3,U_2)} ( B^4, F_2 )
\end{equation}
where the link $U_2$ is an unlink with two components, and the surface $F_2$ is a copy of $\bR\bP^2$ which has two disks deleted. Again we have
\[
	F_2\cdot F_2 = + 2
\]
Furthermore, upon replacing $(B^4, F_2)$ with $(B^4,\Delta)$ where $\Delta$ is two standard disks in the 4-ball, we get the product cobordism:
\begin{equation*}\label{eq:triplecompositedecomp}
    (W, F_1 )\cup_{(S^3,U_2)} ( B^4, \Delta ) \cong (I\times Y, I\times L)
\end{equation*}
For the triple composites $V'=S\circ S''\circ S':L'\to L'$ and $V''=S'\circ S\circ S'':L''\to L''$ the same remarks apply, with the notation following the same pattern as above.

\begin{figure}[t]
    \centering
    \includegraphics[scale=1.1]{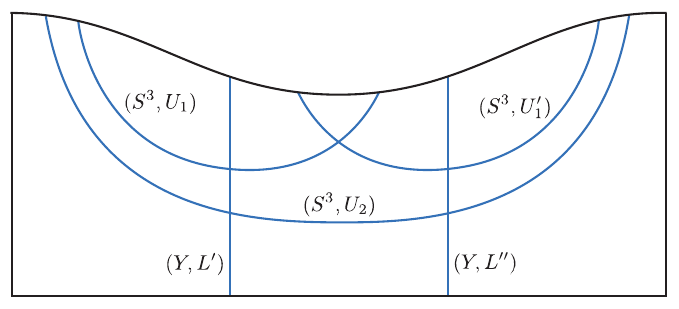}
    \caption{{\small{The 5 hypersurface pairs inside the triple composite $V=S''\circ S' \circ S: L \to L$.}}}
    \label{fig:5hypersurfaces}
\end{figure}

In particular, in the cobordism $(I\times Y, V)$ there are five distinguished hypersurfaces as depicted in Figure \ref{fig:5hypersurfaces}. There is a family of metrics $G_V$ parametrized by a\ pentagon and which stretches along these hypersurfaces. See \cite[\S 7.3]{KM:unknot} for details, and the related antecedent construction in \cite{KMOS}. The metrics on the interior of the pentagon are smooth, while those on the interior of an edge (resp. on a corner) are broken along one hypersurface (resp. two hypersurfaces). Define $Q$, a map associated to $V$, as follows:
\begin{align}\label{eq:qdefn}
    Q:C\to C && \langle Q(\alpha), \beta \rangle & = \# M(I\times Y, V;\alpha,\beta)^{G_V}_0 & Q &= \vcenter{\hbox{\includegraphics[scale=0.35]{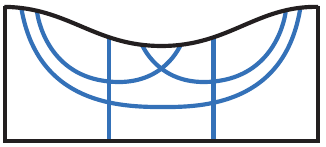} }}
\end{align}

\noindent Here $\alpha$ and $\beta$ are irreducible critical points for the link $(Y,L)$. There are other maps one might associate to $V$ using this 2-dimensional family of metrics by allowing reducible limits, or cutting down by holonomy maps, for example. However, we will see that the only such map necessary for the proof of the exact triangles is the map $Q$. The maps $Q':C'\to C'$ and $Q'':C''\to C''$ are defined analogously, using the cobordisms $V'$ and $V''$.

For the odd (non-orientable) saddle cobordism $S':L'\to L''$ we also have the map $\nu'$ as given in Definition \ref{defn:numap}. In this context, this map is as follows:
\begin{align}
    \nu':\sfR'\to \sfR'' && \langle \nu'(\theta_{\fo'}), \theta_{\fo''} \rangle & = \# M(I\times Y, S';\theta_{\fo'},\theta_{\fo''})^{g_{S'}}_0 & \nu' &= \vcenter{\hbox{\includegraphics[scale=0.35]{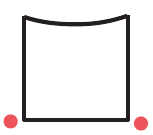} }} \label{eq:nuprime}
\end{align}
This map counts isolated irreducible instantons on $(I\times Y,S')$ with reducible limits. 

The maps defined above (and their modifications for Cases II and III) are the essential ingredients required to assemble the exact triangles. What remains is to prove an assortment of relations, the nature of which will depend on which of Case I--III we are in.

\begin{remark}
	Following our convention from earlier sections, in the case that the family of metrics is a single metric, we often omit the metric from the notation of a moduli space. $\diamd$
\end{remark}

As a preliminary step towards the proof of Theorem \ref{thm:exacttriangles} we have the following.

\begin{lemma}\label{lemma:casesexhaust}
	Suppose $(L,L',L'')$ is a skein triple as above involving links with non-zero determinants. Then exactly one of Cases I, II, or III holds.
\end{lemma}

\begin{proof}
	Write $X$, $X'$, $X''$ for the double covers of $I\times Y$ branched over $S$, $S'$, $S''$, respectively. The assumption on determinants gives that the signature of each of these 4-manifolds is $\pm 1$. The above description of the double composite cobordism yields $X'\circ X \cong \overline{X}''\#\overline{ \bC\bP}^2$ where $\overline{X}''$ is the reverse cobordism of $X''$. Thus
\begin{equation}\label{eq:sumofthreesignatures}
	\sigma(X) + \sigma(X') + \sigma(X'') = -1.
\end{equation}
Proposition \ref{prop:indbranchedcover} says that $\sigma(X)=-\varepsilon(L,L')$ and $\sigma(X'')=-\varepsilon(L'',L)$. Then equation \eqref{eq:sumofthreesignatures} shows that the case $\epsilon(L,L')=\epsilon(L'',L)=-1$, or equivalently $\sigma(X)=\sigma(X'')=1$, is impossible. This leaves Cases I, II, III.
\end{proof}

By the following, it will suffice to prove Cases I and II of Theorem \ref{thm:exacttriangles}.

\begin{lemma}
	Case II of Theorem \ref{thm:exacttriangles} implies Case III, and conversely.
\end{lemma}

\begin{proof}
Suppose we have a skein triple $(L,L',L'')$ as in Case II, so that $\epsilon(L,L')=-1$ and $\epsilon(L'',L)=+1$. Let $mL$ denote the mirror of $L$. Then we have a skin triple $(mL, mL'', mL')$. Using that the signature is negated under taking mirrors, we have $\epsilon(mL,mL'')=-\epsilon(L,L'')=\epsilon(L'',L)=+1$ and similarly $\epsilon(mL',mL)=\epsilon(L,L')=-1$. Thus $(mL,mL'',mL')$ is a skein triple of the type in Case III.

Now suppose the theorem is proved for skein triples $(L,L',L'')$ in Case II. Thus we have an associated exact triangle of $\cS$-complexes, for some morphisms and homotopies:
\begin{equation*}\label{eq:lemmacaseii}
	\begin{tikzcd}[column sep=5ex, row sep=10ex, fill=none, /tikz/baseline=-10pt]
& \widetilde C  \arrow[dr, "\widetilde \lambda"] \arrow[dl, bend right=60, "\widetilde K"'] & \\
\widetilde C''  \arrow[rr, bend right=60, "\widetilde K''"']  \arrow[ur, "\widetilde \lambda''"] & & \Sigma \widetilde  C' \arrow[ll, "{[-1]}"', "\widetilde \lambda'"] \arrow[ul, bend right=60, "\widetilde K'"']
\end{tikzcd}
\end{equation*}
Take the dual of this diagram in the category of $\cS$-complexes. Note that $\widetilde C^\dagger$ is naturally identified with the $\cS$-complex of the link $mL$, and similarly $(\widetilde C'')^\dagger$ is the $\cS$-complex for $mL''$. On the other hand, $(\Sigma \widetilde C')^\dagger$ is isomorphic to $\Sigma^{-1}( (\widetilde C')^{\dagger})$ by Proposition \ref{prop:susdual}, the negative suspension of the $\cS$-complex for $mL'$. The dual diagram is then seen to be an exact triangle for the triple $(mL,mL'',mL')$, falling into Case III. The converse is similar.
\end{proof}

\subsection{Case I} \label{subsec:casei}

Having set the stage in the previous subsection, we now prove Case I of Theorem \ref{thm:exacttriangles}. In this case we assume $\epsilon(L,L')=\epsilon(L'',L)=1$. By Propositions \ref{prop:saddlecobnonor} and \ref{prop:saddlecobor} this means that the saddle cobordisms $S$, $S'$, $S''$ are all unobstructed. More precisely, $S'$ has no reducible instantons, and the minimal reducible instantons on $(I\times Y,S)$ and $(I\times Y,S'')$ are unobstructed and have index $-1$. As already mentioned above, these reducibles correspond to quasi-orientations of $S$ and $S''$, respectively.

There are other reducible instantons of relevance. Consider the pair $(I\times Y\setminus B^4, \overline S''\setminus B^2 ) $ from \eqref{eq:doublecomp}. This has three ends, and it appears as a component of the broken cobordism one boundary point of the metric family $G_T$. This pair is {\emph{obstructed}}: it has minimal reducibles of index $-3$. Note the component $(B^4, F_0)$ appearing in \eqref{eq:doublecomp} does not have any reducibles, as $F_0$ is non-orientable. Similar remarks hold for the components arising from the relevant broken metric at the boundary of $G_{T'}$. In the case of $G_{T''}$, the analogous pair $(I\times Y\setminus B^4, \overline S'\setminus B^2 ) $ has a non-orientable surface and so has no reducibles.

Each of the four types of minimal reducible instantons discussed above on the respective pairs $(I\times Y,S)$, $(I\times Y,S'')$, $(I\times Y\setminus B^4, \overline S''\setminus B^2 ) $, $(I\times Y\setminus B^4, \overline S\setminus B^2 ) $ has a corresponding map counting the reducibles. We depict each of these maps, in the order given, as follows:
\begin{equation}
	\eta = \vcenter{\hbox{\includegraphics[scale=0.35]{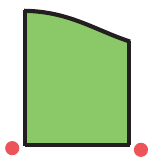} }} \qquad \quad \eta'' = \vcenter{\hbox{\includegraphics[scale=0.35]{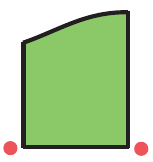} }} 
	\qquad\quad  J = \vcenter{\hbox{\includegraphics[scale=0.35]{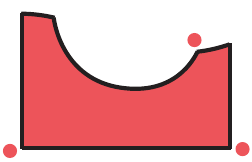} }} \qquad\quad  J' = \vcenter{\hbox{\includegraphics[scale=0.35]{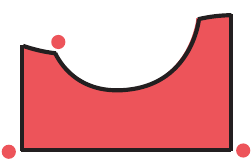} }}\label{eq:minredscasei}
\end{equation}
The colors remind us that the reducibles being counted in the first two cases are unobstructed and of index $-1$; and in the latter two cases the reducibles are obstructed and of index $-3$. The maps $\eta$ and $\eta''$ were described in the previous subsection; the maps $J:\sfR\to \sfR''$ and $J':\sfR'\to \sfR$ are, up to sign-changes, the natural projection and inclusion, respectively, with respect to the decomposition \eqref{eq:reddirectsum}. In particular, we have
\begin{equation}
	J\eta'' \; \dot =\;  \text{id}_{\sfR''} \qquad \quad \eta J' \; \dot =\;  \text{id}_{\sfR'} \qquad \quad \eta'' J \pm  J'\eta \; \dot =\;  \text{id}_{\sfR} \label{eq:caseiredrels}
\end{equation}
where ``$\dot =$'' means equal up to sign-changes.

\begin{remark}\label{rmk:signforjmaps}
	The precise definitions of $J,J'$ are similar to $\eta,\eta''$ in that they involve using fixed I-orientations to determine signed counts of reducibles. However, for $J$ and $J'$ an overall sign $\varepsilon_0 = \pm 1$ is applied in addition to these ones, which is explained below. $\diamd$
\end{remark}

We proceed to construct the exact triangle of Theorem \ref{thm:exacttriangles} in Case I. We define, in the usual fashion as is done in the unobstructed case, maps $\widetilde \lambda:\widetilde C\to \widetilde C'$, $\widetilde \lambda':\widetilde C'\to \widetilde C''$, and $\widetilde \lambda'':\widetilde C''\to \widetilde C$ as follows, with components as prescribed in Subsection \ref{subsec:triangleprelim}:
\[
\widetilde\lambda =  \left[ \begin{array}{ccc} \lambda & 0 & 0 \\ \mu & \lambda & \Delta_2 \\ \Delta_1 & 0 & \eta \end{array} \right] 
\qquad
\widetilde\lambda' =  \left[ \begin{array}{ccc} \lambda' & 0 & 0 \\ \mu' & \lambda' & \Delta'_2 \\ \Delta'_1 & 0 & 0 \end{array} \right]
\qquad
\widetilde\lambda'' =  \left[ \begin{array}{ccc} \lambda'' & 0 & 0 \\ \mu'' & \lambda'' & \Delta''_2 \\ \Delta''_1 & 0 & \eta'' \end{array} \right]
\]
That these are all morphisms of $\cS$-complexes follows from the discussion in Subsection \ref{subsec:cobcyl}. Furthermore, $\widetilde\lambda$ and $\widetilde \lambda''$ are strong height $0$ morphisms of even degree, while $\widetilde \lambda'$, from the non-orientable saddle cobordism, is a morphism of odd degree.

Next, we define maps $\widetilde K:\widetilde C\to \widetilde C''$, $\widetilde K': \widetilde C'\to \widetilde C$ and $\widetilde K'':\widetilde C''\to \widetilde C'$ as follows:
\[
\widetilde K =  \left[ \begin{array}{ccc} K & 0 & 0 \\ L & -K & M_2 \\ M_1 & 0 & J \end{array} \right] 
\quad
\widetilde K' =  \left[ \begin{array}{ccc} K' & 0 & 0 \\ L' & -K' & M'_2 \\ M'_1 & 0 & J' \end{array} \right]
\quad
\widetilde K'' =  \left[ \begin{array}{ccc} K'' & 0 & 0 \\ L'' & -K'' & M''_2 \\ M''_1 & 0 & 0 \end{array} \right]
\]
Before turning to the relations of these maps, we justify the definitions given thus far.

\begin{lemma}\label{lemma:welldefined}
	In Case I, all the maps defined in Subsection \ref{subsec:triangleprelim} are well-defined. More specifically, the moduli space being counted, in each case, is a finite collection of instantons.
\end{lemma}

\begin{proof}
	The maps $\widetilde \lambda$, $\widetilde \lambda'$, $\widetilde \lambda''$ as well as $\nu'$ were already constructed in Section \ref{sec:links}. The only maps that attention are $K,M_1,M_2,L$ and $Q$, as well as their analogues for $S'$ and $S''$. In each case, we must show that the moduli space under consideration is a compact $0$-manifold. The only potential obstacle towards this lies wth instantons in the associated compactified moduli space which are broken solutions involving an obstructed reducible. 
	
The maps $K$,$M_1$,$M_2$, $L$ defined using a $1$-parameter metric family are similar to the maps defined in Subsection \ref{subsec:functoriality}, except that the topology of the situation at hand is different, and the metric family has broken metrics in both directions. Let us consider the map $M_1$, which is defined using moduli spaces of the form
\[
	N:=M(I\times Y, T;\alpha,\theta_{\fo''})^{G_T}_0.
\]
Note that the cobordism $(I\times Y, T)$, for any smooth metric $g_t$ in the interior of $G_T$, contains no reducible instantons, as $T$ is non-orientable. Thus if a sequence $([A_i],g_{t_i})$ in $N$ fails to converge, then, after passing to a subsequence, the metric $g_{t_i}$ must converge to one of the broken metrics at the boundary of $G_T$. Suppose that this broken metric is the one at $t=+\infty$, broken along $(S^3,U_1)$. In this case, $[A_i]$ may a priori chain-converge on the complement of some finite number of points to a broken instanton $([A],[B],[C])$ in
\begin{equation}
	 \breve{M}^+( L;\alpha,\widetilde \theta_{\fo''})\times_{S^1}M(I\times Y\setminus B^4, \overline S''\setminus B^2 ;\widetilde \theta_{\fo},\widetilde \theta, \theta_{\fo''})\times_{S^1} M^+(B^4, F_0 ;\widetilde \theta) \label{eq:modredobstruleout}
\end{equation}
where $[B]$ is one of the obstructed reducibles on $(I\times Y\setminus B^4, \overline S''\setminus B^2 ) $ of index $-3$, while $[A]$ and $[C]$ are possibly broken instantons. This triple corresponds to an instanton decomposed as depicted (with possible further breakings along the two interior cylindrical ends):
\begin{center}
	\includegraphics[scale=0.35]{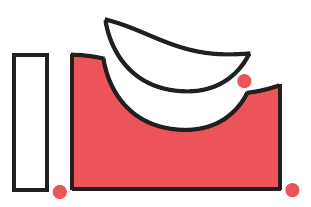}
\end{center}
Here and below, the moduli spaces for $(I\times Y\setminus B^4, \overline S''\setminus B^2 )$ and $(B^4,F_0)$ will always be defined using the metrics restricted from the metric broken along $(S^3,U_1)$. Also, $\theta$ is the unique reducible critical point on $(S^3,U_1)$. As $\alpha$ is irreducible, the (broken) instanton $A$ on the cylinder has $\text{ind}(A)\geq 1$. As remarked earlier, $(B^4, F_0 )$ supports no reducibles, and thus $\text{ind}(C)\geq 0$. Thus the index of $([A],[B],[C])$ satisfies:
\[
	\ind(A) + \dim S^1 + \ind(B) + \dim S^1 + \ind(C) \geqslant 1 + 1 -3 + 1 + 0 =0.
\]
However, the moduli space $N$ is defined using index $-1$ instantons, so that the index of the broken instanton $([A],[B],[C])$ is at most $-1$. This is a contradiction.  If the $g_{t_i}$ converge instead to the metric broken along $(Y,L')$, then the same reasoning leads to a contradiction; in this case, the only reducibles of relevance are the unobstructed ones on $T$ of index $-1$. 

Analogous arguments, involving basic convergence results and our discussion of minimal reducibles above, show that all other maps in Subsection \ref{subsec:triangleprelim} involving metric families of dimension $1$, such as $K,M_2,L$, are well-defined under the assumptions of Case I. 

Finally, consider a map such as $Q$ which is defined using the 2-dimensional metric family $G_V$. Index arguments as above preclude the possibility of the four types of reducible instantons discussed above. Reducible instantons arising from metrics that are broken along the two-component unlink $(S^3,U_2)$ may also enter, a priori; however, under our assumptions, the analysis of the situation at such broken metrics is the same as in \cite[\S 7.3]{KM:unknot}, and we refer the reader there for more details. (The only other possible reducible instantons that appear at broken metrics are unobstructed.) Here it is key that $Q$ involves a moduli space with only {\emph{irreducible}} critical points at the ends.
\end{proof}

It remains to argue that the data thus far described constitutes an exact triangle in the sense of Definition \ref{def:exacttriangle}. The relations \eqref{eq:exacttriangle1} and \eqref{eq:exacttriangle2} in this context hold because $\widetilde C$, $\widetilde C'$, $\widetilde C''$ are $\cS$-complexes and $\widetilde \lambda$, $\widetilde \lambda'$, $\widetilde \lambda''$ are morphisms, as already mentioned above.

Thus the next order of business is \eqref{eq:exacttriangle3}, i.e. the three relations
\begin{align}
	\widetilde d'' \widetilde K + \widetilde K \widetilde d + \widetilde \lambda' \widetilde \lambda = 0 \label{eq:caseirel1}\\
\widetilde d \widetilde K' + \widetilde K' \widetilde d' + \widetilde \lambda'' \widetilde \lambda' = 0 \label{eq:caseirel2}\\
\widetilde d' \widetilde K'' + \widetilde K'' \widetilde d'' + \widetilde \lambda \widetilde \lambda'' = 0 \label{eq:caseirel3}
\end{align}
First consider the equation \eqref{eq:caseirel1}. This expands to the four equations
\begin{align}
	d''K  + K d + \lambda'\lambda = 0 \label{eq:caseirel4} \\
	\delta_1'' K + M_1 d + \Delta_1'\lambda  + J\delta_1 = 0 \label{eq:caseirel5} \\
	-d'' M_2 - K\delta_2 +\lambda'\Delta_2 + \Delta_2' \eta + \delta_2'' J = 0 \label{eq:caseirel6}\\
	v'' K - d''L + \delta_2'' M_1 + L d - Kv + M_2\delta_1 +\mu'\lambda + \lambda'\mu + \Delta_2'\Delta_1 = 0 \label{eq:caseirel7}
\end{align}

\begin{figure}[t]
    \centering
    \centerline{\includegraphics[scale=0.27]{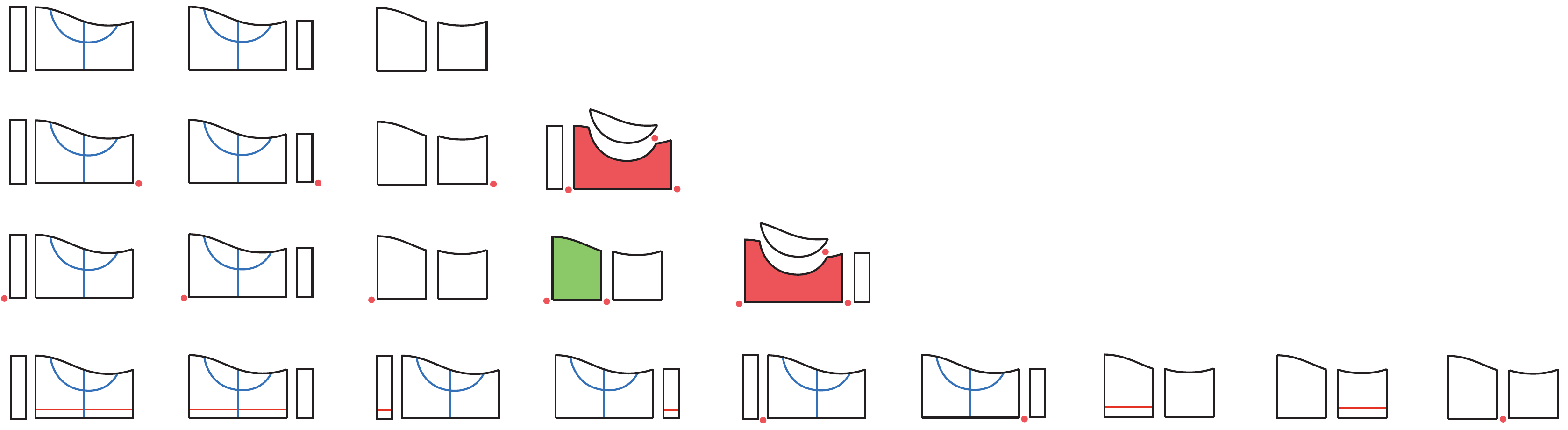}}
    \caption{{\small{Relations \eqref{eq:caseirel4}--\eqref{eq:caseirel7}, row by row.}}}
    \label{fig:caseirels1}
\end{figure}

\begin{lemma}\label{lemma:caseirels1}
	Relations \eqref{eq:caseirel4}--\eqref{eq:caseirel7} hold.
\end{lemma}

\begin{proof}
Relation \eqref{eq:caseirel4} only involves irreducible limits and no holonomy maps, and is essentially is a relation from \cite{KM:unknot}. In our notation, the argument goes as follows. Consider the campactified moduli space $M^+:=M^+(I\times Y, T;\alpha,\alpha'')^{G_T}_1$ where $\alpha$ and $\alpha''$ are irreducible. This is a compact 1-manifold and its boundary is the union of the spaces:
\begin{gather}
	M( T;\alpha,\alpha'')^{\partial G_T}_0 \label{eq:bdycomp1} \\ 
	\textstyle{\bigsqcup_{\beta\in \fC^\text{irr}_{\pi}} } \breve{M}(  L;\alpha,\beta)_0\times M( T;\beta,\alpha'')_0  \label{eq:bdycomp2} \\
		\textstyle{\bigsqcup_{\beta''\in \fC^\text{irr}_{\pi''}}} M( T;\alpha,\beta'')_0 \times \breve{M}( L'';\beta'',\alpha'')_0  \label{eq:bdycomp3}
\end{gather}
The relation is obtained by counting these boundary components. Consider the moduli space \eqref{eq:bdycomp1}. This has two parts, corresponding to the two broken metrics in $\partial G_T$. The part for the metric broken along $(Y,L')$ contributes the term $\langle \lambda'\lambda(\alpha),\alpha''\rangle$. Consider on the other hand $M(I\times Y, T;\alpha,\alpha'')^{g}_0$ where $g$ is the other metric in $\partial G_T$, broken along $(S^3,U_1)$. An element of this moduli space is a pair $([A_1],[A_2])$ where $[A_1]\in M(I\times Y\setminus B^4,\overline{S}''\setminus B^2;\alpha,\theta,\alpha'')$ and $[A_2]\in M(B^4,F_0;\theta)$ and $\theta$ is the reducible flat singular connection class on $(S^3,U_1)$. As remarked earlier, the pair $(B^4,F_0)$ supports no reducible instantons. Then we have
\[
	\ind(A_1) + \dim \text{Stab}(\theta ) + \ind(A_2) \geq 1 
\]
as each $A_i$ is irreducible and has non-negative index. On the other hand, elements of $M^+$ have index $0$. Thus there are no such contributions for the broken metric $g$. Finally, the boundary components \eqref{eq:bdycomp2}  and \eqref{eq:bdycomp3} contribute the terms $\langle Kd(\alpha),\alpha''\rangle$ and $\langle d'' K(\alpha),\alpha''\rangle$. Accounting for signs using our orientation conventions gives \eqref{eq:caseirel4}.

We next address relation \eqref{eq:caseirel5}. Consider $M^+:=M^+(I\times Y, T;\alpha,\theta_{\fo''})^{G_T}_1$ where $\alpha$ is irreducible, and $\theta_{\fo''}$ is reducible. Unlike the previous case, this moduli space is not, in general, a compact $0$-dimensional manifold. The situation is similar to that of Section \ref{sec:obstructed} and concerns the obstructed reducibles $\Theta$ on $(I\times Y\setminus B^4, \overline S''\setminus B^2 ) $ appearing at the metric broken along $(S^3,U_1)$. The statement of Proposition \ref{type-3-obs-gluing} carries over to this setting as follows. There is a continuous section $\Psi_{\alpha,\Theta,\theta_{\fo''}}$ of the complex line bundle $\cH_{\alpha,\theta_{\fo''}}\times \overline \R_+ \times \overline \R_+$ over 
\[
	\breve{M}^+(L;\alpha,\widetilde \theta_{\fo})_1 \times_{S^1} M^+(B^4,F_0;\widetilde \theta)_1 \times  \overline \R_+ \times \overline \R_+
\]
and a map $\Phi_{\alpha,\Theta,\theta_{\fo''}}:\Psi^{-1}_{\alpha,\Theta,\theta_{\fo''}}(0)\to M^+$ which is a homeomorphism onto an open neighborhood of the obstructed solutions (similar to type III solutions in Section \ref{sec:obstructed}):
\[
	\breve{M}^+(L;\alpha,\widetilde \theta_{\fo})_1 \times \{\Theta\} \times M^+(B^4,F_0;\widetilde \theta)_1 / S^1 \subset M^+
\]
The only difference between Proposition \ref{type-3-obs-gluing} and the situation here is that the first copy of $  \overline  \R_+$ above measures the extent to which an instanton is sliding off along the cylindrical end $\R\times (Y, L)$, while the second copy of $\overline \R_+$ embeds as an open neighborhood of the metric in $G_T$ which is broken along $(S^3,U_1)$, sending $\infty\in \overline \R_+$ to $g_{\overline S''\setminus B^2}\sqcup g_{F_0}$. The rest of Proposition \ref{type-3-obs-gluing} carries over to this setting in a straightforward manner. Note that while we use notation such as $\Psi_{\alpha,\Theta,\theta_{\fo''}}$ which is similar to that of Proposition \ref{type-3-obs-gluing}, the role of $\theta_{\fo''}$ here is slightly different to the role of $\alpha'$ in that proposition.

We now proceed as in Proposition \ref{bdry-rel-lambda}. Define $N^+$ to be the disjoint union of 
\begin{equation}
	M^+\setminus \textstyle{ \bigsqcup_{\Theta} }\Phi_{\alpha,\Theta,\theta_{\fo''}}\(\Psi_{\alpha,\Theta,\theta_{\fo''}}^{-1}(0)\cap(\breve M^+(L;\alpha,\widetilde \theta_{\fo})_{1} \times_{S^1} M^+(B,F_0;\widetilde \theta)_{1}\times U_\Theta)\) \label{eq:nplus1part1}
\end{equation}
\begin{equation}
	\text{ and  }\;\;\;\textstyle{ \bigsqcup_{\Theta}} \Xi_{\alpha,\Theta,\theta_{\fo''}}^{-1}(0)\cap \(\breve M^+(L;\alpha,\widetilde \theta_{\fo})_{1} \times_{S^1}M^+(B,F_0;\widetilde \theta)_{1} \times S_\Theta\times [0,1]\) \label{eq:nplus1part2}
\end{equation}
Here we use notation and constructions that follow the same pattern as in Section \ref{sec:obstructed}. The unions appearing in \eqref{eq:nplus1part1}--\eqref{eq:nplus1part2} run over the minimal reducibles $\Theta$ of index $-3$ on the pair $(I\times Y\setminus B^4, \overline S''\setminus B^2 ) $; the reducible $\Theta$ has limits $\theta_{\fo},\theta,\theta_{\fo''}$ at the three ends. The boundary of \eqref{eq:nplus1part1} consists of the following pieces:
\begin{gather}
	\textstyle{ \bigsqcup_{\alpha''\in \fC_{\pi''}^\text{irr}} }M^+(T;\alpha,\alpha'')^{G_T}_0  \times \breve M^+(L'';\alpha'',\theta_{\fo''})_0\label{eq:bdybaseiarg1}\\
	\textstyle{ \bigsqcup_{\beta\in \fC_{\pi}^\text{irr}} } \breve M^+(L;\alpha,\beta)_0  \times  M^+(T;\beta,\theta_{\fo''})^{G_T}_0 \label{eq:bdybaseiarg2}\\
	 \textstyle{ \bigsqcup_{\alpha'\in \fC_{\pi'}^\text{irr}} } M^+(S;\alpha,\alpha')_0  \times M^+(S';\alpha',\theta_{\fo''})_0 \label{eq:bdybaseiarg3} \\
	 -\textstyle{ \bigsqcup_{\Theta} } (\psi_{\alpha,\Theta,\theta_{\fo''}}')^{-1}(0)  \label{eq:bdybaseiarg4}
\end{gather}
The counts of \eqref{eq:bdybaseiarg1}--\eqref{eq:bdybaseiarg3} contribute the respective terms $\langle \delta_1'' K(\alpha),\theta_{\fo''}\rangle$, $\langle M_1 d(\alpha),\theta_{\fo''}\rangle$, $\langle \Delta_1'\lambda(\alpha),\theta_{\fo''}\rangle$ that appear in \eqref{eq:caseirel5}. As in Section \ref{sec:obstructed}, $\psi_{\alpha,\Theta,\theta_{\fo''}}'$ is the restriction of $\Xi_{\alpha,\Theta,\theta_{\fo''}}$ to $\breve M^+(L;\alpha,\widetilde \theta_{\fo})_{1} \times_{S^1}  M^+(B^4,F_0;\widetilde \theta)_{1}\times S_\Theta\times \{1\}$. The boundary of \eqref{eq:nplus1part2} is:
	\begin{equation}\label{eq:caseizerosectionboundaryinpf}
	  \bigsqcup_{\Theta} \left( (\psi_{\alpha,\Theta,\theta_{\fo''}}')^{-1}(0)\sqcup-\psi_{\alpha,\Theta,\theta_{\fo''}}^{-1}(0) \right)
	\end{equation}
The left terms in the union of \eqref{eq:caseizerosectionboundaryinpf} cancel with \eqref{eq:bdybaseiarg4}. The right terms are counted as
\[
	  -\#\breve M^+(L;\alpha,\theta_{\fo})_{0}\cdot \#M^+(B^4,F_0;\theta)_{0}
\]
for each minimal reducible $\Theta$ with limits $\theta_{\fo},\theta,\theta_{\fo''}$. Now we use that $M^+(B^4,F_0;\theta)_{0}$ contains a single instanton. This follows by removable singularities and the fact that there is a unique index $0$ instanton on $(S^4,\bR\bP^2)$ where $\bR\bP^2$ has self-intersection $+2$, when trivial singular bunde data is used. This latter point is explained in \cite[\S 2.7]{KM:unknot}. We set
\begin{equation}\label{eq:oneinstantonishere}
	\#M^+(B^4,F_0;\theta)_{0}=: \epsilon_0 \in \{1,-1\}
\end{equation}
which determines the sign $\varepsilon_0$ of Remark \ref{rmk:signforjmaps}. Thus the terms under consideration contribute the term $\langle J\delta_1(\alpha),\theta_{\fo''}\rangle$. Note that this last term is analogous to the term $\delta_2'\tau_{-1} \delta_1$ that appears in \eqref{lambda-bdry-relation}. This completes the proof of relation \eqref{eq:caseirel5}.

The proof of relation \eqref{eq:caseirel6} is analogous to that of \eqref{eq:caseirel5}. The only difference is that in relation \eqref{eq:caseirel6} there is an additional type of term, $\Delta_2'\eta$, which is handled by the usual unobstructed gluing theory.

Finally, consider \eqref{eq:caseirel7}. Let $M^+:=M^+_\gamma(T;\alpha,\alpha'')_1^{G_T}$ where $\alpha$ and $\alpha''$ are irreducible. We can rule out the appearance of obstructed reducibles, using the usual index additivity argument, which carries through because instantons in $M^+$ are of index $-1$ and $\alpha$, $\alpha''$ are irreducible. Thus $M^+$ is a compact 1-manifold with boundary. The boundary consists of:
\begin{gather*}
	\textstyle{ \bigsqcup_{\beta''\in \fC_{\pi''}^\text{irr}} } M^+(T;\alpha,\beta'')^{G_T}_0 \times \breve M_\gamma^+(L'';\beta'',\alpha'')_0  \\
	\textstyle{ \bigsqcup_{\beta''\in \fC_{\pi''}^\text{irr}} } M_\gamma^+(T;\alpha,\beta'')^{G_T}_0 \times \breve M^+(L'';\beta'',\alpha'')_0  \\
	\textstyle{ \bigsqcup_{\fo''\in \cQ(Y,L'')} } M^+(T;\alpha,\theta_{\fo''})^{G_T}_0 \times \breve M^+(L'';\theta_{\fo''},\alpha'')_0  \\
	\textstyle{ \bigsqcup_{\beta\in \fC_{\pi}^\text{irr}} } \breve M^+(L;\alpha,\beta)_0  \times M_\gamma^+(T;\beta,\alpha'')^{G_T}_0 \\ 	
\textstyle{ \bigsqcup_{\beta\in \fC_{\pi}^\text{irr}} } M_\gamma^+(L;\alpha,\beta)_0 \times  M^+(T;\beta,\alpha'')^{G_T}_0  \\
\textstyle{ \bigsqcup_{\fo\in \cQ(Y,L)} } M^+(L;\alpha,\theta_\fo)_0 \times  M^+(T;\theta_\fo,\alpha'')^{G_T}_0  \\
\textstyle{ \bigsqcup_{\beta'\in \fC_{\pi'}^\text{irr}} } M^+(S;\alpha,\beta')_0 \times  M_\gamma^+(S';\beta',\alpha'')_0  \\
\textstyle{ \bigsqcup_{\beta'\in \fC_{\pi'}^\text{irr}} } M_\gamma^+(S;\alpha,\beta')_0 \times  M^+(S';\beta',\alpha'')_0  \\
\textstyle{ \bigsqcup_{\fo'\in \cQ(Y,L')} } M^+(S;\alpha,\theta_{\fo'})_0 \times  M^+(S';\theta_{\fo'},\alpha'')_0  
\end{gather*}
The last three types of boundary components constitute $M^+_\gamma(T;\alpha,\alpha'')_0^{\partial G_T}$; the other boundary components do not involve metrics in $\partial G_T$. The details involving the modified holonomy maps are essentially the same as in \cite[Section 6]{DS1}. Accounting for the boundary terms above, in order, gives the terms in relation \eqref{eq:caseirel7}.
\end{proof}

Next, equation \eqref{eq:caseirel2} expands to the four equations:
\begin{align}
	dK'  + K' d' + \lambda''\lambda' = 0 \label{eq:caseirel8} \\
	\delta_1 K' + M'_1 d' + \Delta_1''\lambda' + \eta''\Delta_1' + J'\delta_1'  = 0 \label{eq:caseirel9} \\
	-d M'_2 - K'\delta'_2 +\lambda''\Delta'_2 + \delta_2 J' = 0 \label{eq:caseirel10}\\
	v K' - dL' + \delta_2 M'_1 + L' d' - K'v' + M'_2\delta'_1 +\mu''\lambda' + \lambda''\mu' + \Delta_2''\Delta'_1 = 0 \label{eq:caseirel11}
\end{align}
The proofs of these relations are analogous to the arguments given for Lemma \ref{lemma:caseirels1}. These relations are depicted in Figure \ref{fig:caseirels2}. Similarly, equation \eqref{eq:caseirel3} expands to the following four equations, proved in an analogous fashion:
\begin{align}
	d'K''  + K'' d'' + \lambda\lambda'' = 0 \label{eq:caseirel12} \\
	\delta'_1 K'' + M''_1 d'' + \Delta_1\lambda'' + \eta\Delta_1''   = 0 \label{eq:caseirel13} \\
	-d' M''_2 - K''\delta''_2 +\lambda\Delta''_2 + \Delta_2 \eta'' = 0 \label{eq:caseirel14}\\
	v' K'' - d'L'' + \delta_2' M''_1 + L'' d'' - K''v'' + M''_2\delta''_1 +\mu\lambda'' + \lambda\mu'' + \Delta_2\Delta''_1 = 0 \label{eq:caseirel15}\\
	\eta\eta''  = 0 \label{eq:caseirel15half}
\end{align}
Note that relation \eqref{eq:caseirel15half} follows algebraically. The other relations are depicted in Figures \ref{fig:caseirels2} and \ref{fig:caseirels3}. Observe that the verification of \eqref{eq:caseirel3} requires no obstructed gluing theory, in contrast to \eqref{eq:caseirel1} and \eqref{eq:caseirel2}.

\begin{figure}[t]
    \centering
    \centerline{\includegraphics[scale=0.27]{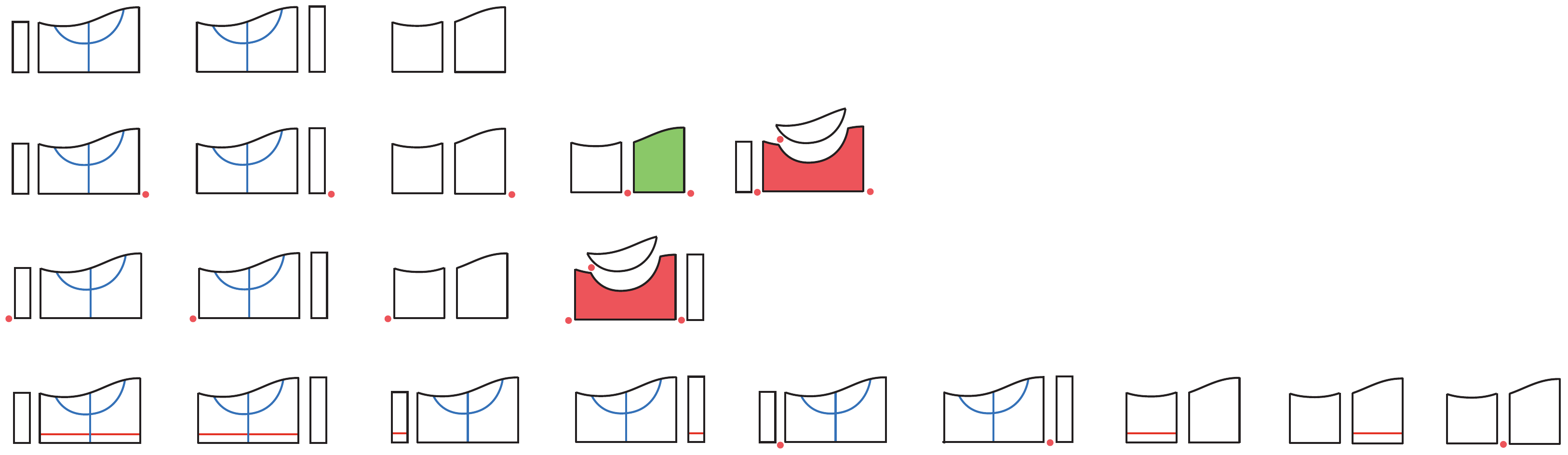}}
    \caption{{\small{Relations \eqref{eq:caseirel8}--\eqref{eq:caseirel11}.}}}
    \label{fig:caseirels2}
\end{figure}

Having established \eqref{eq:caseirel1}--\eqref{eq:caseirel3}, to prove Case I of Theorem \ref{thm:exacttriangles} it remains to show that the following three maps are homotopic to isomorphisms, as in \eqref{eq:antischainhomotopy}:
\begin{align}
	\widetilde \lambda'' \widetilde K - \widetilde K' \widetilde \lambda \label{eq:caseirel16} \\
\widetilde \lambda \widetilde K' - \widetilde K'' \widetilde \lambda' \label{eq:caseirel17} \\
\widetilde \lambda' \widetilde K'' - \widetilde K \widetilde \lambda''  \label{eq:caseirel18}
\end{align}
Note that \eqref{eq:caseirel16} is a morphism $(\widetilde C,\widetilde d)\to (\widetilde C,-\widetilde d)$, and similar remark holds for \eqref{eq:caseirel17}, \eqref{eq:caseirel18}. First consider \eqref{eq:caseirel16}. By Lemma \ref{lemma:morphismuseful}, it suffices to show that the irreducible and reducible components of this map are chain homotopic to isomorphisms. The reducible component of $\widetilde \lambda'' \widetilde K - \widetilde K' \widetilde \lambda$, i.e. the component $\sfR\to \sfR$, is computed to be 
\[
	\eta'' J - J'\eta \;\dot = \; \text{id}_{\sfR}.
\]
Thus it remains to treat the irreducible component. For this we have:

\begin{lemma}\label{lemma:caseilaststep}
	$ \lambda'' K - K' \lambda:(C,d)\to (C,-d)$ is chain homotopic to an isomorphism.
\end{lemma}

\begin{figure}[t]
    \centering
    \centerline{\includegraphics[scale=0.27]{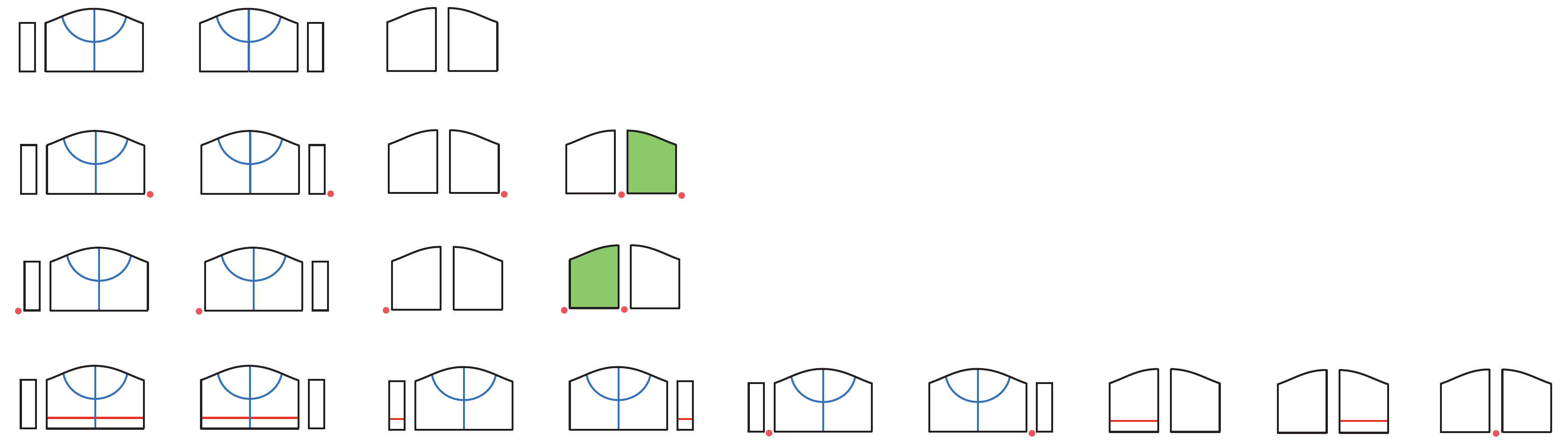}}
    \caption{{\small{Relations \eqref{eq:caseirel12}--\eqref{eq:caseirel15}.}}}
    \label{fig:caseirels3}
\end{figure}

\begin{proof}
We explain how the argument of \cite[\S 7.3]{KM:unknot} carries over to this setting (with different conventions). The argument consists of two steps. 

In the first step, the map $Q$ defined in \eqref{eq:qdefn} is shown to be a chain homotopy between $ \lambda'' K - K' \lambda$ and a map $\bN:(C,d)\to (C,-d)$. This relation,
\[
	dQ - Qd - \mathbf{N} +\lambda'' K - K'\lambda = 0,
\]
is established by considering the moduli space $M^+=M^+(I\times Y, V;\alpha,\beta)_1^{G_V}$ for irreducibles $\alpha,\beta$. There are five boundary components $\partial_i$ of $M^+$ corresponding to the five boundary components $G_i$ ($i\in \{1,2,3,4,5\}$) of the pentagon metric family $G_V$. The boundary component $G_1$ (resp. $G_2$, $G_3$, $G_4$, $G_5$) of the pentagon $G_V$ is homeomorphic to an interval and is the restriction of $G_V$ to those metrics which are broken along the hypersurface pair $(S^3,U_1')$ (resp. $(Y,L')$, $(Y,L'')$, $(S^3,U_1)$, $(S^3,U_2)$). Note $G_2$ is the metric $g_S$ on $S$ with the family $G_T'$ on $T'$, and $G_3$ is the family $G_T$ on $T$ with the metric $g_{S''}$ on $S''$. Thus
\begin{gather*}
	\partial_2 = \textstyle{ \bigsqcup_{\alpha'\in \fC_{\pi'}^\text{irr}} } M^+(S;\alpha,\alpha')_0 \times M^+(T';\alpha',\alpha'')^{G_{T'}}_0 \\
		\partial_3 = \textstyle{ \bigsqcup_{\alpha''\in \fC_{\pi''}^\text{irr}} } M^+(T;\alpha,\alpha'')^{G_T}_0 \times M^+(S'';\alpha'',\beta)_0 
\end{gather*}
Counting these contributions gives the respective terms $-\langle K' \lambda(\alpha),\beta\rangle$ and $\langle \lambda'' K(\alpha),\beta\rangle$. On the other hand, $\partial_1$ and $\partial_4$ are empty by an argument similar to that of relation \eqref{eq:caseirel4}. Counting the component $\partial_5$ defines the term $-\langle \mathbf{N}(\alpha),\beta\rangle$. There are two other types of boundary components for $M^+$ that account for instantons sliding off at the ends: 
\begin{gather*}
	\textstyle{ \bigsqcup_{\delta\in \fC_{\pi}^\text{irr}} } \breve M^+(L;\alpha,\delta)_0 \times M^+(V;\delta,\alpha'')^{G_{V}}_0 \\
		\textstyle{ \bigsqcup_{\delta\in \fC_{\pi}^\text{irr}} } M^+(V;\alpha,\delta)^{G_V}_0 \times \breve M^+(L;\delta,\beta)_0 
\end{gather*} 
These contribute the respective terms $-\langle  Q d(\alpha),\beta\rangle $ and $\langle d Q(\alpha),\beta\rangle $. Thus far we have simply recast the first step of the argument in \cite[\S 7.3]{KM:unknot} using our notation and conventions. In that reference, $M^+$ is a compact 1-manifold with boundary and the above accounts for all boundary components; then $Q$ is a chain homotopy from $ \lambda'' K - K' \lambda$ to $\mathbf{N}$.

Our situation is very similar, given that the limits $\alpha,\beta$ defining $M^+$ are assumed irreducible. However, a priori, $M^+$ may have additional contributions from the four types of reducibles depicted in \eqref{eq:minredscasei}. We must show that none of these reducibles (or reducibles of higher index) occur in the compactification $M^+$. This is done, similar as in the proof of Lemma \ref{lemma:welldefined}, using basic compactness results and index arguments. 

First, suppose there is an element $([A_0], [A_1],[A_2],[A_3])\in M^+$ in
\[
	\left( \breve{M}^+(L;\alpha,\widetilde \theta_{\fo}) \times \{[A_1]\} \times M^+(B^4,F_0;\widetilde \theta) \times M^+(S'';\widetilde \theta_{\fo''},\beta)\right) / S^1\times S^1
\]
where $A_1$ is an obstructed reducible instanton on $(I\times Y\setminus B^4, \overline S''\setminus B^2)$ with limits $ \theta_{\fo}$, $\theta$, $\theta_{\fo''}$. This type of broken instanton is depicted as follows:
\begin{center}
	\includegraphics[scale=0.35]{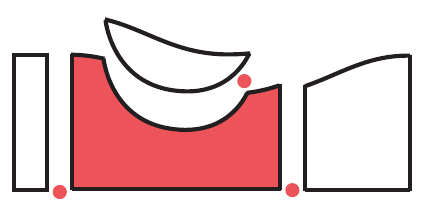}
\end{center}
 Here $S^1\times S^1$ is the product of the stabilizers of $ \theta_{\fo}$, $\theta$, $\theta_{\fo''}$ modulo the stabilizer of $A_1$. As $\alpha$ (resp. $\beta$) is irreducible, so too is $A_0$ (resp. $A_3$), and thus $\ind(A_0)\geq 1$ (resp. $\ind(A_3)\geq 0$). Note $A_0$ is on $\R\times (Y,L)$ while $A_3$ is on $(\R\times Y,S''^+)$. Also, $\ind(A_1) =  -3$, and $A_2$ is irreducible with $\ind(A_2)\geq 0$. Thus the index of $([A_0], [A_1],[A_2],[A_3])$ satisfies
\begin{align*}
	\ind(A_0) + \dim \text{Stab}(\theta_{\fo})& + \ind(A_1) + \dim \text{Stab}(\theta) \\ &+ \dim \text{Stab}(\theta_{\fo''}) + \ind(A_2) + \ind(A_3) \geq 1
\end{align*}
However, instantons in $M^+$ have index at most $-1$, and thus the above scenario is impossible. The argument easily extends to the case that $A_0,A_2,A_3$ are themselves broken instantons. Thus reducibles on $(I\times Y\setminus B^4, \overline S''\setminus B^2)$ do not occur in $M^+$. Reducibles on $(I\times Y\setminus B^4, \overline S\setminus B^2)$ are similarly ruled out, as are reducibles on $(I\times Y,S)$, $(I\times Y, S'')$. 

Finally, the second step of the argument shows that $\mathbf{N}$ is chain homotopic to an isomorphism. The argument here is exactly as in \cite[\S 7.3]{KM:unknot}. Note that the metric family $G_5$ broken along $(S^3,U_2)$ supports none of the minimal reducibles discussed above.
\end{proof}

Thus \eqref{eq:caseirel16} is homotopic to an isomorphism. The verification that the remaining two maps \eqref{eq:caseirel17} and \eqref{eq:caseirel18} are homotopic to isomorphisms follows along the same lines. First, the reducible components of each of these maps are, respectively, given by $\eta J' $ and $J\eta'' $, which are isomorphisms by \eqref{eq:caseiredrels}. Second, the argument that the irreducible components are homotopic to isomorphisms are essentially the same as in Lemma \ref{lemma:caseilaststep}. This completes the proof that the data given forms an exact triangle of $\cS$-complexes, and concludes the proof of Case I in Theorem \ref{thm:exacttriangles}.

\subsection{Case II} \label{subsec:caseii}

Now we prove Case II of Theorem \ref{thm:exacttriangles}. In this case $\epsilon(L,L')=-1$ and $\epsilon(L'',L)=1$. This means that the cobordism $S$ is obstructed, while $S'$ and $S''$ are unobstructed. We retain the notation set up in Subsection \ref{subsec:triangleprelim}, and will explain how some of the definitions there are to be modified in this case.

Recall that the suspension $\Sigma \widetilde C'=\widetilde C'_\Sigma$ of the $\cS$-complex for $(Y,L')$ is defined as:
    \begingroup
\renewcommand{\arraystretch}{1.25}
    \[
        \widetilde C'_\Sigma = C'_\Sigma \oplus C'_\Sigma[-1] \oplus \sfR', \qquad C'_\Sigma = C'[-2]\oplus \sfR'[-1],
    \]
    \[
        \widetilde d'_\Sigma = \left[\begin{array}{cc|cc|c} 
            d' & -\delta'_2 & 0 & 0 & 0 \\
            0 & 0 & 0 & 0 & 0 \\
            \hline
            v' & 0 & -d' & \delta'_2 & v'\delta'_2 \\
            \delta'_1 & 0 & 0 & 0& 0 \\
            \hline
            0 & 1 & 0 & 0 & 0
         \end{array} \right].
    \]
    \endgroup
The goal of this section is to construct an exact triangle of the form:
\begin{equation}\label{eq:exacttrianglecaseiiproof}
	\begin{tikzcd}[column sep=5ex, row sep=10ex, fill=none, /tikz/baseline=-10pt]
& \widetilde C  \arrow[dr, "\widetilde \lambda"] \arrow[dl, bend right=60, "\widetilde K"'] & \\
\widetilde C''  \arrow[rr, bend right=60, "\widetilde K''"']  \arrow[ur, "\widetilde \lambda''"] & &\widetilde C_\Sigma' \arrow[ll, "{[-1]}"', "\widetilde \lambda'"] \arrow[ul, bend right=60, "\widetilde K'"']
\end{tikzcd}
\end{equation}
The complexes are to be considered with their absolute $\Z/2$-gradings.

We begin by defining the maps $\widetilde\lambda,\widetilde\lambda',\widetilde\lambda''$. First, the map $\widetilde \lambda''$ is simply the map induced by the unobstructed cobordism $S''$ as described in Theorem \ref{thm:unobsmaps}:
\[
\widetilde\lambda'' =  \left[ \begin{array}{ccc} \lambda'' & 0 & 0 \\ \mu'' & \lambda'' & \Delta''_2 \\ \Delta''_1 & 0 & \eta'' \end{array} \right]
\]
This is an even degree morphism of $\cS$-complexes which is strong height $0$. Indeed, the map $\eta'':\sfR''\to \sfR$ is as before an inclusion, up to sign-changes, defined via \eqref{eq:reddirectsum}. Next, the map $\widetilde \lambda':\widetilde C_\Sigma'\to \widetilde C''$ is defined as follows, where $\nu'$ was defined in \eqref{eq:nuprime}:
    \begingroup
\renewcommand{\arraystretch}{1.25}
\[
\widetilde\lambda' =  \left[ \begin{array}{cc|cc|c} \lambda'  & \Delta_2' & 0 & 0 & 0  \\ \hline \mu' & 0 &  \lambda' & \Delta'_2 & \mu'\delta_2'+v''\Delta_2' + \delta_2''\nu' \\ \hline \Delta'_1 & \nu' & 0 & 0 & 0 \end{array} \right]
\]
\endgroup
This definition is compatible with the earlier described algebraic construction \eqref{eq:lambdaprimesuspend}, as applied to the morphism of $\cS$-complexes $\widetilde C'\to \widetilde C''$ and the map $\nu':\sfR'\to \sfR''$ which are determined by the odd unobstructed cobordism $S$; the resulting algebraic morphism has the domain $\cS$-complex suspended.

The map $\widetilde \lambda:\widetilde C\to \widetilde C'_\Sigma$ is defined using the height $-1$ obstructed cobordism $S$, following the constructions of Section \ref{sec:obstructed}. Recall from \eqref{eq:morphismfromheight-1} that in terms of the decomposition of $\widetilde C_\Sigma'$:
    \begingroup
\renewcommand{\arraystretch}{1.25}
\[
  \widetilde\lambda=  \left[ \begin{array}{c|c|c} \lambda & 0 & 0 \\ \eta\delta_1 & 0 & 0 \\\hline \mu & \lambda & \Delta_2 \\ \Delta_1 & \eta\delta_1 & \tau_0\\\hline0&0&\eta \end{array} \right]
\]
\endgroup
Here $\eta=\eta_{-1}=\tau_{-1}:\sfR\to \sfR'$ is as before the projection map up to sign-changes, while $\tau_0= - s'\eta - \eta s$. Recall from Section \ref{sec:obstructed} that while $\lambda$ is defined just as in the unobstructed case, the maps $\Delta_1,\Delta_2,\mu$ are defined by suitably modifying the relevant moduli spaces. That $\widetilde \lambda$ defines a morphism $\widetilde C\to \widetilde C'_\Sigma$ was proved in Section \ref{sec:obstructed}.

Although the maps $\eta$, $\eta''$, $J$, $J'$ are defined the same way as in Case I, they take on a different role here. Indeed, the hypotheses $\epsilon(L,L')=-1$ and $\epsilon(L'',L)=1$, and the index formula \eqref{eq:indminredcyl}, imply that now $\eta$ and $J$ count obstructed reducibles of index $-3$, while $\eta''$ and $J'$ count unobstructed reducibles of index $-1$. These are depicted as follows:
\begin{equation}
	\eta = \vcenter{\hbox{\includegraphics[scale=0.35]{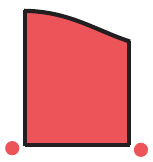} }} \qquad \quad \eta'' = \vcenter{\hbox{\includegraphics[scale=0.35]{graphics/etaprimeprime-pic.pdf} }} 
	\qquad\quad  J = \vcenter{\hbox{\includegraphics[scale=0.35]{graphics/J-obs-pic.pdf} }} \qquad\quad  J' = \vcenter{\hbox{\includegraphics[scale=0.35]{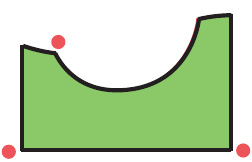} }}\label{eq:minredscaseii}
\end{equation}
Furthermore, the maps $J$ and $J'$ are defined with the same additional sign $\varepsilon_0$ of Remark \ref{rmk:signforjmaps}. 
We next turn to the homotopies $\widetilde K$, $\widetilde K'$, $\widetilde K''$. First we define $\widetilde K$:
\[
\widetilde K =  \left[ \begin{array}{ccc} K & 0 & 0 \\ L & -K & M_2 \\ M_1 & 0 & J \end{array} \right] 
\] 
The terms $K$, $M_1$ are defined just as before. However $M_2$ and $L$ require modified moduli spaces. These are constructed in the same manner as the similarly named maps in Theorem \ref{func-obs}. A minor difference in the setups is that the 1-parameter family of metrics used in that proof is from a fixed (unbroken) metric to a broken metric, while presently our family of metrics interpolates between two broken metrics. We now consider the relation:
\begin{align}
	\widetilde d'' \widetilde K + \widetilde K \widetilde d + \widetilde \lambda' \widetilde \lambda = 0 \label{eq:caseiirel1}
\end{align}
This expands into four relations given as follows:
\begin{align}
	d''K  + K d + \lambda'\lambda + \Delta_2'\eta \delta_1 = 0 \label{eq:caseiirela1} \\
	\delta_1'' K + M_1 d + \Delta_1'\lambda  + J\delta_1  + \nu' \eta \delta_1= 0 \label{eq:caseiirela2} \\
	-d'' M_2 - K\delta_2 + \delta_2'' J  +\lambda'\Delta_2 - \Delta_2' \eta s - \Delta_2' s' \eta + \mu'\delta'_2\eta + v''\Delta_2'\eta + \delta_2''\nu'\eta = 0 \label{eq:caseiirela3}\\
	v'' K - d''L + \delta_2'' M_1 + L d - Kv + M_2\delta_1 +\mu'\lambda + \lambda'\mu + \Delta_2'\Delta_1 = 0 \label{eq:caseiirela4}
\end{align}

\begin{figure}[t]
    \centering
    \centerline{\includegraphics[scale=0.27]{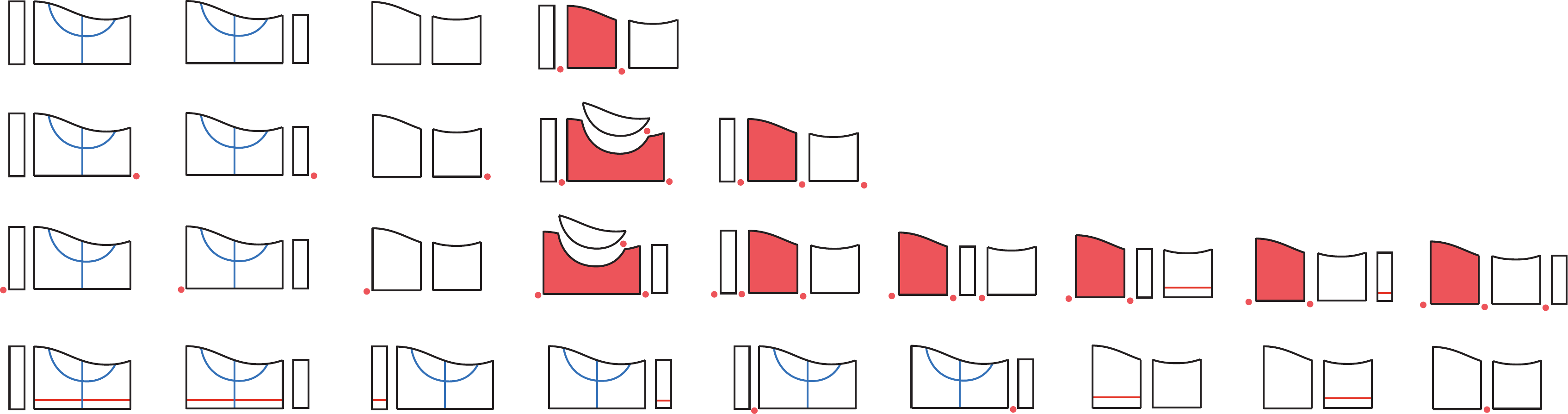}}
    \caption{{\small{Relations \eqref{eq:caseiirela1}--\eqref{eq:caseiirela4}, row by row.}}}
    \label{fig:caseiirelset1}
\end{figure}

\begin{proof}[Proof of relations \eqref{eq:caseiirela1}--\eqref{eq:caseiirela4}]
These are proved in the same way as the relations in Theorem \ref{func-obs}. We remark on the difference in the setups listed mentioned above. By index arguments, there are no contributions from moduli spaces associated to the metric broken along $(S^3,U_1)$ except for the terms $J\delta_1$ and $\delta_2'' J$ in relations \eqref{eq:caseiirela2} and \eqref{eq:caseiirela3}, respectively. These two terms are accounted for in exactly the same way as are the analogous terms in \eqref{eq:caseirel5}--\eqref{eq:caseirel6}. The terms that contribute to these relations are depicted in Figure \ref{fig:caseiirelset1}.
\end{proof}

Next, we define the following map $\widetilde K':\widetilde C_\Sigma'\to \widetilde C$:
    \begingroup
\renewcommand{\arraystretch}{1.25}
\[
\widetilde K' =  \left[ \begin{array}{cc|cc|c} K' & -M_2' & 0 & 0 & 0 \\ \hline L' & 0 & -K' & M_2' & vM_2' +L'\delta_2'  - \mu''\Delta_2' -\Delta_2''\nu' \\ \hline M_1' & 0 & 0 & 0 & J' \end{array} \right] 
\]
\endgroup
The terms $K',L',M_1',M_2'$ are defined in the usual (unobstructed) fashion, without using any modifications. We now consider the verification of the relation:
\begin{align}
\widetilde d \widetilde K' + \widetilde K' \widetilde d_\Sigma' + \widetilde \lambda'' \widetilde \lambda' = 0 \label{eq:caseiirel2}
\end{align}
The matrix entries of the relation \eqref{eq:caseiirel2} are as follows:
\begin{align}
	dK'  + K' d' + \lambda''\lambda' = 0 \label{eq:caseiirelb1} \\
	-d M'_2 - K'\delta'_2 +\lambda''\Delta'_2 = 0 \label{eq:caseiirelb2}\\
	-vM_2' -L'\delta_2'  + \mu''\Delta_2' +\Delta_2''\nu'  +vM_2' +L'\delta_2'  - \mu''\Delta_2' -\Delta_2''\nu'= 0\label{eq:caseiirelb2.5}\\
	\delta_1 K' + M'_1 d' + \Delta_1''\lambda' + \eta''\Delta_1'   = 0 \label{eq:caseiirelb3} \\
		-\delta_1M_2' -M_1'\delta_2' + J' + \Delta_1''\Delta_2' + \eta''\nu' = 0 \label{eq:caseiirelb4} \\
	v K' - dL' + \delta_2 M'_1 + L' d' - K'v' + M'_2\delta'_1 +\mu''\lambda' + \lambda''\mu' + \Delta_2''\Delta'_1 = 0 \label{eq:caseiirelb5}\\
	\delta_2 J' -K'v'\delta_2' +       	\lambda'' \mu' \delta_2' + \lambda'' v'' \Delta_2' \phantom{\lambda''\delta_2''\nu' -dvM_2' -dL'\delta_2' } \nonumber\\ + \lambda''\delta_2''\nu' -dvM_2' -dL'\delta_2'  + d\mu''\Delta_2' +d\Delta_2''\nu' = 0\label{eq:caseiirelb6}
\end{align}

\begin{figure}[t]
    \centering
    \centerline{\includegraphics[scale=0.27]{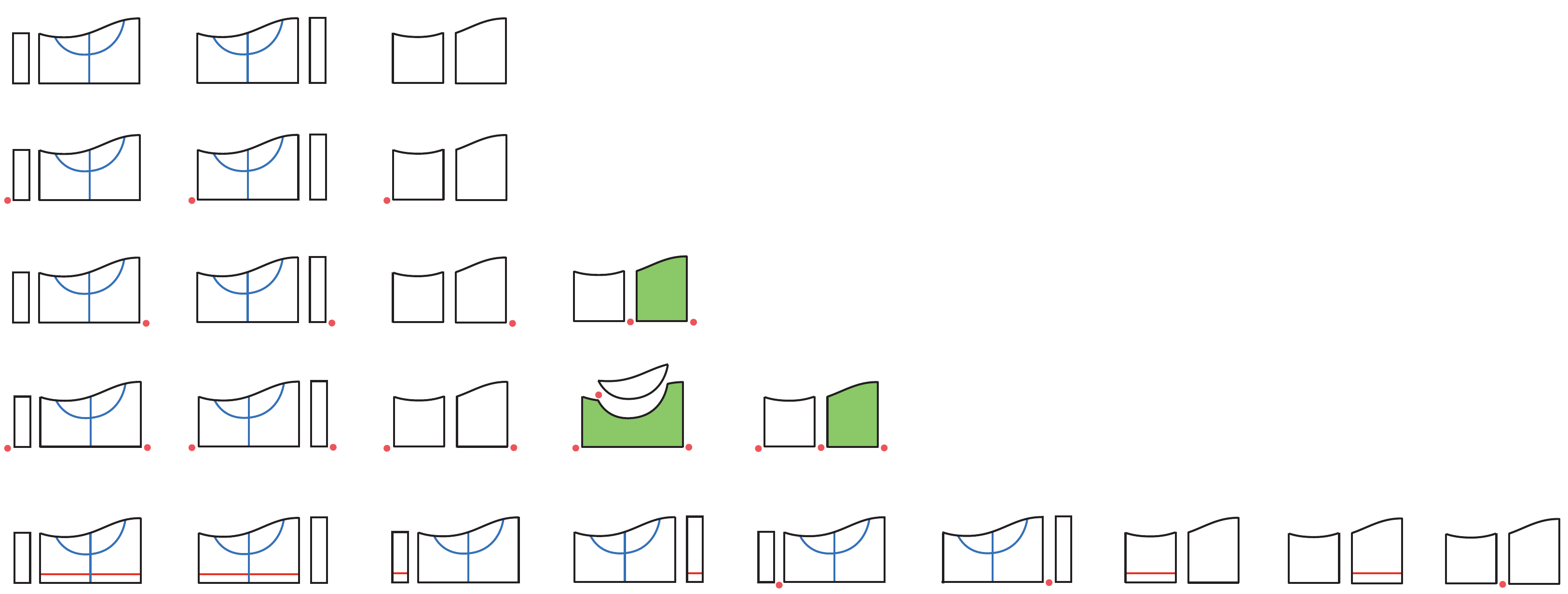}}
    \caption{{\small{Relations \eqref{eq:caseiirelb1}, \eqref{eq:caseiirelb2}, \eqref{eq:caseiirelb3}--\eqref{eq:caseiirelb5}.}}}
    \label{fig:caseiirelset2}
\end{figure}

\begin{proof}[Proof of relations \eqref{eq:caseiirelb1}--\eqref{eq:caseiirelb6}]
Relations \eqref{eq:caseiirelb1}, \eqref{eq:caseiirelb2}, \eqref{eq:caseiirelb3}--\eqref{eq:caseiirelb5} are proved in the usual fashion by looking at the relevant $1$-dimensional moduli spaces; no obstructed gluing theory is needed. The terms that contribute to these relations are depicted in Figure \ref{fig:caseiirelset2}. Relation \eqref{eq:caseiirelb2.5} is true algebraically, as the terms cancel. 

Finally, consider \eqref{eq:caseiirelb6}. We show how this relation follows algebraically from previous ones. First, pre-compose relation \eqref{eq:caseiirelb5} by $\delta_2'$ to obtain
\begin{align}
	v K'\delta_2' - dL'\delta_2' + \delta_2 M'_1\delta_2' + \cancel{L' d'\delta_2'} \phantom{................................} \nonumber \\
	- K'v'\delta_2' + \cancel{M'_2\delta'_1\delta_2'} +\mu''\lambda'\delta_2' + \lambda''\mu'\delta_2' + \Delta_2''\Delta'_1\delta_2' = 0 \label{rel:delta2ppr}
\end{align}
The two cancelled terms follow from $d'\delta_2=0$, $\delta_1'\delta_2'=0$. Now use relation \eqref{eq:caseiirelb2} to write
\[
	vK'\delta_2' = v\lambda'' \Delta_2' -vdM_2'
\]
Now use $dv-vd-\delta_2\delta_1=0$ on $vdM_2'$, and $\mu'' d'' + d\mu'' + \lambda''v''-v\lambda''+\Delta_2''\delta_1''-\delta_2\Delta_1''=0$ on the term $v\lambda'' \Delta_2'$, to obtain the relation
\begin{align*}
vK'\delta_2' &= \mu'' d''\Delta_2' + d\mu''\Delta_2' + \lambda''v''\Delta_2'+\Delta_2''\delta_1''\Delta_2'-\delta_2\Delta_1''\Delta_2' +\delta_2\delta_1M_2' -dvM_2' \\
 &=  -\mu'' \lambda'\delta_2' + d\mu''\Delta_2' + \lambda''v''\Delta_2'-\delta_2\Delta_1''\Delta_2' +\delta_2\delta_1M_2' -dvM_2'
\end{align*}
From the first line to the second we used $d''\Delta_2'+\lambda'\delta_2'=0$ on the term $\mu'' d''\Delta_2' $, and we also used $\delta_1''\Delta_2'=0$, which holds because $\widetilde \lambda'$ is an odd degree morphism, on the term $\Delta_2''\delta_1''\Delta_2'$. For a similar degree reason, the term $\Delta_2''\Delta'_1\delta_2' $ in \eqref{rel:delta2ppr} vanishes.
Substituting this last expression for $vK'\delta_2' $ into \eqref{rel:delta2ppr} we obtain
\begin{align*}
 d\mu''\Delta_2' + \lambda''v''\Delta_2'	 -\delta_2\Delta_1''\Delta_2' +\delta_2\delta_1M_2' -dvM_2' \phantom{...........} \\ - dL'\delta_2' + \delta_2 M'_1\delta_2' 
	- K'v'\delta_2' + \lambda''\mu'\delta_2'  = 0
\end{align*}
Next, use \eqref{eq:caseiirelb4} to replace $-\delta_2\Delta_1''\Delta_2'+\delta_2\delta_1M_2'+\delta_2M_1'\delta_2'$ by $\delta_2J'+\delta_2\eta''\nu'$:
\begin{align*}
 d\mu''\Delta_2' + \lambda''v''\Delta_2' +\delta_2 J' + \delta_2\eta''\nu'	 -dvM_2'  - dL'\delta_2' 
	- K'v'\delta_2' + \lambda''\mu'\delta_2'  = 0
\end{align*}
Finally, using the relation $\lambda''\delta_2'' + d\Delta_2'' -\delta_2\eta''=0$ allows us to replace $\delta_2\eta''\nu'$ with $d\Delta_2''\nu'+\lambda''\delta_2''\nu'$, which gives relation \eqref{eq:caseiirelb6}.
\end{proof}

Next, we define the following map $\widetilde K'': \widetilde C'' \to\widetilde C_\Sigma'$:
    \begingroup
\renewcommand{\arraystretch}{1.25}
\[
\widetilde K'' =  \left[ \begin{array}{c|c|c} K'' & 0 & 0 \\ -\eta\Delta_1'' & 0 & 0 \\ \hline L'' & -K'' & M''_2 \\ M_1'' & \eta\Delta_1'' & 0 \\ \hline 0 & 0 & 0 \end{array} \right] 
\]
\endgroup
Similar to the previous case, the maps $K'',L'',M_1'',M_2''$ are defined just as are the maps in Theorem \ref{func-obs}, the only difference being that here the 1-parameter family of metrics interpolates between two broken metrics. We consider the verification of the relation
\begin{align}
\widetilde d'_\Sigma \widetilde K'' + \widetilde K'' \widetilde d'' + \widetilde \lambda \widetilde \lambda'' = 0 \label{eq:caseiirel3}
\end{align}

\noindent The matrix entries of the relation \eqref{eq:caseiirel3} are as follows:
\begin{align}
	d'K''  + K'' d''+\delta_2'\eta\Delta_1'' + \lambda\lambda'' = 0 \label{eq:caseiirelc1} \\
	-\eta\Delta_1'' d'' + \eta \delta_1\lambda'' = 0 \label{eq:caseiirelc1.5}\\
	\eta\Delta_1'' - \eta\Delta_1'' =0 \label{eq:caseiirelc1.75}\\
	-d' M''_2 - K''\delta''_2 +\lambda\Delta''_2 + \Delta_2 \eta'' = 0 \label{eq:caseiirelc2}\\
	\eta\Delta_1'' v'' + \delta'_1 K'' + M''_1 d'' + \Delta_1\lambda''  +\eta\delta_1\mu'' -s'\eta\Delta_1''  -\eta s \Delta_1''= 0 \label{eq:caseiirelc3} \\
	\eta \delta_1\Delta_2'' -s'\eta\eta'' - \eta s \eta'' + \eta\Delta_1'' \delta_2'' = 0\label{eq:caseiirelc4}\\
	v' K'' - d'L'' + \delta_2' M''_1 + L'' d'' - K''v'' + M''_2\delta''_1 +\mu\lambda'' + \lambda\mu'' + \Delta_2\Delta''_1 = 0 \label{eq:caseiirelc5}\\
	\eta\eta'' = 0 \label{eq:caseiirelc5half}
\end{align}

\begin{figure}[t]
    \centering
    \centerline{\includegraphics[scale=0.27]{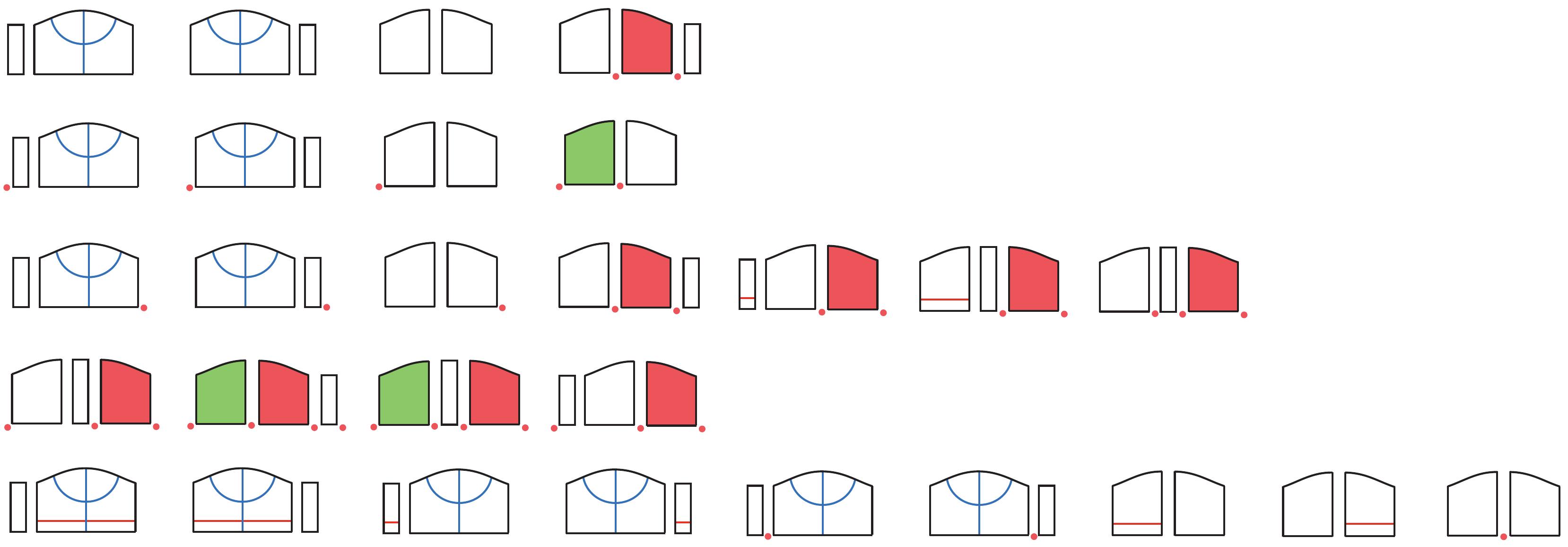}}
    \caption{{\small{Relations \eqref{eq:caseiirelc1}, \eqref{eq:caseiirelc2}--\eqref{eq:caseiirelc5}.}}}
    \label{fig:caseiirelset3}
\end{figure}

\begin{proof}[Proof of relations \eqref{eq:caseiirelc1}--\eqref{eq:caseiirelc5half}]
The relations \eqref{eq:caseiirelc1} and \eqref{eq:caseiirelc2}--\eqref{eq:caseiirelc5} are proved just as are the five relations that appear in Theorem \ref{func-obs}. These are depicted in Figure \ref{fig:caseiirelset3}. We note that there are no contributions from instantons at the metric broken along $(S^3,U''_1)$ by the usual index arguments. Relation \eqref{eq:caseiirelc1.75} is identically true, and \eqref{eq:caseiirelc5half}, which says $\eta\eta''=0$, follows from the definitions of $\eta$ and $\eta''$. Finally, \eqref{eq:caseiirelc1.5} follows by applying $\eta$ to the relation $\Delta_1''d''+\eta''\delta_1''-\delta_1\lambda''=0$ and using $\eta\eta''=0$.
\end{proof}

Finally, we show that each of the following is homotopic to an isomorphism:
\begin{align}
	\widetilde \lambda'' \widetilde K - \widetilde K' \widetilde \lambda \label{eq:caseiirel16} \\
\widetilde \lambda \widetilde K' - \widetilde K'' \widetilde \lambda' \label{eq:caseiirel17} \\
\widetilde \lambda' \widetilde K'' - \widetilde K \widetilde \lambda''  \label{eq:caseiirel18}
\end{align}
As in the situation of Case I, \eqref{eq:caseiirel16} is a morphism $(\widetilde C,\widetilde d)\to (\widetilde C,-\widetilde d)$, and similarly for \eqref{eq:caseiirel17}, \eqref{eq:caseiirel18}; the homotopies to follow should be understood in this context.
We follow the same strategy as in Case I, and show that the irreducible and reducible components of each morphism are respectively homotopic to isomorphisms, and then appeal to Lemma \ref{lemma:morphismuseful}. The reducible components of \eqref{eq:caseiirel16}, \eqref{eq:caseiirel17}, \eqref{eq:caseiirel18} are respectively
\[
	\eta'' J - J'\eta \; \dot = \;  \text{id}_{\sfR}, \qquad  \eta J'  \; \dot = \;  \text{id}_{\sfR'} , \qquad J\eta''  \; \dot = \;  \text{id}_{\sfR''}.
\]
Thus it remains to show that the irreducible components of \eqref{eq:caseiirel16}--\eqref{eq:caseiirel18} are chain homotopic to isomorphisms.

The irreducible component of \eqref{eq:caseiirel16} is given as follows:
\begin{equation}
	\lambda'' K - K' \lambda  + M_2'\delta_1\eta \label{eq:caseiiirrcomp1}
\end{equation}
At this stage, the map $Q$ enters. We remark that the map $Q$, which here is defined the same as in Case I, is well-defined by an argument which is similar to the one given in the proof of Lemma \ref{lemma:caseilaststep}. Next, we have:

\begin{lemma}\label{lemma:qfamily1}
The map $Q$ is a chain homotopy from \eqref{eq:caseiiirrcomp1} to a map $\mathbf{N}$. That is,
\begin{equation}\label{eq:caseiiqrel1}
	dQ - Qd  - \mathbf{N} + \lambda'' K- K' \lambda+ M_2'\eta \delta_1 =0
\end{equation}
Furthermore, $\mathbf{N}:(C,d)\to (C,-d)$ is chain homotopic to an isomorphism.
\end{lemma}

\begin{proof}
As in the proof of Lemma \ref{lemma:caseirels1}, this is a variation on the argument of Proposition \ref{bdry-rel-lambda}, where the term $M_2'\eta \delta_1$ appears in an analogous way to the term $\delta_2'\tau_{-1}\delta_1$. The relation comes from analyzing the ends of moduli spaces of the form $M^+(V;\alpha,\beta)^{G_V}_1$. The terms that involve no obstructed gluing theory account for all the terms in the relation apart from $M_2'\eta \delta_1$. Note that the term $\mathbf{N}$ is handled exactly as in \cite[\S 7.3]{KM:unknot}.

The gluing data relevant to Proposition \ref{type-3-obs-gluing} here adapts to give a description, in the moduli space $M^+(V;\alpha,\beta)^{G_V}_1$, of a neighborhood of obstructed reducibles $([A],[B],[C])$ where $B$ is an obstructed reducible on $S$ and $([A],[C])$ is in
\[
	\breve{M}^+(L;\alpha,\widetilde \theta_{\fo}) \times_{S^1}	M^+(T';\widetilde \theta_{\fo'},\beta)^{G_{T'}}.
\]
The first factor in $\overline \R_{+}\times \overline \R_{+}$ from Proposition \ref{type-3-obs-gluing} plays a similar role as before, while the second $\overline \R_{+}$ factor is part of the metric family. Specfically, a neighborhood of the interval $G_{T'}$ in the pentagon $G_V$ may be described as $\overline \R_{+}\times G_{T'}$ where $\{\infty\}\times G_{T'}$ corresponds to $G_{T'}$, and this is the relevant factor of $\overline \R_{+}$. With these modifications, the accounting of the term  $M_2'\eta \delta_1$ is entirely analogous to the appearance of $\delta_2'\tau_{-1}\delta_1$ in Proposition \ref{bdry-rel-lambda}.
\end{proof}

The irreducible component of \eqref{eq:caseiirel17} is a chain map $(C'_\Sigma, d'_\Sigma)\to (C'_\Sigma,-d'_\Sigma)$, which with respect to the decomposition $C'_\Sigma  = C'[-2]\oplus \sfR'[-1]$ is:
\begin{equation} \label{eq:caseiiirrcompmatrix2}
		  \left[ \begin{array}{cc}  \lambda K' - K'' \lambda'  & -\lambda M_2' - K''\Delta_2' \\[2mm] \eta\delta_1 K' + \eta\Delta_1''\lambda' & -\eta \delta_1 M_2' + \eta\Delta_1'' \Delta_2'\end{array} \right]
\end{equation}

\begin{lemma}
The map \eqref{eq:caseiiirrcompmatrix2} is chain homotopic to an isomorphism.
\end{lemma}

\begin{proof}
First, let $Q'_\Sigma:C_\Sigma'\to C_\Sigma'$ be the map defined by
	\[
		Q'_\Sigma = \left[ \begin{array}{cc} Q' & 0 \\[2mm] \eta M_1' & 0 \end{array} \right].
	\]
The expression $d_\Sigma' Q_\Sigma' - Q_\Sigma'd_\Sigma'$ added to \eqref{eq:caseiiirrcompmatrix2} is given by
\begin{equation}
	\left[ \begin{array}{cc} d'Q'  - Q' d'+\delta_2'\eta M_1' + \lambda K' - K'' \lambda'& Q'\delta_2' - \lambda M_2' - K'' \Delta_2' \\[2mm]
	 \eta M_1' d' + \eta\delta_1 K' +\eta \Delta_1'' \lambda' & -\eta M_1' \delta_2' -\eta \delta_1 M_2' +\eta \Delta_1'' \Delta_2' \end{array} \right] \label{eq:matrixqstep2}
\end{equation}
and we now analyze these terms. Analogous to \eqref{eq:caseiiqrel1}, we have
\begin{equation}
	d'Q' - Q'd'  - \mathbf{N}' +\lambda K' -  K'' \lambda' + \delta_2'\eta M_1' =0 \label{eq:caseiiqrel2}
\end{equation}
where $\mathbf{N}'$ is chain homotopic to an isomorphism. The bottom left entry of \eqref{eq:matrixqstep2} is zero, as follows from applying $\eta$ to the relation $\delta_1 K' + M_1' d' + \Delta_1'' \lambda' + \eta'' \Delta_1' =0$, and using $\eta\eta''=0$. By relation \eqref{eq:caseiirelb4}, the bottom right entry is equal to $-\eta J'  \; \dot = \;  -\text{id}_{\mathsf{R}'}$. Let $\psi$ be a chain homotopy from $\mathbf{N}'$ to an isomorphism. With the observations thus far, 
\[
\left[ \begin{array}{cc} \psi & 0\\ 0 & 0 \end{array} \right]
\]
provides a chain homotopy from  \eqref{eq:matrixqstep2} to a matrix with isomorphisms on the diagonal. 
\end{proof}

\begin{figure}[t]
    \centering
    \centerline{\includegraphics[scale=0.37]{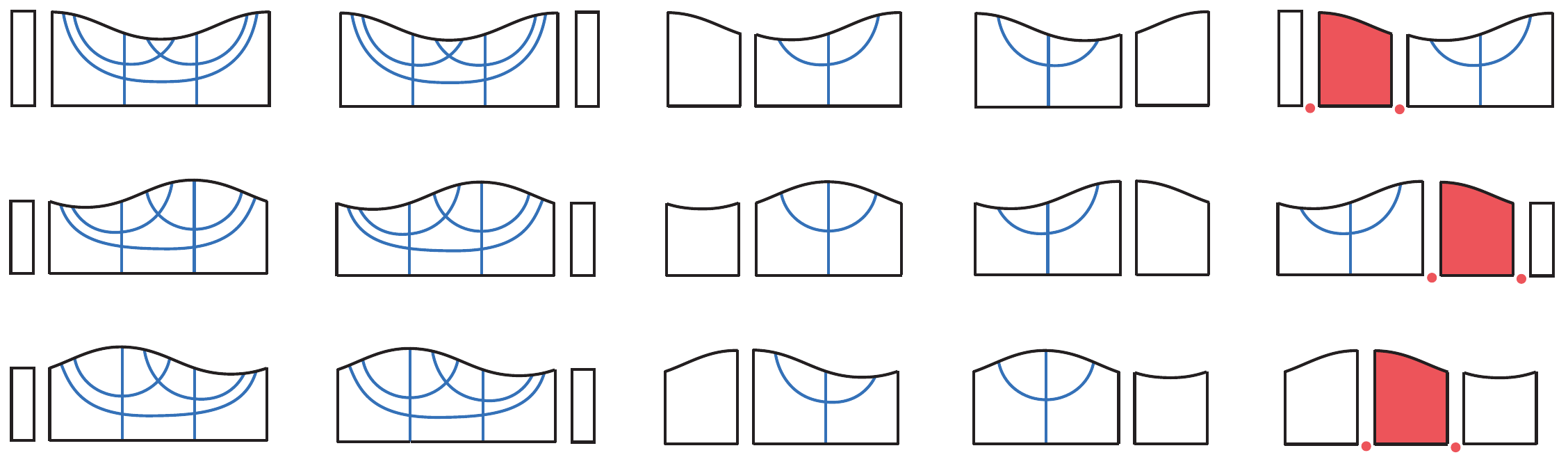}}
    \caption{{\small{Relations \eqref{eq:caseiiqrel1}, \eqref{eq:caseiiqrel2}, \eqref{eq:caseiiqrel3} without the $\mathbf{N}$, $\mathbf{N}'$, $\mathbf{N}''$ terms.}}}
    \label{fig:caseiirelset4}
\end{figure}

Finally, consider the irreducible component of the morphism \eqref{eq:caseiirel18}:
\begin{equation}
	\lambda' K'' - \Delta_2'\eta\Delta_1'' - K\lambda'' : (C'',d'') \to (C'',-d'')\label{eq:caseiiirrcomp3}
\end{equation}
The proof of the following is a minor variation of the proofs for \eqref{eq:caseiiqrel1} and \eqref{eq:caseiiqrel2}.

\begin{lemma}
The map $Q''$ is a chain homotopy from \eqref{eq:caseiiirrcomp3} to a map $\mathbf{N}''$. That is,
\begin{equation}
	d''Q'' - Q''d''  - \mathbf{N}'' +\lambda' K'' - \Delta_2'\eta\Delta_1'' - K\lambda''  =0 \label{eq:caseiiqrel3}
\end{equation}
Furthermore, $\mathbf{N}'': (C'',d'') \to (C'',-d'')$ is chain homotopic to an isomorphism.
\end{lemma}

Note that the only geometric input used that involves the 2-dimensional pentagon metric families are the three relations \eqref{eq:caseiiqrel1}, \eqref{eq:caseiiqrel2}, \eqref{eq:caseiiqrel3}. These are depicted in Figure \ref{fig:caseiirelset4}.

Thus \eqref{eq:caseiirel16}--\eqref{eq:caseiirel18} are all $\cS$-chain homotopic to isomorphisms. This completes the proof that the data \eqref{eq:exacttrianglecaseiiproof} consitutes an exact triangle of $\cS$-complexes, which is the claim of Case II in Theorem \ref{thm:exacttriangles}, and also concludes the proof of Theorem \ref{thm:exacttriangles}.

\newpage


\section{Non-trivial bundles}\label{sec:nontrivbundles}

In this section we prove Theorem \ref{thm:exacttriangles-withbundles}, the two exact triangles involving a non-trivial bundle. This is used to prove Theorem \ref{thm:eulerchar-withbundles}, which computes the Euler characteristic of Kronheimer and Mrowka's instanton homology $I^\omega(Y,L)$ for admissible links $(Y,L,\omega)$ where $Y$ is an integer homology 3-sphere and $L$ is a null-homotopic link. As a preliminary to this latter result, in the first subsection we define absolute $\Z/2$-gradings on the groups $I^\omega(Y,L)$.

\subsection{Absolute gradings for admissible links}\label{subsec:absz2gr}

Let $(Y,L,\omega)$ be an admissible link, where $Y$ is an integer homology 3-sphere. In particular, $\omega$ has an odd number of boundary points on some component of $L$. Note that there is no assumption on the determinant of the link. Fix a ground ring $R$, and choose metric and perturbation data for $(Y,L,\omega)$ similar to that of Subsection \ref{subsec:linkinv}. Recall that $C^\omega(Y,L)$ is freely generated over $R$ by $\fC^\omega_\pi$, the critical set of the $\pi$-perturbed Chern--Simons functional. This consists entirely of irreducibles by the admissibility condition. The differential $d$ counts isolated non-constant instantons as usual, and the homology is denoted
\[
	I^\omega(Y,L) = H( C^\omega(Y,L), d)
\]
This homology group is constructed by Kronheimer and Mrowka \cite{KM:YAFT,KM:unknot}. The $\cS$-complex $\widetilde C^\omega(Y,L)$, which depends on a choice of basepoint on $L$, is the mapping cone complex for the corresponding $v$-map. Any grading defined on $C^\omega(Y,L)$ determines one on $\widetilde C^\omega(Y,L)$ via \eqref{eq:mappingconeadmissible}. Thus we restrict our attention to $C^\omega(Y,L)$.

We follow the conventions of Section \ref{sec:links}, where a singular connection $A$ on a pair $(W,S)$ with boundary $(Y,L)$ implicitly involves having already attached cylindrical ends to the boundary. We may consider $\text{ind}(A)$, the index of the linearized (perturbed) ASD operator associated to $A$, defined using weighted Sobolev spaces which force exponential decay along the cylindrical ends. For such a connection on $\R\times Y$ with limits $\alpha$ at $-\infty$ and $\beta$ at $+\infty$, recall that $\text{ind}(A)=\text{gr}_z(\alpha,\beta)\in \Z$ where $z$ is the relative homotopy class of $A$, viewed as a path from $\alpha$ to $\beta$ in the configuration space of connections mod gauge. Setting 
\begin{equation}\label{eq:relz4grnontriv}
	\text{gr}(\alpha,\beta):=\text{gr}_z(\alpha,\beta)\pmod{4}
\end{equation}
defines a relative $\Z/4$-grading on $C^\omega(Y,L)$. We proceed to explain how auxiliary topological data can be used to fix an absolute $\Z/2$-grading, and also absolute $\Z/4$-grading, each compatible with \eqref{eq:relz4grnontriv}. In the sequel, only the absolute $\Z/2$-grading will be relevant.

\subsubsection{Absolute $\Z/2$-grading}

To fix an absolute $\Z/2$-grading on $C^\omega(Y,L)$, we draw inspiration from the constructions in \cite[\S 5.6]{donaldson-book} and \cite{froyshov}. Fix a critical point $\alpha \in \fC_\pi^\omega$. Choose a quasi-orientation $\fo=\{o,-o\}$ of $L$. Consider a pair $(W,S)$ with boundary $(Y,L)$. We assume the conditions $H_1(W;\Z/2)=0$ and $[S]=0\in H_2(W;\Z/2)$, which can always be arranged. Let $c\subset W$ be a surface representing singular bundle data for $(W,S)$ such that $c\cap Y=\omega$. In this situation the boundary of $c$ can be written as
\[
	\partial c = \omega \cup \partial_h c
\]
where  $\partial_h c$ is the ``horizontal'' boundary of $c$ and $\omega= \partial c \cap Y $ is the ``vertical'' boundary. Let $A$ be any singular connection on $(W,S,c)$ with limit $\alpha$ along its cylindrical end. Now let $B$ be a singular connection on $(W,S)$ with trivial singular bundle data, and having limit a reducible flat connection. Given these choices, define $\text{gr}_\fo:\fC^\omega_\pi\to \Z/2$ by:
\begin{equation}\label{eq:grmod2notrivbundle}
	\text{gr}_\fo (\alpha ) := \text{ind}(A) - \text{ind}(B)  + \partial_h c \cdot_\fo \partial_h c  \pmod{2} 
\end{equation}
The term $ \partial_h c \cdot_\fo \partial_h c \pmod{2}$ is defined as follows. Let $\gamma\subset S$ be a curve transverse and isotopic to $\partial_h c$ which at $L=\partial S$ is slightly pushed off from $\omega\cap L$ in the direction determined by the orientation $o$ of $L$. Then $ \partial_h c \cdot_\fo \partial_h c \equiv \#(\gamma \cap \partial_h c)\pmod{2}$. See Figure \ref{fig:hboundary}. As the notation suggests, this number only depends on the quasi-orientation $\fo$.

\begin{figure}[t]
    \centering
    \centerline{\includegraphics[scale=1.00]{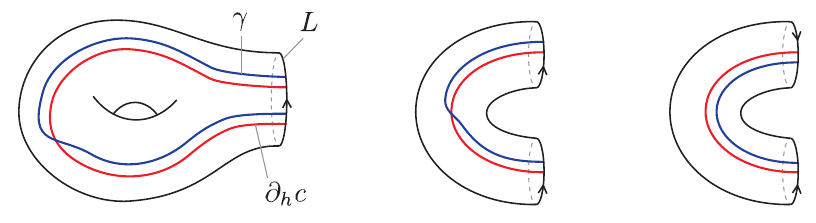}}
    \caption{{\small{The number $\partial_h c\cdot_\fo \partial_hc$ is odd in the first two illustrations, and even for the right illustration. The second and third illustrations show how $\partial_h c\cdot_\fo \partial_hc$ depends on the quasi-orientation of $L$.}}}
    \label{fig:hboundary}
\end{figure}

\begin{prop}\label{prop:mod2grnontriv}
	The quantity $\text{{\emph{gr}}}_\fo (\alpha ) \in \Z/2$ defined by \eqref{eq:grmod2notrivbundle} is independent of the choices of $(W,S,c)$ and the singular connections $A$ and $B$.
\end{prop}

\begin{proof}
Let $(W',S',c')$ be data similar to $(W,S,c)$, except that its boundary is the orientation-reversal of $(Y,L,\omega)$. Let $A'$ be a singular connection on $(W',S',c')$ with limit $\alpha$, and $B'$ a singular connection on $(W',S')$ with trivial singular bundle data with reducible limit $\theta_\fo$. Denote by $(\overline{W},\overline{S},\overline{c})$ the result of gluing $(W,S,c)$ and $(W',S',c')$ along $(Y,L,\omega)$, and write $\overline{A}$ and $\overline{B}$ for the gluings of $A$ to $A'$ and $B$ to $B'$, respectively. Then
\begin{align*}
	\text{gr}_\fo(\alpha) & + \text{ind}(A') - \text{ind}(B') + \partial_h c' \cdot_\fo \partial_h c' \\
	& = \text{ind}(A) - \text{ind}(B) + \partial_h c \cdot_\fo \partial_h c  + \text{ind}(A') -  \text{ind}(B') + \partial_h c' \cdot_\fo \partial_h c' \\
	& = \text{ind}(\overline{A}) - \text{ind}(\overline{B}) + \partial \overline{c} \cdot \partial \overline{c} - h^0(\theta_\fo) - h^1(\theta_\fo) - h^0(\alpha) - h^1(\alpha)\\
	& = \text{ind}(\overline{A}) - \text{ind}(\overline{B}) + \partial \overline{c} \cdot \partial \overline{c} -2\eta(Y,L)- 1 
\end{align*}
We have used index additivity, additivity of the intersection pairing terms, and the relations $h^0(\alpha)=h^1(\alpha)=0$, $h^0(\theta_\fo)=1$, $h^1(\theta_\fo)=2\eta(Y,L)$. Here $\partial \overline{c} \cdot \partial \overline{c}$ is shorthand for the intersection pairing on $\overline{S}$. By \eqref{eq:indclosedcase} and Lemma \ref{lemma:4kappamodz} we have:
\begin{equation}\label{eq:stepgrinv}
	\text{ind}(\overline{A}) - \text{ind}(\overline{B}) \equiv 8\kappa(\overline{A}) - 8\kappa(\overline{B}) \equiv  -\partial \overline{c} \cdot \partial \overline{c} \pmod{2}
\end{equation}
Altogether, we obtain the relation
\[
	\text{gr}_\fo(\alpha) \equiv -1-\text{ind}(A') + \text{ind}(B') - \partial_h c' \cdot_\fo \partial_h c' \pmod{2}
\]
The right side is independent of the choices $(W,S,c)$, $A$, $B$, thus the result is proved.
\end{proof}

	For quasi-orientations $\fo=\{o,-o\}$ and $\fo'=\{o',-o'\}$ of $L=\cup_{i=1}^n L_i$, and a 1-manifold $\omega$ representing singular bundle data, we define
	\begin{equation}
		\fo \cdot_\omega \fo' \equiv \sum_{k=1}^{n} \#(\partial \omega\cap L_k)\cdot (o\cdot o')_k \pmod{2}\label{eq:ooprimeomega}
	\end{equation}
where $(o\cdot o')_k$ is even if $o$ and $o'$ agree on the component $L_k\subset L$, and is odd otherwise. It is straightforward to verify that this definition depends on the given quasi-orientations.

\begin{prop}\label{prop:grchangenontriv}
	For quasi-orientations $\fo$ and $\fo'$ of $L\subset Y$, we have
	\begin{equation*}
		\text{{\emph{gr}}}_\fo - \text{{\emph{gr}}}_{\fo'} \equiv \fo \cdot_\omega \fo' \pmod{2}
	\end{equation*}
\end{prop}

\begin{proof}
	Let $\alpha\in \fC^\omega_\pi$ and choose $(W,S,c)$, $A$, $B$ as in the definition of $\text{gr}_\fo(\alpha)$. Let $B'$ be a connection on $(W,S)$ with trivial singular bundle data and reducible limit $\theta_{\fo'}$. Then
	\begin{align*}
	\text{gr}_\fo(\alpha) - \text{gr}_{\fo'}(\alpha) & \equiv \text{ind}(A) - \text{ind}(B) + \partial_h c \cdot_\fo \partial_h c  -\text{ind}(A) + \text{ind}(B') - \partial_h c \cdot_{\fo'} \partial_h c \\
							& \equiv  \partial_h c \cdot_\fo \partial_h c -  \partial_h c \cdot_{\fo'} \partial_h c \pmod{2}
\end{align*}
where we have used $\ind(B)\equiv \ind(B') \pmod{2}$, which follows from Corollary \ref{cor:indexred}.  To finish the computation, consider the case in which $o$ and $o'$ agree on $L_1$ and disagree on $L_2$. Assume further that $\omega$ contains exactly one arc, and this arc connects $L_1$ and $L_2$. Then, as illustrated in Figure \ref{fig:hboundary} (with some inessential topological simplifications), we have 
\[
	 \partial_h c \cdot_\fo \partial_h c -  \partial_h c \cdot_{\fo'} \partial_h c  \equiv 1 \pmod{2}
\]
The general case follows from this case and additivity of terms involved.
\end{proof}

The notation $\text{gr}_\fo$ defines a $\Z/2$-grading on $C^\omega(Y,L)$ and should be distinguished from our notation $\text{gr}[\fo]$ in Section \ref{sec:links} which defines, in the case $\det(Y,L)\neq 0$, a $\Z/4$-grading on $C(Y,L)$, which is an entirely different complex. However, $\text{gr}_\fo$ can be defined in this latter case as well, and agrees with the mod 2 reduction of $\text{gr}[\fo]$, for any $\fo\in \cQ(Y,L)$:

\begin{prop}
	Let $(Y,L)$ be a link in an integer homology 3-sphere with $\det \neq 0$, and $\fo \in \cQ(Y,L)$. Then the mod 2 grading $\rm{gr}_\fo:\fC_\pi\to \Z/2$ defined using the recipe \eqref{eq:grmod2notrivbundle} with $\omega=\emptyset$ agrees with the absolute mod $2$ grading defined in Subsection \ref{subsec:gradings}.
\end{prop}

\begin{proof}
Fix $\alpha\in \fC_\pi$. Let $A_0$ be a singular conection on $\R\times (Y,L)$ with limits $\alpha$ at $-\infty$ and flat reducible $\theta_\fo$ at $+\infty$. Let $(W,S)$ be a pair with boundary $(Y,L)$. Let $B$ be a connection on $(W,S)$ with trivial singular bundle data, and limit $\theta_\fo$. Then glue $A_0$ and $B$ to obtain a connection on $(W,S)$ with limit $\alpha$. We compute
\begin{align*}
	\text{gr}[\fo](\alpha) & \equiv \text{ind}(A_0 )  + h^0(\alpha) \\
				 & \equiv (\text{ind}(A) - \text{ind}(B) - h^0(\alpha))  + h^0(\alpha)\\
				  & \equiv \text{ind}(A )  -\text{ind}(B) \equiv \text{gr}_\fo(\alpha) \pmod{2}. \qedhere
\end{align*}
\end{proof}

Let $(W,S,c):(Y,L,\omega)\to (Y',L',\omega')$ be a cobordism between links. Assume that either both links are admissible, or one is admissible and the other has determinant non-zero. Furthermore, choose a path $\gamma$ in $S\subset W$ whose endpoints are basepoints $p\in L$ and $p'\in L'$. Then there is an induced morphism of $\cS$-complexes
\begin{equation}\label{eq:gencobmapnontriv}
	\widetilde \lambda = \widetilde \lambda_{(W,S,c),\gamma}: \widetilde C^\omega(Y,L,p)\to \widetilde C^{\omega'}(Y',L',p')
\end{equation} 
Indeed, the hypotheses guarantee that there exist no reducible instantons on $(W,S,c)$ for any metric, and so the morphism is constructed as in Theorem \ref{thm:unobsmaps}, setting $\rho=0$. Recall that we typically omit $\gamma$, $p$, $p'$ from notation. 

\begin{prop}\label{prop:degmorphismnontriv}
Choose quasi-orientations $\fo,\fo'$ to fix $\Z/2$-gradings of the complexes $\widetilde C^\omega(Y,L), \widetilde C^{\omega'}(Y',L')$ respectively. Then the mod $2$ degree of the map \eqref{eq:gencobmapnontriv} is given by 
    \[
         \frac{1}{2}\left( \chi(W) + \sigma(W) \right) + \chi(S) +|L| + |L'| + (\partial c|_S)\cdot_{\fo\fo'} (\partial c|_S)   \pmod{2}
    \]
    The term $(\partial c|_S)\cdot_{\fo\fo'} (\partial c|_S)$ is computed using a framing of $\partial c$ determined by pushing its boundary points along $L$ in directions determined by $o,o'$ (orientations inducing $\fo,\fo'$).
    \end{prop}
    
\noindent The proof is similar to that of Proposition \ref{prop:degmorphism}.\\

\begin{remark}
In general, if $\gamma\subset S$ has an even number of boundary points on the incoming and outgoing ends of the surface cobordism $S$, then we can define the operation $\gamma \cdot_{\fo\fo'} \gamma$ as above, and it depends only on $\gamma$ and the quasi-orientations $\fo$ and $\fo'$. However, if $\gamma$ has an odd number of boundary points on the incoming and outgoing ends of $S$, then the described operation generally depends on the choices of orientations, as the case of the cylinder $S=I\times S^1$ with $\gamma=I\times \{p\}$ shows. $\diamd$
\end{remark}

\subsubsection{Absolute $\Z/4$-grading}\label{subsec:absz4grading}

The absolute $\Z/2$-grading can be upgraded to a $\Z/4$-grading, as we now explain. This construction is only mentioned in passing; in the sequel, only the $\Z/2$-grading is used. 

Let $(Y,L,\omega)$ be an admissible link, and $\fo=\{o,-o\}$ a quasi-orientation of $L$. We define an {\emph{arc-framing}} of the quasi-oriented admissible link $(Y,L,\omega,\fo)$ to be a pushoff $\omega'$ of $\omega$ such that $\partial \omega'\subset L$ is pushed off from $\partial \omega$ along $L$ in the direction of $o$ (or $-o$). In particular, with the given boundary conditions induced by $o$, the arc $\omega'$ is equivalent to a choice of framing of $\omega$ in the standard sense, i.e. a trivialization of its normal bundle. See Figure \ref{fig:arcframing}.

\begin{figure}[t]
    \centering
    \centerline{\includegraphics[scale=0.90]{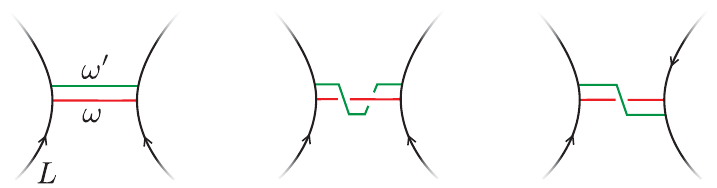}}
    \caption{{\small{Examples of arc-framings for a quasi-oriented admissible link $(Y,L,\omega)$.}}}
    \label{fig:arcframing}
\end{figure}

A quasi-oriented admissible link $(Y,L,\omega,\fo)$ with a choice of arc-framing $f$ will determine the $\Z/4$-grading, as we now explain. Let $(W,S)$ be a pair which has boundary $(Y,L)$. We choose $W$ so that it is simply-connected, and $S$ so that it is a Seifert surface for $L$ compatible with the orientation $o$, whose interior is pushed into $W$. Let $c$ be an orientable surface in $W$ representing singular bundle data, which restricts on $Y$ to the unoriented $1$-manifold $\omega$. Fix $\alpha\in \fC_\pi^\omega$. Choose a singular connection $A$ on $(W,S,c)$ with limit the irreducible connection $\alpha$, and a connection $B$ on $(W,S)$ with trivial singular bundle data and limit the reducible $\theta_\fo$. Define $\text{gr}^f_\fo:\fC^\omega_\pi\to \Z/4$ by:
\begin{equation}\label{eq:grmod2notrivbundle}
	\text{gr}^f_\fo (\alpha ) := \text{ind}(A) - \text{ind}(B)  + \widetilde c \cdot_{f} \widetilde c  \pmod{4} 
\end{equation}
The number $\widetilde c \cdot_{f} \widetilde c$ is computed as follows. Let $\widetilde c$ be the pullback of $c$ to the double branched cover $\widetilde W$ of $(W,S)$. The surface $\widetilde c$ is an orientable surface with boundary $\widetilde \omega\subset \partial \widetilde W=\widetilde Y$, the preimage of $\omega$. The arc-framing $f$ determines a framing, in the usual sense, of $\widetilde \omega$. Then $\widetilde c \cdot_{f} \widetilde c$ is the self-intersection of $\widetilde{c}$ computed with respect to this framing.

The proof that $\text{gr}^f_\fo$ is well-defined is similar to that of Proposition \ref{prop:mod2grnontriv}. At the step \eqref{eq:stepgrinv}, we use that a singular connection $A$ on the closed pair $(\overline{W},\overline{S})$ with singular bundle data $\overline{c}$ satisfies $4\kappa(\overline{A}) \equiv -\frac{1}{2}\gamma\cdot \gamma \pmod{2}$, which follows from \eqref{eq:4kappaupgrade} and the topological assumptions. Here $\gamma$ is the pullback of $\overline{c}$ to the double branched cover of $(\overline{W},\overline{S})$.

Clearly $\text{gr}^f_\fo$ reduces modulo $2$ to the grading $\text{gr}_\fo$ defined previously. Furthermore, if the arc-framing is changed by one twist, as for example in the two left-most illustrations of Figure \ref{fig:arcframing}, then $\text{gr}^f_\fo$ changes by an overall shift of $2\pmod{4}$.

\subsection{Setting up the exact triangles}\label{subsec:nontrivbundlesetup}
We now set the stage for proving Theorem \ref{thm:exacttriangles-withbundles}. In this subsection and the two that follow, we ignore gradings. We will return to this issue in Subsection \ref{subsec:euler}.

Let $L,L',L''$ be an unoriented skein triple in an integer homology 3-sphere $Y$, where $L$ has one more component than $L'$ and $L''$. There are two components of $L$ that intersect the skein 3-ball, and we let $\omega$ be an arc connecting these two components. We can arrange that $\omega$ is outside of the skein 3-ball. As the links are identified outside this ball, we may also view $\omega$ as an arc with endpoints on $L'$ and $L''$, respectively. See Figure \ref{fig:nontrivbundlessetup}.

\begin{figure}[t]
    \centering
    \centerline{\includegraphics[scale=0.6]{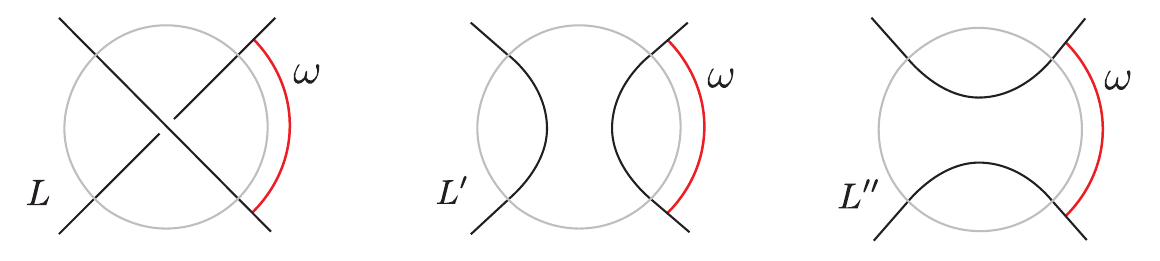}}
    \caption{{\small{The skein triple with uniform singular bundle data determined by an arc $\omega$.}}}
    \label{fig:nontrivbundlessetup}
\end{figure}

Choose metric and perturbation data, and form the $\cS$-complexes
\begin{equation}\label{eq:complexesinnontrivbundlecase}
    \widetilde C = \widetilde C^\omega(Y,L), \qquad \widetilde C' = \widetilde C^\omega(Y,L'), \qquad \widetilde C'' = \widetilde C^\omega(Y,L'').
\end{equation}
Note that $\omega$ has endpoints on a single component of $L'$, and also for $L''$, and thus in these two cases the singular bundle data determined by $\omega$ is topologically trivial, in contrast to the case for $L$; indeed, the isomorphism class of bundle data for $(Y,L')$, for example, is determined by the class of $\omega$ in $H_1(Y,L';\Z/2)$, which is zero. We assume that the links $L'$ and $L''$ have non-zero determinant, while there is no restriction on $L$.

All of the cobordism constructions from Subsection \ref{subsec:triangleprelim} are equipped with the singular bundle data determined by $I\times \omega$. This is sensible, as all the cobordisms are products outside of the skein region. See for example Figure \ref{fig:cobwbundle}. This choice of singular bundle data is uniform throughout this discussion and will often be omitted from notation.

The other key difference between the current situation and the setup in Subsection \ref{subsec:triangleprelim} is the accounting of reducibles. First, as the link $(Y,L,\omega)$ is admissible, it admits no reducible critical points. The differential of $\widetilde C=\widetilde C^\omega(Y,L)$ has the simpler form 
\begin{equation}
\widetilde d =  \left[ \begin{array}{cc} d & 0\\ v & -d  \end{array} \right] \label{eq:tildednontrivbun}
\end{equation}
The $\cS$-complexes $\widetilde C'$ and $\widetilde C''$ have the usual differentials
\begin{equation}
\widetilde d' =  \left[ \begin{array}{ccc} d' & 0 & 0 \\ v' & -d' & \delta'_2 \\ \delta'_1 & 0 & 0 \end{array} \right]
\qquad
\widetilde d'' =  \left[ \begin{array}{ccc} d'' & 0 & 0 \\ v'' & -d'' & \delta''_2 \\ \delta''_1 & 0 & 0 \end{array} \right]\label{eq:tildedppnontrivbun}
\end{equation}
While the reducible summands of the complexes $\widetilde C'$ and $\widetilde C''$ will be isomorphic to the free $R$-modules $\mathsf{R}'=\mathsf{R}(Y,L')$ and $\mathsf{R}''=\mathsf{R}(Y,L'')$ generated by quasi-orientations, these isomorphisms are not canonical. We now explain this point.

Let the components of $L'$ be written $L'_1,\ldots, L'_n$, and assume $\omega$ is an arc with endpoints on $L'_1$. Then $L'_1\setminus \partial \omega$ is two arcs, $A'_1,A'_2$. Denote by $\mathsf{O}(Y,L',\omega)$ the set of orientations of $L'\setminus \omega = A'_1\cup A'_2 \cup L_2'\cup \cdots L_n'$ such that the orientations of the two arcs $A'_1$ and $A'_2$ are opposite where they meet at their boundary points.

\begin{figure}[t]
    \centering
    \centerline{\includegraphics[scale=0.65]{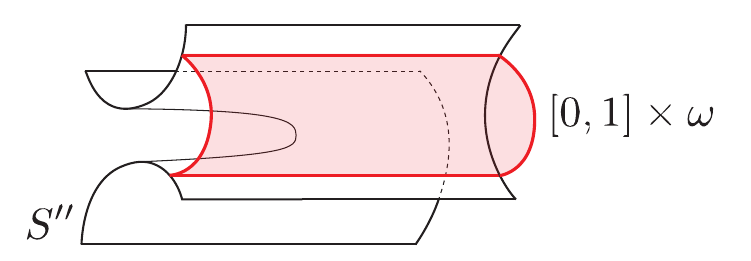}}
    \caption{{\small{The reverse of the saddle cobordism $S'':L''\to L$ with the product singular bundle data.}}}
    \label{fig:cobwbundle}
\end{figure}

The flat reducible singular connections on $(Y,L')$, mod gauge, are in natural bijection with conjugacy classes of representations $\pi_1(Y\setminus (L\cup \omega))\to SU(2)$ with abelian image and which send meridians of $\omega$ to $-1\in SU(2)$. These reducibles are in natural bijection with 
\begin{equation}\label{eq:orientationsetnontriv}
	\mathsf{O}(Y,L',\omega)/\pm
\end{equation}
where $\pm$ is the action which reverses all orientations. The correspondence is as follows. Given an orientation $\fo\in\mathsf{O}(Y,L',\omega)$, we obtain distinguished meridians around each of $A'_1,A'_2,L_2',\ldots, L_n'$ by the right-hand rule. Each of these meridians is sent to $i\in SU(2)$, and of course any meridian around $\omega$ is sent to $-1$. There is a remaining conjugation symmetry, say by $j\in SU(2)$, which maps this represention to the one determined by $-\fo$. We write $\theta_\fo$ for the reducible flat singular connection which corresponds to $\fo$.

The reducible summand $\mathsf{R}'$ of the $\cS$-complex $\widetilde C'=\widetilde C(Y,L',\omega)$ may then be identified with $\mathsf{R}(Y,L',\omega)$, the free $R$-module with generating set $\mathsf{O}(Y,L',\omega)/\pm$:
\[
    \sfR' = \sfR(Y,L',\omega) = \bigoplus_{\qo \in \mathsf{O}(Y,L',\omega)/\pm } R\cdot \theta_\qo  
\]
 The reducible summand $\mathsf{R}''$ of the $\cS$-complex $\widetilde C''$ for $L''$ is similarly defined.

Now we turn to the cobordism $S':L'\to L''$. As this is non-orientable, it admits no flat reducibles for the trivial singular bundle data. However, with the bundle data $I\times \omega$, for each reducible on $(Y,L',\omega)$ (resp. on $(Y,L'',\omega)$) there is a unique flat reducible extension over $(I\times Y,S',I\times \omega)$. This is of course a statement about fundamental groups, in fact homology; the key point is that $S'\setminus I\times \omega$ is orientable. 

\begin{remark}
	If $L'$ and $L''$ are unknots, the unique flat reducible instanton above is a restriction of the flat reducible instanton $A_+^1$ on $(S^4,\bR\bP^2_+)$, where $\bR\bP^2_+$ has self-intersection $+2$. This is an example from \cite[\S 2.7]{KM:unknot} and was discussed in Subsection \ref{sec:unobscobmaps}. In general, our choice of singular bundle data on $S'$ coincides, on the non-orientable part, with the description in \cite[Section 6]{DS2}. A key point is that the reducibles discussed above, at least at the level of adjoint connections, are the ones which can be obtained as orbifold quotients from the trivial adjoint connection on the double branched cover. $\diamd$
\end{remark}

\begin{figure}[t]
    \centering
    \centerline{\includegraphics[scale=0.33]{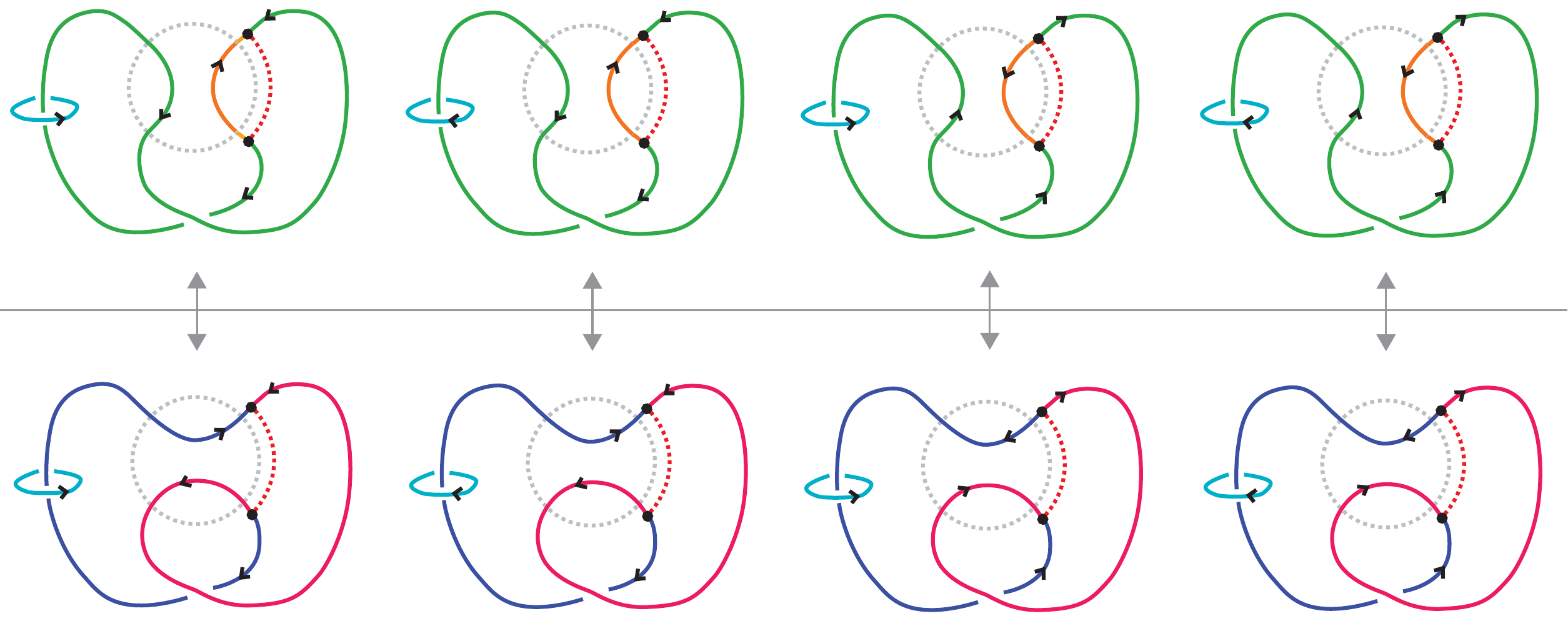}}
    \caption{{\small{In the top row is depicted a $2$-component link with an arc $\omega$ in dotted red which represents singular bundle data. This link has four orientations up to equivalence in the set \eqref{eq:orientationsetnontriv}, which are all displayed. The link of the bottom row is obtained by performing a skein resolution in the 3-ball indicated by the dotted circle. This resolution, and hence the associated saddle cobordism, is non-orientable. The bijection between the orienation sets \eqref{eq:nontrivbundorbij} is depicted.}}}
    \label{fig:orbijex}
\end{figure}

Thus the map induced by counting minimal reducibles on $(I\times Y,S',I\times \omega)$,
\[
	\eta':\mathsf{R}'\longrightarrow \mathsf{R}'',
\]
is an isomorphism. To describe this in terms of orientations, we have a natural bijection 
\begin{equation}\label{eq:nontrivbundorbij}
	\mathsf{O}(Y,L',\omega) \cong \mathsf{O}(Y,L'',\omega)
\end{equation}
which is as follows. Write $L'\setminus \omega = A'_1\cup A'_2 \cup L_2'\cup \cdots L_n'$ and $L''\setminus \omega = A''_1\cup A''_2 \cup L_2''\cup \cdots L_n''$ where $L_i''=L_i'$ for $i\geq 2$. Now suppose $L'\setminus \omega$ is oriented in such a way that $A_1'$ and $A_2'$ are oppositely oriented where they meet. Since $L'\setminus \omega$ and $L''\setminus \omega$ agree outside of the skein 3-ball, $L''\setminus \omega$ inherits an orientation from $L'\setminus \omega$, and the resulting orientation of $L''\setminus \omega$ has the property that $A_1''$ and $A_2''$ are oriented oppositely where they meet. This defines the bijection \eqref{eq:nontrivbundorbij}. An example of the bijection is given in Figure \ref{fig:orbijex}. The map $\eta'$ is then induced by the bijection \eqref{eq:nontrivbundorbij} up to sign-changes in the natural bases. 

Choosing one of the arcs, say $A_1'\subset L'\setminus \omega$, induces a bijection between $\mathsf{O}(Y,L',\omega)$ and orientations of $L'$, and hence an isomorphism $\sfR(Y,L',\omega)\cong \sfR(Y,L')$. Choosing the other arc $A_2'$ induces another bijection which is distinct from the previous one if $|L'|\geq 2$. Topologically, the first bijection is induced by a null-homotopy of $\omega$ that moves its two boundary points together to collapse the arc $A_2'$; the other bijection collapses $A_1'$.

We now turn to compute the indices of the minimal reducibles described above. For the purposes of these computations, we may use the formula of Proposition \ref{prop:redind}. Indeed, the choice of singular bundle data $I\times \omega$ on $(I\times Y, S')$ is such that any reducible on this cobordism is the quotient of the trivial connection on the double branched cover. In particular, the arguments of \cite[Lemma 11]{DS2} and Proposition \ref{prop:redind} adapt.  

Let $A'$ be any of the reducibles on $(I\times Y, S')$ from above, with limits $\theta_{\fo'}$ and $\theta_{\fo''}$. Applying the formula of Proposition \ref{prop:redind}, noting that $A'$ is flat and $\chi(S)=-1$, we obtain
\begin{equation}
	\text{ind}(A') = \frac{1}{2}S'\cdot S' + \sigma(Y,L',\fo') - \sigma(Y,L'',\fo'') -2. \label{eq:indnontrivsprimefirst}
\end{equation}
Write $X$, $X'$, $X''$ for the double covers of $I\times Y$ branched over $S$, $S'$, $S''$, respectively. A straightforward extension of \cite[Theorem 3.1]{kauffman-taylor} to this setting shows that 
\begin{equation}
	\sigma(X') = -\frac{1}{2}S'\cdot S' - \sigma(Y,L',\fo') + \sigma(Y,L'',\fo''). \label{eq:indnontrivsprimesign}
\end{equation}
We also recall from relation \eqref{eq:sumofthreesignatures} that the following holds:
\begin{equation}
	\sigma(X) + \sigma(X') + \sigma(X'') = -1. \label{eq:indnontrivsprimesignsum}
\end{equation}
Relations \eqref{eq:indnontrivsprimefirst}, \eqref{eq:indnontrivsprimesign} and \eqref{eq:indnontrivsprimesignsum}, combined with the identities $\sigma(X)=-\varepsilon(L,L')$ and $\sigma(X'')=-\varepsilon(L'',L)$ from Lemma \ref{prop:indbranchedcover}, give
\begin{equation}
	\text{ind}(A') = -1 - \epsilon(L,L') - \epsilon(L'',L).  \label{eq:nontrivbundleindreds}
\end{equation}

There is one other type of reducible singular instanton relevant for what follows. Consider the double composite cobordism $(I\times Y, T'' )$ with the projective plane $(B^4, F_0)$ removed; this is identified with $(I\times Y\setminus B^4, \overline S'\setminus B^2 ) $. The index of a reducible $A''$ on this cobordism is the same as the index of such a reducible on $\overline S' :L''\to L'$. The computation is similar to that of \eqref{eq:nontrivbundleindreds}, except that the orientation of the cobordism is reversed, so the signatures all change sign. The resulting formula is as follows:
\begin{equation}
	\text{ind}(A'')  = -3 + \epsilon(L,L') + \epsilon(L'',L)  \label{eq:nontrivbundleindreds-2}
\end{equation} 
Let $J'':\sfR''\to \sfR'$ count the minimal reducibles on $(I\times Y\setminus B^4, \overline S'\setminus B^2 ) $. Note that
\[
		\eta' J'' \; \dot = \; \text{id}_{\mathsf{R}''} \qquad \quad J'' \eta'  \; \dot = \; \text{id}_{\mathsf{R}'}
\]
where we remind the reader that $\dot =$ means ``up to sign-changes'' in the natural bases.

Recall that Theorem \ref{thm:exacttriangles-withbundles} is divided into two cases depending on the value of 
\[
	 \epsilon(L,L') + \epsilon(L'',L) \in \{0,2\}
\]
That these are the only possible values follows a similar argument as in Lemma \ref{lemma:casesexhaust}. The case of value $0$ is Case A, and that of value $2$ is Case B. With these preliminaries out of the way, we now turn to construct the exact triangles of Theorem \ref{thm:exacttriangles-withbundles}. Gradings will be omitted from the proofs given below, and revisted in Subsection \ref{subsec:euler}. Auxiliary data such as metrics, perturbations and I-orientations are fixed in a way similar to what was done in Section \ref{sec:proofs}.

\subsection{Case A}

In Case A of Theorem \ref{thm:exacttriangles-withbundles}, we assume $\epsilon(L,L') + \epsilon(L'',L)=0$. Recall that the isomorphism $\eta':\mathsf{R}'\to \mathsf{R}''$ counts minimal reducibles on $S'$, which by \eqref{eq:nontrivbundleindreds} are of index $-1$ and hence unobstructed. The map $J'':\mathsf{R}''\to \mathsf{R}'$ counts minimal reducibles on the cobordism $T''=S\circ S''$ after removing the copy of $\bR\bP^2$, which as discussed above may be identified with the reversed cobordism of $S'$ with a disk removed. (The specific sign convention for $J''$ follows a similar convention as for $J,J'$ in Remark \ref{rmk:signforjmaps}.) In Case A, by \eqref{eq:nontrivbundleindreds-2} the reducibles that $J''$ counts are of index $-3$ and obstructed. These maps are depicted as:

\begin{equation}
	\eta' = \vcenter{\hbox{\includegraphics[scale=0.35]{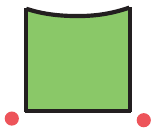} }} \qquad \quad J'' = \vcenter{\hbox{\includegraphics[scale=0.35]{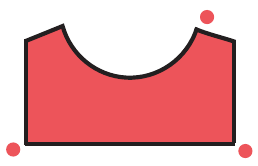} }} 
\end{equation}

\begin{figure}[t]
    \centering
    \centerline{\includegraphics[scale=0.27]{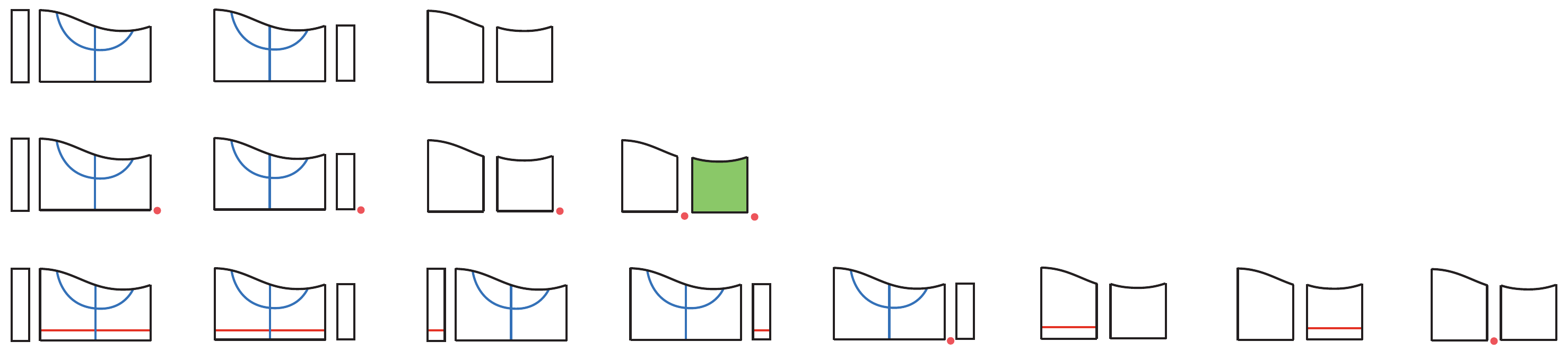}}
    \caption{{\small{Relations \eqref{eq:casearelc1}--\eqref{eq:casearelc3}.}}}
    \label{fig:casearelset1}
\end{figure}

We now proceed to construct an exact triangle of the form
\begin{equation*}\label{eq:exacttrianglecaseaproof}
	\begin{tikzcd}[column sep=5ex, row sep=10ex, fill=none, /tikz/baseline=-10pt]
& \widetilde C  \arrow[dr, "\widetilde \lambda"] \arrow[dl, bend right=60, "\widetilde K"'] & \\
\widetilde C''  \arrow[rr, bend right=60, "\widetilde K''"']  \arrow[ur, "\widetilde \lambda''"] & &\widetilde C' \arrow[ll, "\widetilde \lambda'"] \arrow[ul, bend right=60, "\widetilde K'"']
\end{tikzcd}
\end{equation*}
where the complexes $\widetilde C,\widetilde C',\widetilde C''$ are given in \eqref{eq:complexesinnontrivbundlecase}. First, we define $\widetilde \lambda, \widetilde\lambda', \widetilde \lambda''$:
\[
\widetilde\lambda =  \left[ \begin{array}{cc} \lambda & 0  \\ \mu & \lambda \\ \Delta_1 & 0  \end{array} \right] 
\qquad
\widetilde\lambda' =  \left[ \begin{array}{ccc} \lambda' & 0 & 0 \\ \mu' & \lambda' & \Delta'_2 \\ \Delta'_1 & 0 & \eta' \end{array} \right]
\qquad
\widetilde\lambda'' =  \left[ \begin{array}{ccc} \lambda'' & 0 & 0 \\ \mu'' & \lambda'' & \Delta''_2 \end{array} \right]
\]
The components are defined just as in Subsection \ref{subsec:triangleprelim}, but with the modifications for non-trivial bundles as described Subsection \ref{subsec:nontrivbundlesetup}. 

Similarly, we define $\widetilde K,\widetilde K',\widetilde K''$:
\[
\widetilde K =  \left[ \begin{array}{cc} K & 0  \\ L & -K  \\ M_1 & 0  \end{array} \right] 
\quad
\widetilde K' =  \left[ \begin{array}{ccc} K' & 0 & 0 \\ L' & -K' & M'_2 \end{array} \right]
\quad
\widetilde K'' =  \left[ \begin{array}{ccc} K'' & 0 & 0 \\ L'' & -K'' & M''_2 \\ M''_1 & 0 & J'' \end{array} \right]
\]

With this setup, the details of the proof run parallel to the proof of Case I of Theorem \ref{thm:exacttriangles}, but with fewer reducibles. We give the necessary ingredients.

The relation $\widetilde d'' \widetilde K + \widetilde K \widetilde d + \widetilde \lambda' \widetilde \lambda=0$ is verified by the following:
\begin{align}
	d''K  + K d + \lambda'\lambda = 0 \label{eq:casearelc1} \\
	\delta_1'' K + M_1 d + \Delta_1'\lambda  + \eta' \Delta_1= 0  \label{eq:casearelc2} \\
	v'' K - d''L + \delta_2'' M_1 + L d - Kv +\mu'\lambda + \lambda'\mu + \Delta_2'\Delta_1 = 0 \label{eq:casearelc3}
\end{align}
The arguments for these proceed, in the usual way, by studying the ends of certain 1-dimensional moduli spaces, and involve only unobstructed gluing theory. Similarly, the relation $\widetilde d \widetilde K' + \widetilde K' \widetilde d' + \widetilde \lambda'' \widetilde \lambda'=0$ follows from the relations:
\begin{align}
	dK'  + K' d' + \lambda''\lambda' = 0  \label{eq:casearelc4} \\
	-d M'_2 - K'\delta'_2 +\lambda''\Delta'_2 + \Delta_2 \eta' = 0  \label{eq:casearelc5}\\
	v K' - dL'  + L' d' - K'v' + M'_2\delta'_1 +\mu''\lambda' + \lambda''\mu' + \Delta_2''\Delta'_1 = 0  \label{eq:casearelc6}
\end{align}
Next, the equation $\widetilde d' \widetilde K'' + \widetilde K'' \widetilde d'' + \widetilde \lambda \widetilde \lambda''=0$ follows from:
\begin{align}
	d'K''  + K'' d'' + \lambda\lambda'' = 0  \label{eq:casearelc7} \\
	-d' M''_2 - K''\delta''_2 +\lambda\Delta''_2 +\delta_2' J''  = 0  \label{eq:casearelc8}\\
	\delta'_1 K'' + M''_1 d'' + \Delta_1\lambda''   + J''\delta_1''= 0  \label{eq:casearelc9} \\
	v' K'' - d'L'' + \delta_2' M''_1 + L'' d'' - K''v'' + M''_2\delta''_1 +\mu\lambda'' + \lambda\mu''  = 0 \label{eq:casearelc10}
\end{align}
We remark that relations \eqref{eq:casearelc8} and \eqref{eq:casearelc9} involve obstructed gluing theory; the arguments are essentially the same as that of relation \eqref{eq:caseirel5}. All of the relations listed above are depicted in Figures \ref{fig:casearelset1}, \ref{fig:casearelset2} and \ref{fig:casearelset3}.

\begin{figure}[t]
    \centering
    \centerline{\includegraphics[scale=0.27]{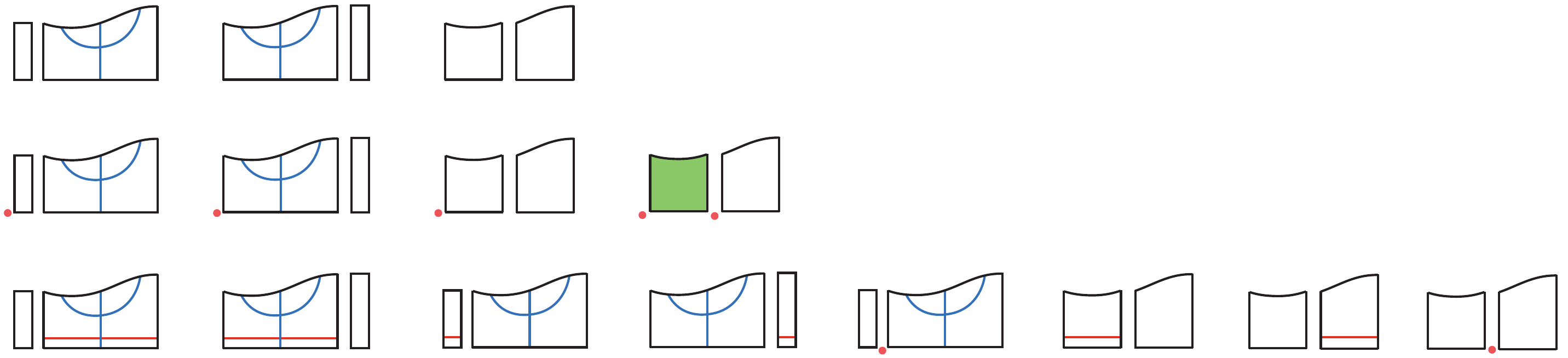}}
    \caption{{\small{Relations \eqref{eq:casearelc4}--\eqref{eq:casearelc6}.}}}
    \label{fig:casearelset2}
\end{figure}

To complete Case A we must show that the expressions \eqref{eq:caseirel1}--\eqref{eq:caseirel3}, as defined here, are each homotopic to an isomorphism. Just as in the proofs of Section \ref{sec:proofs} which used Lemma \ref{lemma:morphismuseful}, this is done by showing that the irreducible and reducible parts of these morphisms are homotopic to isomorphisms.

The irreducible part of \eqref{eq:caseirel1} is $d''  K + K  d + \lambda'\lambda$. This is homotopic to an isomorphism using the pentagonal metric family; no reducibles appear. The argument is similar to that of Lemma \ref{lemma:caseilaststep}. The irreducible parts of \eqref{eq:caseirel2} and \eqref{eq:caseirel3} are essentially the same.

There is no reducible part of \eqref{eq:caseirel1}, as it is a morphism defined on $\widetilde C=\widetilde C^\omega(Y,L)$. The reducible parts of \eqref{eq:caseirel2} and \eqref{eq:caseirel3} are $J'' \eta'  \; \dot = \; \text{id}_{\mathsf{R}'}$ and $\eta' J'' \; \dot = \; \text{id}_{\mathsf{R}''}$, respectively. This completes the proof of Theorem \ref{thm:exacttriangles-withbundles} in Case A.

\begin{figure}[t]
    \centering
    \centerline{\includegraphics[scale=0.27]{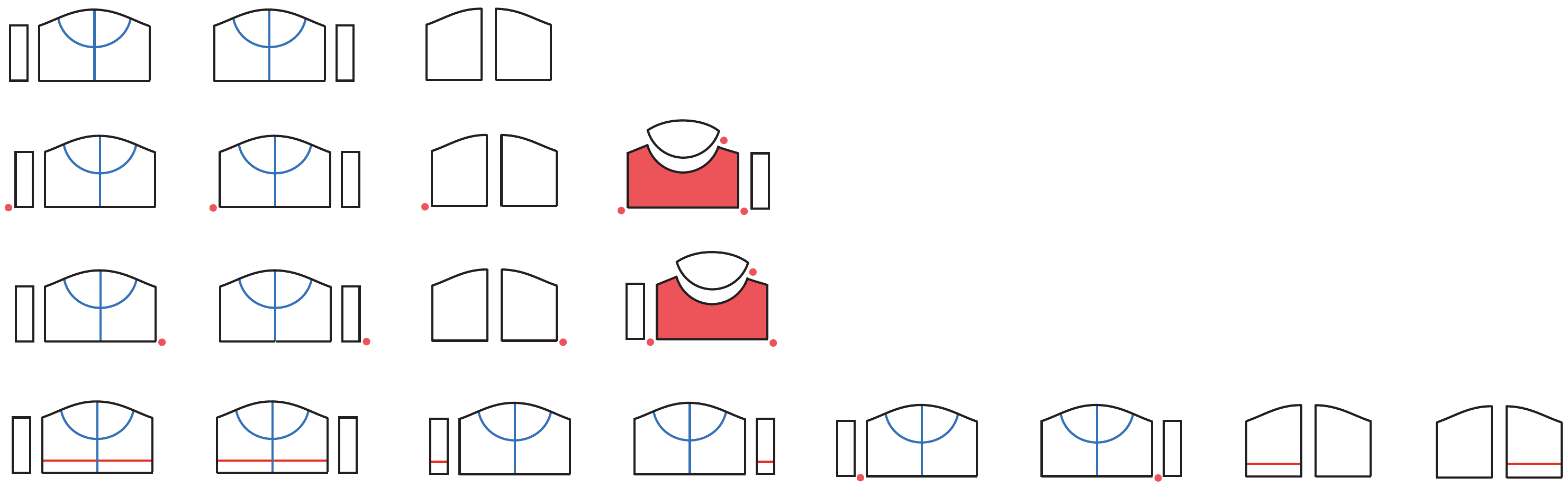}}
    \caption{{\small{Relations \eqref{eq:casearelc7}--\eqref{eq:casearelc10}.}}}
    \label{fig:casearelset3}
\end{figure}

\subsection{Case B}

We now consider Case B of Theorem \ref{thm:exacttriangles-withbundles}, in which we assume $\epsilon(L,L') + \epsilon(L'',L)=2$. From \eqref{eq:nontrivbundleindreds} it follows that the reducibles that $\eta'$ counts are obstructed, while by \eqref{eq:nontrivbundleindreds-2} those that $J''$ counts are unobstructed. These are depicted as follows:
\begin{equation}
	\eta' = \vcenter{\hbox{\includegraphics[scale=0.35]{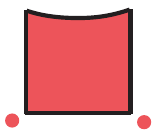} }} \qquad \quad J'' = \vcenter{\hbox{\includegraphics[scale=0.35]{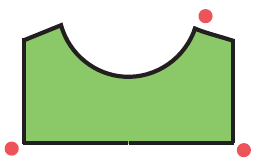} }} 
\end{equation}
We now construct an exact triangle of $\cS$-complexes of the form:
\begin{equation*}\label{eq:exacttrianglecaseaproof}
	\begin{tikzcd}[column sep=5ex, row sep=10ex, fill=none, /tikz/baseline=-10pt]
& \widetilde C  \arrow[dr, "\widetilde \lambda"] \arrow[dl, bend right=60, "\widetilde K"'] & \\
\widetilde C_\Sigma''  \arrow[rr, bend right=60, "\widetilde K''"']  \arrow[ur, "\widetilde \lambda''"] & &\widetilde C' \arrow[ll, "\widetilde \lambda'"] \arrow[ul, bend right=60, "\widetilde K'"']
\end{tikzcd}
\end{equation*}
The technical details of the construction run parallel to those of Case II in Theorem \ref{thm:exacttriangles}, but with fewer reducibles. We proceed to define all the terms and outline the proof. The complexes $\widetilde C = \widetilde C^\omega(Y,L)$ and $\widetilde C'=\widetilde C(Y,L')$ are as in Case A, with differentials $\widetilde d$, $\widetilde d'$ given as in \eqref{eq:tildednontrivbun}, \eqref{eq:tildedppnontrivbun}. The suspended complex $\widetilde C''_\Sigma = \Sigma \widetilde C(Y,L'')$ has differential
    \begingroup
\renewcommand{\arraystretch}{1.25}
\[
        \widetilde d''_\Sigma = \left[\begin{array}{cc|cc|c} 
            d'' & -\delta''_2 & 0 & 0 & 0 \\
            0 & 0 & 0 & 0 & 0 \\
            \hline
            v'' & 0 & -d'' & \delta''_2 & v''\delta''_2 \\
            \delta''_1 & 0 & 0 & 0& 0 \\
            \hline
            0 & 1 & 0 & 0 & 0
         \end{array} \right]
\]
\endgroup
The expression for $\widetilde \lambda$ is the same as in Case A, while $\widetilde \lambda'$ and $\widetilde \lambda''$ are given by:
    \begingroup
\renewcommand{\arraystretch}{1.25}
\[
\widetilde\lambda' = \left[ \begin{array}{c|c|c} \lambda' & 0 & 0 \\ \eta'\delta'_1 & 0 & 0 \\\hline \mu' & \lambda' & \Delta'_2 \\ \Delta'_1 & \eta'\delta'_1 & -s''\eta' - \eta' s' \\\hline0&0&\eta' \end{array} \right]
\qquad
\widetilde\lambda'' =  \left[ \begin{array}{cc|cc|c} \lambda''  & \Delta_2'' & 0 & 0 & 0  \\ \hline \mu'' & 0 &  \lambda'' & \Delta''_2 & \mu''\delta_2''+v\Delta_2''  \end{array} \right]
\]
\endgroup
The expression for $\widetilde K'$ is the same as in Case A, while $\widetilde K$ and $\widetilde K''$ are given by:
    \begingroup
\renewcommand{\arraystretch}{1.25}
\[
\widetilde K = \left[ \begin{array}{c|c} K & 0  \\ -\eta'\Delta_1  & 0 \\\hline L & -K  \\ M_1 & \eta'\Delta_1 \\\hline0&0\end{array} \right]
\quad
\widetilde K'' = \left[ \begin{array}{cc|cc|c} K'' & -M_2'' & 0 & 0 & 0 \\ \hline L'' & 0 & -K'' & M_2'' & v'M_2'' + L''\delta_2'' -\mu\Delta_2'' \\\hline M_1'' & 0 & 0 & 0 & J'' \end{array} \right]
\]
\endgroup
In this case, $\widetilde d'' \widetilde K + \widetilde K \widetilde d + \widetilde \lambda' \widetilde \lambda=0$ follows from the relations
\begin{align}
	d''K  + K d + \lambda'\lambda + \delta_2''\eta'\Delta_1= 0 \label{eq:casebrelc1} \\
	  - \eta'\Delta_1 d + \eta' \delta_1'\lambda= 0 \label{eq:casebrelc1half} \\
	\delta_1'' K + M_1 d + \Delta_1'\lambda  +\eta'\Delta_1 v -s''\eta'\Delta_1 -\eta's'\Delta_1 + \eta'\delta'_1\mu= 0  \label{eq:casebrelc2} \\
	v'' K - d''L + \delta_2'' M_1 + L d - Kv +\mu'\lambda + \lambda'\mu + \Delta_2'\Delta_1 = 0 \label{eq:casebrelc3}\\
	-\eta'\Delta_1 + \eta'\Delta_1 = 0 \label{eq:casebrelc4}
\end{align}
Three of these relations are depicted in Figure \ref{fig:casebrelset1} and are proved using obstructed gluing theory. For example, \eqref{eq:casebrelc1} and \eqref{eq:casebrelc2} are completely analogous to \eqref{eq:caseiirelc1} and \eqref{eq:caseiirelc3}, respectively. Relation \eqref{eq:casebrelc1half} follows from $-\Delta_1 d + \delta_1'\lambda=0$. 

Similarly, the following relations, depicted in Figure \ref{fig:casebrelset2}, imply $\widetilde d \widetilde K' + \widetilde K' \widetilde d' + \widetilde \lambda'' \widetilde \lambda'=0$:

\begin{figure}[t]
    \centering
    \centerline{\includegraphics[scale=0.27]{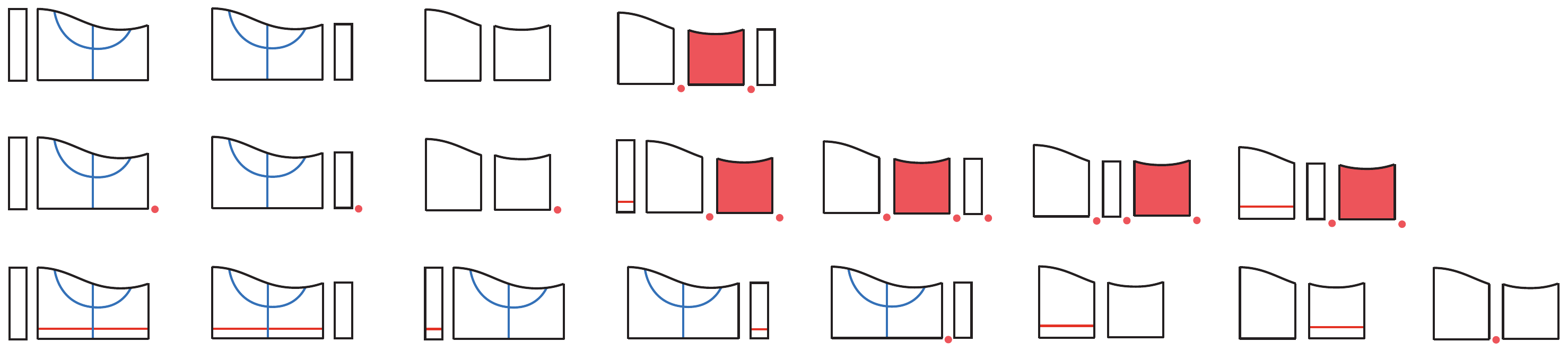}}
    \caption{{\small{Relations \eqref{eq:casebrelc1}, \eqref{eq:casebrelc2}, \eqref{eq:casebrelc3}.}}}
    \label{fig:casebrelset1}
\end{figure}

\begin{align}
	dK'  + K' d' + \lambda''\lambda'  + \Delta_2'' \eta'\delta_1'= 0  \label{eq:casebrelc5} \\
	-d M'_2 - K'\delta'_2 +\lambda''\Delta'_2 -\Delta_2'' s'' \eta' - \Delta_2'' \eta' s' +\mu''\delta_2''\eta' +v\Delta_2''\eta'= 0  \label{eq:casebrelc6}\\
	v K' - dL'  + L' d' - K'v' + M'_2\delta'_1 +\mu''\lambda' + \lambda''\mu' + \Delta_2''\Delta'_1 = 0  \label{eq:casebrelc7}
\end{align}

\begin{figure}[t]
    \centering
    \centerline{\includegraphics[scale=0.27]{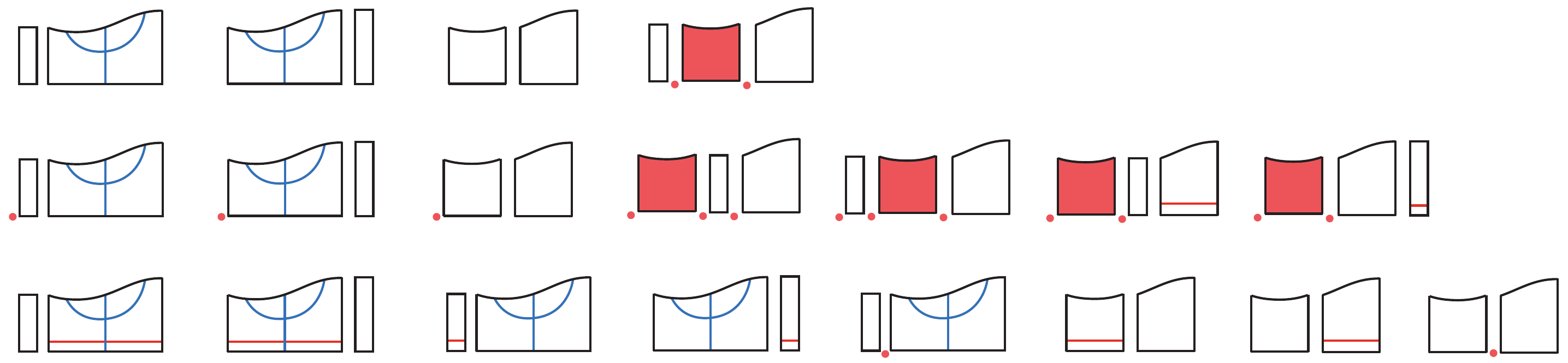}}
    \caption{{\small{Relations \eqref{eq:casebrelc5}--\eqref{eq:casebrelc7}.}}}
    \label{fig:casebrelset2}
\end{figure}

The relations that give $\widetilde d' \widetilde K'' + \widetilde K'' \widetilde d'' + \widetilde \lambda \widetilde \lambda''=0$ are:
\begin{align}
	d'K''  + K'' d'' + \lambda\lambda'' = 0  \label{eq:casebrelc8} \\
	-d' M''_2 - K''\delta''_2 +\lambda\Delta''_2 = 0  \label{eq:casebrelc9}\\
	\delta'_1 K'' + M''_1 d'' + \Delta_1\lambda'' = 0  \label{eq:casebrelc10} \\
		-\delta'_1 M_2'' - M''_1 \delta_2'' + \Delta_1\Delta_2'' + J''= 0  \label{eq:casebrelc11} \\
	v' K'' - d'L'' + \delta_2' M''_1 + L'' d'' - K''v'' + M''_2\delta''_1 +\mu\lambda'' + \lambda\mu''  = 0 \label{eq:casebrelc12}\\
	-v'M_2'' - L''\delta_2'' + v'M_2'' + L''\delta_2'' -\mu\Delta_2'' + \mu\Delta_2'' =0 \label{eq:casebrelc13}\\
	-d'v'M_2'' - d'L''\delta_2'' +d'\mu\Delta_2'' -K''v''\delta_2'' +\lambda\mu'' \delta_2'' + \lambda v \Delta_2'' +\delta_2' J'' = 0\label{eq:casebrelc14}
\end{align}

\begin{figure}[t]
    \centering
    \centerline{\includegraphics[scale=0.27]{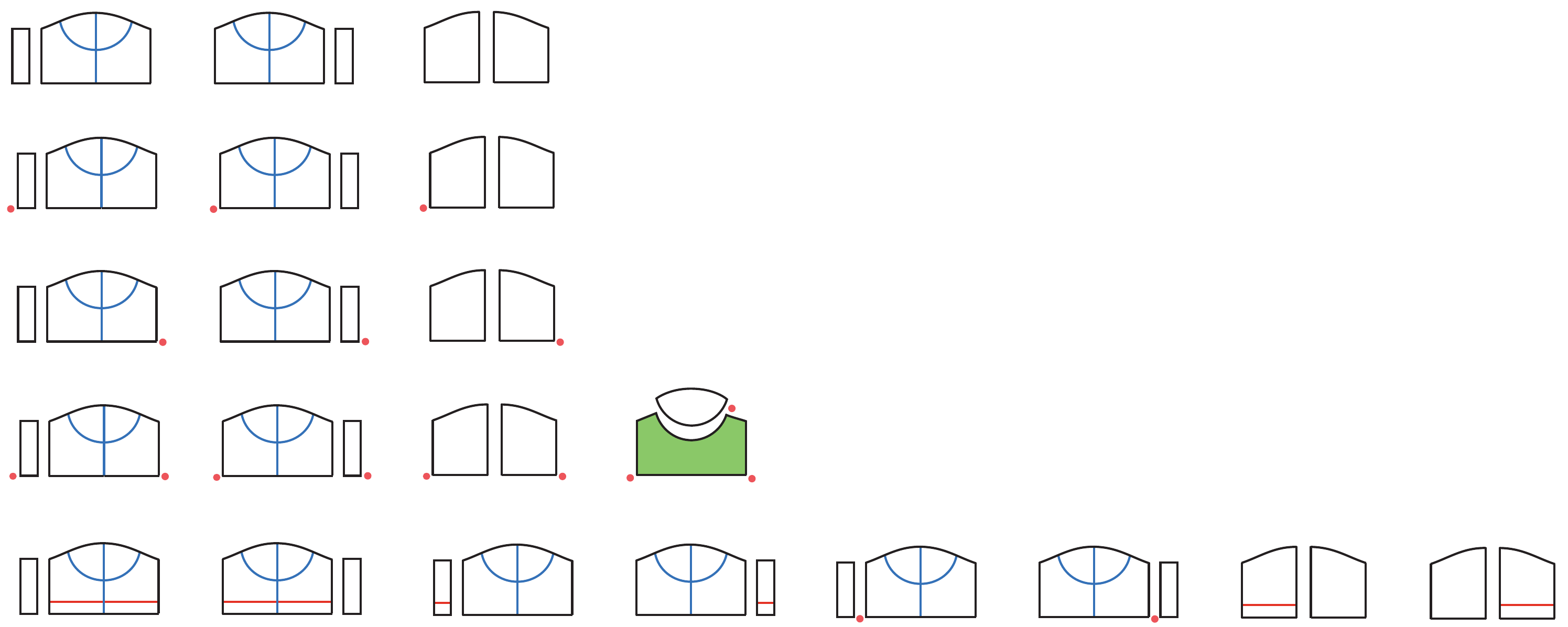}}
    \caption{{\small{Relations \eqref{eq:casebrelc8}--\eqref{eq:casebrelc12}.}}}
    \label{fig:casebrelset3}
\end{figure}

The first five of these relations are proved in the usual way, studying the ends of certain 1-dimensional moduli spaces, and in fact no obstructed reducibles appear. Relation \eqref{eq:casebrelc13} is identically true. The last relation, \eqref{eq:casebrelc14}, follows algebraically from other relations, similar to the case of \eqref{eq:caseiirelb6}.

Finally, we must show that the expressions \eqref{eq:caseirel1}--\eqref{eq:caseirel3}, as defined here, are homotopic to isomorphisms. The same strategy is followed. The reducible parts of these morphisms are the same as in Case A, and are isomorphisms. The irreducible parts are, respectively,
\begin{align}
	\lambda'' K - K'\lambda - \Delta_2''\eta'\Delta_1  \\
	\lambda K'  - K''\lambda' + M_2''\eta'\delta'_1 \\
	 \left[ \begin{array}{cc} \lambda'K'' - K \lambda'' & -\lambda' M_2'' - K\Delta_2'' \\[1mm] \eta'\delta'_1 K'' + \eta'\Delta_1 \lambda'' & -\eta' \delta'_1 M_2''+\eta' \Delta_1 \Delta_2''  \end{array} \right]
\end{align}
The arguments that these are homotopic to isomorphisms are completely analogous to those for the respective terms \eqref{eq:caseiiirrcomp3}, \eqref{eq:caseiiirrcomp1}, \eqref{eq:caseiiirrcompmatrix2}. This completes the proof of Case B, and hence of Theorem \ref{thm:exacttriangles-withbundles}, ignoring gradings. This issue is discussed in the next subsection.

\subsection{Euler characteristics}\label{subsec:euler}

Recall from Subsection \ref{subsec:absz2gr} that for an admissible link $(Y,L,\omega)$ where $Y$ is an integer homology 3-sphere, a choice of quasi-orientation on $L$ determines an absolute $\Z/2$-grading on $I^\omega(Y,L)$. We denote this homology group equipped with this absolute $\Z/2$-grading by $I_\fo^\omega(Y,L)$. In this subsection we prove the following, which implies Theorem \ref{thm:eulerchar-withbundles}. 

\begin{theorem}\label{thm:eulerchar-withbundles-signs}
	Let $(Y,L,\omega)$ be an admissible link, where $Y$ is an integer homology 3-sphere. Choose a quasi-orientation $\fo$ for $L$. If $\omega$ consists of a single arc which connects two distinct components $L_1,L_2\subset L$, then we have
	\begin{equation}
		\chi\left( I_{\fo}^\omega(Y,L) \right)= -2^{|L|-2}\text{{\emph{lk}}}_{\fo}(L_1,L_2) \label{eq:eulerchar-onearc}
	\end{equation}
	If the number of components $L_0\subset L$ with $\#( L_0\cap \partial \omega)\equiv 1 \pmod{2}$ is greater than $2$, then
	\[
		\chi\left( I_{\fo}^\omega(Y,L) \right) = 0
	\]
\end{theorem}

The grading calculations below are stated for irreducible homology groups, as our goal is to prove Theorem \ref{thm:eulerchar-withbundles-signs}. However, everything carries over without change to the corresponding $\cS$-complex morphisms. For simplicity of notation, we fix the integer homology 3-sphere $Y$ throughout, and often omit it from notation below.

Suppose $L,L',L''$ is an unoriented skein triple, where $|L|=|L'|+1=|L''|+1$. Let $\omega$ be an arc that connects the two components of $L$ that interact with the skein 3-ball. Assume $L'$ and $L''$ have non-zero determinant. Then we have one of the two exact triangles in Figure \ref{fig:exacttriangles-withbundles-irr}, which are the irreducible homology exact triangles derived from Theorem \ref{thm:exacttriangles-withbundles}. Recall that $I_+(L'')$ is the irreducible homology of the suspended $\cS$-complex $\Sigma \widetilde C(Y,L'')$.

Let $\fo'$ and $\fo''$ be quasi-orientations of $L'$ and $L''$ that agree as quasi-orientations when restricted to the components that do not interact with the skein 3-ball. We will slighly abuse notation and write $\fo'$ (resp. $\fo''$) for the quasi-orientation of $L$ that is compatible with $\fo'$ (resp. $L''$) via an oriented resolution to $L'$ (resp. $L''$). In what follows below, we will always choose our quasi-orientations in this manner.

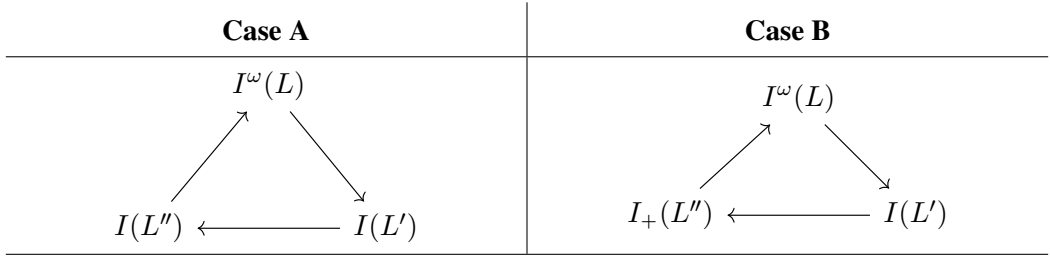
\begin{figure}[t]
    \centering
    \renewcommand{\arraystretch}{1.5}
    \begin{tabular}{ c | c }
        \textbf{Case A} & \textbf{Case B} \\ 
        \hline
        \begin{minipage}{6.5 cm}{
            \begin{equation*}
                \begin{tikzcd}[column sep=1.7ex, row sep=7ex, fill=none, /tikz/baseline=-10pt]
            & I^\omega(L)  \arrow[dr]  & \\
             I(L'')  \arrow[ur] & & I(L')  \arrow[ll] 
            \end{tikzcd}
        \end{equation*}
        }\end{minipage} & 
        \begin{minipage}{6.5 cm}{
            \begin{equation*}
               \begin{tikzcd}[column sep=1.7ex, row sep=5ex, fill=none, /tikz/baseline=-10pt]
            & I^\omega(L)  \arrow[dr]  & \\
            I_+(L'')  \arrow[ur] & & I(L')  \arrow[ll] 
            \end{tikzcd}
        \end{equation*}
        }\end{minipage} \\
        \hline 
    \end{tabular}
    \caption{{\small{The irreducible homology exact triangles of Theorem \ref{thm:exacttriangles-withbundles}.}}}
    \label{fig:exacttriangles-withbundles-irr}
\end{figure}

\begin{prop}\label{prop:degreesfornontrivbundletriangle}
With quasi-orientations as chosen above, the $\Z/2$-degrees of the maps in the Case A exact triangle of Figure \ref{fig:exacttriangles-withbundles-irr} are described as follows:
 \begin{equation}\label{eq:degstwocasesnontriv}
                \begin{tikzcd}[column sep=1.7ex, row sep=5ex, fill=none, /tikz/baseline=-10pt]
            & I^\omega_{\fo'}(L)  \arrow[dr,"\text{deg } 0"]  & \\
             I(L'')  \arrow[ur, "\text{deg } 1"] & & I(L')  \arrow[ll,"\text{deg } 0"] 
            \end{tikzcd}
\hspace{1.05cm}
                \begin{tikzcd}[column sep=1.7ex, row sep=5ex, fill=none, /tikz/baseline=-10pt]
            & I^\omega_{\fo''}(L)  \arrow[dr,"\text{deg } 1"]  & \\
             I(L'')  \arrow[ur, "\text{deg } 0"] & & I(L')  \arrow[ll,"\text{deg } 0"] 
            \end{tikzcd}
\end{equation}
The same result holds for the Case B exact triangle of Figure \ref{fig:exacttriangles-withbundles-irr}.
\end{prop}

\begin{proof}
	In what follows, it is convenient to write $I(L'')=I_{\fo''}^\omega(L'')$ and $I(L')=I_{\fo'}^\omega(L')$. Proposition \ref{prop:degmorphismnontriv} gives the mod $2$ degree of $I_{\fo'}^\omega(L)\to I_{\fo'}^\omega(L')$ as
	\[
	 \chi(S) +|L| + |L'| + (\partial c|_S)\cdot_{\fo\fo} (\partial c|_S)\; \, \equiv\;\,  (\partial c|_S)\cdot_{\fo'\fo'} (\partial c|_S) \equiv 0 \pmod{2}
	 \]
	 where $c=I\times \omega$. Indeed, in this case $S$ is orientable, and, by the choices of $\fo'$ on $L$ and $\fo'$ on $L'$, $(\partial c|_S)\cdot_{\fo'\fo'}(\partial c|_S)$ is computed with respect to orientations induced by one on $S$. Thus a pushoff of $(\partial c|_S)$ can be chosen disjoint from $(\partial c|_S)$. The degree of $I_{\fo''}^\omega(L)\to I_{\fo'}^\omega(L')$ then follows from Proposition \ref{prop:grchangenontriv}, noting $\text{gr}_{\fo'}-\text{gr}_{\fo''}\equiv \fo'\cdot \fo'' \equiv 1 \pmod{2}$. The other two non-horizontal maps in \eqref{eq:degstwocasesnontriv} are similar.
	 
	 Finally, consider $I_{\fo'}^\omega(L')\to I_{\fo''}^\omega(L'')$. Proposition \ref{prop:degmorphismnontriv} gives its degree as
	\[
	 \chi(S') +|L'| + |L''| + (\partial c|_{S'})\cdot_{\fo'\fo''} (\partial c|_{S'})\; \, \equiv\;\,  1 + (\partial c|_{S'})\cdot_{\fo'\fo''} (\partial c|_{S'}) \equiv 0 \pmod{2}
	 \]
	 The term $(\partial c|_{S'})\cdot_{\fo'\fo''} (\partial c|_{S'})$ is odd, as shown in Figure \ref{fig:mod2degofnonorsaddle}. Note that this cobordism map having even degree was already implicit in the proof of Theorem \ref{thm:exacttriangles-withbundles}: the presence of reducible instantons induces the non-triviality of the maps $\eta'$ and $J''$, which each have domain and co-domain supported in even degree.
\end{proof}

We will also make use of the exact triangle of \cite{KM:unknot}, which is the irreducible homology of \eqref{eq:kmequivtriangle}. This is the case where all bundles involved are non-trivial. The proof of the following is similar to that of Proposition \ref{prop:degreesfornontrivbundletriangle}.

\begin{prop}\label{prop:kmtriangledegsofmaps}
	Let $L,L',L''$ be a skein triple in an integer homology 3-sphere $Y$. Suppose bundle data $\omega$ is chosen uniformly, outside of the skein 3-ball, so that all three links are admissible. Without loss of generality, assume $L$ has one more component than $L'$ and $L''$. Choose quasi-orientations as described above. Then the $\Z/2$-degrees of the maps in the Kronheimer--Mrowka exact triangle for $L,L',L''$ are determined as follows:
		 \begin{equation*}\label{eq:degstwocasesnontrivkm}
                \begin{tikzcd}[column sep=1.7ex, row sep=5ex, fill=none, /tikz/baseline=-10pt]
            & I^\omega_{\fo''}(L)  \arrow[dr,"\text{deg  } \, \fo'\cdot_\omega\fo'' "]  & \\
             I_{\fo''}^\omega(L'')  \arrow[ur, "\text{deg  }\, 0"] & & I_{\fo'}^\omega(L')  \arrow[ll,"\text{deg  }\, 1 + \fo'\cdot_\omega\fo'' "] 
            \end{tikzcd}
\end{equation*}
where $\fo'\cdot_\omega\fo''$ is defined, with respect to the link $L$, as in \eqref{eq:ooprimeomega}.
\end{prop}

\begin{figure}[t]
    \centering
    \centerline{\includegraphics[scale=0.75]{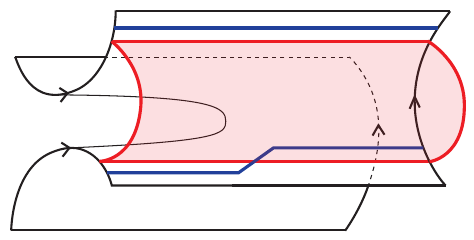}}
    \caption{{\small{Computation of $(\partial c|_{S'})\cdot_{\fo'\fo''} (\partial c|_{S'})$ where $c=I\times \omega$.}}}
    \label{fig:mod2degofnonorsaddle}
\end{figure}

\begin{proof}[Proof of Theorem \ref{thm:eulerchar-withbundles-signs}]
We assume that quasi-orientations appearing below are compatible with the conventions listed prior to Proposition \ref{prop:degreesfornontrivbundletriangle}. 

	The proof is by induction on the number of link components. Suppose $L$ has two components, $L_1$ and $L_2$, with the arc $\omega$ connecting them. We can fit $L$ into a skein triple $L,L',L''$ where $L'$ and $L''$ are knots. By \eqref{eq:degstwocasesnontriv} we have
	\[
		\chi\left( I_{\fo''}^\omega(L) \right) = \chi\left( I(L'') \right)  - \chi\left( I(L') \right) - \frac{1}{2}(\epsilon(L,L')  + \epsilon(L'',L))
	\]
	Note the last term is zero if we are in Case A of Figure \ref{fig:exacttriangles-withbundles-irr}, and is $-1$ in Case B, as determined by Lemma \ref{fig:exacttriangles-withbundles-irr}. Now by Theorem 1.3 of \cite{DS1} we have
	\[
		\chi\left( I(L'') \right) = 4\lambda(Y) + \frac{1}{2}\sigma(L'')
	\]
	where $\lambda(Y)$ is the Casson invariant of $Y$, and similarly for $L'$. We obtain
	\begin{align*}
		\chi\left( I_{\fo''}^\omega(L) \right) &= \frac{1}{2}\sigma(L'')  - \frac{1}{2}\sigma(L') - \frac{1}{2}(\epsilon(L,L')  + \epsilon(L'',L)) \\
		& =  \frac{1}{2}\left( - \sigma(L,\fo') + \sigma(L,\fo'')   \right) \\
		& = \frac{1}{2}\left(\text{lk}_{\fo'}(L_1,L_2) -  \text{lk}_{\fo''}(L_1,L_2) \right) = -\text{lk}_{\fo''}(L_1,L_2)
	\end{align*}
	where we used \eqref{eq:signaturechangeqo}. Thus the case of two component links is proved. 
	
	Now suppose $L$ has three components $L_1,L_2,L_3$. Assume $\omega$ is an arc connecting $L_1$ and $L_2$. We fit $L$ into a skein triple $L,L',L''$ where $L'$ and $L''$ each have two components, and $L_1\subset L$ does not intersect the skein 3-ball. Thus $L'=L_1\cup L_2'$ and $L''=L_1\cup L_2''$. We also specify our quasi-orientations so that they are represented by orientations which all agree on $L_1$, and which are related in the skein 3-ball as in Figure \ref{fig:ors3comp}. These choices give $\fo'\cdot_\omega\fo''\equiv 0 \pmod{2}$. Thus Proposition \ref{prop:kmtriangledegsofmaps} yields
			\begin{align*}
		\chi\left( I_{\fo''}^\omega(L) \right) &= \chi\left( I_{\fo'}^\omega(L')  \right)  + \chi\left( I_{\fo''}^\omega(L'')  \right) \\[1mm]
		&= - \text{lk}_{\fo'}(L_1,L_2')- \text{lk}_{\fo''}(L_1,L_2'')\\[1mm]
		& = -2\;\text{lk}_{\fo''}(L_1,L_2) = -2^{|L|-2} \; \text{lk}_{\fo''}(L_1,L_2)
	\end{align*}
	where the last line follows again by inspecting the linking numbers determined by Figure \ref{fig:ors3comp}. To see this, note that if $B$ is the skein $3$-ball, then the arc $L_2\setminus B$ oriented by $o''$ is oriented the same way when viewed as an arc in the diagrams for $(L',o')$ and $(L'',o'')$, while the arc $L_3\setminus B$ oriented by $o''$ is oriented in opposite ways in the diagrams for $(L',o')$ and $(L'',o'')$. This proves the statement for three component links. 
	
\begin{figure}[t]
    \centering
    \centerline{\includegraphics[scale=0.55]{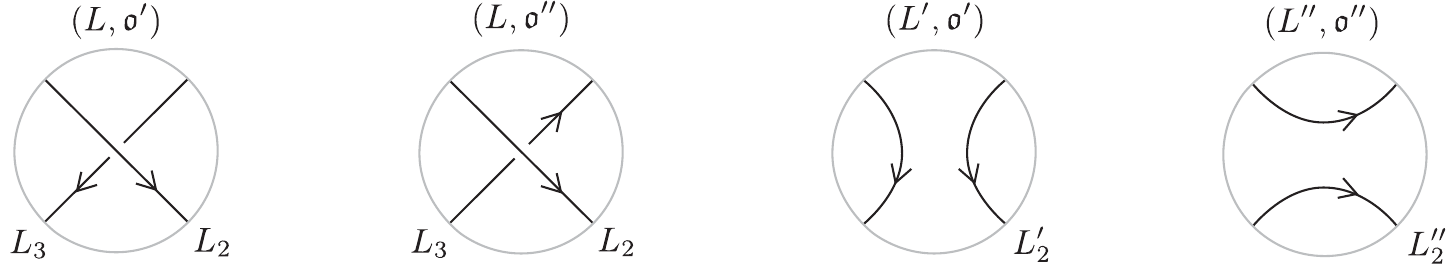}}
    \caption{{\small{The orientations inducing quasi-orientations that are used for a 3-component link. They agree on the component $L_1$ (not shown).}}}
    \label{fig:ors3comp}
\end{figure}
	
	Now suppose $L$ has $n>3$ components and $\omega$ is an arc joining two components $L_1,L_2$ of $L$. Assume the result in this case is proved for links with less than $n$ components. Choose a skein triple $L,L',L''$ such that $L'$, $L''$ have one few component than $L$, and the skein 3-ball does not interact with $L_1,L_2$. Recalling that $\fo',\fo''$ are chosen to agree away from the skein 3-ball, we have $\fo'\cdot_\omega\fo''\equiv 0\pmod{2}$ and Proposition \ref{prop:kmtriangledegsofmaps} yields
	\begin{align*}
		\chi\left( I_{\fo''}^\omega(L) \right) &= \chi\left( I_{\fo'}^\omega(L')  \right)  + \chi\left( I_{\fo''}^\omega(L'')  \right) \\[1mm]
		&= -2^{|L'|-2} \, \text{lk}_{\fo'}(L_1',L_2')-2^{|L''|-2} \,\text{lk}_{\fo''}(L_1'',L_2'')
	\end{align*}
where in the second line the induction hypothesis was used. For $i\in \{1,2\}$ we have written $L_i'$ and $L_i''$ for the copies of $L_i$ in the respective links $L'$ and $L''$. Each of the two final terms is equal to $-2^{|L|-3} \text{lk}_{\fo''}(L_1,L_2)$ by our setup, and the result follows. Thus the case in which $L$ is any link and $\omega$ is an arc connecting two distinct components is proved.
	
	Next, consider the case that $L$ has four components $L_1,L_2,L_3,L_4$ and $\omega=\omega_1\cup \omega_2$ where $\omega_1$ is an arc connecting $L_1$ and $L_2$, and $\omega_2$ is an arc connecting $L_3$ and $L_4$. Form a skein triple $L,L',L''$ where $L'$, $L''$ have one fewer component than $L$, and such that $L_1,L_2$ intersect the skein 3-ball. Then $\fo'\cdot_\omega\fo''\equiv 1 \pmod{2}$ on $L$, and Proposition \ref{prop:kmtriangledegsofmaps} gives
		\begin{align*}
		\chi\left( I_{\fo''}^\omega(L) \right) &= \chi\left( I_{\fo''}^\omega(L'')  \right)  - \chi\left( I_{\fo'}^\omega(L')  \right) \\[1mm]
		&= -2\; \text{lk}_{\fo''}(L_3'',L_4'')+2 \; \text{lk}_{\fo''}(L_3',L_4')  =0
	\end{align*}
	Next, let $L$ be any link with bundle data $\omega$, and $m_L$ be the number of components $L_0\subset L$ with $\#( L_0\cap \partial \omega)\equiv 1 \pmod{2}$. If $m_L\geq 4$, one can reduce to the four component case by a sequence of skein triples such that every link $L'$ appearing in each skein triple has $m_{L'}\geq 4$. Proposition \ref{prop:kmtriangledegsofmaps} then implies the result by induction. 
\end{proof}

\newpage


\section{Further constructions} \label{sec:further}

In this section we describe how a number of constructions in \cite{DS1,DS2} for instanton $\cS$-complexes of knots generalize to the case of non-zero determinant links, and at various points indicate how the exact triangles of Theorem \ref{thm:exacttriangles} fit in with these constructions. In the first subsection, we define equivariant homology groups, and prove Theorem \ref{thm:equivexacttriangle}. We then discuss local coefficient systems and a mild generalization of Theorem \ref{thm:localcoeffintro}. As an illustration, we study the case of alternating torus links. In the final subsection we introduce generalized Fr\o yshov-type invariants of links.

\subsection{Equivariant homology groups}\label{subsec:equivhomgrps}

In this subsection we discuss equivariant homology groups associated to an $\cS$-complex and their behavior under suspension. We then define the equivariant singular instanton Floer groups for links with non-zero determinant and deduce Theorem \ref{thm:equivexacttriangle}.

Fix a commutative ring $R$. Anaolgous to the constructions in \cite[Section 4]{DS1}, to any $\cS$-complex $(\widetilde C,\widetilde d, \chi)$ over $R$ we  may associate three chain complexes over $R[x]$:
\begin{equation}
	(\hrC,\widehat d), \qquad (\crC, \widecheck d), \qquad (\brC ,\overline d) \label{eq:largeeqcomplexes}
\end{equation}
These are the (large) equivariant chain complexes associated to $\widetilde C$; they are defined  by
\begin{align*}
	   \hrC & :=R[x]\otimes_R \widetilde C\hspace{1cm}  & \widehat d(x^i \cdot \zeta) & =-x^i \cdot \widetilde d\zeta+x^{i+1}\cdot \chi(\zeta)\\[1mm]
   \crC  & :=(R[\![x^{-1},x]/R[x])\otimes_R \widetilde C\hspace{1cm}  & \widecheck d(x^{i} \cdot \zeta) & = \phantom{-}x^{i} \cdot
  \widetilde d\zeta-x^{i+1}\cdot \chi(\zeta)\\[1mm]
 \brC & := R[\![x^{-1},x] \otimes_R \widetilde C\hspace{1cm}  & \overline d(x^{i} \cdot \zeta) & = -x^{i} \cdot
  \widetilde d\zeta-x^{i+1}\cdot \chi(\zeta)
\end{align*}
Here $R[\![x^{-1},x]$ is the Laurent power series ring in the variable $x^{-1}$ and with coefficients in $R$. The differentials given are extended $R$-linearly and to Laurent power series in the obvious way; similar remarks should be applied for formulas appearing below. These complexes inherit gradings using the tensor product grading and the convention that the grading of $x^i$ is $-2i$. For a morphism $\widetilde \lambda: \widetilde C\to \widetilde C'$ of $\cS$-complexes we obtain $\widehat\lambda:  \hrC\to  \hrC'$ defined by $\widehat\lambda(x^i\cdot \alpha) = x^i\cdot \widetilde \lambda(\alpha)$, and similarly $\widecheck\lambda,\overline\lambda$ are defined. These constructions are functorial, and are also compatible with chain homotopies, as in \cite[Proposition 4.9]{DS1}.

There are two triangles of chain complexes inducing exact triangles on homology:
 \begin{equation}\label{equiv-triangles}
                \begin{tikzcd}[column sep=4ex, row sep=8ex, fill=none, /tikz/baseline=-10pt]
            \crC \arrow[rr, "j"] & & \hrC \arrow[dl, "i"] \\
	& \brC \arrow[ul, "p"] &
            \end{tikzcd}
            \qquad\qquad
                            \begin{tikzcd}[column sep=4ex, row sep=8ex, fill=none, /tikz/baseline=-10pt]
           \widetilde C \arrow[rr, "z"] & & \hrC \arrow[dl, "x"] \\
	&  \hrC \arrow[ul, "y"] &
            \end{tikzcd}
 \end{equation}
Here $j(x^i\alpha) = 0$ for $i<-1$ and $j(x^{-1}\alpha)= -\chi(\alpha)$; $p$ is the projection map composed with the sign map $\epsilon$; and $i$ is the inclusion. The map $z$ is given by $\chi$, $x$ is multiplication by $x$, and $y$ is the projection to the constant term composed with the sign map. The maps $j$ and $z$ are only $R[x]$-module maps up to chain homotopy, while the other maps appearing in \eqref{equiv-triangles} are $R[x]$-module maps. The complex $\widetilde C$ is viewed as an $R[x]$-module with $x$ acting trivially. These triangles are functorial with respect to morphisms. In short, all of the constructions of \cite[Subsection 4.2]{DS1} directly apply to the more general definition of $\cS$-complexes considered in the present work.

Now assume, as we will for the rest of this section, that all $\cS$-complexes are $r$-perfect; these are all that are needed for the sequel. Then the construction of the small equivariant complexes of \cite[Subsection 4.3]{DS1} carries through with minor variations. Namely, to our $\cS$-complex $\widetilde C = C\oplus C[-1]\oplus \sfR$ we associate $R[x]$-chain complexes
\begin{equation}
(\fhrC,\widehat \fd), \qquad (\fcrC, \widecheck \fd), \qquad (\fbrC ,\overline \fd). \label{eq:smalleqcomplexes}
\end{equation}
To define these, first introduce the $R$-modules
\[
	\sfR[x]:=\sfR\otimes_R R[x], \qquad \sfR[\![x^{-1},x]:=\sfR\otimes_R R[\![x^{-1},x].
\]
The complexes \eqref{eq:smalleqcomplexes} are then defined as follows, where $\alpha\in C$ and $\theta\in \sfR$:
\begin{align*}
  \fhrC & :=C[-1]\oplus \sfR[x]  & \widehat \fd(\alpha, \theta x^i) & =(d\alpha-v^i\delta_2(\theta),0)\\[0.1mm]
  \fcrC & :=C \oplus \sfR[\![x^{-1},x]/\sfR[x]  & \widecheck \fd(\alpha, \theta x^i)  & =(d\alpha,
  \textstyle{\sum_{i\leqslant -1 }}\delta_1v^{-i-1}(\alpha)x^i)\\[1mm]
	\fbrC & := \sfR[\![x^{-1},x] & \overline \fd & = 0
\end{align*}
These small equivariant complexes are naturally homotopy equivalent to the respective large complexes \eqref{eq:largeeqcomplexes} as chain complexes over $R[x]$. For then small equivariant complexes \eqref{eq:smalleqcomplexes}, the $R[x]$-module structures are defined as follows:
\begin{align*}
  \fhrC: & &x\cdot (\alpha, \theta x^i) & =(v\alpha,\delta_1(\alpha) + \theta x^{i+1})\\[1mm]
  \fcrC: & & x\cdot (\alpha, \theta x^i)  & =\begin{cases} (v\alpha + \delta_2(\theta),  0 ) & i=-1 \\[1mm]   (v\alpha , \theta x^{i+1} ) & i \leqslant -2 \end{cases}
\end{align*}
while the module structure of $\fbrC$ is the usual one. The small equivariant package, apart from being ``smaller'', has the advantage that $\fbrC$ is exactly equal to $\sfR[\![x^{-1},x]/\sfR[x]$. 

To an $\cS$-complex $\widetilde C$ as above we have the equivariant homology groups $\widecheck H$, $\widecheck H$ and $\overline H = \sfR[\![x^{-1},x]$. These are graded $R[x]$-modules and may be defined as the homology groups of either the large or small equivariant chain complexes associated to $\widetilde C$.

The following result says that the equivariant homology groups of an $\cS$-complex are naturally isomorphic to the ones defined using the suspension complex, with the isomorphism between the ``bar'' homology groups given by multiplication by $x$.

\begin{prop}\label{prop:susequivar}
Let $\widetilde C = C \oplus C[-1] \oplus \sfR$ be an $\cS$-complex with equivariant homology groups $\widecheck H$, $\widehat H$ and let the suspension $\cS$-complex $\widetilde C_\Sigma$ have equivariant homology groups $\widecheck H_\Sigma$, $\widehat H_\Sigma$. There are natural isomorphisms $\widecheck f_\ast$  and $\widehat f_\ast$ which make the following diagram commute:
\begin{equation*} 
                \begin{tikzcd}[column sep=4ex, row sep=8ex, fill=none, /tikz/baseline=-10pt]
            	\cdots \arrow[rr, "p_\ast"]  && \widecheck H_\Sigma \arrow[rr, "j_\ast"] \arrow[d, "\widecheck f_\ast"] \arrow[rr]  &&\widehat H_\Sigma \arrow[rr, "i_\ast"] \arrow[d, "\widehat f_\ast"] && \sfR[\![x^{-1},x] \arrow[d, "x\cdot"] \arrow[rr, "p_\ast"] & & \cdots \\
            	\cdots \arrow[rr, "p_\ast"]  && \widecheck H \arrow[rr, "j_\ast"]\arrow[rr]  &&\widehat H \arrow[rr, "i_\ast"] && \sfR[\![x^{-1},x] \arrow[rr, "p_\ast"] & & \cdots
            \end{tikzcd}
            \end{equation*}	
\end{prop}

\begin{proof}
	Using Definition \ref{defn:suspension} we write $\widehat \fC_\Sigma = C[-3] \oplus \sfR[-2] \oplus \sfR[x]$ and $\widecheck\fC_\Sigma = C[-2] \oplus \sfR[-1] \oplus \sfR[\![x^{-1},x]/\sfR[x]$. We compute the respective differentials to be
\begin{align*}
	\widehat \fd_\Sigma(\alpha, \theta,  \theta' x^i) &  = (d\alpha  - \delta_2(\theta)- v^{i+1}\delta_2(\theta'), 0 ,0) \\[1mm]
	\widecheck \fd_\Sigma(\alpha, \theta,  \theta' x^i) &  = (d\alpha - \delta_2(\theta), 0 , \theta x^{-1} + \sum_{i\leqslant -2} \delta_1 v^{-i-2}(\alpha) x ^i)
\end{align*}
where $\alpha\in C$ and $\theta,\theta'\in \sfR$. The $R[x]$-module structures are given by
\begin{align*}
  \fhrC_\Sigma: & &x\cdot (\alpha, \theta,  \theta' x^i)  & =(v\alpha,\delta_1(\alpha), \theta + \theta' x^{i+1})\\[1mm]
  \fcrC_\Sigma: & & x\cdot (\alpha, \theta,  \theta' x^i)  & =\begin{cases} (v\alpha + v\delta_2(\theta'),  \delta_1(\alpha),0  ) & i=-1 \\[1mm]   (v\alpha , \delta_1(\alpha), \theta' x^{i+1} ) & i \leqslant -2 \end{cases}
\end{align*}
Define the following maps, where $\alpha\in C$ and $\theta,\theta_i\in \sfR$:
\begin{align*}
	\widehat{f}:\widehat{\fC}_\Sigma\to \widehat{\fC} &&  \widehat{f}(\alpha,\theta,\sum_{i\geq 0} \theta_i x^i)  & = (\alpha, \theta + \sum_{i\geq 0} \theta_i x^{i+1})\\
	\widehat{f}':\widehat{\fC}\to \widehat{\fC}_\Sigma && \widehat{f}'(\alpha,\sum_{i\geq 0} \theta_i x^i)  & = (\alpha, \theta_0, \sum_{i\geq 1} \theta_i x^{i-1})\\
	\widecheck{f}:\widecheck{\fC}_\Sigma\to \widecheck{\fC} &&  \widecheck{f}(\alpha,\theta,\sum_{i\leq -1} \theta_i x^i)  & = (\alpha + \delta_2(\theta_{-1}),\sum_{i\leq -2} \theta_i x^{i+1})\\
	\widecheck{f}':\widecheck{\fC}\to \widecheck{\fC}_\Sigma &&  \widecheck{f}'(\alpha,\sum_{i\leq -1} \theta_i x^i)  & = (\alpha, 0 , \sum_{i\leq -1} \theta_i x^{i-1})
\end{align*}
The following identities are direct computations: each of $\widehat{f}$, $\widehat{f}'$, $\widecheck{f}$, $\widecheck{f}'$ is a chain map; $\widehat{f}$ is an isomorphism of $R[x]$-chain complexes with inverse $\widehat{f}'$;  $\widecheck f x = x \widecheck f$; $\widecheck f \widecheck f' = 1$ and $1-\widecheck f' \widecheck f = \widecheck K \widecheck \fd_\Sigma  + \widecheck \fd_\Sigma  \widecheck K$; and $\widecheck f x - x \widecheck f = \widecheck \fd_\Sigma \widecheck K' + \widecheck K' \widecheck \fd$. Here $\widecheck K$ and $\widecheck K'$ are defined by

\begin{align*}
	\widecheck K(\alpha ,\theta, \sum_{i\leq -1} \theta_i x^i ) &= ( 0, \theta_{-1} , 0) \\
	\widecheck K'(\alpha , \sum_{i\leq -1} \theta_i x^i ) &= (0, -\theta_{-1} , 0)
\end{align*}
Next we recall that the maps $i_\ast$, $j_\ast$, $p_\ast$ have the following representations on the level of small equivariant chain complexes:
\begin{align*}
	\mathfrak{i}:\widehat{\fC}\to \overline{\fC} &&  \mathfrak{i}(\alpha,\sum_{i\geq 0} \theta_i x^i)  & = \sum_{i\leq -1 }\delta_1 v^{-i-1}(\alpha)x^i + \sum_{i\geq 0 }\theta_i x^i\\
	\fj:\widecheck{\fC} \to \widehat{\fC} &&  \fj(\alpha,\sum_{i\leq -1 }\theta_i x^i)  & = (-\alpha, 0)\\
	\fp:\overline{\fC} \to \widecheck{\fC} &&  \fp(\sum_{i\in \Z} \theta_i x^i)  & = (\sum_{i\geq 0} v^i \delta_2(\theta_i), \sum_{i\leq -1} \theta_i x^i ) 
\end{align*}
Then $\mathfrak{i}\widehat{f} = x \circ \mathfrak{i}_\Sigma$, $\fp \circ x = \widecheck f \fp_\Sigma$, and $\widehat f \fj_\Sigma - \fj \widecheck f = K \widecheck \fd_\Sigma + \widehat \fd K$ where the homotopy $K$ is defined by $K(\alpha,\theta,\textstyle{\sum_{i\leq -1} }\theta_i x^i) = (0,-\theta_{-1})$. Applying homology gives the result.
\end{proof}

We have already remarked in Proposition \ref{eq:tildehomologysuspension} that the homology of $\widetilde C$ is invariant under suspension. The exact triangle induced on homology from the right triangle in \eqref{equiv-triangles} can also be shown to be isomorphic to the one associated to the suspension $\cS$-complex. This essentially follows because $\widehat{f}$ commutes with multiplication by $x$.

\begin{definition}\label{defn:equivarianthomologygroups}
	Let $(Y,L)$ be a based link in a homology 3-sphere with non-zero determinant. Define the equivariant singular instanton Floer groups associated to $(Y,L)$, denoted
\[
	\widehat{I}(Y,L), \qquad \;\;\; \widecheck I(Y,L), \qquad \;\;\; \overline{I}(Y,L)
\] 
to be the equivariant homology groups associated to the $\cS$-complex $\widetilde C(Y,L)$. These are $\Z/2$-graded $R[x]$-modules. $\diamd$
\end{definition}

A choice of quasi-orientation of $(Y,L)$ lifts these to $\Z/4$-graded modules. Note
\[
	\overline{I}(Y,L) \cong R^{2^{|L|-1}}\otimes_R R[\![x^{-1},x]
\]
Definition \ref{defn:equivarianthomologygroups} extends to based admissible links $(Y,L,\omega)$. However, this case is much simpler: the absence of reducibles implies that the associated equivariant homology $\overline I$ is zero, while $ \widehat{I}$ and $\widecheck{I}$ are naturally isomorphic as $R[x]$-modules. For this reason, we focus on the case of non-zero determinant links in homology $3$-spheres.

\begin{lemma}\label{lemma:equivtriangle}
	Suppose $\cS$-complexes $\widetilde C,\widetilde C', \widetilde C''$ fit into an exact triangle as in Definition \ref{def:exacttriangle}. Then $\widehat C,\widehat C', \widehat C''$ fit into an exact triangle of $R[x]$-chain complexes, and similarly for the other two flavors of equivariant chain complexes.
\end{lemma}

\begin{proof} This follows from the functoriality of the equivariant homology constructions from an $\cS$-complex as described in \cite[Proposition 4.9]{DS1}.
\end{proof}

\begin{proof}[Proof of Theorem \ref{thm:equivexacttriangle}]
	This follows from Theorem \ref{thm:exacttriangles}, that equivariant homology groups are invariant under suspension as shown in Proposition \ref{prop:susequivar}, and Lemma \ref{lemma:equivtriangle}.
\end{proof}

\begin{figure}[t]
    \centering
    \begin{equation*}
                \begin{tikzcd}[column sep=4ex, row sep=7ex, fill=none, /tikz/baseline=-10pt]
            	\cdots \arrow[rr]  &&  \widehat I(Y,L)  \arrow[rr] \arrow[d]  &&  \widehat I(Y,L) \arrow[rr] \arrow[d] && I^\natural(Y,L) \arrow[d] \arrow[rr] & & \cdots \\
            	            	\cdots \arrow[rr]  &&  \widehat I(Y,L')  \arrow[rr] \arrow[d]   &&  \widehat I(Y,L') \arrow[rr] \arrow[d] && I^\natural(Y,L') \arrow[d] \arrow[rr] & & \cdots \\
            	            \cdots \arrow[rr]  &&  \widehat I(Y,L'')  \arrow[rr] \arrow[d]   &&  \widehat I(Y,L'') \arrow[rr] \arrow[d] && I^\natural(Y,L'') \arrow[d] \arrow[rr] & & \cdots \\
            	                        	            	\cdots \arrow[rr]  &&  \widehat I(Y,L)  \arrow[rr]   &&  \widehat I(Y,L) \arrow[rr]  && I^\natural(Y,L)  \arrow[rr] & & \cdots 
            \end{tikzcd}
\end{equation*}
    \caption{{\small{The exact triangle for the equivariant homology group $\widehat{I}$ fits into an exact triangle with Kronheimer and Mrowka's exact triangle for $I^\natural$. Rows and columns are exact.}}}
    \label{fig:exacttrianglefromnatural}
\end{figure}

The exact triangle for $\widehat{I}$ in Theorem \ref{thm:equivexacttriangle} also fits into an exact triangle relating it to Kronheimer and Mrowka's exact triangle for $I^\natural$. This follows from the above discussion and the identification of $H(\widetilde C(Y,L))$ with $I^\natural(Y,L)$ via Theorem \ref{eq:inatural}. See Figure \ref{fig:exacttrianglefromnatural}.

The following will be useful when discussing concordance invariants for links.

\begin{prop}\label{prop:equivlevel}
Let $(W,S):(Y,L)\to (Y',L')$ be a cobordism between links in homology 3-spheres with non-zero determinant, and $b^1(W)=b^+(W)=0$. Suppose there is given an embedded path $\gamma\subset S$ between basepoints on $L$, $L'$, and also singular bundle data $c$. If this cobordism is unobstructed then there is an associated commutative diagram
\begin{equation*}
                \begin{tikzcd}[column sep=4ex, row sep=8ex, fill=none, /tikz/baseline=-10pt]
            	\cdots \arrow[rr, "p_\ast"]  && \widecheck{I}(Y,L)  \arrow[rr, "j_\ast"] \arrow[d, "\widecheck \lambda_\ast"]   && \widehat{I}(Y,L) \arrow[rr, "i_\ast"] \arrow[d, "\widehat \lambda_\ast"] && \overline{I}(Y,L) \arrow[d, "\overline\lambda_\ast"] \arrow[rr, "p_\ast"] & & \cdots \\
            	\cdots \arrow[rr, "p_\ast"]  && \widecheck{I}(Y',L') \arrow[rr, "j_\ast"]  &&\widehat{I}(Y',L') \arrow[rr, "i_\ast"] &&  \overline{I}(Y',L') \arrow[rr, "p_\ast"] & & \cdots
            \end{tikzcd}
\end{equation*}
with vertical maps induced by the cobordism data. If $(W,S,c)$ is negative definite of height $k\in \Z_{\geq 0}$ then identifying $\overline{I}(Y,L)=\sfR(Y,L)\otimes R[\![x^{-1},x]$ we have
\[
	\overline{\lambda}_\ast = \sum_{i=-\infty}^k f_i\otimes x^i, \quad\quad  \text{ for some }\qquad f_i:\sfR(Y,L)\longrightarrow \sfR(Y',L').
\]
Furthermore, $f_i=\tau_i$ with notation as in Subsection \ref{subsec:poslevels}. In particular, if the cobordism is strong height $k$, then $f_k$ is an isomorphism. The same holds in the obstructed case when $(W,S,c)$ is negative definite of height $k=-1$. In the cobordism is odd, then $\overline{\lambda}_\ast=0$.
\end{prop}

\begin{proof}
	The case of height $0$ cobordisms follows from the construction of $\widetilde \lambda$ and a straightforward generalization of \cite[Corollary 4.12]{DS1}. The unobstructed case follows from the height $0$ case by Proposition \ref{prop:levelsuspension}, Theorem \ref{thm:unobsmaps}, and Proposition \ref{prop:susequivar}. The obstructed case follows from Proposition \ref{decomp-neg-level}, the construction of Section \ref{sec:obstructed}, and Proposition \ref{prop:susequivar}.
\end{proof}

The cobordism maps above, all defined for ``negative definite'' cobordisms, i.e. ones for which $b^1(W)=b^+(W)=0$, are independent of metric and perturbation data; the usual continuation arguments relate different choices by $\cS$-chain homotopy.

\subsection{Local coefficients}\label{subsec:local}

In \cite{DS1,DS2} an additional layer of structure is defined for the equivariant singular instanton invariants of a knot: a local coefficient system modelled on the construction of \cite{KM:YAFT}. As explained below, this structure extends to the more general setting of $\cS$-complexes for links. At the end of the subsection we also describe the Chern--Simons filtration in this setting.

Define the local coefficient ring for an $n$-component link to be 
\[
	\sS_n = \Z[T_1^{\pm 1}, \ldots, T_n^{\pm 1}].
\]
The general local coefficient system, denoted $\Delta_n$, is defined on the configuration space of singular $SU(2)$ connections for a link $(Y,L)$. First label the components of the link $L_1,\ldots, L_n$. Then $\Delta_n$ is defined over a singular connection class $\alpha$ to be the $\sS_n$-module
\begin{equation}\label{eq:loccoefffiber}
	(\Delta_n)_\alpha = \sS_n\cdot T_1^{\text{hol}_{L_1}(\alpha)}\cdots T_n^{\text{hol}_{L_n}(\alpha)}	
\end{equation}
Here the terms $\text{hol}_{L_j}(\alpha)$ are well-defined in $\R/\Z$; any lifts to $\R$ will make do for the above expression. The term $\text{hol}_{L_j}(\alpha)$ is roughly defined as follows. First, we frame our link. The connection $\alpha$ is compatible, in the limit near $L_j$, with a reduction $F_j\oplus F_j^{-1}$ of the bundle; the factor $F_j$ is distinguished as the one for which the holonomy of $\alpha$ around meridians (determined by the framing) near $L_j$ is given by $i\in U(1)$. Then $\text{hol}_{L_j}(\alpha)$ is twice the holonomy of $\alpha$ restricted to $F_j$ along the longitude of $L_j$. In general, changing the framing of $L$ will change this local coefficient system $\Delta_n$ by a canonical isomorphism which is multiplication by a monomial in the $T_i's$. For details see \cite{KM:YAFT, DS1}.

To choose the framing of $L$ in the definition of the local system, we may choose for each component $L_j$ a Seifert surface in $Y$ and use the induced framing of $L_j$. This defines $\Delta_n$ and does not depend on any choice of orientation of the link. With this framing, 
let us compute \eqref{eq:loccoefffiber} for $\alpha=\theta_{\fo}$ a reducible. Such a reducible is compatible with a splitting $F\oplus F^{-1}$ over $Y$ which extends the splittings $F_j\oplus F_{j}^{-1}$ near each component mentioned above. Then 
\begin{equation}\label{eq:holofreducible}
	\text{hol}_{L_j}(\theta_\fo) \equiv  \frac{1}{2}\sum_{i\neq j} \text{lk}(L_i,L_j) \pmod{\Z},
\end{equation}
as the longitude for the Seifert surface of $L_j$ is mod 2 homologous to $\sum_{i\neq j} \text{lk}(L_i,L_j)\mu_i$ where $\mu_i$ is the meridian for $L_i$. While \eqref{eq:holofreducible} is a half-integer in general, we have
\begin{equation}\label{eq:holofreduciblesum}
	\sum_{i=1}^n \text{hol}_{L_i}(\theta_\fo)\equiv 0 \pmod{\Z}.
\end{equation}

 For $(Y,L)$ a non-zero determinant link in a homology sphere define
\[
	\widetilde C(Y,L,p;\Delta_n)	
\]
to be the singular instanton $\cS$-complex of $(Y,L)$, with some basepoint $p$, defined with the local coefficent system $\Delta_n$; this is a $\Z/2$-graded $\cS$-complex over $\sS_n$. The construction follows Section \ref{subsec:linkinv} with modifications as in \cite[Section 7]{DS1}. Further, a choice of quasi-orientation of $L$ can also be used to lift the grading to $\Z/4$.

The local coefficient systems above are functorial with respect to component-wise cobordisms. Suppose $(Y,L)$ and $(Y',L')$ are two links as above each with $n$ components, labelled $L_1,\ldots, L_n$ and $L'_1,\ldots,L'_n$ respectively. Consider a cobordism
\begin{equation}
	(W,S):(Y,L)\to (Y',L'), \qquad S = \bigsqcup_{i=1}^n S_i \qquad S_i: L_i\to L_i'  \label{eq:componentwise}
\end{equation}
Each $S_i$ in \eqref{eq:componentwise} is allowed to be non-orientable.
Choose singular bundle data $c$ which is trivial in a neighborhood of $S$. Then we obtain a morphism of local coefficient systems. In particular, if $(W,S,c)$ is negative definite of height $k\geq -1$, the constructions of Theorem \ref{thm:unobsmaps} for ($k\geq 0$) and Section \ref{sec:obstructed} (for $k=-1$) yield
\begin{equation}\label{eq:morphismlocalcoeff}
	\widetilde \lambda_{(W,S,c),\gamma} : \widetilde C(Y,L,p;\Delta_n)	\longrightarrow \widetilde C(Y',L',p';\Delta_n)	
\end{equation}
where $\gamma$ is some path in $W$ from the basepoint $p$ to $p'$; as usual, these are often omitted from notation. The map $\widetilde \lambda_{(W,S,c),\gamma}$ is a height $k$ morphism of $\cS$-complexes over $\sS_n$. In particular, all of the relations in Sections \ref{sec:links} and \ref{sec:obstructed} for cobordisms carry through unchanged. In the case that $S$ is non-orientable, to construct \eqref{eq:morphismlocalcoeff}, we again follow \cite{DS1} but with modifications from \cite[Subsection 2.4]{km-barnatan}.

As an ingredient in the construction of cobordism maps from earlier, recall from Subsection \ref{unob-cob-map} that $\eta_i(W,s,c):\sfR(Y,L)\to \sfR(Y',L')$ was defined by letting 
\begin{equation}\label{eq:morphismlocalcoeffred}
	\langle \eta_i(W,S,c),\theta_{\fo},\theta_{\fo'}\rangle
\end{equation}
equal a signed count of index $2i-1$ reducible instantons with reducible limits $\theta_\fo$ and $\theta_{\fo'}$. In the local coefficient setting, \eqref{eq:morphismlocalcoeffred} is an element of the ring $\sS_n$ which in general is a sum of monomials in the variables $T_i$. In particular, each reducible instanton contributes one such monomial, the powers of which are determined by monopole numbers, with a coefficient of $\pm 1$ determined by orientation conventions. There is a cohomological formula for \eqref{eq:morphismlocalcoeffred}, generalizing the case for knots as given in \cite{DS2}. However, we will not need an explicit formula for what follows.

Suppose $\sT$ is a module over $\sS_n$. Define a local coefficient system $\Delta_{n,\sT}=\Delta_n\otimes \sS_n$ by base change. An important special case is when $\sT$ is the ring
\begin{equation}\label{eq:coeffringt}
	\sS = \Z[ T^{\pm 1}]
\end{equation}
The module structure over $\sS_n$ is defined by $T_i\mapsto T$. We write $\Delta := \Delta_{n,\sS}$. Thus
\[
	\widetilde C(Y,L;\Delta)
\]
is an $\cS$-complex over $\sS$. The benefit of $\Delta$ is that there is functoriality for all cobordisms $(W,S)$ with bundle data trivial near $S$, not just component-wise cobordisms. Thus we obtain morphisms \eqref{eq:morphismlocalcoeff} without the constraint that $S$ is of the form \eqref{eq:componentwise}. Note that for $\alpha$ a reducible, $\Delta_\alpha=\sS$, as follows from \eqref{eq:loccoefffiber} and \eqref{eq:holofreduciblesum}.

The cobordisms that appear in the exact triangles induce maps with respect to the local coefficient systems $\Delta$. Consider a skein triple $L$, $L'$, $L''$ where $|L|-1=|L'|=|L''|$. Suppose $L$ has $n+1$ components. Identify the $n-1$ components of the links which do not interact with the skein region, and label them by $\{1,\ldots,n-1\}$. The remaining component for $L',L''$ is labelled by $n$, while $L$ has two remaining components, labelled by $n,n+1$. For $L'$ and $L''$ use the local system $\Delta_n$, and for $L$ use the one obtained from $\Delta_{n+1}$ by identifying $T_{n}=T_{n+1}$; also call this $\Delta_n$. The following implies Theorem \ref{thm:localcoeffintro}.

\begin{theorem}\label{thm:localcoeffgen}
	There are exact triangles just as in Theorem \ref{thm:exacttriangles} but where each $\cS$-complex has the local coefficient system $\Delta_{n}$ over the ring $\sS_n$, where $n=\min\{|L|,|L'|,|L''|\}$.
\end{theorem}

\begin{proof}
The saddle cobordisms involved in the skein triangle relate $L_n$ and $L_{n+1}$ to $L'_n$ (resp. $L''_n$) by a connected component, and this necessitates the use of the local systems $\Delta_n$ described above. With this said, all of the maps in the proof of Theorem \ref{thm:exacttriangles} may be defined in the context of these local coefficient systems, and most of the proof adapts without changes. We discuss some of the more salient aspects of this adaptation.	
	
	First, recall that the map $J:\sfR\to \sfR''$ counts the instantons $([A],[B])$ in compactified moduli spaces $M^+(I\times Y, T;\theta_\fo,\theta_{\fo''})$ defined with respect to the metric broken along $(S^3,U_1)$, where $[A]$ is a minimal reducible instanton on $(I\times Y\setminus B^4, \overline S''\setminus B^2)$ and $[B]$ is an instanton on $(B^4,F_0)$ of index $0$, each with the appropriate cylindrical end metrics. Note that because there is only one instanton on $(B^4,F_0)$ of index $0$, this is the same as counting the minimal reducibles $[A]$ on $(I\times Y\setminus B^4, \overline S''\setminus B^2)$, up to a sign. See Remark \ref{rmk:signforjmaps}. In the setting of local coefficients, we note that the one instanton $[B]$ on $(S^4,\bR\bP^2)$ of index $0$ has non-trivial monopole number (see \cite[Section 2.7]{KM:unknot}). Thus it is important that $J$ is defined in the way described, to avoid powers of $T_n$ from entering some relations that appear in Section \ref{sec:proofs}. Similar remarks for $J'$ and $J''$.	
	
	Recall from Section \ref{sec:proofs} that a {\emph{sign-change}} is an automorphism of $\sfR=\sfR(Y,L)$ that is a diagonal matrix with entries in $\{1,-1\}$ with respect to this natural basis of quasi-orientations. In the local coefficient setting at hand, sign-changes should be generalized to automorphisms that are diagonal matrices with entries monomials in $\sS_n$ with coefficient $\pm 1$. Relations such as \eqref{eq:caseiredrels} should be understood as equalities up to these more general automorphisms.
	
	Finally, we remark on the map $\mathbf{N}:(C(Y,L),d)\to (C(Y,L),-d)$, defined using moduli spaces $M^+(I\times Y,V;\alpha,\beta)_0^{G_5}$ where $G_5$ is the edge of the pentagonal metric family $G_V$ which consists of metrics broken along the two-component unlink $(S^3,U_2)$. Kronheimer and Mrowka show in \cite[\S 7.3]{KM:unknot} that $\mathbf{N}$ is chain homotopic to an isomorphism. (Note that we use different orientation and sign conventions.) We explain how this step, used several times in the proof of Theorem \ref{thm:exacttriangles}, adapts to the setting of local coefficients.
	
	We recall the relevant ingredients. No perturbation is used for the Chern--Simons functional on $(S^3,U_2)$, so that the critical set $\fC(S^3,U_2)$ is an interval. Recall the decomposition \eqref{eq:twocompunlinkdecomp}, where $F_2$ is a copy of $\bR\bP^2$ with two disks removed. There is a map 
	\begin{equation}\label{eq:endpointmapfornmap}
		M^+(B^4, F_2)^{G_5}_1\to \fC(S^3,U_2)
	\end{equation}
	where the domain is the $1$-dimensional moduli space of instantons on $(B^4, F_2)$ for the restriction of the metric family $G_5$; the limiting critical points are unconstrained, and \eqref{eq:endpointmapfornmap} sends an instanton to its flat limit on $(S^3,U_2)$. Kronheimer and Mrowka argue that \eqref{eq:endpointmapfornmap} has degree $\pm 1$. More precisely, the domain of \eqref{eq:endpointmapfornmap} is an interval and possibly some $S^1$ components, and the endpoints of the interval are sent bijectively to those of $\fC(S^3,U_2)$, while all $S^1$ components are sent to the interior.
	
	This description of \eqref{eq:endpointmapfornmap} allows one, for the purposes of defining $\mathbf{N}$, to replace $(B^4,F_2)$ (with its $1$-parameter metric family) in a fiber product description of $M^+(I\times Y,V;\alpha,\beta)_0^{G_5}$ with $(B^4,F)$ where $F$ is a standard annulus (with a single metric). The 4-manifold after this replacement is the cylinder $(I\times Y, I\times L)$, equipped with a broken metric. A chain homotopy induced by a path of metrics from the new broken metric to the cylindrical one completes the argument. We refer to \cite[\S 7.3]{KM:unknot} for the remaining details. 
	
	In the local coefficient case, the above argument adapts to show that the map $\mathbf{N}$ is chain homotopic to an isomorphism. Indeed, the instantons in the interval component of $M^+(B^4, F_2)^{G_5}_1$ all have the same monopole number, as they belong to a connected component of the configuration space. Thus, in the replacement step above, the chain homotopy goes from $\mathbf{N}$ to some power of $T_n$, determined by the monopole number, times the identity map up to a sign-change. The circle components of $M^+(B^4, F_2)^{G_5}_1$ do not play an important role, in that for each such component, there is a signed count of zero. 
	
	To make this last aspect of the proof more direct, one may argue that the choices can be made such that $M^+(B^4, F_2)^{G_5}_1$ is itself an interval, and \eqref{eq:endpointmapfornmap} is a diffeomorphism. To see this, one may take double branched covers and observe that this is the case for the analogous proof of Floer's exact triangle in the non-singular setting \cite[Lemma 5.5]{scadutothesis}.	
	
	Similar remarks hold for the analogous maps $\mathbf{N}'$ and $\mathbf{N}''$.
\end{proof}

We end this subsection with a brief discussion regarding the Chern--Simons filtration on $\mathcal{S}$-complexes of links.  Following \cite{DS2}, if we fix one reducible $\theta_\fo$, each connection $A$ on $\R\times (Y,L)$ with limit $\alpha\in \fC_\pi$ at $-\infty$ and $\theta_\fo$ at $\infty$ determines a homotopy class of paths $\widetilde \alpha$ in the configuration space for $(Y,L)$. This is called a {\emph{lift}} of $\alpha$. We define
\begin{gather}
	\text{gr}[\fo]_\Z(\widetilde \alpha) = \text{ind}(A) + \dim\text{Stab}(\alpha),\label{eq:zgrdef}\\
	 \text{gr}[\fo]_I(\widetilde \alpha) = 2\kappa(A).
\end{gather}
where $\kappa(A)$ is the energy. The quantity $\text{gr}[\fo]_\Z(\widetilde \alpha)$ acts as a homological (Floer) grading and reduces modulo $4$ to \eqref{eq:mod4grdef}. On the other hand, $\text{gr}[\fo]_I(\widetilde \alpha)$, called the $I$-grading, is the Chern--Simons invariant of $\widetilde\alpha$. This structure fits into a $\Z\times \R$-graded $\cS$-complex in which the differential is filtered with respect to $\text{gr}[\fo]_I$. See \cite{DS2} for details.

Proposition \ref{prop:redind} computes \eqref{eq:zgrdef} for $\alpha=\theta_{\fo'}$ to be $8 \kappa(A)$. Indeed, in the index formula, $\frac{1}{2}S\cdot S + \sigma(Y,L,\fo) - \sigma(Y,L,\fo')$ vanishes, see for example \cite{kauffman-taylor}, specifically the proof of Theorem 3.3. Thus for lifts of reducibles, we have
\begin{equation}
	\text{gr}[\fo]_\Z(\widetilde \theta_{\fo'}) = 4\, \text{gr}[\fo]_I(\widetilde \theta_{\fo'}).
\end{equation}
Together with Proposition \ref{prop:grdiff} this gives the possible $\Z\times \R$-gradings of reducibles with respect to the fixed reducible $\theta_{\fo}$. Note in particular that $\text{gr}[\fo]_I(\widetilde \theta_{\fo'})\in \frac{1}{2}\Z$. In the next subsection, this structure is determined for the alternating torus links.

\subsection{Alternating torus links}

We now compute the $\cS$-complex of the torus link $T_{2,2k}$ over the ring $\sS=\Z[T^{\pm 1}]$, using an exact triangle and knowledge of the complex for $T_{2,2k-1}$. Then we compute the maps $s$, introduced in Definition \ref{defn:smap}, for these torus links.

Let $k\in \Z_{> 0}$. Then we have an unoriented skein triple given by
\[
	L = T_{2,2k}, \quad L' = U_1, \quad L'' = T_{2,2k-1}
\]
obstained by resolving any of the crossings in a standard diagram for $T_{2,2k}$. We have $\sigma(T_{2,2k-1})=-2k+2$ and $\sigma(T_{2,2k},\fo_+)=-2k+1$, $\sigma(T_{2,2k},\fo_-)=1$. Here $\fo_+$ (resp. $\fo_-$) is the quasi-orientation which makes the crossings of $T_{2,2k}$ positive (resp. negative). From these computations we see that the exact triangle of Theorem \ref{thm:localcoeffintro} is in Case I of Figure \ref{fig:exacttriangles}:
            \begin{equation}\label{eq:exacttrianglealttoruslink}
                \begin{tikzcd}[column sep=2ex, row sep=5ex, fill=none, /tikz/baseline=-10pt]
            &  \widetilde C(T_{2,2k};\Delta)  \arrow[dr, "\widetilde \lambda "]  & \\
           \widetilde C(T_{2,2k-1};\Delta)    \arrow[ur, "\widetilde \lambda'' "] & & \widetilde  C(U_1;\Delta)   \arrow[ll, "\widetilde \lambda' "]
            \end{tikzcd}
        \end{equation}
        We assume throughout that we are working with $\cS$-complexes defined, using local coefficients, over $\sS=\Z[T^{\pm 1}]$. 
        
        Let $\lambda',\mu',\Delta_1',\Delta_2',\rho'$ be the components of $\widetilde \lambda'$. As $\widetilde C(U_1;\Delta)$ has no irreducible part, $\lambda',\mu',\Delta_1'$ all vanish. Because $\widetilde \lambda'$ comes from a non-orientable saddle cobordism with trivial singular bundle data, which supports no reducible instantons, we have $\rho'=0$. Next, we note that $\widetilde C(T_{2,2k-1};\Delta)$, defined using zero perturbation and a spherical orbifold metric, has irreducible complex supported in odd gradings (see \cite[Subsection 3.2]{DS2}). As $\widetilde \lambda'$ is of odd degree and $\widetilde C(U_1;\Delta)$ has only a reducible in even degree, $\Delta_2'=0$. Thus $\widetilde \lambda'=0$. 

The exact triangle \eqref{eq:exacttrianglealttoruslink} gives an $\cS$-chain homotopy equivalence between $\widetilde C(T_{2,2k};\Delta)$ and the mapping cone $\cS$-complex of $\widetilde \lambda'$. Having established $\widetilde \lambda'=0$, the latter is simply a direct sum of $\cS$-complexes. Thus we have an $\cS$-chain homotopy equivalence
\[
	\widetilde C(T_{2,2k};\Delta) \simeq \widetilde C(T_{2,2k-1};\Delta) \oplus \widetilde C(U_1;\Delta)
\]	
With the computation of the $\widetilde C(T_{2,2k-1};\Delta)$ from \cite[Proposition 3.10]{DS2}, we obtain:

\begin{prop}\label{prop:toruslinkcomplex}
The $\cS$-complex $\widetilde C(T_{2,2k};\Delta)$, with local coefficient system defined over $\sS=\Z[T^{\pm 1}]$, is homotopy equivalent to the $\cS$-complex
\[
	\widetilde C = C \oplus C[-1] \oplus \mathsf{R}, \qquad C = \bigoplus_{i=1}^{k-1} \sS\cdot \xi^i, \qquad \mathsf{R} = \sS\cdot \theta_{\fo_+}\oplus \sS\cdot \theta_{\fo_-}
\]
where the differential $\widetilde d$ has components $d=\delta_2=0$ and
\begin{gather*}
	\delta_1(\xi^1) = (T^2-T^{-2})\theta_{\fo_+}, \qquad v(\xi^i) = (T^2-T^{-2})\xi^{i-1} \;\; (2\leq i \leq k-1)
\end{gather*}
with all other components of $\delta_1$ and $v$ equal to zero.
\end{prop}

	This computation can easily be extended to the include the Chern--Simons filtration structure, using the conventions given in the previous subsection. Write $\zeta^i$ for a path of flat connections from $\xi^i$ to the reducible $\theta_{\fo_+}$. These choices are made such that
	\[
		\text{gr}[\fo_+]_\Z(\zeta^i)= 2i-1+\epsilon, \qquad \text{gr}[\fo_+]_I(\zeta^i) = i^2/2k
	\]
	where $\epsilon=1$ if $i\in\{0,k\}$ and is $0$ otherwise. Here $\zeta^0$ is the trivial path of connections from $\theta_{\fo_+}$ to itself, and $\zeta^k$ is a particular path of connections from $\theta_{\fo_-}$ to $\theta_{\fo_+}$. This follows from the computations of \cite{austin}, which in particular includes a criterion for when moduli spaces of instantons on $L(p,1)\times \R$ of low expected dimension are non-empty. These results are related to the singular instanton moduli for $(S^3,T_{2,p})\times \R$ following the discussion in \cite[Section 9]{DS1}, which passes through double branched covers.
 
\begin{remark}\label{rmk:flipdiscussion}
The exact triangle does not directly compute the $\cS$-complex $\widetilde C(T_{2,2k};\Delta_{\sS_2})$ over the more general ring $\sS_2=\Z[T_1^{\pm 1}, T_2^{\pm}]$. In this case, we have
\begin{equation}\label{eq:t1t2formula}
	(\delta_1+v)(\xi^i) = (T_1T_2 - T_1^{-1}T_2^{-1})\xi^{i-1} \pm  (T_1T_2^{-1} - T_1^{-1}T_2)\xi^{i+1}
\end{equation}
for $1\leq i\leq k-1$, after possibly renormalizing the generators by units, while still $\delta_2=d=0$. Here we have written $\xi^0=\theta_{\fo_+}$ and $\xi^k=\theta_{\fo_-}$. If one sets $T_1=T_2=T$, the computation of Proposition \ref{prop:toruslinkcomplex} is recovered. This more symmetric complex defined over $\sS_2$ is a consequence of certain ``flip-symmetries'' that act on the $\cS$-complex. This structure will be further developed elsewhere. $\diamd$
\end{remark}

We now compute the $s$-map of Definition \ref{defn:smap} for the Hopf link $H=T_{2,2}$. 

\begin{prop}\label{prop:smaphopf}
	With local coefficient system defined over $\sS=\Z[T^{\pm 1}]$, the Hopf link $\cS$-complex is given by $\widetilde C(H;\Delta)=\mathsf{R}(H)= \sS\cdot\theta_{\fo_+}\oplus \sS\cdot\theta_{\fo_-}$ with $\widetilde d=0$. In this basis, the map $s:\mathsf{R}(H)\to \mathsf{R}(H)$ of Definition \ref{defn:smap} is given, up to a unit of $\sS$, by the matrix
	\[
		\left[\begin{array}{cc} 0 & T^2-T^{-2} \\ 0 & 0  \end{array} \right]
	\]
\end{prop}

\begin{proof}
We have a skein triple $H,T_{2,3},U_1$. Inspection of the signatures shows that the associated exact triangle of $\cS$-complexes is in Case III of Figure \ref{fig:exacttriangles}:
            \begin{equation}\label{eq:exacttrianglehopftrefoil}
                \begin{tikzcd}[column sep=2ex, row sep=5ex, fill=none, /tikz/baseline=-10pt]
            &  \widetilde C(H;\Delta)  \arrow[dr, "\widetilde \lambda "]  & \\
           \Sigma^{-1}\widetilde C(U_1;\Delta)    \arrow[ur, "\widetilde \lambda'' "] & & \widetilde  C(T_{2,3};\Delta)   \arrow[ll, "\widetilde \lambda' "]
            \end{tikzcd}
        \end{equation}
Now, $\widetilde C(H;\Delta)$ has no irreducible part and two reducible generators, $\theta_{\fo_+}$ and $\theta_{\fo_-}$. Here $\fo_+$ is the quasi-orientation of $H$ which makes the crossings positive. Further, $\widetilde C(T_{2,3};\Delta)$ has one irreducible generator $\xi^1$ and one reducible generator $\theta_{\fo}$.

The morphism $\widetilde \lambda$ in the triangle \eqref{eq:exacttrianglehopftrefoil} has $\rho(\theta_{\fo_+})=\theta_\fo$ and $\rho(\theta_{\fo_-})=0$. The only other possible non-zero component of $\widetilde \lambda$ is $\Delta_2$. For grading reasons, $\Delta_2(\theta_{\fo_+})=0$, and we can write $\Delta_2(\theta_{\fo_-})=c \xi^1$ for some ring element $c\in \sS$.

Next, $\Sigma^{-1}\widetilde C(U_1;\Delta)$ has one reducible generator $\theta$, and one irreducible generator $\xi$, and $\delta_2(\theta)=\xi$. The exact triangle \eqref{eq:exacttrianglehopftrefoil} implies $\Sigma^{-1}\widetilde C(U_1;\Delta)$ is $\cS$-homotopy equivalent to $\text{Cone}(\widetilde \lambda)$, the mapping cone of $\widetilde \lambda$, up to a grading shift. Thus $\text{Cone}(\widetilde \lambda)$ must trivial total homology. From this it follows that $c$ is a unit, so that $c= \pm T^n$ for some $n\in \Z$.

 Writing $(W,S)$ for the cobordism defining $\widetilde \lambda$, and considering the ends of the moduli space $M(W,S;\theta_{\fo_-},\theta_{\fo_+})_1$, we arrive at a special case of \eqref{rel-tau-i+1}:
\[
	\langle \delta_1\Delta_2(\theta_{\fo_-}), \theta_{\fo_+}\rangle  = \langle \rho s(\theta_{\fo_-}),\theta_{\fo_+}\rangle,
\]
This implies $\langle s(\theta_{\fo_-}), \theta_{\fo_+}\rangle=\pm T^n (T^2-T^{-2})$ for some $n\in \Z$.

Write $(W'',S'')$ for the cobordism defining $\widetilde \lambda''$ in \eqref{eq:exacttrianglealttoruslink} with $k=1$. Thus $(W'',S'')$ is a cobordism from $U_1$ to $H$. Considering the ends of the moduli space $M(W'',S'';\theta,\theta_{\fo_-})_1$ gives $\langle  s\rho''(\theta),\theta_{\fo_-}\rangle=0$, and as $\rho''(\theta)=\theta_{\fo_+}$, we obtain $s(\theta_{\fo_+})=0$.
\end{proof}

Similar to Remark \ref{rmk:flipdiscussion}, when the more general coefficient ring $\sS_2$ is used, the $s$-map for the Hopf link is in fact non-zero for both $\theta_{\fo_+}$ and $\theta_{\fo_-}$.
        
      For the $(2,2k)$ torus link with $|k|>1$, the map $s:\mathsf{R}(T_{2,2k})\to \mathsf{R}(T_{2,2k})$ is zero, defined using any coefficient ring. Indeed, there are no nonempty $0$-dimensional moduli spaces $\breve M(\theta_{\fo_{\pm}},\theta_{\fo_{\pm}})_0$ with distinct limits. This follows from a correspondence parallel to \cite[Section 9]{DS1} and computations for instantons on $\R\times L(2k,1)$ from \cite{austin}.
       
       \begin{remark}\label{rmk:twobridgecomps}
	The authors expect that the $\cS$-complex and $s$-maps for $T_{2,2k}$, and more generally for any two-bridge link, can be computed from equivariant ADHM data following arguments parallel to the case of knots as in \cite{DS1,DS2}, and using \cite{austin-lens}. $\diamd$
\end{remark}

\subsection{Fr\o yshov-type invariants for links}\label{subsec:froyshovinvs}

In \cite{DS1} the authors introduced Fr\o yshov-type invariants for knots in integer homology 3-spheres defined using the singular instanton $\cS$-complex associated to the knot. These invariants were further investigated in \cite{DS2}. Here we adapt this material to the more general setting of links with non-zero determinant.

For a knot, there is a unique reducible critical point for the Chern--Simons functional, with corresponding rank 1 reducible summand in the associated $\cS$-complex. Roughly, the Fr\o yshov invariant for the knot measures the interaction of this reducible with the irreducibles in the $\cS$-complex. For a link there is no longer a unique reducible, but one for every quasi-orientation. We would like to extract algebraic invariants from this more general situation that, like in the case for knots, lead to concordance invariants. 

To proceed, we begin with some algebra. Let $\widetilde C = C \oplus C[-1] \oplus \mathsf{R}$ be an $\cS$-complex, where $\mathsf{R}$ is a free module over the ground ring $R$. We assume that $R$ is an integral domain. As in Subsection \ref{subsec:equivhomgrps}, we have a long exact sequence involving the homology groups of the small equivariant complexes \eqref{eq:smalleqcomplexes}:
\begin{equation*} 
                \begin{tikzcd}[column sep=4ex, row sep=8ex, fill=none, /tikz/baseline=-10pt]
            	\cdots \arrow[rr, "p_\ast"] && \widecheck H \arrow[rr, "j_\ast"]  \arrow[rr]  &&\widehat H \arrow[rr, "i_\ast"] && \sfR[\![x^{-1},x] \arrow[rr, "p_\ast"] & & \cdots
            \end{tikzcd}
            \end{equation*}	
We denote by $\fI$ the image of $i_\ast$, or equivalently the kernel of $p_\ast$. Thus $\fI$ is an $R[x]$-submodule of $\sfR[\![x^{-1},x]=\mathsf{R}\otimes_R R[\![x^{-1},x]$. In the case that $\mathsf{R}$ is the ground ring $R$, the Fr\o yshov invariant of the $\cS$-complex $\widetilde C$ is defined by
\[
	h(\widetilde C) = -\text{min}\{ \deg Q(x) : \; Q(x) \in \fI \} \in \Z
\]
In the more general case, we proceed as follows. First, define
\[
	J_i(\widetilde C ) = \{ a_0 \in \mathsf{R} : \; \exists a_0 x^{-i} + a_{-1} x^{-i-1} + \cdots \in \fI \} \subset \mathsf{R}
\]
Then $J_i(\widetilde C)$ is an $R$-submodule of $\mathsf{R}$. Furthermore, there are inclusions
\[
	\cdots \subset J_{i+1}(\widetilde C) \subset J_i(\widetilde C) \subset J_{i-1}(\widetilde C) \subset \cdots \subset \mathsf{R}
\]
which follow because $\fI$ is an $R[x]$-module. This nested sequence of modules is a natural generalization of the nested sequence of ideals considered in \cite[Subsection 4.7]{DS1}.

Given a height $0$ morphism $\widetilde \lambda:\widetilde C\to \widetilde C'$, we obtain
\begin{equation}\label{eq:jiinclusion}
	\rho (J_i(\widetilde C)) \subset J_i(\widetilde C')
\end{equation}
where $\rho:\mathsf{R}\to \mathsf{R}'$ is the reducible component of $\widetilde\lambda$. The inclusion \eqref{eq:jiinclusion} follows from the diagram in Proposition \ref{prop:equivlevel}, along with the description of $\overline{\lambda}_\ast$ there (all of which hold for general $\cS$-complexes, not just those coming from links).

From Proposition \ref{prop:susequivar} we have the following relation for suspension:
\[
	J_i(\Sigma \widetilde C)  = J_{i-1}(\widetilde C)
\]
Then Propositions \ref{prop:levelsuspension} and \ref{decomp-neg-level}  imply that for a height $n$ morphism, $\tau_n(J_i(\widetilde C))\subset J_{i-n}(\widetilde C')$ where $\tau_n$ is as in Definition \ref{level--n-def}. It is also straightforward to verify
\[
	J_i(\widetilde C ) \otimes J_j(\widetilde C') \subset J_{i+j}(\widetilde C  \otimes\widetilde C') 
\]
For direct sums of $\cS$-complexes we have the simple relation 
\[
	J_i(\widetilde C \oplus \widetilde C') = J_{i}(\widetilde C ) \oplus J_i( \widetilde C') 
\]
For dual $\cS$-complexes, there is in general an inclusion
\[
	J_i(\widetilde C^\dagger) \subset \text{Ann}_R J_{-i}(\widetilde C) 
\]
where the right hand side is the annihilator of the $R$-module $J_{-i}(\widetilde C)$; this inclusion is equality if $R$ is a field. In particular, the sum of the ranks of $J_i(\widetilde C^\dagger)$ and $J_{-i}(\widetilde C) $ equal the rank of $\mathsf{R}$. These properties are straightforward generalizations of properties described in \cite[Subsection 4.7]{DS1}.

Given an $\cS$-complex $\widetilde C$ over $R$ we define a function
\[
	d_{\widetilde C}:\Z \to \Z_{\geq 0}
\]
\[
	d_{\widetilde C}(i ) = \text{rk}_R \left( J_i( \widetilde C) \right)
\]
The following result organizes some properties of $d_{\widetilde C}$, all of which follow from the properties of the ideals $J_i(\widetilde C)$ discussed above.

\begin{prop}
Let $\widetilde C=C \oplus C[-1]\oplus \mathsf{R}$ be an $\cS$-complex over an integral domain $R$. Then the associated function $d_{\widetilde C}:\Z\to \Z_{\geq 0}$ satisfies the following:
\begin{itemize}
	\item[\rm{(i)}] $d_{\widetilde C}(i)= 0$ for $i \gg 0$.
	\item[\rm{(ii)}] $d_{\widetilde C}(i)\geq d_{\widetilde C}(i+1)$.
	\item[\rm{(iii)}] $d_{\widetilde C}(i)={\rm{rk}}_R(\mathsf{R})$ for $i\ll 0$.
	\item[\rm{(iv)}] $d_{\Sigma^n \widetilde C}(i) = d_{\widetilde C}(i-n)$.
	\item[\rm{(v)}] if there is a strong height $n$ morphism $\widetilde \lambda :\widetilde C \to \widetilde C'$, then $d_{\widetilde C}(i) \leq d_{\widetilde C'}(i-n)$.
	\item[\rm{(vi)}] $d_{\widetilde C}(i) d_{\widetilde C'}(j) \leq d_{\widetilde C\otimes \widetilde C'}(i+j)$.
	\item[\rm{(vii)}] $d_{\widetilde C \oplus \widetilde C'}(i) = d_{\widetilde C}(i) + d_{\widetilde C'}(i)$.
	\item[\rm{(viii)}] $d_{\widetilde C^\dagger}(i) = {\rm{rk}}_R( \mathsf{R} ) - d_{\widetilde C}(-i)$. 
\end{itemize}
\end{prop}

\noindent Note that when $\mathsf{R}=R$, the function $d_{\widetilde{C}}$ is a non-increasing function on $\Z$ with values in $\{0,1\}$. The Fr\o yshov invariant of $\widetilde C$ in this case is given by
\begin{equation}\label{eq:froyshovfromgeneralconst}
		h(\widetilde C)  = \max\{ i: \;  d_{\widetilde C}(i) = 1 \}
\end{equation}
and this completely determines $d_{\widetilde C}$. It is in this sense that the function $d_{\widetilde C}:\Z\to \Z_{\geq 0}$ is a generalization of the Fr\o yshov invariant to $\cS$-complexes with reducible summand $\mathsf{R}$ not necessarily of rank $1$. Note that identity \eqref{eq:froyshovfromgeneralconst} and property (v) above yield the following, a relation mentioned earlier in Subsection \ref{subsec:suspensiondefn}:
\[
	h(\Sigma^n \widetilde C ) = h(\widetilde C) + n.
\]

We now apply the above construction to the singular instanton $\cS$-complexes of links. 

\begin{definition}
	Let $L\subset Y$ be a based link with non-zero determinant in an integer homology $3$-sphere, and let $n=|L|$ be the number of components. Given an integral domain $\sS$ which is an algebra over the universal ring $\sS_n$, we define 
	\[
		d_{(Y,L)}^\sS:\Z\to \Z_{\geq 0}
	\]
	to be the function $d_{\widetilde C}$ associated to the $\cS$-complex $\widetilde C = \widetilde C(Y,L;\Delta_\sS)$. $\diamd$
\end{definition}

Recall that two links $L$ and $L'$ in the $3$-sphere are {\emph{concordant}} if there is an embedding $S\subset [0,1]\times S^3$ with $\partial S=-L\sqcup L'$, and such that $S$ is a disjoint union of annuli, with the boundary of each annulus containing one component each of $L$ and $L'$. This notion generalizes as follows. We call $(W,S):(Y,L)\to (Y',L')$ a {\emph{homology concordance}} of links if $W$ is a homology cobordism $Y\to Y'$ and $S\subset W$ is a disjoint union of embedded annuli satisfying the same conditions as above. A homology concordance of {\emph{based}} links is required to have an annulus which connects the components of $L$ and $L'$ that have basepoints.

\begin{prop}
	The function $d_{(Y,L)}^\sS:\Z\to \Z_{\geq 0}$ only depends on the homology concordance class of the based link $(Y,L)$.
\end{prop}

\begin{proof}
	A homology concordance between based links $(Y,L)$ and $(Y',L')$ induces
	\[
		\widetilde C(Y,L;\Delta_\sS)\to \widetilde C(Y',L';\Delta_\sS),
\] 
which is strong height $0$ morphism. By property (v) above, $d_{(Y,L)}^\sS(i)\leq d_{(Y',L')}^\sS(i)$ for each $i\in \Z$. The reverse cobordism gives a strong height $0$ morphism in the other direction, and gives the reverse inequality.
	\end{proof}

As mentioned earlier, the invariants defined above for the case of knots, in the guise of the Fr\o yshov invariants \eqref{eq:froyshovfromgeneralconst}, were studied in \cite{DS1,DS2}. We expect that many of the techniques and computations done in those works can be adapted to this more general setting. Here we only mention the following application of the unoriented skein exact triangles.
	
	To state the result, we introduce the following terminology. Let $\widetilde C=C\oplus C[-1]\oplus \sfR$ be an $\cS$-complex over a ring $R$. We say that $d_{\widetilde C}$ is {\emph{trivial}} if
	\[
		d_{\widetilde C}(i) = \begin{cases}  0 & i > 0\\  \text{rk}_R\left( \sfR\right) & i \leq  0\end{cases}
	\]
	For example, if $\widetilde C$ is an $\cS$-complex with irreducible summand $C=0$, then $d_{\widetilde C}$ is trivial. 
	
	In what follows, the $\sS_n$-module $\sS=\Z$ is obtained by setting each $T_i=1$; this is the setting of $\cS$-complexes of links with ordinary $\Z$ coefficients.  

\begin{prop}\label{prop:alternatingfroyshov}
	Let $L\subset S^3$ be a non-split alternating link, or more generally a quasi-alternating link. Then the invariant $d_L^\Z$ is trivial.
\end{prop}

\begin{proof}
	Given an $\cS$-complex $\widetilde C=C\oplus C[-1]\oplus \sfR$ over $R$, a straightforward verification from the definitions shows the following: if $(\delta_2)_\ast:\sfR\to H_\ast(C,d)$ and $(\delta_1)_\ast:H_\ast(C,d)\to \sfR$ are zero, then the invariant $d_{\widetilde C}:\Z\to \Z_{\geq 0}$ is trivial. We then utilize Theorem \ref{thm:quasialtintro}, which is proved in the next section as Theorem \ref{thm:qa} using the unoriented skein exact triangles. This theorem implies that the spectral sequence \eqref{eq:ssirrtonatintro} collapses at the $E_2$-page. The differential on the $E_2$-page is given by the maps $(\delta_1)_\ast$ and $(\delta_2)_\ast$. The proposition follows.
\end{proof}

\noindent Proposition \ref{prop:alternatingfroyshov} generalizes \cite[Proposition 17]{DS2}, which is for knots. The argument given above, which relies on the unoriented skein exact triangles of Theorem \ref{thm:irredexacttrianglesintro}, provides an alternative proof, in the case of knots, to the one given in the above citation. 

The invariant $d_L^\Z$ in the case that $L$ is a knot is in general non-trivial and is related to the unoriented 4-ball genus. As explained in \cite{DS2}, the invariant appears to be an instanton analogue of invariants defined in the context of Heegaard Floer theory \cite{oss} and Khovanov homology \cite{ballinger}.

\newpage


\section{Applications to irreducible instanton homology} \label{sec:comps}

In this section we use the skein exact triangles to study $I(L)$, the irreducible instanton homology of links. In the first subsection, we review and expand upon some aspects of Theorem \ref{thm:irredexacttrianglesintro}. As a corollary, we verify that the Euler characteristic of $I(L)$, for $L$ a non-zero determinant link in the 3-sphere, is a multiple of the Murasugi signature. In the next subsection, we compute $I(L)$ for quasi-alternating links. In the remaining subsections, we study knots and links which are $I$-basic; these have the rank of $I(L)$ equal to the Euler characteristic, up to sign. After some general results about $I$-basic links, we show that certain pretzel knots and twisted torus knots are $I$-basic.

\subsection{Exact triangles for irreducible homology}\label{eq:irrexacttrisubsec}

Theorem \ref{thm:irredexacttrianglesintro} gives the unoriented skein exact triangles for irreducible instanton homology with trivial singular bundle data, with the condition that the links have non-zero determinant and are in the $3$-sphere. These triangles are displayed in Figure \ref{fig:irredexacttrianglesintro}. Note that Theorem \ref{thm:irredexacttrianglesintro} follows algebraically from Theorem \ref{thm:exacttriangles} by taking the irreducible homology, as explained in Subsection \ref{subsec:scomplextriangles}. In this section most links are in $S^3$, although much of what is done below holds inside a general integer homology 3-sphere. We also assume throughout that the coefficient ring is $\Z$, unless stated otherwise. 

On occasion we will also make use of the exact triangles with non-trivial bundles. For the relevant tools, see Figure \ref{fig:exacttriangles-withbundles-irr} and Propositions \ref{prop:degreesfornontrivbundletriangle} and \ref{prop:kmtriangledegsofmaps}.

In what follows, we typically choose an identification of the free abelian group $\mathsf{R}(L)$ generated by quasi-orientations with $\Z^{2^{|L|-1}}$.
The suspension homologies appearing in these triangles have the following concrete descriptions:
\begin{align*}
	I_+(L) = H_\ast\left( C(L) \oplus \Z_{(1)}^{2^{|L|-1}} , \left[ \begin{array}{cc} d & \delta_2 \\ 0 & 0 \end{array}\right] \right)  \cong H_\ast (\Sigma^{+1} C(L)) \\
	I_-(L) = H_\ast\left( C(L) \oplus \Z_{(0)}^{2^{|L|-1}} , \left[ \begin{array}{cc} d & 0 \\ \delta_1 & 0 \end{array}\right] \right)  \cong H_\ast (\Sigma^{-1} C(L))
\end{align*}
See also the exact triangles \eqref{eq:irrhomdeftriangles} from the introduction, both of which follow from these mapping cone descriptions. Here and throughout this section, all homology groups will be considered as $\Z/2$-graded, unless stated otherwise. 

Write $(\delta_1)_\ast$ and $(\delta_2)_\ast$ for the maps induced by $\delta_1$ and $\delta_2$ on homology:
\begin{align*}
  (\delta_2)_\ast: \Z^{2^{|L|-1}}\to I(L) \\
 (\delta_1)_\ast:I(L)\to \Z^{2^{|L|-1}}
\end{align*}
Note that in the special case that $L$ is a knot, the relation $dv-vd - \delta_2\delta_1=0$ implies that one of $(\delta_2)_\ast=0$ or $(\delta_1)_\ast=0$ must hold. Thus in this case, one of the relations $\text{rk}_\Z I_+(L)=\text{rk}_\Z I(L) + 1$ or $\text{rk}_\Z I_-(L)=\text{rk}_\Z I(L) + 1$ holds. The descriptions of $I_+(L)$ and $I_-(L)$ here run parallel to analogously defined instanton homology groups for integer homology 3-spheres studied in \cite[Section 7.1]{donaldson-book}.

We now compute the Euler characteristic of $I(L)$, stated in the introduction as \eqref{eq:irredhomlinkeuler}.

\begin{prop}\label{prop:irreuler}
	If $L$ is a link in $S^3$ with $\det\neq 0$, then we have
	\begin{equation}\label{eq:irreulerchar}
		\chi\left( I(L)\right) = 2^{|L|-2} \xi(L).
	\end{equation}
\end{prop}

\begin{proof}
	The Euler characteristic $\chi(I(L))$ satisfies the skein recursion \eqref{eq:eulercharskeinrel}, as follows from Theorem \ref{thm:irredexacttrianglesintro}. As shown in \cite{karan}, this recursion formula for $\det\neq 0$ links, together with $\chi(I(U_1))=0$, completely determines $\chi(I(L))$. The expression $ 2^{|L|-2} \xi(L)$ satisfies the same recursion and is also zero for the unknot, thus the result follows.
\end{proof}

We expect that the above result generalizes as follows. For any non-zero determinant link $L$ in an integer homology 3-sphere $Y$, we have the relation
\begin{equation}\label{eq:eulmurasugigen}
	\chi\left( I(Y,L) \right) = 2^{|L|-2} \xi(Y,L) + 4\lambda(Y)
\end{equation}
where $\lambda(Y)$ is the Casson invariant of $Y$. This is already established in the case of knots in $Y$ by \cite[Theorem 1.3]{DS1}. It is clear that \eqref{eq:eulmurasugigen} can be proved for many links in a fixed $Y$ using the unoriented skein exact triangles in a way similar to the argument above, using the base case of knots that is known. Following comments in Subsection \ref{subsec:concludingremarks}, it is also possible that this strategy works without the restriction that $\det\neq 0$, upon extending the construction of $I(Y,L)$ to the case of arbitrary links in $Y$. An alternative approach to proving \eqref{eq:eulmurasugigen} is to generalize the arguments of Herald \cite{herald}.

\subsection{Quasi-alternating links}

Ozsv\'{a}th and Szab\'{o} define in \cite{os:branched} the set $\cQ$ of {\emph{quasi-alternating links}} to be the smallest set of links such that the following conditions hold: (i) the unknot is in $\cQ$, and (ii) if $L,L',L''$ is an unoriented skein triple such that $L',L''\in \cQ$ and
	\[
		\det(L) = \det(L') + \det(L'')
	\]
	with $\det(L'), \det(L'')\neq 0$, then $L\in \cQ$. Clearly every quasi-alternating link has non-zero determinant. Examples of quasi-alternating links are non-split alternating links.

For a based quasi-alternating link $L$, Kronheimer and Mrowka \cite{KM:unknot} proved
\begin{equation}\label{eq:inaturalqa}
	I^\natural(L) \cong \Z^{\det(L)}
\end{equation}
In general, identifying $I^\natural(L;\Z)$ with the homology of the $\cS$-complex $\widetilde C=\widetilde C(L;\Z)$ using Theorem \ref{eq:inatural}, the filtration $C[-1] \subset C[-1]\oplus\sfR \subset \widetilde C$ induces a spectral sequence
\begin{equation}
	I(L)\oplus I(L)[-1] \oplus \Z_{(0)}^{2^{|L|-1}} \rightrightarrows I^\natural(L) \label{eq:ssirrtonat}
\end{equation}
already mentioned in the introduction. The following result says that this spectral sequence collapses for quasi-alternating links. In other words, given \eqref{eq:inaturalqa}, the irreducible instanton homology for a quasi-alternating link is as small as possible.

\begin{theorem}\label{thm:qa}
	Let $L$ be a quasi-alternating link. Then the irreducible singular instanton homology $I(L)$ is free abelian of rank $\frac{1}{2}(\det(L)-2^{|L|-1})$.
\end{theorem}

The proof of this theorem is by induction on $\det(L)$ and uses the exact triangles of irreducible instanton homology from Theorem \ref{thm:irredexacttrianglesintro}. We will need:

\begin{lemma}\label{lemma:manoz}
	Let $L,L',L'' \subset Y$ be an unoriented skein triple as in Figure \ref{fig:skein}, with no assumption on the number of components of each link. Assume 
\[
	\det(L) = \det(L') + \det(L'')
\]
and that each determinant is non-zero. Then if the saddle cobordism $L\to L'$ (resp. $L''\to L$) is orientable, we have $\epsilon(L,L') =  1$ (resp. $\epsilon(L'',L)=1$).
\end{lemma}

\begin{proof}
	One case follows directly from \cite[Lemma 3]{manolescu-ozsvath}, attributed to Murasugi, and the other case from an application of the same result to the mirrors of the links.
\end{proof}

\begin{proof}[Proof of Theorem \ref{thm:qa}]
Note that the result is true for the base case of the unknot, the unique quasi-alternating link with determinant $1$. Next, let $L_0$ be a quasi-alternating link, and suppose the theorem is true for quasi-alternating links with determinant less than $\det(L_0)$. 

We remark that \eqref{eq:inaturalqa} and the spectral sequence mentioned in \eqref{eq:ssirrtonat} imply
\begin{equation}\label{eq:ineqfromss}
	\text{rk}_\Z I(L_0) \geq  \frac{1}{2}(\det(L_0) - 2^{|L_0|-1} )
\end{equation}
The link $L_0$ fits into a skein triple $L,L',L''$ of quasi-alternating links with $|L|=|L'|+1=|L''|+1$; this is a quasi-alternating skein triple in that the determinant of $L_0$ is the sum of the other two links. There are now $3$ cases, corresponding to $L_0$ being $L$, $L'$, or $L''$.

First suppose $L_0=L$ so that we have the relation
\[
	 \det(L_0) = \det(L') + \det(L'')
\]
Lemma \ref{lemma:manoz} implies that we have an exact triangle as in Case I of Figure \ref{fig:irredexacttrianglesintro}. Thus by the exact triangle and the induction hypothesis we have
\begin{align*}
	\text{rk}_\Z I(L_0) & \leq \text{rk}_\Z I(L')  + \text{rk}_\Z I(L'') \\
						& = \frac{1}{2}(\det(L') - 2^{|L'|-1} ) +  \frac{1}{2}(\det(L'') - 2^{|L''|-1}) \\
						& = \frac{1}{2}(\det(L_0) - 2^{|L_0|-1} ) 
\end{align*}
Together with \eqref{eq:ineqfromss} this computes the rank of $I(L_0)$; the same argument works for any coefficient ring and this implies the free abelian claim (similar remarks hold below).

Second, suppose $L_0=L'$. In this case we have the relation
\[
	\det(L_0)  = \det(L) + \det(L'')
\]
and $|L_0|=|L''|=|L|-1$. By Lemma \ref{lemma:manoz} we have $\epsilon(L,L_0)=\epsilon(L,L')=1$. Thus we have an exact triangle in Figure \ref{fig:irredexacttrianglesintro} which is Case I or Case III. In either case we have
\begin{align*}
	\text{rk}_\Z I(L_0) & \leq \text{rk}_\Z I(L)  + \text{rk}_\Z I(L'') + 2^{|L''|-1} \\
						& = \frac{1}{2}(\det(L) - 2^{|L|-1} ) +  \frac{1}{2}(\det(L'') - 2^{|L''|-1}) + 2^{|L''|-1} \\
						& = \frac{1}{2}(\det(L_0) - 2^{|L_0|-1} ) 
\end{align*}
With \eqref{eq:ineqfromss} we again conclude the inequality above must be an equality. We remark that as a consequence the exact triangle here must in fact be the one of Case III, not Case I. 

Finally, the case in which $L_0=L''$ is similar to the previous case. One finds that the exact triangle is the one of Case II. (In fact this case is the mirror of the previous case.)
\end{proof}

As the total rank and Euler characteristic for the irreducible instanton homology of quasi-alternating links are determined, we obtain the following graded isomorphism.

\begin{cor}\label{cor:qa}
	Let $L$ be a quasi-alternating link. Then as $\Z/2$-graded abelian groups,
	\[
		I(L) \cong \Z_{(0)}^{\frac{1}{4}\det(L) - 2^{|L|-3}(1-\xi(L))} \oplus \Z_{(1)}^{\frac{1}{4}\det(L) - 2^{|L|-3}(1+\xi(L))} 
	\]
\end{cor}

\noindent  As the degrees of all the maps in the proof of Theorem \ref{thm:qa} are determined by formulas in Section \ref{sec:links}, in principle the $\Z/4$-gradings can also be computed.

\subsection{$I$-basic knots}

In this subsection we introduce and study a class of links which have particularly simple irreducible instanton homology. 

\begin{definition}
	A non-zero determinant link $L$ in an integer homology 3-sphere $Y$ is {\emph{$I$-basic}} if the irreducible instanton homology has the rank of its Euler characteristic:
\[
	\text{rk}_\Z I(Y,L) = 2^{|L|-2} \left| \xi(Y,L)\right|. \quad \diamd
\] 
\end{definition}

\noindent Our primary focus, for simplicity, will be on $I$-basic knots in the 3-sphere; links will only enter through our use of the exact triangles.

\begin{example}
	Torus knots are $I$-basic. This follows because the $SU(2)$-traceless character variety for a torus knot $T_{p,q}$ contains exactly $|\sigma(T_{p,q})/2|$ irreducible points, all of which are non-degenerate, and therefore the complex that computes $I(T_{p,q})$ has rank equal to the absolute value of its Euler characteristic. See also \cite{PS,DS2}. $\diamd$
\end{example}

\begin{example}
	The unknot is the unique knot which is $I$-basic and has zero signature. Indeed, if $I(K)=0$ then \eqref{eq:ssirrtonat} implies $I^\natural(K)=\Z$, and this implies $K$ is the unknot by work of Kronheimer and Mrowka \cite{KM:unknot}. More generally, the non-zero determinant links with $\xi(L)=0$ which are $I$-basic are connected sums of Hopf links (see \cite{zhang-xie-hopf}). $\diamd$
\end{example}

Define $|\Delta_L|$ to be the sum of the absolute values of the coefficients of the (single variable) Alexander polynomial $\Delta_L(t)\in \Z[t^{\pm 1/2}]$ of the link $L$. In \cite{KM:unknot} it is proved that
\begin{equation}\label{eq:kmalexineq}
	\text{rk}_\Z I^\natural (L) \geq |\Delta_L|.
\end{equation}
From the definition of $I$-basic, inequality \eqref{eq:kmalexineq} and the spectral sequence \eqref{eq:ssirrtonat} we have:

\begin{lemma}\label{eq:ibasicineq}
	Suppose a non-zero determinant link $L\subset S^3$ is $I$-basic. Then
\[
	2^{|L|-1}|\xi(L)| \geqslant |\Delta_L| - 2^{|L|-1}
\]
\end{lemma}

\begin{prop}
	Among knots with crossing number at most $12$ the only $I$-basic knots are
\[
	U_1, \quad T_{2,3}, \quad T_{2,5}, \quad T_{2,7}, \quad T_{2,9}, \quad T_{2,11}, \quad T_{3,4}, \quad T_{3,5}, \quad P(-2,3,7)
\]
\end{prop}

\begin{proof}
	The inequality of Lemma \ref{eq:ibasicineq} rules out all knots except for those on the above list. As already mentioned, torus knots are $I$-basic. The pretzel knot $P(-2,3,7)$ is also $I$-basic, as can be seen in the next subsection.
\end{proof}

\begin{prop}
	If an alternating knot $K\subset S^3$ is $I$-basic, then $K$ is an alternating torus knot $T_{2,2k+1}$ for some integer $k$. 
\end{prop}

\begin{proof}
Assume that $K$ is a non-trivial alternating $I$-basic knot. By Theorem \ref{thm:qa} and the definition of $I$-basic, we have an equality
\begin{equation}\label{eq:altibasicproof1}
	|\sigma(K)| = \det(K)-1.
\end{equation}
An inequality of Crowell \cite[4.9]{crowell} says that for any alternating knot, we have
\begin{equation}\label{eq:altibasicproof2}
	\det(K) \geqslant c(K)
\end{equation}
where $c(K)$ is the crossing number of $K$. Let $G$ be the black graph associated to a minimal alternating diagram of $K$. Let $V$ be the number of vertices of $G$. Then
\begin{equation}\label{eq:altibasicproof3}
	\sigma(K) = V - P -1
\end{equation}
where $P$ is the number of positive crossings in the diagram; see \cite{lee-alt}. Let $E$ and $F$ be the number of edges and faces of the planar graph $G$. Note $E=c(K)$. As $K$ is non-trivial, $G$ is connected and has $V\geq 2, F\geq 2$. We have by Euler's formula $V=E-F+2\leq E = c(K)$. Also, $0\leq P \leq c(K)$. We then obtain 
\begin{equation}\label{eq:altibasicproof4}
	|\sigma(K)|\leq c(K)-1.
\end{equation}
Now \eqref{eq:altibasicproof1}, \eqref{eq:altibasicproof2} and \eqref{eq:altibasicproof4} imply $|\sigma(K)|=c(K)-1$. Replacing $K$ with its mirror image if necessary, we may assume that $\sigma(K)\leq 0$. Then by \eqref{eq:altibasicproof3}, $V-P-1=1-c(K)$ and consequently $V=2$ and $P=c(K)$. Thus $G$ is a planar graph with $2$ vertices and $c(K)$ edges. As the diagram for $K$ is minimal, $G$ has no loops. This uniquely determines $G$, and it comes from the torus knot $T_{2,2k+1}$ where $k$ is given by $(c(K)-1)/2$.
\end{proof}

The notion of $I$-basic generalizes to links with non-trivial bundle data in a straightforward manner. An admissible link $(Y,L,\omega)$ where $Y$ is an integer homology 3-sphere is {\emph{$I$-basic}} if
\begin{equation}\label{eq:ibasicnontrivbun}
	\text{rk}_\Z I^\omega(Y,L) =  \left|\chi( I^\omega(Y,L) )\right|.
\end{equation}
Theorem \ref{thm:eulerchar-withbundles} computes the right-hand side of condition \eqref{eq:ibasicnontrivbun}. In particular, if $\omega$ has an odd number of boundary points on more than two components of $L$, then $(Y,L,\omega)$ is $I$-basic if and only if $I^\omega(Y,L)$ vanishes.

\begin{example}
	Consider an admissible link $(L,\omega)$. Suppose $L$ is split by a separating $2$-sphere $S\subset S^3$, and $\omega\cap S$ is an odd number of points. Then $I^\omega(L,\omega)=0$, as there are no flat $SU(2)$ connections on a 2-sphere with an odd number of punctures, such that the holonomy around each puncture is $-1$. Thus $(L,\omega)$ is $I$-basic. $\diamd$
\end{example}

To obtain further examples of $I$-basic knots we use the exact triangles for irreducible singular instanton homology. Call any homology group {\emph{basic}} if it has rank the absolute value of its Euler characteristic. Suppose given an exact triangle of homology groups $H,H',H''$ where $H',H''$ are basic. If $|\chi(H)|=|\chi(H')|+|\chi(H'')|$ then $H$ is also basic. This principle is used in the following lemma as a strategy for generating more $I$-basic knots.

\begin{lemma}\label{lemma:ibasic}
	Let $K_0$ be a knot which is $I$-basic and $\sigma(K_0)\leqslant 0$. Let $K_m$ be obtained from $K_0$ by inserting $m$ positively oriented half-twists; see Figure \ref{fig:ibasiclemma}. Similarly, let $K_\infty$ be obtained from $K_0$ by the local replacement given in Figure \ref{fig:ibasiclemma}. If $K_\infty$ is an unknot, then the following holds.
	\begin{itemize}
		\item[{\rm{(i)}}] If $m$ is a positive even integer, then the knot $K_m$ is $I$-basic.
		\item[{\rm{(ii)}}] If $m$ is a positive odd integer, and $\omega$ is an arc connecting the two components of $K_m$, then $(K_m,\omega)$ is $I$-basic. Also, if $\det(K_m)\neq 0$, then $K_m$ is $I$-basic.
	\end{itemize} 
\end{lemma}

\begin{figure}[t]
    \centering
    \centerline{\includegraphics[scale=0.65]{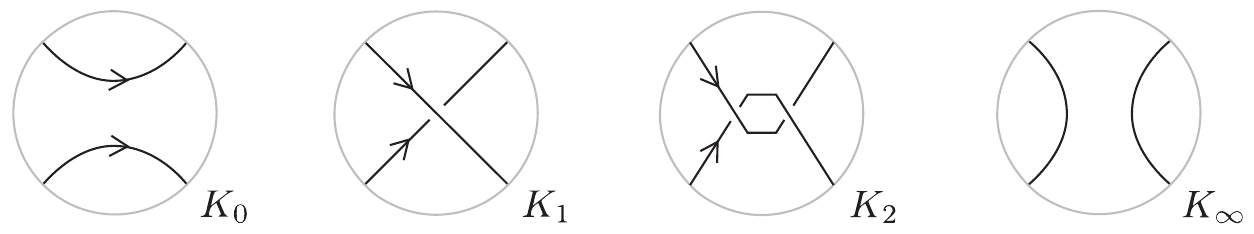}}
    \caption{{\small{The knots $K_0$, $K_2$, $K_\infty$ and the link $K_1$ which appear in Lemma \ref{lemma:ibasic}.}}}
    \label{fig:ibasiclemma}
\end{figure}

\begin{proof}
	Under the assumptions, $K_1$ is a 2-component link. We have an unoriented skein triple $(K_1,K_\infty,K_0)$. Let $\omega$ be an arc on $K_1$ that connects the two components. Then we consider the corresponding exact triangles with this non-trivial bundle data as described in Theorem \ref{thm:exacttriangles-withbundles}; see Figure \ref{fig:exacttriangles-withbundles-irr} for the irreducible homology exact triangles. Suppose we are in Case A of Figure \ref{fig:exacttriangles-withbundles-irr}. Then we have
\[
                \begin{tikzcd}[column sep=1.7ex, row sep=5ex, fill=none, /tikz/baseline=-10pt]
            & I_\fo^\omega(K_1)  \arrow[dr,"\text{deg } 1"]  & \\
             I(K_0)  \arrow[ur, "\text{deg } 0"] & & I(K_\infty) = 0  \arrow[ll,"\text{deg } 0"] 
            \end{tikzcd}
\]
Here $\fo$ is the quasi-orientation of $K_1$ that is given in Figure \ref{fig:ibasiclemma}, and $I_\fo^\omega(K_1)$ is the group $I^\omega(K_1)$ equipped with the $\Z/2$-grading determined by $\fo$ as described in Subsection \ref{subsec:absz2gr}. The mod $2$ degrees are computed from Proposition \ref{prop:degreesfornontrivbundletriangle}. We obtain a $\Z/2$-graded isomorphism $I_\fo^\omega(K_1)\cong I(K_0)$. As $K_0$ is $I$-basic, so too is $(K_1,\omega)$. Note that $\sigma(K_0)\leq 0$ implies that $I(K_0)$ and hence $I_\fo^\omega(K_1)$ is supported in odd gradings.

If instead we are in Case B of Figure \ref{fig:exacttriangles-withbundles-irr}, then we obtain a $\Z/2$-graded isomorphism $I_\fo^\omega(K_1)\cong I_+(K_0)$. As $I(K_0)$ is supported in odd gradings, we have $(\delta_2)_\ast=0$, and thus $I_+(K_0)$ is also supported in odd gradings (see the discussion in Subsection \ref{eq:irrexacttrisubsec}). Thus again $(K_1,\omega)$ is $I$-basic.

Next, we have a skein triple $(K_1,K_2,K_\infty)$. We continue to equip $K_1$ with the arc $\omega$, and look at the corresponding exact triangles. If we are in Case A of Figure \ref{fig:exacttriangles-withbundles-irr} then we have
\[
                \begin{tikzcd}[column sep=1.7ex, row sep=5ex, fill=none, /tikz/baseline=-10pt]
            & I_\fo^\omega(K_1)  \arrow[dr,"\text{deg } 0"]  & \\
            0= I(K_\infty)  \arrow[ur, "\text{deg } 1"] & & I(K_2)  \arrow[ll,"\text{deg } 0"] 
            \end{tikzcd}
\]
and the $\Z/2$-graded isomorphism $I_\fo^\omega(K_1)\cong I(K_2)$ implies $K_2$ is $I$-basic. If instead we are in Case B of Figure \ref{fig:exacttriangles-withbundles-irr}, then $I(K_\infty)$ is replaced in the above exact triangle by $I_+(K_\infty)=\Z_{(1)}$. As the group $I_\fo^\omega(K_1) $ is supported in odd gradings, we obtain that so too is $I(K_2)$. Thus in this case $K_2$ is also $I$-basic.

Inductively we obtain that $(K_m,\omega)$ for $m>0$ odd is $I$-basic, and $K_m$ for $m$ even is $I$-basic. It remains to show that $K_m$ for $m$ odd, assuming $\det(K_m)\neq 0$, is $I$-basic. It suffices to show that $K_1$ is $I$-basic, assuming $\det(K_1)\neq 0$. Indeed, the same argument applied to each $I$-basic $K_m$, where $m\geq  0$ is even, shows that $K_{m+1}$ is $I$-basic. 

Consider the unoriented skein triple $(K_1,K_\infty,K_0)$. We look at the exact triangles with trivial singular bundle data, as given in Theorem \ref{thm:irredexacttrianglesintro}. For any skein triple $(L,L',L'')$ with non-zero determinants there is a relation $\epsilon \det(L) + \epsilon' \det(L') + \epsilon'' \det(L'') =0 $ where $\epsilon,\epsilon',\epsilon''\in \{ \pm 1\}$. In the case at hand we have
\begin{align}
	\det(K_1) &= \det(K_0) \pm 1  \label{eq:ibasiclemmadets1}
\end{align}
Indeed, if $\det(K_1)+\det(K_0)=1$, then, as the determinant of a knot is positive, we must have $\det(K_1)=0$, contrary to our assumption on $K_1$. 

Relation \eqref{eq:ibasiclemmadets1} with Lemma \ref{lemma:manoz} implies that the exact triangle of Theorem \ref{thm:irredexacttrianglesintro} is either in Case I or Case II. In Case I the exact triangle and $I(K_\infty)=0$ yield $I(K_0)\cong I(K_1)$. Thus if $K_0$ is $I$-basic then so too is the link $K_1$. If instead we are in Case II, we have
\begin{equation*}
                \begin{tikzcd}[column sep=1.7ex, row sep=5ex, fill=none, /tikz/baseline=-10pt]
            & I(K_1)  \arrow[dr,  "\text{deg } 0"]  & \\
             I(K_{0})  \arrow[ur,  "\text{deg } 0"] & &   I_+(K_{\infty}) =\Z_{(1)}  \arrow[ll,  "\text{deg } 1"]
            \end{tikzcd}
\end{equation*}
As $I(K_0)$ is supported in odd gradings, we obtain that $I(K_1)$ is supported in odd gradings. Thus in this case $K_1$ is again $I$-basic.
\end{proof}

An alternative perspective is as follows. Let $K_0$ be a knot obtained from the unknot by some band move. Then the lemma says that if $K_0$ has $\sigma(K_0)\leqslant 0$ and is $I$-basic, so too are any of the links $K_{i}$ (with $\det\neq 0$) obtained by adding $i$ postive half-twists to the band.

\begin{example} Lemma \ref{lemma:ibasic} gives another proof that the torus knots $T_{2,2k+1}$ are $I$-basic. For $k\neq 0$, we also obtain that the torus links $T_{2,2k}$ are $I$-basic, and that each $(T_{2,2k},\omega)$ is $I$-basic, where $\omega$ is an arc connecting the two components.
$\diamd$
\end{example}

\begin{remark}
Motivated by Question \ref{question:lspaceknots}, it would interesting to see whether Lemma \ref{lemma:ibasic} is also true for Heegaard Floer L-space knots. $\diamd$
\end{remark}

\begin{remark}
	In Lemma \ref{lemma:ibasic}, if $I(K_0)$ is free abelian, then so too are the instanton homology groups for the $I$-basic links that are determined. Indeed, the arguments with the exact triangles work for each field $\Z/p$ ($p$ prime), and one concludes that no torsion appears. $\diamd$ \label{rmk:lemmaibasicskeintorsion}
\end{remark}

The $\Z/4$-gradings for the ranks of an $I$-basic knot are determined as follows.

\begin{prop}\label{prop:z4equidist}
	If a knot $K$ is $I$-basic, then as $\Z/4$-graded vector spaces we have
	\[
	I(K)\otimes \Q   \cong \begin{cases}\Q_{(1)}^{\lceil -\sigma(K)/4\rceil}\oplus  \Q_{(3)}^{\lfloor -\sigma(K)/4\rfloor} \qquad  & \sigma(K)\leq 0\\
	& \\
					 \Q_{(0)}^{\lfloor \sigma(K)/4\rfloor}\oplus  \Q_{(2)}^{\lceil \sigma(K)/4\rceil}    \qquad  & \sigma(K)\geq 0 \end{cases}
\]
\end{prop}

\begin{proof}
	The proof is similar to the argument of \cite[Corollary 3]{DS2}. Let $\sS'$ be the field of fractions of the ring $\sS=\Z[T^{\pm 1}]$. For any knot $K$, if we use local coefficients over the ring $\sS'$, then we have that the $\Z/4$-graded $\sS'$-vector space $I(K;\Delta_{\sS'})$ is as described in the statement of the proposition, with $\sS'$ replacing $\Q$ throughout. Furthermore, we also have
	\[
		\text{rk}_\Z I(K)_j \geq \dim_{\sS'} I(K;\Delta_{\sS'})_j
	\]
	for each $j\in \Z/4$. Combined with the condition that $K$ is I-basic determines the ranks of $I(K)_j $ as desired.
\end{proof}

\subsection{Pretzel links}

In this subsection we compute the irreducible instanton homology of the pretzel links $P_{n}:=P(-2,3,n)$ for $n>0$. In this subsection and the one to follow, for simplicity we focus entirely on the case of $I$-basic knots and links with trivial singular bundle data. We show that these are $I$-basic using Lemma \ref{lemma:ibasic}. To state the results in an explicit form we first compute the signatures of these links. 

For $n>0$ odd, $P_n$ is a knot and of course has a unique quasi-orientation. For $n>0$ even, $P_n$ is a link and has two quasi-orientations $\fo_+$ and $\fo_-$, see Figure \ref{fig:pretzel}.

\begin{figure}[t]
    \centering
    \centerline{\includegraphics[scale=0.7]{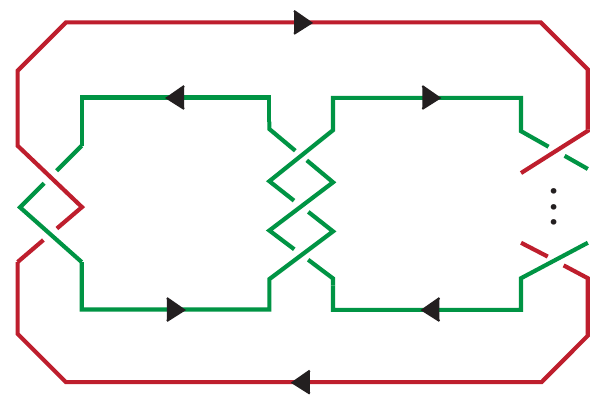}}
    \caption{{\small{The pretzel link $P_n=P(-2,3,n)$ with $n>0$, $n$ even. There are $n$ crossings in the right-most vertical part of the link. The link shown is $P_n$ with the quasiorientation $\fo_+$. To obtain the other quasi-orientation $\fo_-$, reverse the orientation of either of the two components.}}}
    \label{fig:pretzel}
\end{figure}

\begin{lemma}\label{lemma:pretzelsignatures}
	For odd $n>0$ the knot $P_n=P(-2,3,n)$ has signature
\[
	\sigma(P_n) = \begin{cases} -n-1 & n\geq 7 \\ -n-3 & n\in \{1,3,5\} \end{cases}
\]
For even $n>0$ the link $P_n=P(-2,3,n)$ has signatures
\begin{align*}
	\sigma(P_n,\fo_+) &= \begin{cases} -n-1 & n\geq 8 \\
									-n-2 =-8 \;\;\;\;\; & n=6 \\
									-n-3 & n\in \{2,4\}\end{cases}\\
	\sigma(P_n,\fo_-) &= \begin{cases} +1 \;\;\;\;\; & n\geq 8 \\ 0 & n=6 \\ -1 & n\in\{2,4\}  \end{cases}
\end{align*}
The determinant of the Pretzel link $P_n$ for all $n>0$ is given as follows:
\[
	\det(P_n) = |n-6|
\]
\end{lemma}

\begin{proof}
The knot $P_1$ is the $(2,5)$ torus knot, and $P_0$ is a trefoil connect summed with a Hopf link. Using the oriented skein relation for the Alexander polynomial applied to the oriented link as given in the middle of Figure \ref{fig:pretzel} we inductively compute:
\[
	\Delta_{P_n}(t) = \sum_{\substack{-3\leqslant i<n\\i \text{ odd}, i\neq 1}}(-1)^{\frac{i+1}{2}}(t^{\frac{n-i}{2}} + t^{-\frac{n-i}{2}} ) + \epsilon
\]
where $\epsilon=1$ if $n$ is odd and $\epsilon=0$ if $n$ is even. Evaluating $|\Delta_{P_n}(-1)|$ gives the determinants as claimed. The Murasugi condition $|\sigma(K)|+1\equiv \det(K) \pmod{4}$ for a knot $K$ along with $\sigma(L_+)-\sigma(L_-)\in \{0,-2\}$ for two non-zero determinant links related by a positive crossing change can then be used to compute the signatures, for $n$ odd, inductively from the base cases. For $n$ even, as $P_n$ with quasi-orientation $\fo_+$ and $P_{n-1}$ are related by an oriented band move, we have $|\sigma(P_n,\fo_+)-\sigma(P_{n-1})|\in \{0,1\}$, and this computes $\sigma(P_n,\fo_+)$ from the knot cases; for $\sigma(P_n,\fo_-)$ then use \eqref{eq:signaturechangeqo}. Alternatively, a $2\times 2$ Goeritz matrix and the Gordon--Litherland formula can be used to compute these quantities. 
\end{proof}

\begin{prop}\label{lemma:-237pretzel}
	For the pretzel link $P(-2,3,n)$ with $n>0$ and $n\neq 6$ we have:
\[
	I(P(-2,3,n)) \cong \begin{cases}  \Z_{(1)}^{\lceil \frac{n+3}{4}\rceil} \oplus  \Z_{(3)}^{\lfloor \frac{n+3}{4} \rfloor} & 1\leq n \leq 5 \\ & \\ \Z_{(1)}^{\lceil \frac{n+1}{4}\rceil}\oplus \Z_{(3)}^{\lfloor \frac{n+1}{4}\rfloor} & n\geqslant 7 \end{cases}
\]
In particular, these links are $I$-basic.
\end{prop}

\begin{proof}
The knot $P_1=T_{2,5}$ is $I$-basic. That $P_n$ for $n>1$ and $n\neq 6$ are $I$-basic now follows from Lemma \ref{lemma:ibasic}, using the right vertical strand of $n$-half twists. (The case $n=6$ is omitted, as $\det(P_6)=0$.) The rank of each group is determined by the signatures as computed in Lemma \ref{lemma:pretzelsignatures}, and the $\Z/4$-gradings follow from Proposition \ref{prop:z4equidist}. That the groups are free abelian follows from Remark \ref{rmk:lemmaibasicskeintorsion}. 
\end{proof}

We note that $P_1=T_{2,5}$, $P_3=T_{3,4}$, $P_5=T_{3,5}$ are torus knots and thus are $I$-basic from earlier remarks. Furthermore, that $P_7=P(-2,3,7)$ is $I$-basic can be seen by a more direct computation: the irreducible traceless $SU(2)$ character variety consists of $4=|\sigma(P_7)/2|$ non-degenerate points, see for example \cite[\S 7.2]{PS}. Such a direct computation does not work for the knot $P_n$ in the cases when $n>7$.

\subsection{Twisted torus knots}
Let $p,q$ be relatively prime positive integers. We denote by $T(p,q;a,b)$ the twisted torus knot obtained from taking $a>0$ parallel strands in a standard braid presentation of $T(p,q)$ and applying $b\in \Z$ full twists. We allow $b$ to be a half-integer as well, indicating that half twists are inserted, and in this case we obtain a twisted torus link. See Figure \ref{fig:twistedtorus52} for a summary of our conventions in the case $a=2$. Note that we also write $T(p,q;2,\infty)$ for the result of a vertical resolution of the 2 horizontal strands; this link is relevant when applying the exact triangles in this setting. 

\begin{figure}[t]
    \centering
    \centerline{\includegraphics[scale=0.85]{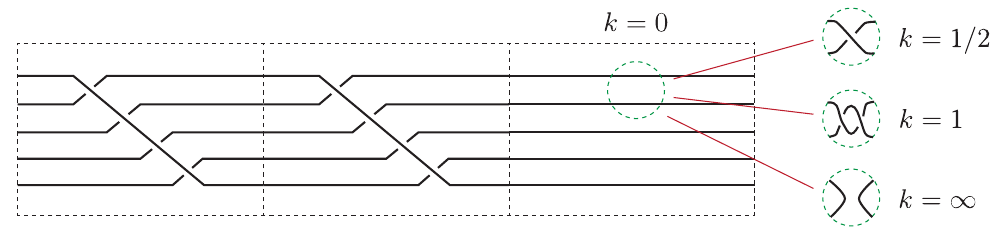}}
    \caption{{\small{The twisted torus link $T(5,2;2,k)$ for various $k$, represented as a braid closure. When $k$ is an integer (resp. half-integer), the result is a knot (resp. 2-component link).}}}
    \label{fig:twistedtorus52}
\end{figure}

\begin{prop}\label{prop:twistedtorusknots}
For $p>1$ an odd integer, $m>0$ any integer, and $k\in \frac{1}{2}\Z_{\geq 0}$, the twisted torus link $T(p,pm \pm 2 ; 2,k)$ is $I$-basic whenever it has $\det \neq 0$.
\end{prop}

\begin{proof}
First, we claim that $T(p,pm\pm 2;2,\infty)$ is an unknot. The case in which $p=5$, $m=1$ and $\pm = -$ is established in Figure \ref{fig:twistedtorusunknot}. It is easy to see that the same isotopy works regardless of $p$, the number of strands. Varying $m\in \Z$ amounts to adding full twists, and the isotopy extends to these cases; this is because the isotopy can be done so that its support is disjoint from the braid axis. The case with $\pm =+$ is similar. The result then follows from Lemma \ref{lemma:ibasic} applied to $K_0=T_{p,pm+ \epsilon 2}$, which is $I$-basic and has negative signature.
\end{proof}

We remind the reader that the technical condition $\det\neq 0$ is included only because the invariant $I(L)$ has not yet been defined for links with zero determinant. In any case, the determinant of the $T(p,q;2,k)$ twisted torus link is straightforward to compute. Write $r$ for the unique integer such that $0<r<p$ and $r\equiv -q^{-1}\pmod{p}$.

\begin{figure}[t]
    \centering
    \centerline{\includegraphics[scale=1.0]{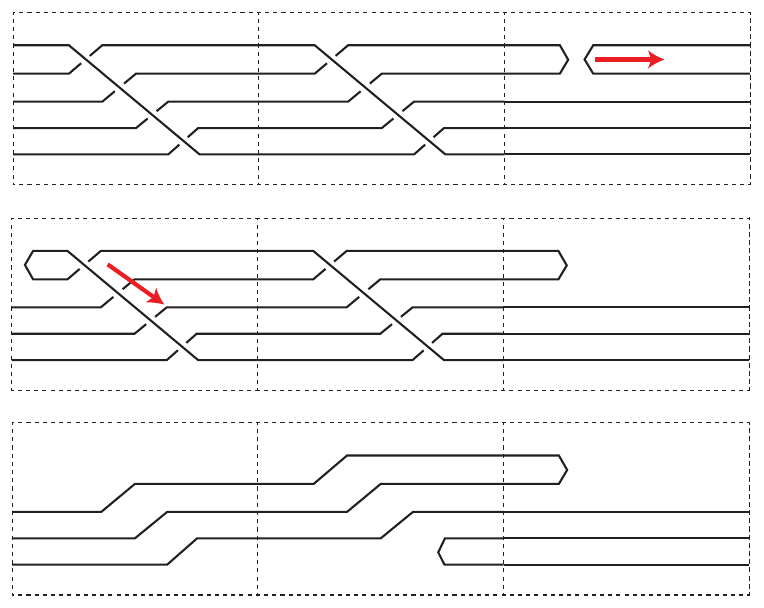}}
    \caption{{\small{Isotoping $T(5,2;2,\infty)$ into an unknot.}}}
    \label{fig:twistedtorusunknot}
\end{figure}

\begin{prop}
	Let $p,q$ be relatively prime positive integers. The Alexander polynomial of the twisted torus link $T(p,q;2,k)$ for $k\in \frac{1}{2}\Z$ is given as follows:
\[
	\Delta_{T(p,q;2,k)}(t) = \frac{t^{-\frac{(p-1)(q-1)}{2}-k}(t-1)}{(t^p-1)(t^q-1)}\left( t^{pq+2k}-1 - \frac{(t^{2k}-1)}{(t+1)}\left(t^{rq+1}+ t^{(p-r)q}\right)\right)
\]
\end{prop}

\begin{proof}
	The computation for knots, i.e. when $k$ is even, is done by Morton \cite{morton} (we have normalized his results to fit our conventions), and the case for links follows from the case for knots via the oriented skein relation for the Alexander polynomial.
\end{proof}

Evaluating (or possibly taking the limit of) the above polynomial at $t=-1$ and taking the absolute value yields the following computation.

\begin{cor}\label{cor:twistedtorusdet}
The determinant of $T(p,q;2,k)$ for $k\in \frac{1}{2}\Z$ is as follows:
\begin{equation*}
	\det(T(p,q;2,k)) = \begin{cases} |2k+(-1)^r| \qquad  & p,q \text{ odd}\\
											|q(1-2k) + \frac{4k}{p}(1+rq)| \qquad \quad & p \text{ even}, q \text{ odd}\\ 
											|(2k+1)p - 4kr | & p \text{ odd}, q \text{ even}\\ 
										 \end{cases}
\end{equation*}
\end{cor}

Finally, let us describe the irreducible instanton homology of the particular family of twisted torus knots $T(3,3n+2;2,k)$, where $n,k>0$. To this end, we note that well-known formulas for the signatures of torus knots \cite{litherland} give
\begin{equation}\label{eq:signtwistedtorus3}
	\sigma(T(3,3n+2;2,0)) = \begin{cases} -4n-2 \;\; & n \text{ odd} \\ -4n-4 \;\; & n \text{ even} \end{cases}
\end{equation}
Next, Corollary \ref{cor:twistedtorusdet} in this case yields
\begin{equation}\label{eq:dettwistedtorus3}
	\det(T(3,3n+2;2,k)) = \begin{cases} |2k-1| \;\; & n \text{ odd} \\ |2k+3| \;\; & n \text{ even} \end{cases}
\end{equation}
Using \eqref{eq:signtwistedtorus3} and \eqref{eq:signtwistedtorus3} and the constraints $|\sigma(K)|+1\equiv \det(K) \pmod{4}$ for any knot $K$ and $\sigma(K_+)-\sigma(K_-)\in \{0,-2\}$ for knots differing by a positive crossing, we obtain
\[
	\sigma(T(3,3n+2;2,k)) = -4n -2 - 2k
\]
for any integer $k>0$. Finally, using Proposition \ref{prop:twistedtorusknots}, and following an argument similar to that of Proposition \ref{lemma:-237pretzel}, we obtain the following.

\begin{cor}\label{cor:3twistedtoruscomp}
	For the twisted torus knot $T(3,3n+2;2,k)$ where $n,k\in \Z_{>0}$ we have:
\[
	I(T(3,3n+2;2,k)) \cong   \Z_{(1)}^{n+\lceil \frac{k+1}{2}\rceil} \oplus  \Z_{(3)}^{n+\lfloor \frac{k+1}{2}\rfloor}
\]
\end{cor}

\vspace{.1cm}

\noindent We note that $T(3,5;2,1)$ is in fact isotopic to the pretzel knot $P(-2,3,7)$, and this gives an overlapping case in Corollaries \ref{lemma:-237pretzel} and \ref{cor:3twistedtoruscomp}.

\newpage

\addcontentsline{toc}{section}{References}

\bibliography{references}
\bibliographystyle{alpha.bst}

\Addresses

\end{document}